\documentclass[12pt]{amsart}
\usepackage{amsfonts,amssymb,amsthm,eucal,amsmath}
\usepackage{amscd}
\usepackage{graphicx}
\usepackage[T1]{fontenc}
\usepackage{latexsym,url}
\usepackage{pinlabel}
\usepackage{microtype}
\usepackage{hyperref}
\usepackage[inner=20mm, outer=20mm, textheight=225mm]{geometry}
\usepackage{color}
\usepackage{tikz}
\usepackage{tikz-cd}
\usetikzlibrary{decorations.pathmorphing}
\usepackage{subfig}
\usepackage{caption}
\theoremstyle{plain} 
\newtheorem{thm}{Theorem}[section]
\newtheorem{lemma}[thm]{Lemma}
\newtheorem{prop}[thm]{Proposition}

\theoremstyle{definition}

\newtheorem{rmk}[thm]{Remark}
\newcommand{\bdry}{\partial}
\newcommand{\cT}{\mathcal{T}}

\newcommand{\tri}{\bigtriangleup}

\DeclareMathOperator\Lk{Lk}

\newcommand{\co}{\colon\thinspace} 

\begin{document}

\title{Traversing three-manifold triangulations and spines}
\author{J. Hyam Rubinstein, Henry Segerman and Stephan Tillmann}
\address{School of Mathematics and Statistics, The University of Melbourne, VIC 3010, Australia} 
\address{Department of Mathematics, Oklahoma State University, Stillwater,  Oklahoma, OK 74078, USA}
\address{School of Mathematics and Statistics, The University of Sydney, NSW 2006, Australia} 
\email{joachim@unimelb.edu.au} 
\email{segerman@math.okstate.edu}
\email{stephan.tillmann@sydney.edu.au} 

\begin{abstract}
A celebrated result concerning triangulations of a given closed 3--manifold is that any two triangulations with the same number of vertices are connected by a sequence of so-called 2-3 and 3-2 moves. A similar result is known for ideal triangulations of topologically finite non-compact 3--manifolds. These results build on classical work that goes back to Alexander, Newman, Moise, and Pachner. The key special case of 1--vertex triangulations of closed 3--manifolds was independently proven by Matveev and Piergallini. The general result for closed 3--manifolds can be found in work of Benedetti and Petronio, and Amendola gives a proof for topologically finite non-compact 3--manifolds. These results (and their proofs) are phrased in the dual language of spines.

The purpose of this note is threefold. We wish to popularise Amendola's result; we give a combined proof for both closed and non-compact manifolds that emphasises the dual viewpoints of triangulations and spines; and we give a proof replacing a key general position argument due to Matveev with a more combinatorial argument inspired by the theory of subdivisions. \end{abstract}

\subjclass[2000]{57Q15, 57Q25, 52B70, 57N10, 57M50, 57M27}


\keywords{triangulation, 3-manifold, surface, combinatorial topology, triangulated manifolds, bistellar moves, standard spines}

\maketitle

\section{Introduction}

Suppose you have a finite number of triangles. If you identify edges in pairs such that no edge remains unglued, then the resulting identification space looks locally like a plane and one obtains a closed surface, a two-dimensional manifold without boundary. The classification theorem for surfaces, which has its roots in work of Camille Jordan and August M\"obius in the 1860s, states that each closed surface is topologically equivalent to a sphere with some number of handles or crosscaps. For example, the torus is a sphere with a single handle, and the projective plane is a sphere with a single crosscap. A modern proof of the classification theorem is due to John Conway and presented by Francis and Weeks~\cite{Francis-conway-1999}. 

Now suppose you have a finite number of tetrahedra and identify triangular faces in pairs such that no face remains unglued. In this case, the resulting identification space is again closed (it has no boundary). The resulting space looks everywhere like three-dimensional euclidean space \emph{except} possibly at the vertices and at mid-points of edges. This can be seen as follows. 

The local picture near a vertex in the identification space can be understood by tracing the small triangles in the tetrahedra cutting off the vertices that are identified under the face gluings (see Figure~\ref{vertex_link}). If these small triangles globally glue up to a sphere, then in the identification space the sphere formed by the triangles bounds a ball and the vertex in the identification space looks like the centre of a ball. In this case, the vertex is called \emph{material}. However, if the triangles glue up to a sphere with at least one handle or crosscap, then the space near this vertex looks like a cone on this surface. Such a vertex is called \emph{ideal}.

\begin{figure}[htbp]
\centering
\includegraphics[width=0.4\textwidth]{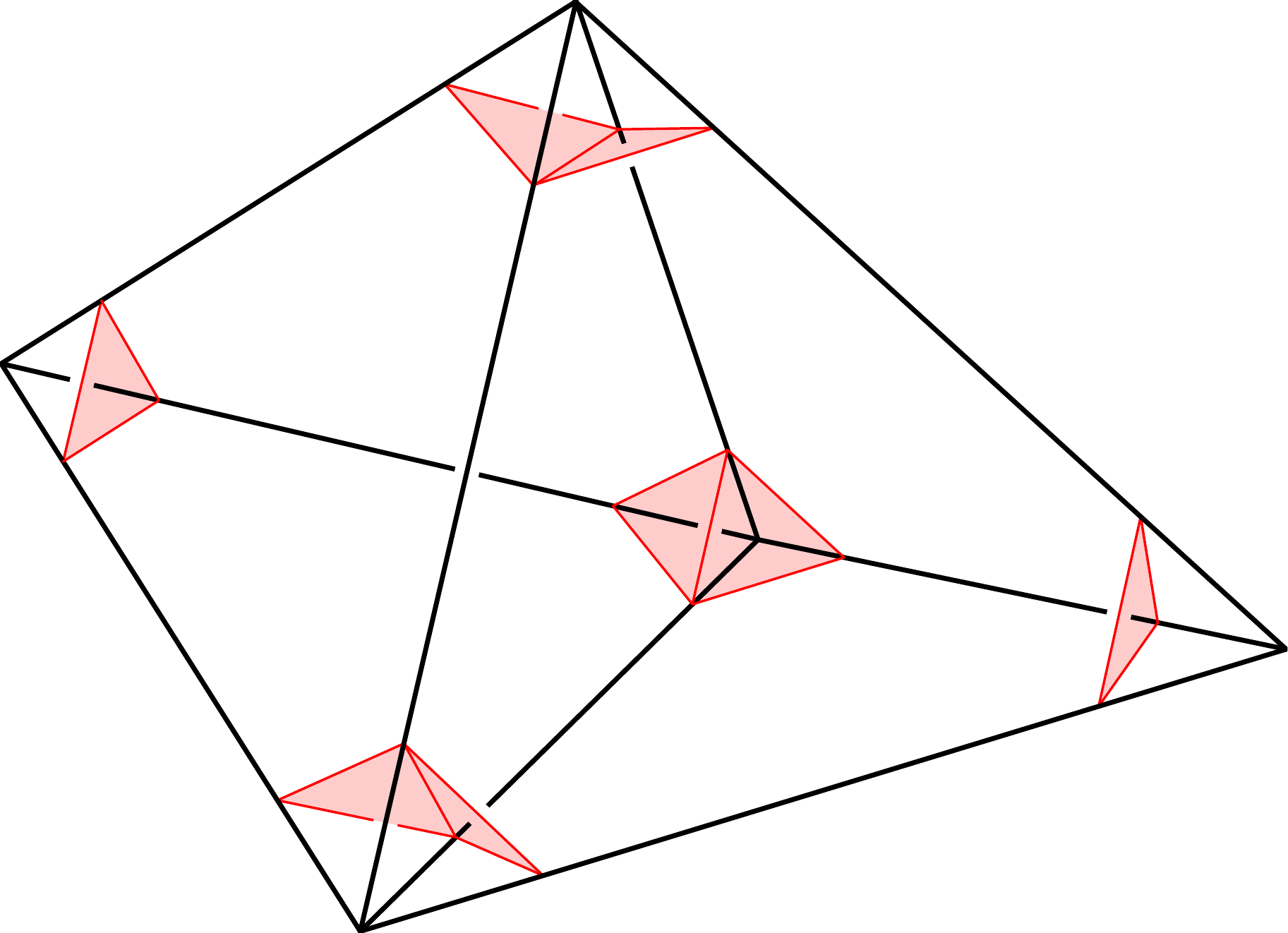}
\caption{Small triangles near the vertices of the tetrahedra glue up to form a surface.} 
\label{vertex_link}
\end{figure}

Only the cone on a sphere is a ball; the cone on a different surface is more difficult to visualise. In general one can imagine this cone as obtained as follows. First take a product of the surface $S$ with an interval $I$. The resulting three-dimensional space has two boundary components, each a copy of $S$. One then collapses one of these boundary components to a point, giving the cone on the surface. If $S$ is orientable, one can picture the product $S\times I$ in three-dimensional euclidean space. If $S$ is a sphere, it is now also easy to see that after collapsing one of the boundary components, one obtains a ball. 

The situation at the midpoints of edges arises from the issue that an edge may be identified with itself but in the opposite direction. In this case, a small neighbourhood of the midpoint of such an edge is bounded by a projective plane (see Figure~\ref{dr_projective_plane_midpoint}). If this happens, we can turn the mid-point into an ideal vertex by subdividing the tetrahedra. 

\begin{figure}[htbp]
\centering
\subfloat[A projective plane bounds the neighbourhood of the midpoint of an edge.]{
\includegraphics[width=0.4\textwidth]{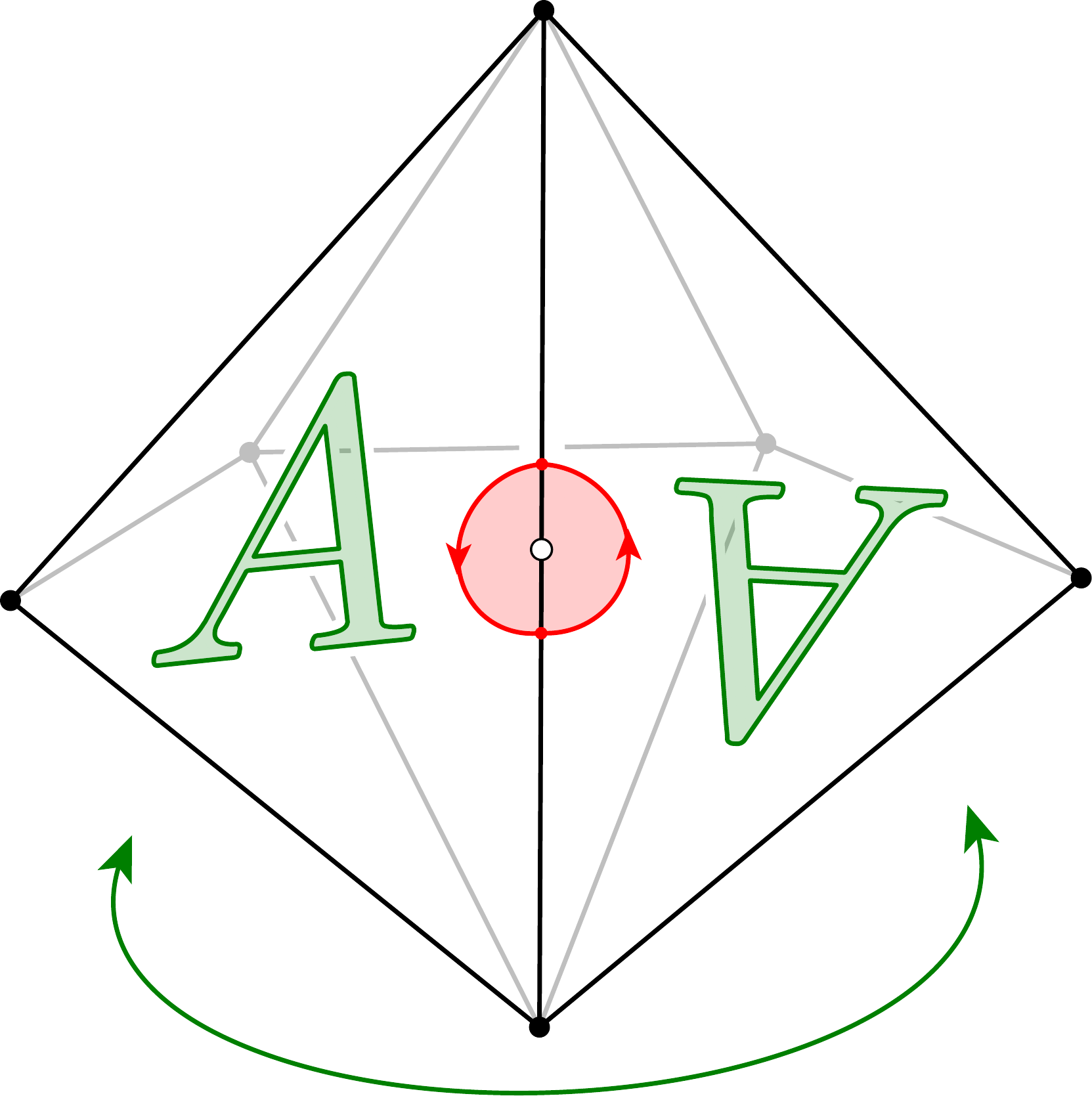}}
\quad 
\subfloat[After subdivision we get an ideal vertex whose small triangles glue up to produce the projective plane.]{
\includegraphics[width=0.4\textwidth]{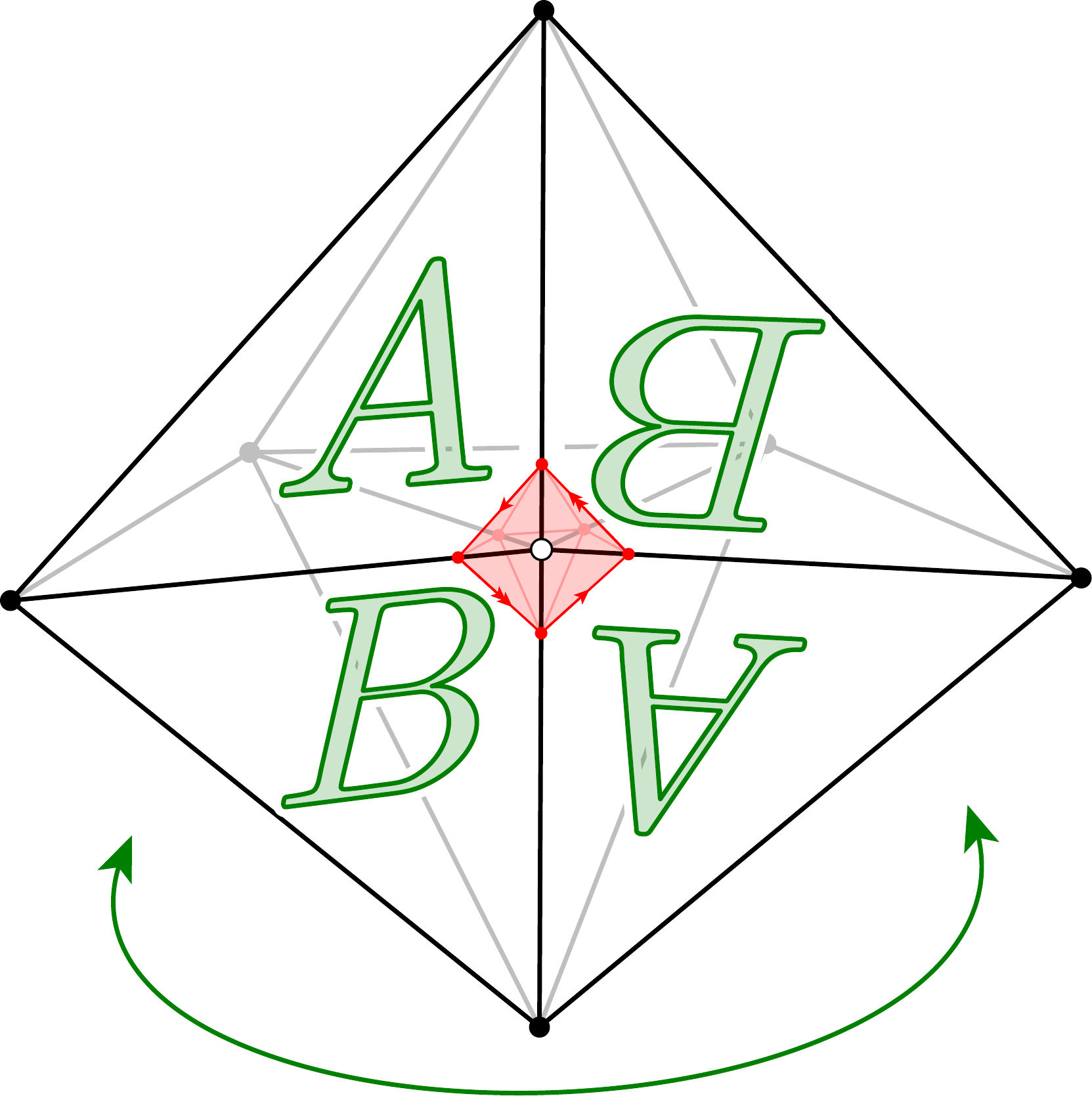}}
\caption{} 
\label{dr_projective_plane_midpoint}
\end{figure}

Hence we will assume that when we identify faces of a collection of tetrahedra, then the resulting space looks \emph{everywhere} like three-dimensional euclidean space \emph{except} possibly at the vertices. If all vertices are material, we have a \emph{three-dimensional manifold} and if there is at least one ideal vertex, we have a \emph{three-dimensional pseudo-manifold}. In each case, the space also comes with a \emph{triangulation}, namely the collection of tetrahedra and face pairings with the property that the only non-manifold points are at the vertices.

\medskip

\textbf{Why think about triangulations?}
As Bill Thurston~\cite{Thurston-three-1997} puts it, ``Manifolds are around us in many guises.'' We will not give the technical definition of a manifold here.
A key result of Moise~\cite{Moise-affine-1952} states that every three-dimensional manifold can be triangulated in the way described above. If one is given a compact three-dimensional manifold with non-empty boundary and one cones each boundary component to a point (one cone point for each boundary component), then it is a closed three-dimensional pseudo-manifold. A relative version of Moise's theorem allows the compact three-dimensional manifold with boundary to be triangulated, and coning each triangulation of a boundary component to a point gives a triangulation of the pseudo-manifold. In fact, there is a more general definition of three-dimensional pseudo-manifolds than that considered in this paper, where not only the vertices but also points at edges have neighbourhoods that don't look like balls. Such spaces arise naturally in algebraic geometry. Banagl and Friedman~\cite{banagl-friedman} showed that these more general three-dimensional pseudo-manifolds can also be triangulated.

This is in contrast with higher dimensions. There is a four-dimensional manifold, called Freedman's $E_8$ manifold, which was shown in the mid 1980s through work of Casson~\cite{Akbulut-Casson-1990} to have no triangulation. The existence of such manifolds in all dimensions greater or equal to five was recently shown by Manolescu~\cite{Manolescu-Pin-2016}. An overview of the so-called triangulation conjecture can be found in \cite{Manolescu-triangulations-2014, Manolescu-lectures-2016}.

\medskip

The interest in triangulations lies in their combinatorial framework that allows the study of geometric and topological properties of a manifold both in theory and in practice. In three-dimensional geometry and topology this is completely general due to Moise's result, and hence gives a framework for proving important theoretical results~\cite{Kneser-geschlossene-1929, Haken-theorie-1961, Haken-some-1968, Jaco-minimal-1988, Jaco-equivariant-1989, Jaco-0-efficient-2003}, 
design algorithms for decisions problems~\cite{Haken-verfahren-1961, Schubert-bestimmung-1961, Haken-uber-1962, Jaco-algorithm-1984, Thompson-thin-1994, Rubinstein-algorithm-1995, Li-Heegaard-2011, Dunfield-spanning-2011}, 
and allows the analysis of the computational complexity of such decision problems~\cite{Hass-computational-1999, Agol-computational-2006, Coward-reidemeister-2014, Burton-courcelle-2017, Schleimer-NP-2011}. 
Triangulations also carry the geometry of a three-manifold~\cite{Neumann-realizing-2011, Neumann-volumes-1985, Neumann-arithmetic-1992, Frigerio-constructing-2004}
 and allow rigorous computation of geometric invariants~\cite{Neumann-bloch-1999, Coulson-computing-2000, Hoffman-verified-2016}, such as the volume or the length of a shortest curve that cannot be contracted to a point. In addition, there are also numerous topological or geometric invariants that can be computed from triangulations, see, for example, \cite{Baseilhac-analytic-2015, Kashaev-matrix-2011, Carbone-wigner-2000, Kabaya-pre-2007}.
Many of these algorithms are implemented in publicly available software packages~\cite{snappy, regina}.

\medskip

Turaev and Viro~\cite{Turaev-state-1992, Turaev-quantum-1994} showed that this can also be turned around. They \emph{define} powerful invariants of closed 3--manifolds \emph{using} triangulations, and there is currently much activity in trying to link these invariants to known topological or geometric invariants, see \cite{Murakami-current-2013, Dimofte-quantum-2011} and the references therein.

To show that a quantity computed using a triangulation is indeed independent of the triangulation requires means to compare the quantities associated to different triangulations. A similar situation arises in \emph{knot theory}. In this case, there are quantities, properties or mathematical objects that can be computed from a diagram, which is a particular way of drawing the knot in the plane. These are computed using a prescribed set of rules. Reidemeister showed that any two diagrams of the same knot are related by a particular set of moves (see Figure~\ref{dr_Reidemeiser}). So if the result computed before applying any one of Reidemeister's moves is the same as the result computed after performing that move, then one has shown that the result is independent of the particular diagram and hence an invariant of the knot. Examples of this are the property of tricolourability, the Jones polynomial or the fundamental group. An excellent, elementary introduction to these ideas is given by Adams~\cite{Adams-knot-2004}.

\begin{figure}[htbp]
\centering
\includegraphics[width=0.7\textwidth]{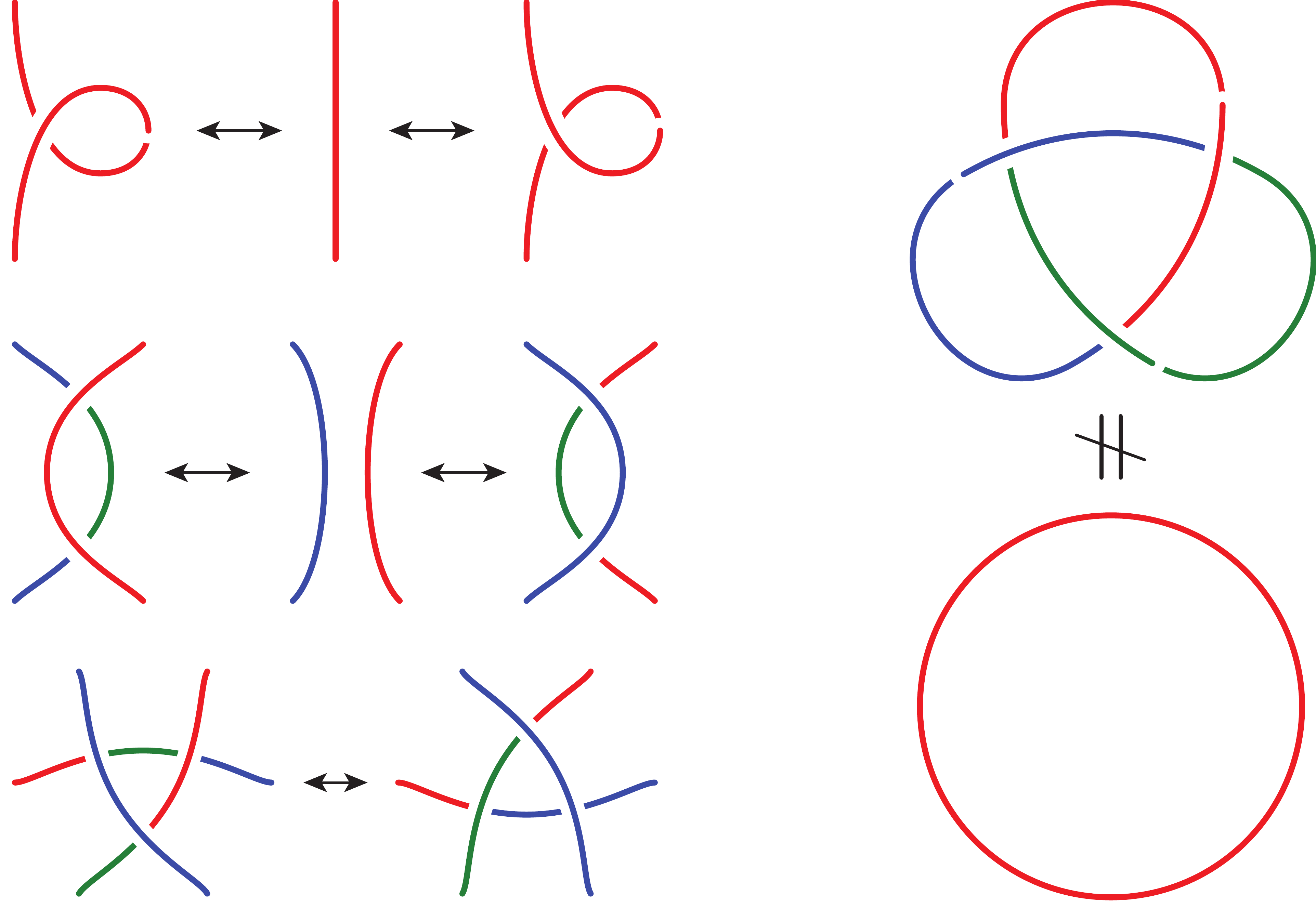}
\caption{Reidemeister moves. The colouring indicates that they preserves tricolourability. In particular, the trefoil knot cannot be deformed into the unknot.} 
\label{dr_Reidemeiser}
\end{figure}

The invariants of Turaev and Viro work similarly. 
Suitable moves that change one triangulation to another are the \emph{bistellar moves} that have been popularised by Pachner~\cite{Pachner-bistellare-1978}.
The bistellar moves play the role of Reidemeister's moves in this setting.
We drop down by one dimension to introduce the bistellar moves and highlight some subtle points.

\begin{figure}[htbp]
\centering
\includegraphics[width=0.8\textwidth]{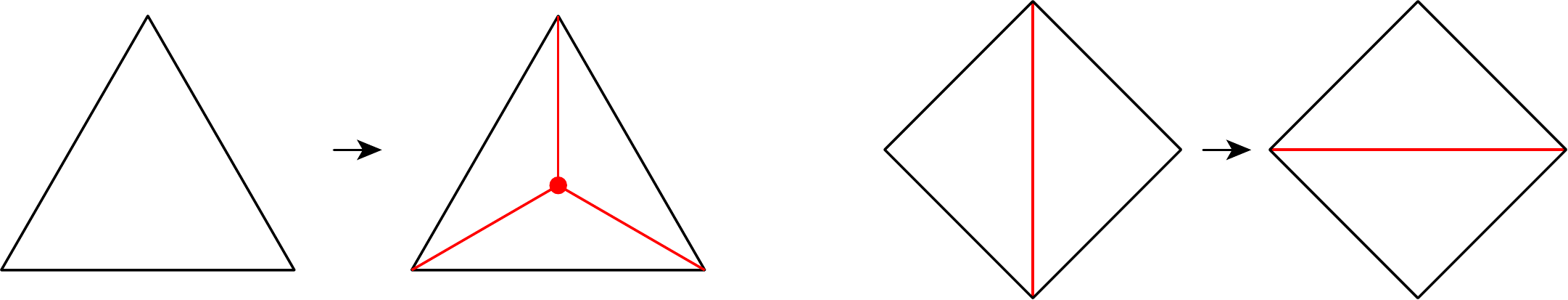}
\caption{The 1-3 and 2-2 moves.} 
\label{1-3_and_2-2_moves}
\end{figure}

\textbf{An interlude about surfaces.} For surfaces, the corresponding moves are shown in Figure~\ref{1-3_and_2-2_moves}. The 1-3 move introduces a vertex at the centre of a triangle and connects this to the three vertices of the triangle, thus dividing it into three triangles. Hence the name 1-3 move. Its inverse is called a 3-1 move, and can be performed on any vertex at which three different triangles meet. The second move is the 2-2 move. Two triangles sharing an edge can be thought of as a quadrilateral with a diagonal drawn in. The 2-2 move changes this diagonal to the other diagonal of the quadrilateral. Note that the 1-3 and 3-1 moves change the number of vertices, edges and triangles, whilst the 2-2 move leaves all of these quantities unchanged. There is a conceptually nice way of thinking about these moves. The boundary of a tetrahedron has a natural triangulation with four triangles. The above moves can be thought of as swapping any subset of triangles on the boundary of the tetrahedron with their complementary triangles. The three bistellar moves we have just described can be used to change a given triangulation into a different triangulation. 

However, in the setting we are interested in, one is given two different triangulations of the same space. What does this mean? One way to interpret this is that our topological space, in this case a surface $S$, comes with two sets of triangles drawn on it, and we'd like to apply the moves to change one set to another. The immediate idea is to superimpose both triangulations. This introduces new vertices at the places where edges of the two triangulations meet. The resulting cells may not all be triangles, but can be easily be subdivided into triangles. So the result is a third triangulation that is a \emph{common refinement} of both triangulations. It is therefore enough to show that one can transform a triangulation to an arbitrary refinement using the given set of moves. There is a hidden assumption in this argument, though, namely that the edges of both triangulations only meet in finitely many points.
Think of a small disc on the surface as parameterised as the euclidean plane. It might be that an edge in one triangulation looks like $y=0$ whilst an edge in the other triangulation looks like $y = x\ \text{sin}(\frac{1}{x})$ for $x>0$ and includes the origin as a vertex. These edges meet in infinitely many points! 
{We appear to have gone down the rabbit hole further: given a triangulation of a surface, one can define a smooth structure on the surface that makes the triangulation smooth. A nice account of this is given by Thurston~\cite[\S3.10]{Thurston-three-1997}. So now each triangulation gives us a smooth structure on the surface.}
A theorem of Whitehead says that one can indeed perturb the triangulations in order to arrange the finite intersection property. In applying this theorem, there is yet again a hidden assumption: that one of the triangulations may be curvy but at least sufficiently nice with respect to the smooth structure induced on the surface by the other triangulation. To concoct something unpleasant, imagine that an edge with respect to one triangulation looks like a fractal curve with respect to the smooth structure induced by the other. 
That both triangulations are piecewise smooth with respect to the same structure can also be arranged by a perturbation. This was shown by Epstein~\cite{Epstein-curves-1966} in 1966. Hatcher~\cite{Hatcher-kirby-2013} recently gave a very nice proof of this fact using the so-called Kirby torus trick. So we may (and will) assume that the triangulations have the finite intersection property after performing a small perturbation.

How do we transform a given triangulation $\cT_0$ of the closed surface $S$ to an arbitrary subdivision $\cT_1$ of $\cT_0$ only using 1-3, 2-2 and 3-1 moves? First apply a small isotopy of the triangulation $\cT_0$ that keeps all vertices fixed and moves each edge to a position where it does not meet any of the additional vertices of $\cT_1.$ Then repeatedly apply 1-3 moves to the isotopic copy of $\cT_0$ to introduce vertices precisely where the vertices of $\cT_1$ are. Denote the resulting triangulation $\cT'_0.$
We now have two triangulations of the surface that share the same set of vertices. It is now a pleasant exercise to show that there is a sequence of 2-2 moves that transforms $\cT'_0$ into a triangulation that is isotopic to $\cT_1$, again keeping the vertices fixed. 
So we have argued that any triangulation of a surface can be transformed to a subdivision by first applying an isotopy fixing the vertices of the initial triangulation, then applying a sequence of 1-3 moves followed by a sequence of 2-2 moves, and completing with an isotopy fixing the vertices of the target triangulation. The isotopies are merely for cosmetic reasons: they put edges neatly on top of each other, rather than roughly in the right place. 
In conclusion, given two arbitrary triangulations, we can transform one to the other by applying an isotopy, a sequence of 1-3 moves, a sequence of 2-2 moves, an isotopy, a sequence of 2-2 moves, a sequence of 3-1 moves, and finishing with an isotopy. Topological proofs of the existence of the sequence of 2-2 moves can be found in work of Mosher~\cite{Mosher-tiling-1988} and Hatcher~\cite{Hatcher-triangulations-1991}. A geometric proof using the Epstein-Penner decomposition~\cite{Epstein-euclidean-1988} transforms each triangulation with the same number of vertices via 2-2 moves to a \emph{canonical triangulation} by first choosing a geometric structure on the surface. This can be found in \cite{Tillmann-algorithms-2016}.

\begin{figure}[htbp]
\centering
\subfloat[]{
\includegraphics[width=0.3\textwidth]{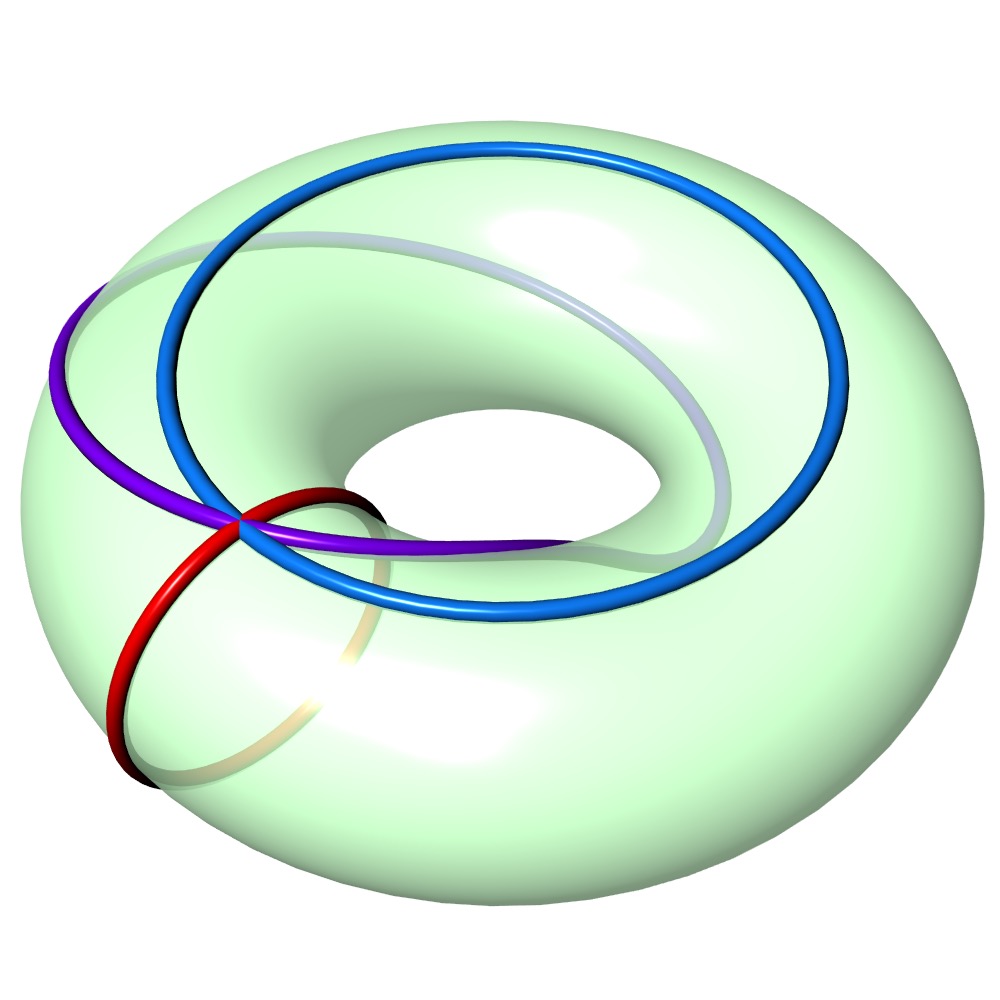}
}
\subfloat[]{
\includegraphics[width=0.3\textwidth]{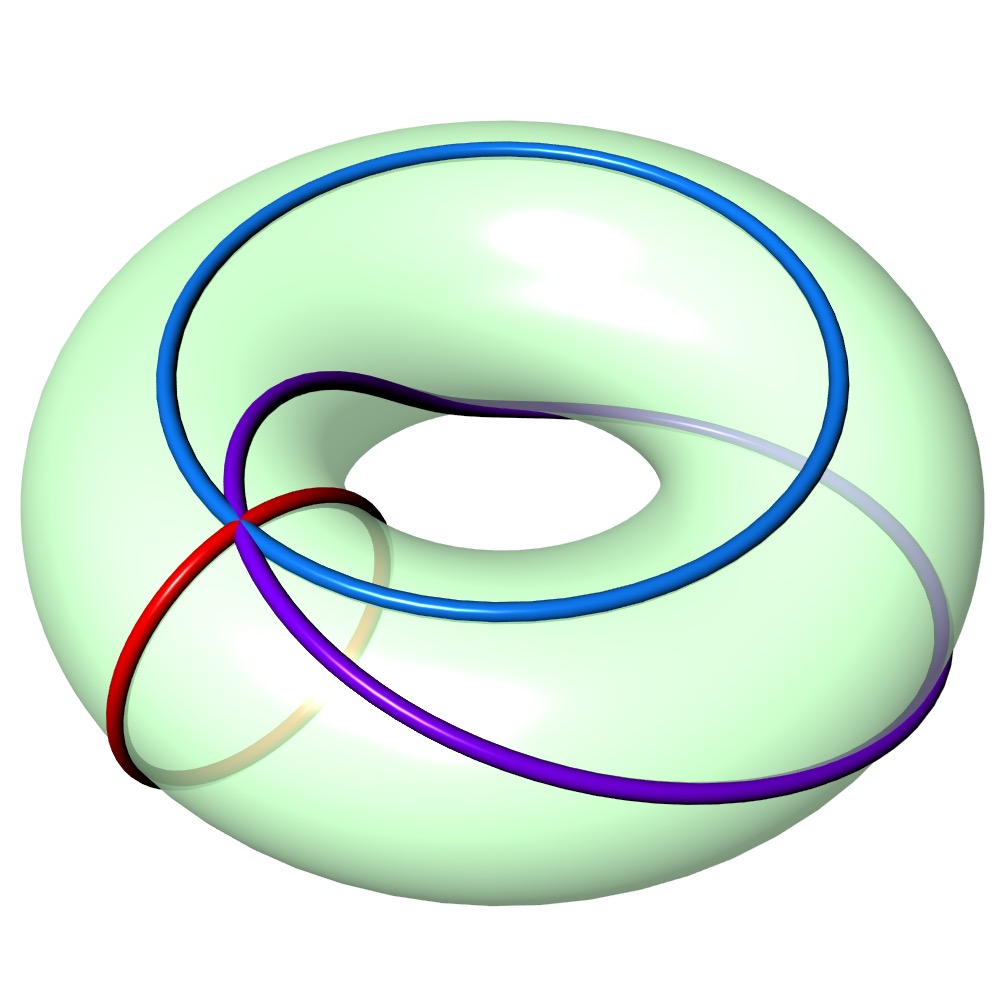}
}
\subfloat[]{
\includegraphics[width=0.3\textwidth]{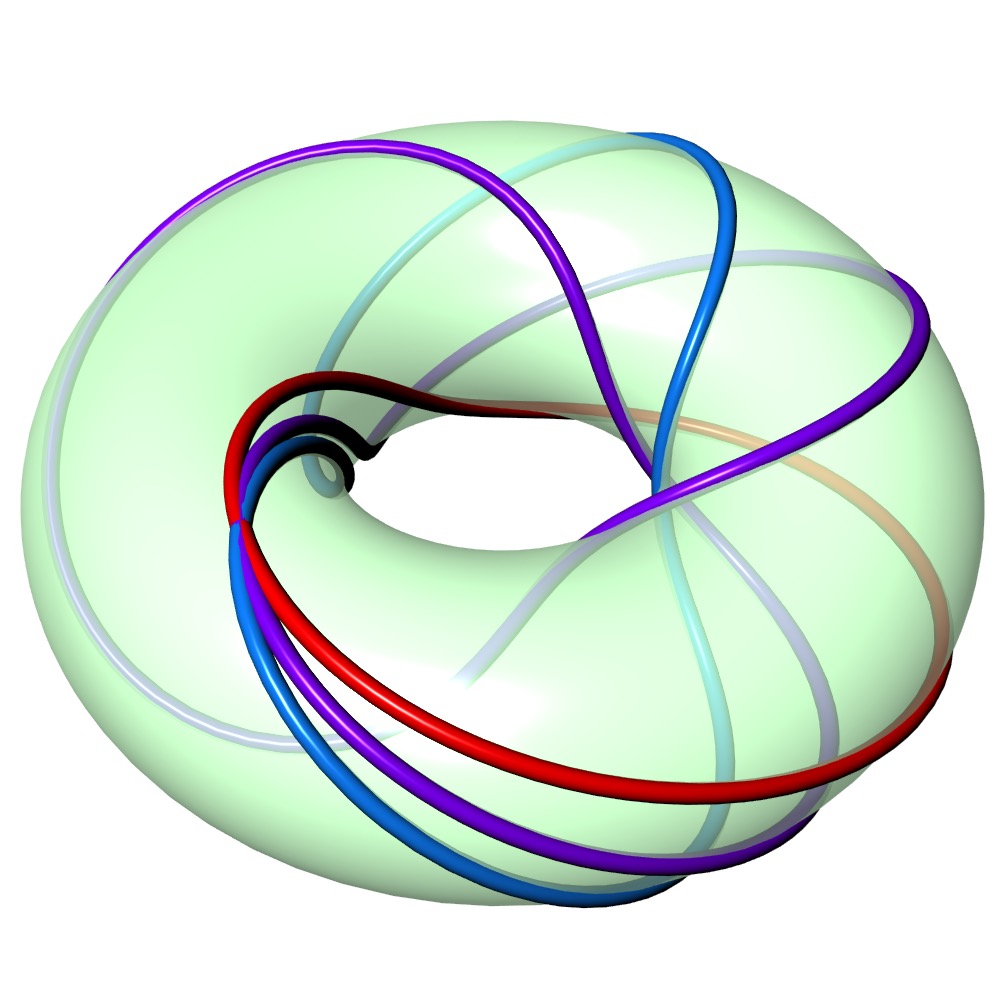}
}
\caption{Two one-vertex triangulations of the torus related by a single 2-2 move, and another one-vertex triangulation obtained by performing two more flips.} 
\label{torus_triangulations}
\end{figure}

Many authors also allow combinatorial isomorphisms when moving between different triangulations. These are maps from $S$ to $S$ that take vertices to vertices, edges to edges, and triangles to triangles. This has the following effect.
Combinatorially, there is a unique one-vertex triangulation of the torus: it has two triangles, three edges and one vertex. 
If one allows combinatorial isomorphism as a move to go from one triangulation to another, then for any two one-vertex triangulations of the torus, no bistellar moves would be required at all, but only a single combinatorial isomorphism to go from one to another. If one is not allowed to perform any combinatorial isomorphisms but only bistellar moves and isotopies, then there are infinitely many inequivalent one-vertex triangulations of a torus. However, one can apply a sequence of 2-2 moves and isotopies to connect any one-vertex triangulation to any other. See Figure~\ref{torus_triangulations} for an example of two such triangulations related by a single 2-2 move. The set of all isotopy classes of one-vertex triangulations of the two--torus connected by 2-2 moves thus has the structure of an infinite trivalent tree. In particular, for any two such triangulations, there is a unique minimal sequence of moves connecting them.
{All of this has beautiful stories in itself, of which we mention three. One story links triangulations to geometry and the Farey tesselation of the hyperbolic plane~\cite{Mosher-tiling-1988, Hatcher-triangulations-1991}. In particular, we can not only tell any two of them apart, but can easily work out the minimal number of moves needed to transform one to another. Another story gives a geometric interpretation of Markov triples and diophantine equations~\cite{Penner-decorated-1987, Penner-decorated-2012}, see \cite{Springborn-hyperbolic-2017} for a nice account of this. A third story uses these triangulations of the torus to construct special triangulations of the so-called lens spaces, a special class of 3--manifolds~\cite{Jaco-layered-2006, Jaco-minimal-2009}.}

Here is another reason why one might not want to allow combinatorial isomorphisms. To return to the analogy with knot theory, there are knots that cannot be transformed to their mirror image by a sequence of Reidemeister moves. An example is the trefoil knot. If we allowed the operation that takes a knot to its mirror image as an additional transformation, then the distinction between a right-handed and a left-handed trefoil knot is lost. 

\medskip

\textbf{Back to 3--dimensional spaces.} The situation for 3--dimensional manifolds regarding common refinements of triangulations is similar to surfaces, albeit slightly more difficult to visualise. Details and references are given in \S\ref{sec:refinements}.
The moves on 3--dimensional triangulations are also completely analogous. The 1-4 move introduces a vertex inside a tetrahedron and connects it to the four vertices of the tetrahedron with four edges. This creates six new triangular faces, spanned by the edges of the tetrahedron and the new vertex, and four new tetrahedra. The inverse move is called a 4-1 move. Then there is the 2-3 move, which is performed on two different tetrahedra meeting in a triangular face. The 2-3 move deletes this face by introducing a new edge connecting the opposite corners of the tetrahedra. Its inverse is the 3-2 move. Each of these moves changes the number of tetrahedra, faces and edges. However, the 2-3 and 3-2 moves do not affect the number of vertices in the triangulation. The 1-4 and 4-1 moves only change the number of material vertices, but not the number of ideal vertices. One may again imagine all of these moves as swapping subsets of tetrahedra on the boundary of a 4--dimensional simplex! 

\begin{figure}[htbp]
\centering
\includegraphics[width=0.8\textwidth]{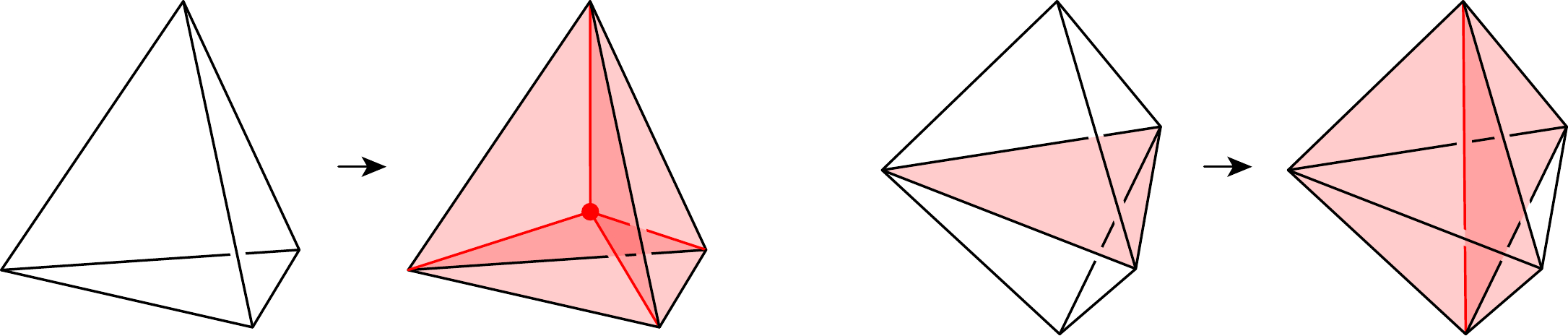}
\caption{The 1-4 and 2-3 moves.} 
\label{1-4_and_2-3_moves}
\end{figure}

The use of bistellar moves in the work of Pachner has its roots in the stellar moves used by Alexander~\cite{Alexander-combinatorial-1930} and Newman~\cite{Newman-foundation-1926}. The bistellar moves are equivalent to moves on a dual structure, called a spine, due to Matveev~\cite{Matveev-transformations-1987}. Classically, the stellar and bistellar moves were only defined, and results involving them only proved, for special types of triangulations that are less general than the ones we defined above. Vertices, edges, triangular faces and tetrahedra are called \emph{simplices}, and one indicates the dimension $d$ of a simplex by saying that it is a $d$--simplex. 
{
A simplex contained in another simplex $\sigma$ is said to be a
\emph{face} of $\sigma.$
A \emph{simplicial} triangulation requires any two simplices to meet in either a face or not at all. 
}
A \emph{combinatorial} triangulation has the additional requirement that the manifold structure is completely evident from the combinatorics --- this is technically made precise by requiring that the so-called link of every simplex is a sphere. For surfaces and 3--dimensional manifolds, every simplicial triangulation is combinatorial, and the distinction is only of relevance in higher dimensions (which are of no concern for this paper). 
Our triangulations are sometimes called \emph{singular} or \emph{semi-simplicial} in the literature. They can be turned into simplicial triangulations by performing at most two \emph{barycentric subdivisions}. We define these in \S\ref{sec:Barycentric subdivision} and show that barycentric subdivision can be achieved using the bistellar moves 1-4, 2-3 and 3-2. We also show that stellar moves can be achieved using these moves---the caveat here is that one stellar move may turn into an arbitrarily long sequence of bistellar moves. In this sense, there are infinitely many stellar moves but only four bistellar moves. This discussion connects our triangulations to the general theory of simplicial or combinatorial triangulations via bistellar moves, and we therefore use the word \emph{triangulation} throughout this paper without further qualification. We also always allow the implicit use of isotopy --- so our statements are really statements about triangulations up to isotopy, rather than fixed triangulations.
We summarise this discussion in the following theorem.

\begin{thm}[Alexander, Newman, Moise, Pachner]
\label{pachner}
The set of all triangulations of a closed three-dimensional manifold $M$ is connected under 1-4, 2-3, 3-2 and 4-1 moves.
\end{thm}

An excellent account of the history and proof of a similar result that holds in all dimensions was recently given by Lickorish~\cite{lickorish_simplicial_moves}.
A more general definition of triangulations, and a different set of moves, is used by Ludwig and Reitzner~\cite{Ludwig-elementary-2006}.

Banagl and Friedman~\cite[Proposition 2.16]{banagl-friedman} give a version of this theorem for closed three-dimensional pseudo-manifolds starting with a more general definition of a pseudo-manifold, and use this to extend the Turaev-Viro invariants to three-dimensional pseudo-manifolds:

\begin{thm}[Banagl-Friedman]
\label{thm:Banagl-Friedman}
The set of all triangulations of a closed three-dimensional pseudo-manifold is connected under 1-4, 2-3, 3-2 and 4-1 moves.
\end{thm}

We remark that in the above, the number of ideal vertices remains unchanged, but the number of material vertices may vary. Proofs of Theorems~\ref{pachner} and \ref{thm:Banagl-Friedman} are sketched in \S\ref{sec:refinements}.

\medskip

The interest of results such as the above is indeed the construction of invariants. This is made even easier by limiting the number of moves required, which is achieved by the following refinement of Theorem~\ref{pachner} of Matveev~\cite{matveev_book} and Piergallini~\cite{Piergallini-standard-1988}.

\begin{thm}[Matveev, Piergallini]
\label{matveev}
The set of all triangulations of a closed three-dimensional manifold $M$ with exactly one material vertex is connected under 2-3 and 3-2 moves, excepting triangulations with a single tetrahedron.
\end{thm}

This theorem was originally stated in the dual language of spines, which we discuss in \S\ref{sec:spines}.
There are precisely three closed three-manifolds that admit a triangulation with a single tetrahedron: the three-sphere and two other lens spaces, namely $L(4,1)$ and $L(5,2).$ See Figure~\ref{dr_one_vertex}. All of these manifolds are orientable. Moreover, up to combinatorial equivalence, there are precisely four triangulations with a single tetrahedron. The three-sphere has two such triangulations, but one of them has two vertices. So the exclusion only concerns three triangulations of three three-dimensional manifolds. No 2-3 or 3-2 moves can be applied to these triangulations.

\begin{figure}[htbp]
\centering
\subfloat[$S^3$]{
\includegraphics[width=0.22\textwidth]{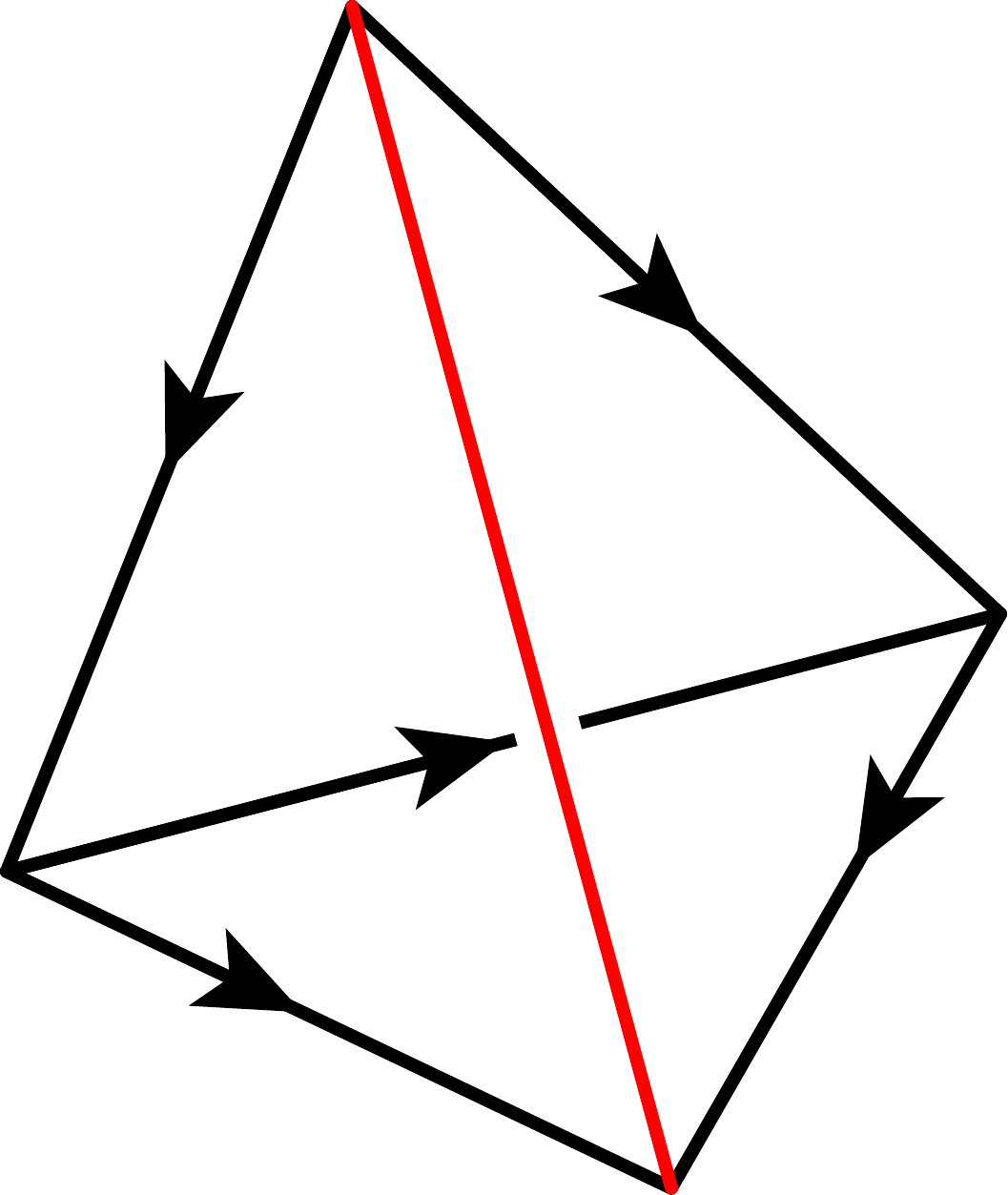}
}
\subfloat[$S^3$]{
\includegraphics[width=0.22\textwidth]{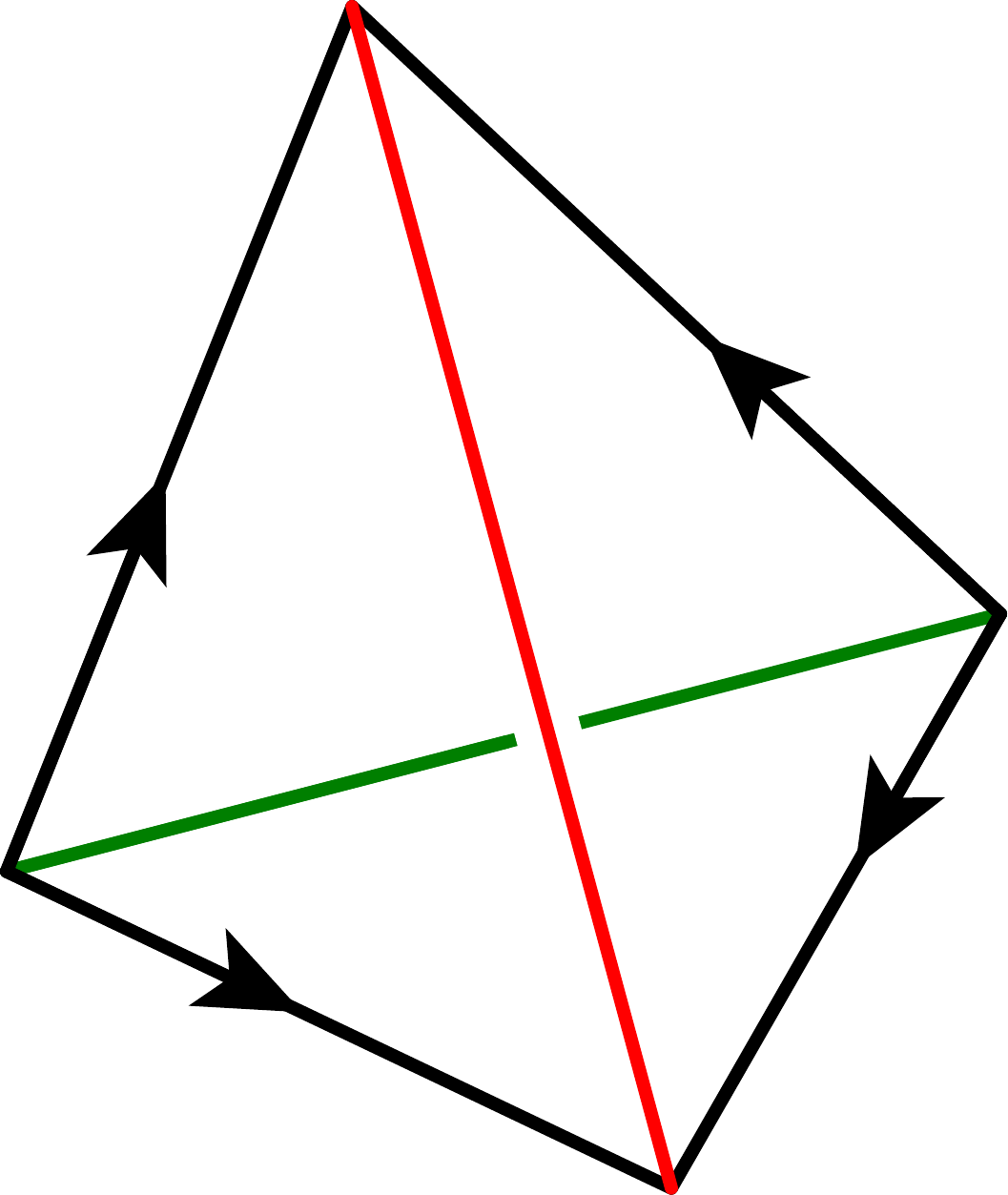}
}
\subfloat[$L(4,1)$]{
\includegraphics[width=0.22\textwidth]{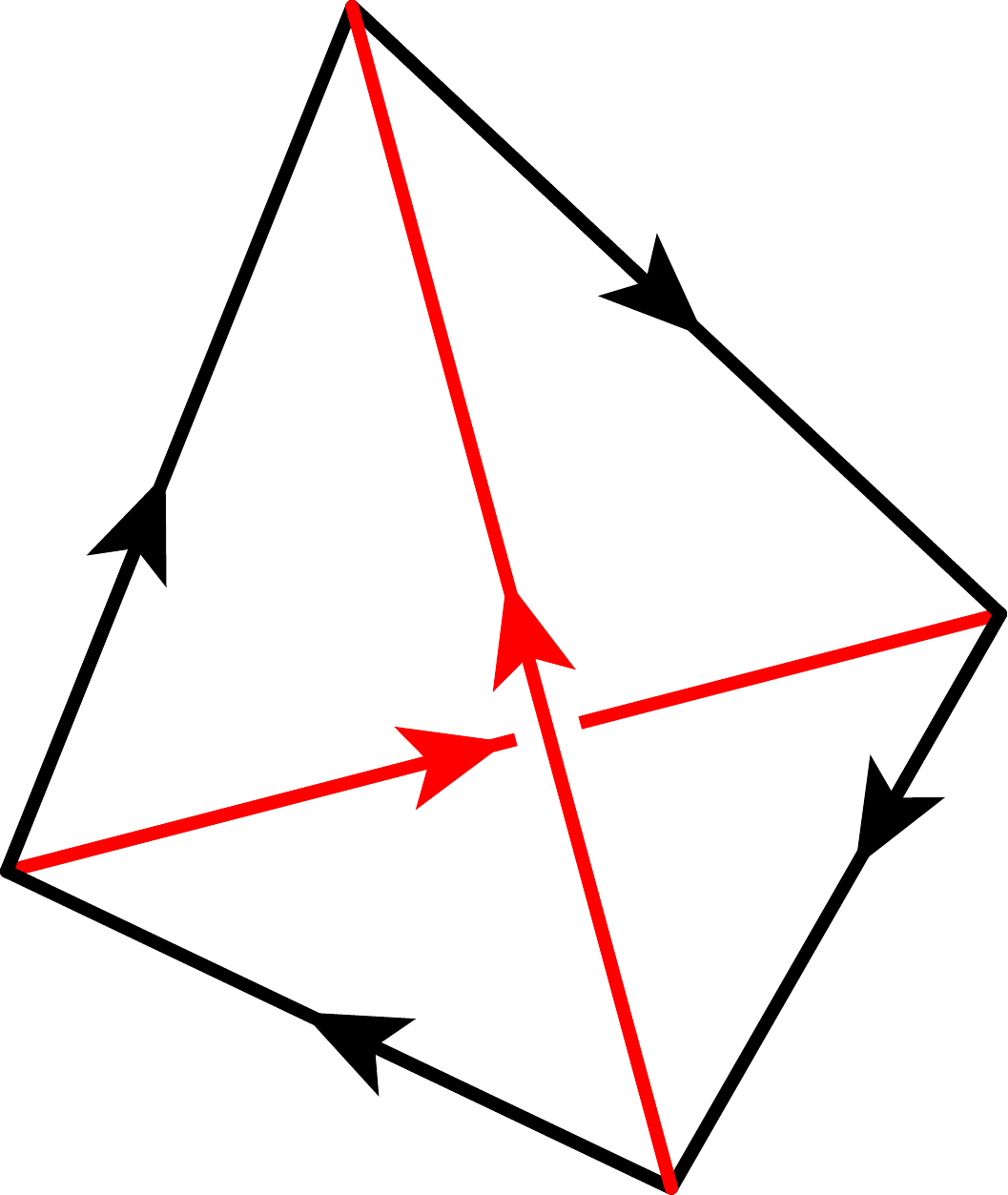}
}
\subfloat[$L(5,2)$]{
\includegraphics[width=0.22\textwidth]{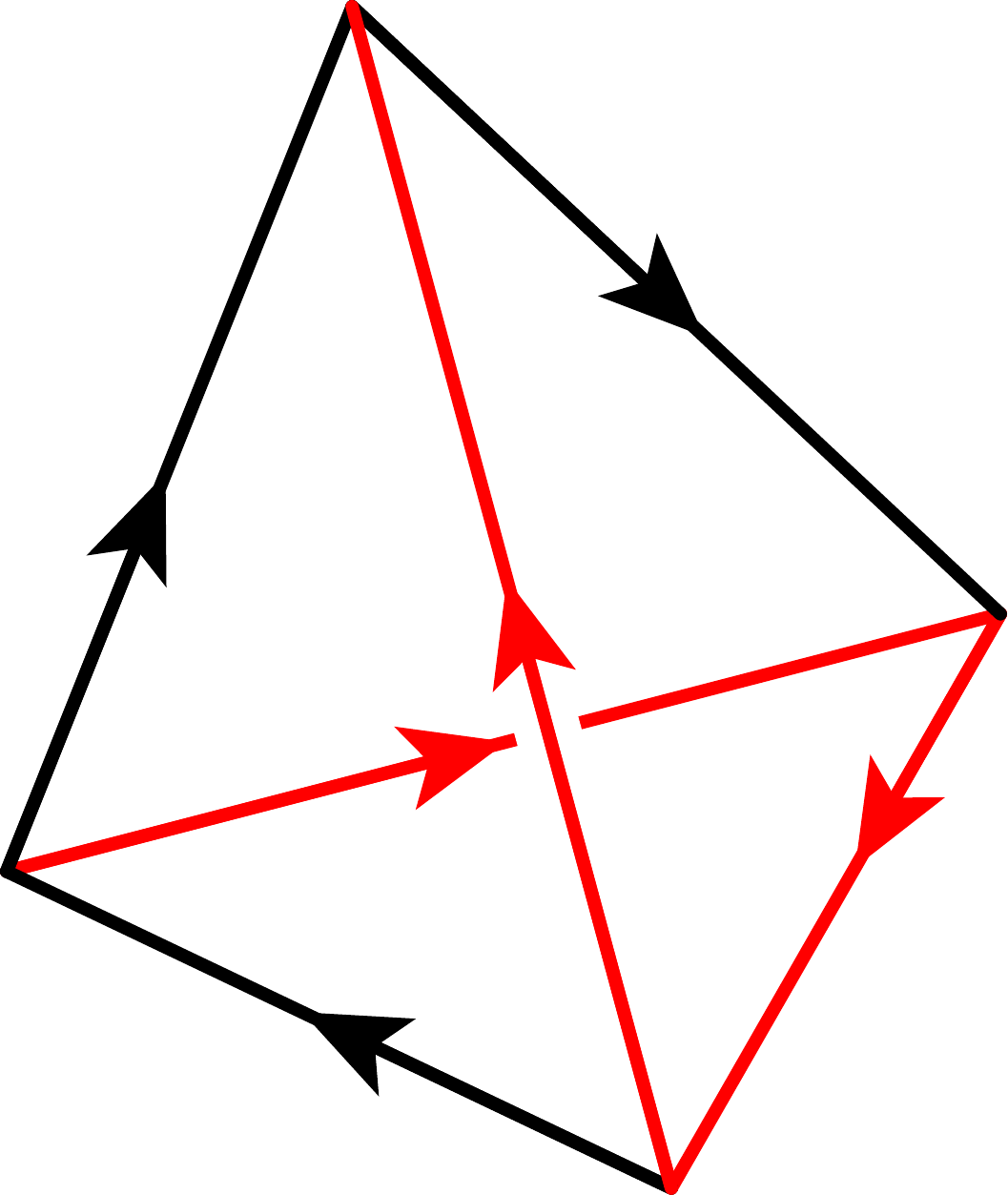}
}
\caption{Single tetrahedron triangulations of closed manifolds.} 
\label{dr_one_vertex}
\end{figure}

\medskip

Benedetti and Petronio~\cite{Benedetti-finite-1995, Benedetti-branched-1997} outlined a proof of the fact that the above result also holds for any prescribed number of material vertices in a triangulation of a 3--dimensional manifold, and Amendola~\cite{Amendola-calculus-2005} shows that this extends to 3--dimensional pseudo-manifolds. In each case, key results of Matveev~\cite{matveev_book} are used and the techniques involve the dual viewpoint of spines. The main purpose of this note is threefold. We wish to popularise Amendola's result; we give a proof that emphasises the dual viewpoints of triangulations and spines; and we give a proof replacing a key general position argument of \cite{matveev_book} with a more combinatorial argument inspired by the theory of subdivisions.

\begin{thm}[Amendola]
\label{thm:main}
Let $M$ be a three-dimensional manifold or pseudo-manifold.
The set of triangulations of $M$ with a fixed number (possibly zero) of material vertices is connected under 2-3 and 3-2 moves, excepting triangulations with a single tetrahedron.
\end{thm}

As for the additional exceptions to the theorem: The only closed orientable three-dimensional pseudo-manifolds having a triangulation with a single tetrahedron are the manifolds given above. The only non-orientable pseudo-manifold having a triangulation with a single tetrahedron is the so-called Gieseking manifold shown in Figure~\ref{dr_Gieseking}. Here, the vertex link is a Klein bottle. 

\begin{figure}[htbp]
\centering
\includegraphics[width=0.3\textwidth]{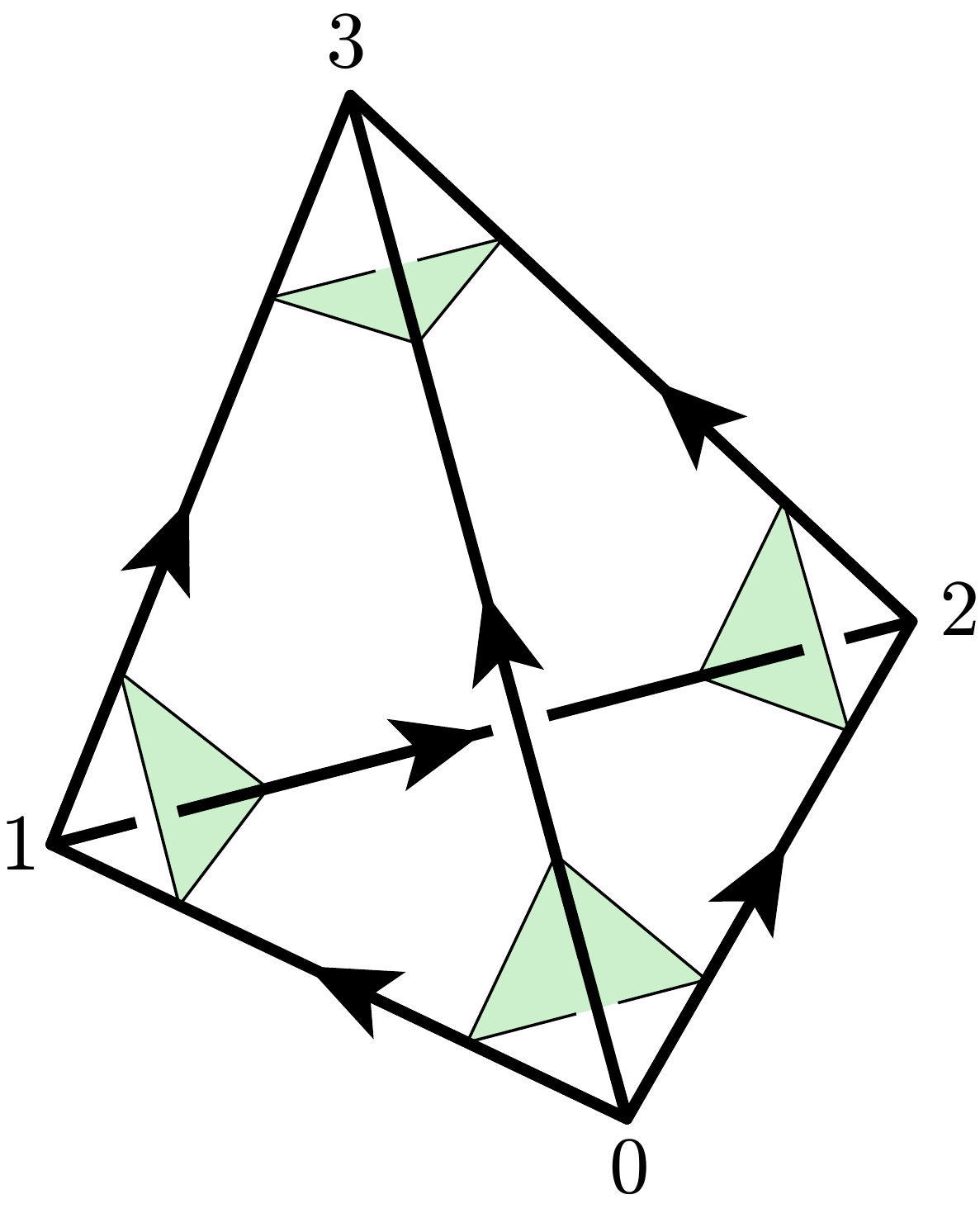}
\caption{The Gieseking manifold. Vertices are labelled by the number of arrows pointing towards the vertex. The faces are identified by $012 \leftrightarrow 023$ and $013 \leftrightarrow 123$.} 
\label{dr_Gieseking}
\end{figure}

\medskip

Theorems~\ref{pachner} and \ref{thm:Banagl-Friedman} are important connectivity results. They are useful for defining invariants in terms of triangulations, and for building censuses of triangulations. Theorem~\ref{matveev} is a great  improvement over Theorem~\ref{pachner} in the view of the theory of \emph{0--efficient triangulations} of Jaco and Rubinstein~\cite{Jaco-0-efficient-2003}, which require only a single vertex for the most important classes of closed three-dimensional manifolds. In the same vein, Theorem~\ref{thm:main} allows us to pass between any two triangulations with no material vertices without introducing material vertices along the way.
We also remark that Amendola~\cite{Amendola-calculus-2005} proves a more general theorem that keeps track of finer information on the manifold than Theorem~\ref{thm:main}.

\medskip

In the literature, there are a number of papers that attribute Theorem~\ref{thm:main} to Matveev or Piergallini. Personal communication with both of these authors revealed that both think they only proved Theorem~\ref{matveev}, but that their techniques may extend to that effect. 
This paper affirms this belief.

\medskip

We give a complete proof of Theorem~\ref{thm:main}. We will assume Theorems~\ref{pachner} and \ref{thm:Banagl-Friedman} (proofs are sketched in \S\ref{sec:refinements}), but we do not assume Theorem~\ref{matveev}. Our main task is thus to convert any appearance of 1-4 or 4-1 moves into sequences of 2-3 and 3-2 moves. 
Our contribution in reproving Theorem~\ref{matveev} in the proof of Theorem~\ref{thm:main} is to view the passage between triangulations from two perspectives: the triangulations on the one hand and the dual notion of \emph{special spines} on the other. 

\medskip

The structure of this paper is as follows. 
More examples, as well as the definitions of triangulations and spines are given in Section~\ref{sec:Triangulations}. Here, we also describe classical constructions, such as stellar moves and barycentric subdivision, and outline proofs of classical results going back to Alexander and Newman.

We set the scene in Section~\ref{sec:warm-up}, proving a lemma that simplifies the set-up for the proof of the main result. Whilst the main ideas for the main proof will be present, the interplay between triangulations and spines is not yet required. The strategy for the proof of the main result is presented in  Section~\ref{proof_strategy}. The remainder of the paper aims at filling in the missing details. This starts in Section~\ref{sec:arch and friends} by discussing Matveev's arch, and describing arch constructions as moves on triangulations. The first serious steps in providing details are taken in Section~\ref{sec:details}, which can be viewed as a pre-processing step for the two triangulations, reducing the problem of connecting them to the problem of \emph{sweeping a membrane across a ball}. This last task is performed in detail in Section~\ref{move_membrane}, thus completing this paper.

In writing this paper, we have considered other strategies to prove this particular connectivity result for triangulations. Inevitably, the problem was always reduced to a situation that could not easily be swept under the rug. In the end, we settled for sweeping across a ball, as here the generic situation is easy to deal with and hence the overall strategy easy to follow.  


\section{Triangulations and spines}
\label{sec:Triangulations}

In the hope of providing a reference that is useful for novices in the field, we provide this section to make this paper as self-contained as possible. We give further examples~(\S\ref{subsec:Examples}); 
define manifolds and pseudo-manifolds~(\S\ref{subsec:Manifolds and pseudo-manifolds}) and the triangulations we consider~(\S\ref{subsec:Triangulations}). We then discuss classical moves on triangulations, namely stellar moves~(\S\ref{sec:stellar}) and barycentric subdivision~(\S\ref{sec:Barycentric subdivision}) and explain how these can be performed using bistellar flips. We use these in \S\ref{sec:refinements} with sketches of proofs of Theorems~\ref{pachner} and \ref{thm:Banagl-Friedman}. The dual notion of spines~(\S\ref{sec:spines}) completes the summary of background material required for this paper.


\subsection{Examples}
\label{subsec:Examples}

Before we formalise the definition of a triangulation given in the introduction, we give some examples. In the introduction, we have already seen triangulations just involving one tetrahedron. Figure~\ref{dr_one_vertex} shows that we can build the 3--dimensional sphere with just one tetrahedron. A simplicial triangulation of the 3--dimensional sphere requires at least five tetrahedra.
The advantage of the more general triangulations thus lies in their efficiency.

Ideal triangulations similarly turn out to be an efficient way of encoding manifolds. An example is given by Thurston's ideal triangulation of the complement of the figure eight knot shown in Figure~\ref{dr_figure8}. This is a space obtained by deleting the figure eight knot from the three--dimensional sphere. One can build a pseudo-manifold with just two tetrahedra with the property that the complement of the vertex in this pseudo-manifold is homeomorphic to the complement of the figure eight knot. If one considers the space obtained by deleting a small regular neighbourhood of the knot from the 3--sphere, one obtains a compact 3--manifold with boundary a torus. To divide this into (material) tetrahedra requires ten --- where now we allow some faces of tetrahedra not to be identified so as to form the boundary. 

\begin{figure}[htbp]
\centering
\subfloat[]{
\includegraphics[width=0.5\textwidth]{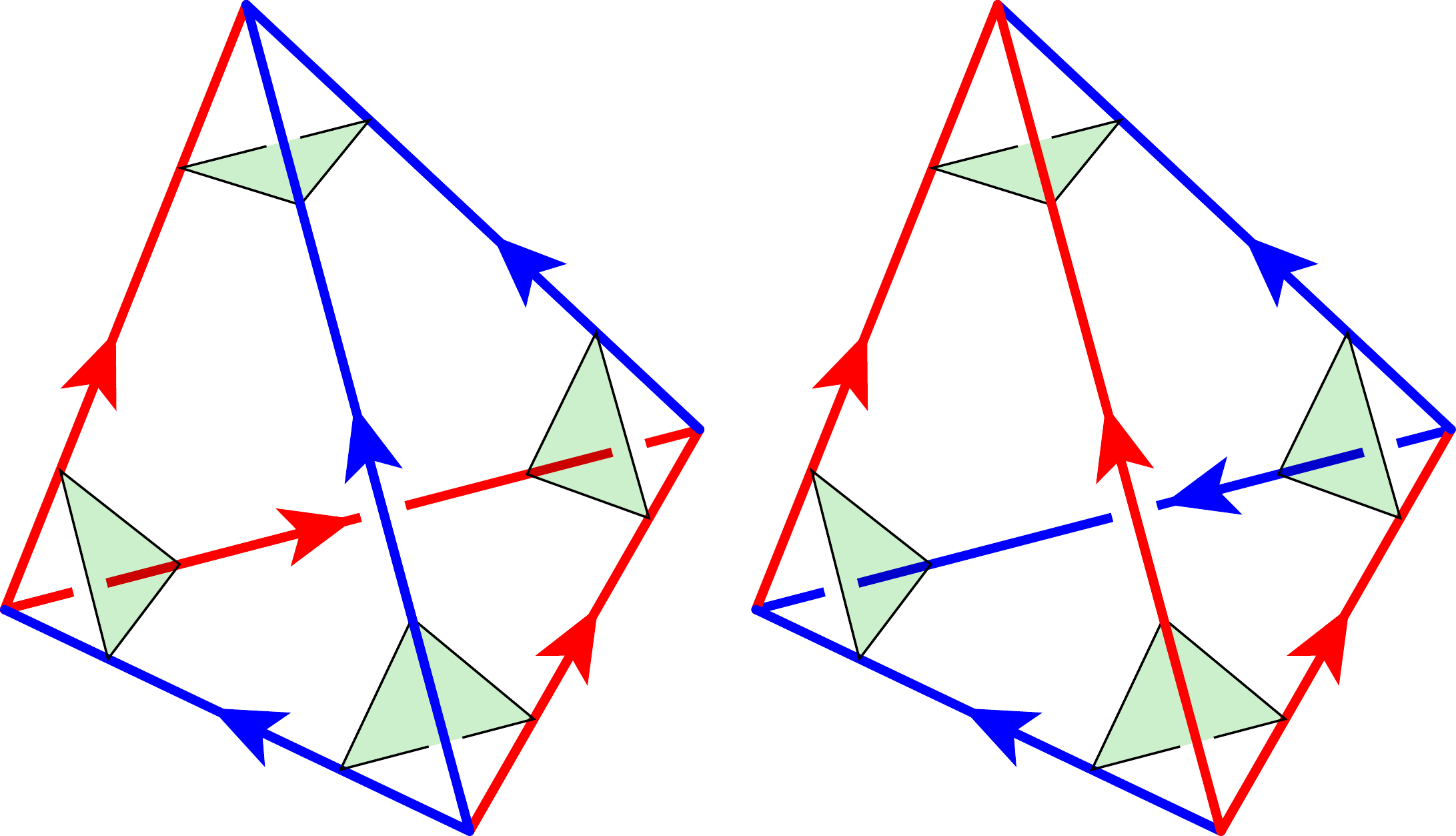}
}
\quad
\subfloat[]{
\includegraphics[width=0.3\textwidth]{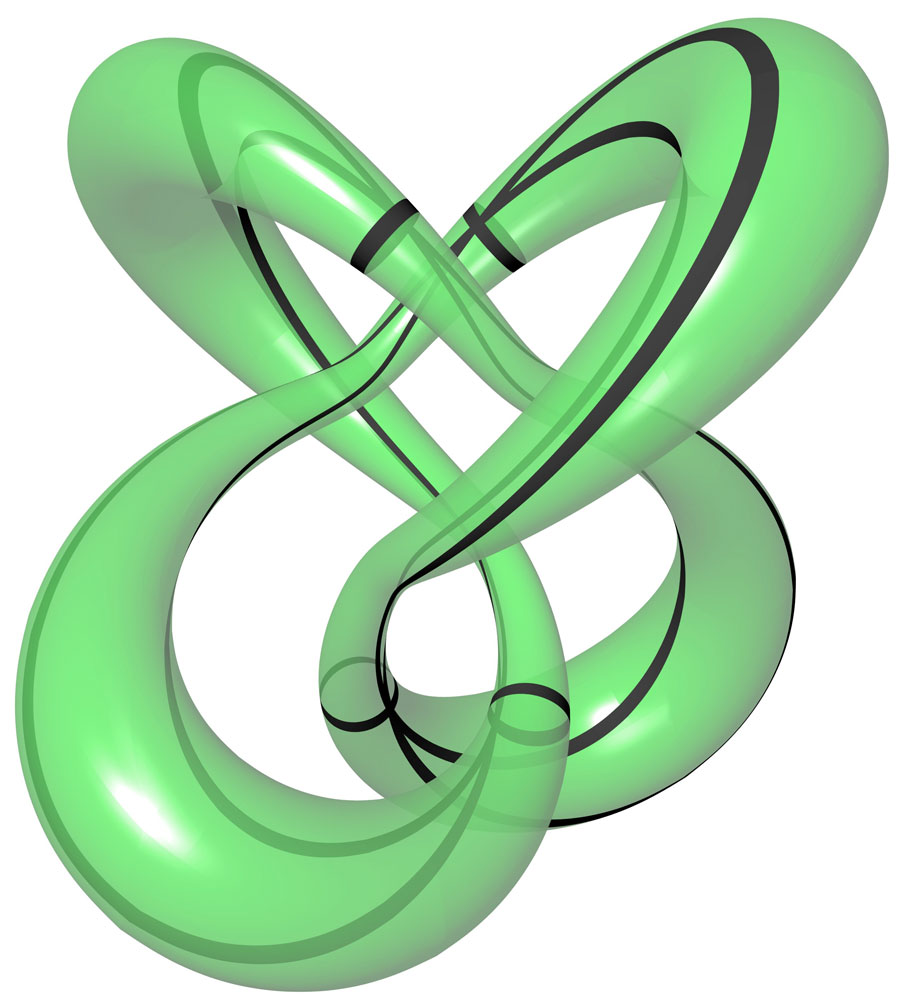}
}
\caption{The ideal triangulation of the complement of the figure eight knot in the three-sphere. The shown triangles form a tube enclosing the knot.} 
\label{dr_figure8}
\end{figure}


\subsection{Manifolds and pseudo-manifolds}
\label{subsec:Manifolds and pseudo-manifolds}

{
A euclidean 3--simplex $\sigma$ is the convex hull of four points $p_0, p_1, p_2, p_3$ in euclidean 3--space that are in general position. What matters in our discussion is not the euclidean, but just the affine structure. A unified discussion of simplices and subsimplices is aided by the following definition. An $n$--simplex $\sigma$ is the convex hull of $n+1$ affinely independent points $p_0, \ldots, p_n$ in some affine space of dimension at least $n.$
The convex hull of a subset of $\{p_0, \ldots, p_n\}$ is a \emph{face} or \emph{subsimplex} of $\sigma.$ 
We refer to 3--simplices as tetrahedra, 2--simplices as triangles, 1--simplices as edges and 0--simplices as vertices.
}

Let $\widetilde{\Delta}$ be a finite union of pairwise disjoint euclidean $3$--simplices. 
Every $k$--simplex $\tau$ in $\widetilde{\Delta}$ is contained in a unique $3$--simplex $\sigma_\tau.$ A $2$--simplex in $\widetilde{\Delta}$ is called a \emph{facet}.
Let $\Phi$ be a family of affine isomorphisms pairing the facets in $\widetilde{\Delta},$ with the properties that $\varphi \in \Phi$ if and only if $\varphi^{-1}\in \Phi,$ and every facet is the domain of a unique element of $\Phi.$ The elements of $\Phi$ are termed \emph{face pairings}.

The quotient space $P = \widetilde{\Delta}/\Phi$ with the quotient topology is then a \emph{closed $3$--dimensional pseudo-manifold}, and the quotient map is denoted $p\co \widetilde{\Delta} \to P.$ 
We will always assume that $P$ is connected. In the case where $P$ is not connected, the results of this paper apply to its connected components.

If one considers any point $p\in P$ that is not a vertex or a mid-point of an edge, then a small neighbourhood of $p$ looks like the neighbourhood of a point in 3--dimensional euclidean space. This is easy to see for points in the interior of a tetrahedron, not too difficult for points in faces, and a little more involved for points that lie in edges. So what makes $P$ a \emph{pseudo-}manifold is the fact that we have no control over small neighbourhoods of its vertices, and over small neighbourhoods of mid-points of edges. Place a small triangle in the corner of each tetrahedron so that these triangles glue up to form a closed surface. This surface is called the \emph{link} of the vertex, and the neighbourhood of the vertex is a cone over this surface, with cone point the vertex. This cone is a ball precisely when the surface is a sphere. We thus call a vertex \emph{material} when its link is a sphere and \emph{ideal} otherwise. To see what happens at midpoints of edges, take the intersection of the tetrahedron with a small sphere based at the midpoint. This gives a lune in the tetrahedron. Now trace the chain of tetrahedra around this edge. With each tetrahedron, this adds another lune until we get to the last tetrahedron. Here again, a lune is added but this is also identified to the initial lune. There are two different ways of making the identification: one gives a sphere, and the other gives a projective plane. In the second case, the edge is identified with itself but in opposite orientation. To rule out this case and to align the structure of the triangulation with the structure of our pseudo-manifolds, we require that 
\begin{enumerate}
\item the restriction of the map $p\co \widetilde{\Delta} \to P$ to the interior of each {$k$-simplex} in $\widetilde{\Delta}$ is a homeomorphism for all $k \in \{ 0, 1, 2, 3\}.$
\end{enumerate}
This is only a restriction on the triangulations we allow, but not on the pseudo-manifolds. The restriction has consequences only for $k=1$, indeed it only avoids the situation where an edge may fold back onto itself, but it allows extra identifications along the boundary of a simplex. With this requirement, the pseudo-manifold $P$ is a \emph{manifold} if and only if every vertex is material.

We still refer to the image under $p$ of a triangle in $\widetilde{\Delta}$ as a triangle in $P$, but the reader should bear in mind that this may not be an embedded triangle and have identifications along its boundary. The same applies to tetrahedra or edges in $P.$

The quotient space $P$ is studied via the map $p\co \widetilde{\Delta} \to P.$ The \emph{degree} of the triangle, edge or vertex $\tau$ in $P$ is the number of preimages it has in $\widetilde{\Delta};$ equivalently, the degree of the restriction $p \co p^{-1}(\tau)\to \tau.$

There is a more restrictive class, where such confusion cannot arise. We say that $(\widetilde{\Delta}, \Phi)$ is \emph{simplicial} if
\begin{enumerate}
\item the restriction of the map $p\co \widetilde{\Delta} \to P$ to each $k$--simplex $\widetilde{\Delta}$ is a homeomorphism for all $k \in \{0,1,2,3\},$ and
\item for any two simplices $\tau_0$ and $\tau_1$ in $\widetilde{\Delta},$ $p(\tau_0) \cap p(\tau_1)$ is either empty or a single simplex. 
\end{enumerate}


\subsection{Triangulations}
\label{subsec:Triangulations}

In the previous section, we have defined a 3--dimensional manifold or pseudo-manifold using tetrahedra. Now suppose you are given such a manifold or pseudo-manifold $M.$
We say that the triple $\cT = ( \widetilde{\Delta}, \Phi, h)$ is a \emph{triangulation} of $M$ if 
\begin{itemize}
\item[] $h\co P \to M$ is a piecewise linear homeomorphism.
\end{itemize}
The images of vertices, edges, triangles and tetrahedra under the
composition of the map $p\co  \widetilde{\Delta} \to P$ with $h$ give us the 
vertices, edges, triangles and tetrahedra of the triangulation in $M.$ Again, these generally have self-identifications along their boundaries.
The triangulation $\cT = ( \widetilde{\Delta}, \Phi, h)$ is a \emph{simplicial triangulation} of $M$ if $( \widetilde{\Delta}, \Phi)$ is simplicial.

We say that two triangulations $\cT_0 = ( \widetilde{\Delta}_0, \Phi_0, h_0)$ and $\cT_1 = ( \widetilde{\Delta}_1, \Phi_1, h_1)$ are \emph{equivalent} or \emph{the same} if there is a simplicial bijection $s \co \widetilde{\Delta}_0 \to \widetilde{\Delta}_1$ such that $h_0 \circ p_0 \circ s$ is isotopic to $h_1 \circ p_1.$ This in particular allows the same union of tetrahedra and set of face pairings to give inequivalent triangulations that are combinatorially equivalent.

We will be concerned with changing a given triangulation. On a formal level, this amounts to changing a triangulation $\cT_0 = ( \widetilde{\Delta}_0, \Phi_0, h_0)$ of $M$ to another triangulation $\cT_1 = ( \widetilde{\Delta}_1, \Phi_1, h_1)$. Such a change can be very local, such as replacing one simplex and the associated face pairings in $(\widetilde{\Delta}_0, \Phi_0)$ with a small number of simplices and corresponding face-pairings. Or it could be a change that replaces all simplices and face-pairings. 
As an example, we mention that it is always possible to transform a triangulation into a simplicial triangulation by applying at most two \emph{barycentric subdivisions}. This replaces each 3--simplex by
 $24 \times 24 = 576$ 3--simplices. We give the definition in \S\ref{sec:Barycentric subdivision} and show how to achieve this using 1-4, 2-3 and 3-2 moves. 

In the remainder of this paper, we try to keep notation to a minimum in order to not distract from the main ideas. A precise notation is, of course, relevant for an actual implementation in software.
There is a natural way to track vertices, edges, and so forth through a sequence of moves on a triangulation. Our moves are typically performed topologically inside the pseudo-manifold, and the simplices can then be disassembled to form a triangulation together with a preferred quotient map $\widetilde{\Delta} \to M.$


\subsection{Stellar moves}
\label{sec:stellar}

There are three types of stellar moves. The first type is simply the 1-4 move. See Figure~\ref{1-4_and_2-3_moves}. It places a vertex in the interior of a tetrahedron and divides the tetrahedron into four tetrahedra. Each of these can be viewed as a cone from a boundary face of the original tetrahedron to the central vertex. 

\begin{figure}[htbp]
\centering
\subfloat[Stellar move on a facet.]{
\labellist
\small\hair 2pt
\pinlabel {1-4} at 75 200
\pinlabel {2-3} at 335 170
\endlabellist
\includegraphics[width=0.3\textwidth]{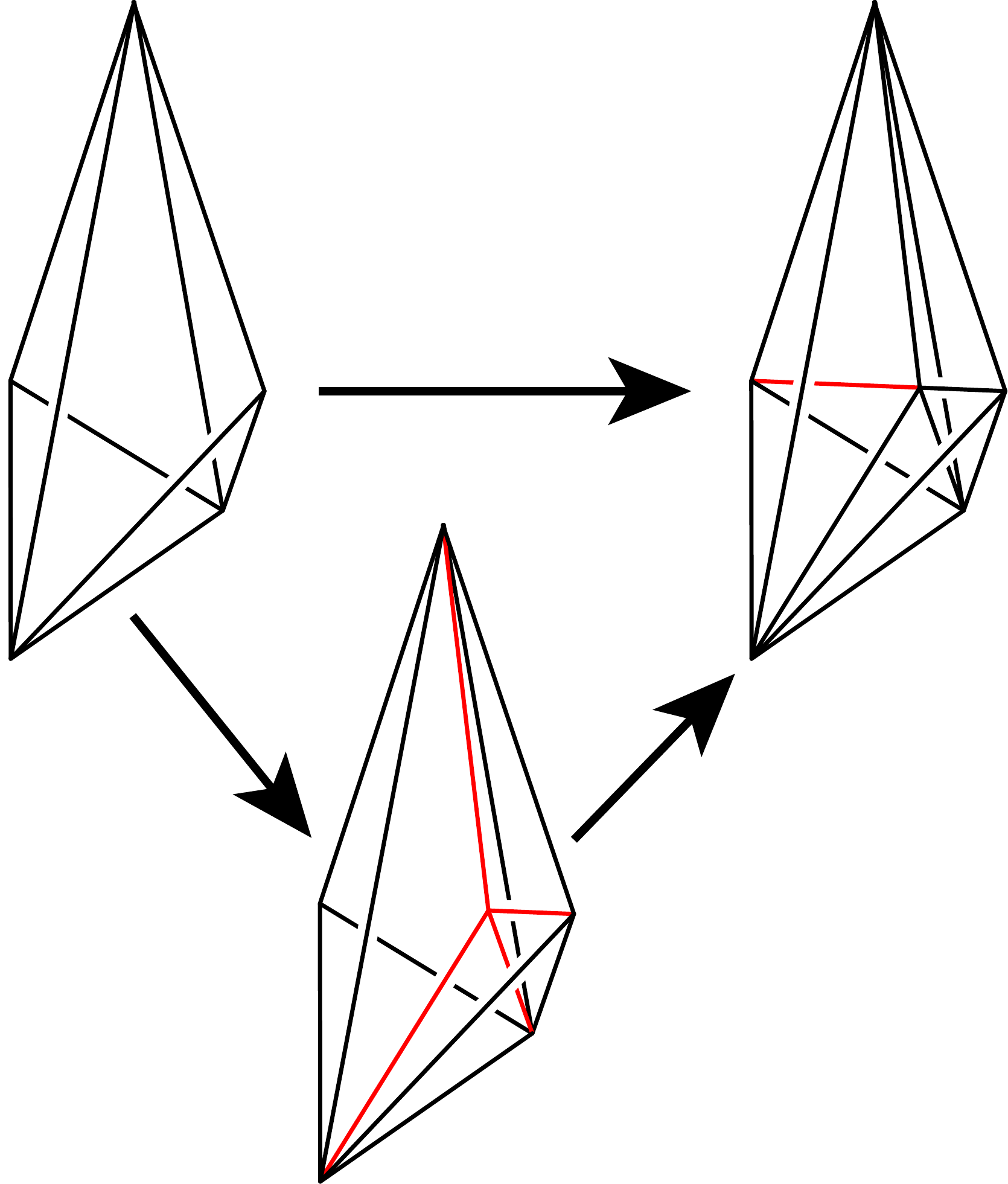}
\label{stellar_move_face}
}
\quad 
\subfloat[Stellar move on an edge.]{
\includegraphics[width=0.5\textwidth]{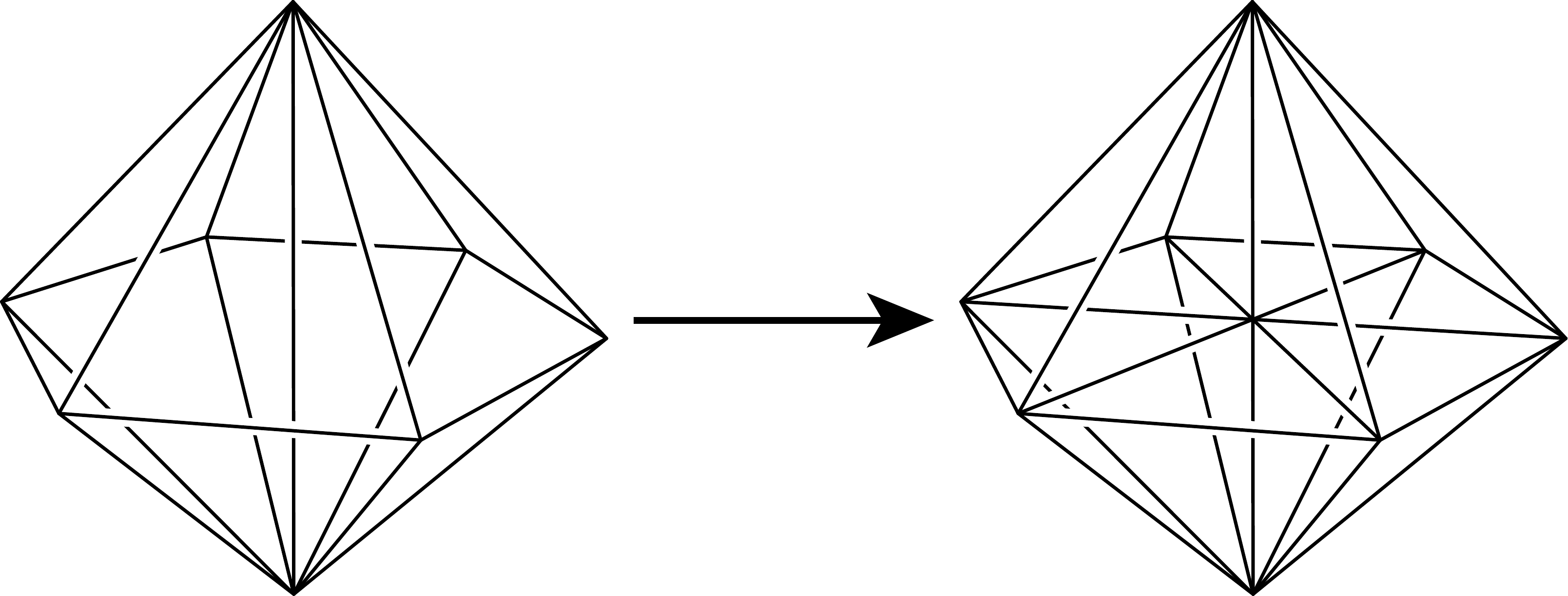}
\label{stellar_move_edge}
}
\caption{Stellar moves.} 
\label{stellar_moves}
\end{figure}

The second type of stellar move is conceptually similar, this time performed on a facet $\tau$ that is contained in two distinct tetrahedra $\sigma_0$ and $\sigma_1.$
Here, one first performs a 1-3 move on $\tau.$ One then subdivides each of the two tetrahedra into three by coning the triangulation of the subdivided face to the remaining vertex. See Figure~\ref{stellar_move_face}. This operation can be described as an operation on $\widetilde{\Delta},$ so it is irrelevant that some of the edges or vertices of the tetrahedra may be identified. To achieve this move by our four bistellar moves, one can first apply a 1-4 move on $\sigma_0.$ The face $\tau$ is not affected by this, and one can now perform a 2-3 move that replaces the two tetrahedra meeting in $\tau$ with three tetrahedra meeting in an edge. The result is the same as the stellar move on the face $\tau.$

The third move is performed on an edge $e$ that is only contained at most once in any tetrahedron and of degree at least two. We place a new vertex on $e$. In each tetrahedron, one now creates a new triangle by coning the edge opposite $e$ to the new vertex, thus cutting the tetrahedron into two. See Figure~\ref{stellar_move_edge}. Again, this operation can be done consistently in $\widetilde{\Delta},$ and so is well-defined. We need to show how to realise this move by bistellar moves. Choose a cyclic order of the tetrahedra abutting $e,$ and starting with one tetrahedron label them $\sigma_0, \ldots, \sigma_k.$ Suppose $\sigma_i$ and $\sigma_{i+1}$ meet in $e$ along $\tau_i,$ and $\sigma_k$ and $\sigma_{0}$ meet in $e$ along $\tau_k.$
First apply a 1-4 move on $\sigma_0.$ This increases the degree of $e$ by one. Then apply a 2-3 move along the face  $\tau_0$ (interpreted as a face in the new triangulation). This decreases the degree of $e$ by one. Now we iteratively apply 2-3 moves along $\tau_1,$ $\tau_2,$ \ldots, until only $\tau_{k-1}$ and $\tau_k$ are left. At this stage, $e$ (interpreted as an edge in the resulting triangulation) has degree three. Hence apply a 3-2 move (noting that all three tetrahedra meeting in $e$ are distinct). 
See Figure~\ref{barycentric_subdiv_3d_edge_top_view}.


\subsection{Barycentric subdivision}
\label{sec:Barycentric subdivision}

A \emph{barycentric subdivision} can be defined inductively. A barycentric subdivision of a 1--simplex introduces a vertex at its midpoint and divides it into two 1--simplices. A barycentric subdivision of a 2--simplex introduces a vertex at its centre and cones this to the barycentric subdivision of its boundary. This divides it into six 2--simplices. For a 3--simplex, we again introduce a vertex at its centre and then cone to the barycentric subdivision of its boundary, thus dividing it into 24 3--simplices. See Figure~\ref{fig_barycentric_subdivision}.

\begin{figure}[htbp]
\centering
\labellist
\small\hair 2pt
\pinlabel {0} at -15 80
\pinlabel {0} at 212 470
\pinlabel {0} at 436 80
\pinlabel {1} at 92 279
\pinlabel {1} at 332 279
\pinlabel {1} at 212 63
\pinlabel {2} at 225 233
\endlabellist
\subfloat[The result of barycentric subdivision on a 2-simplex.]{
\includegraphics[width=0.4\textwidth]{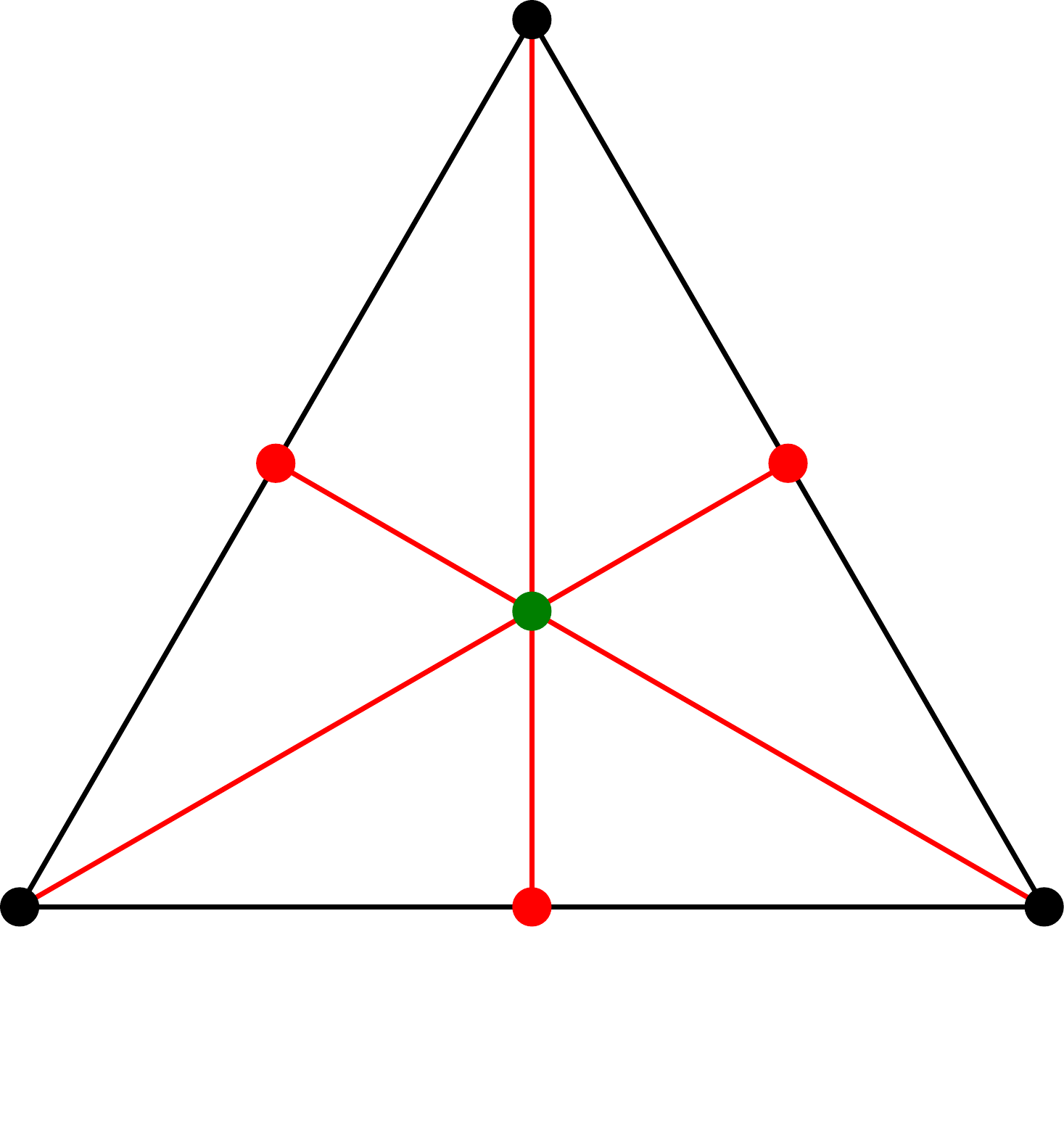}
}
\qquad
\subfloat[The result of barycentric subdivision on a 3-simplex.]{
\includegraphics[width=0.45\textwidth]{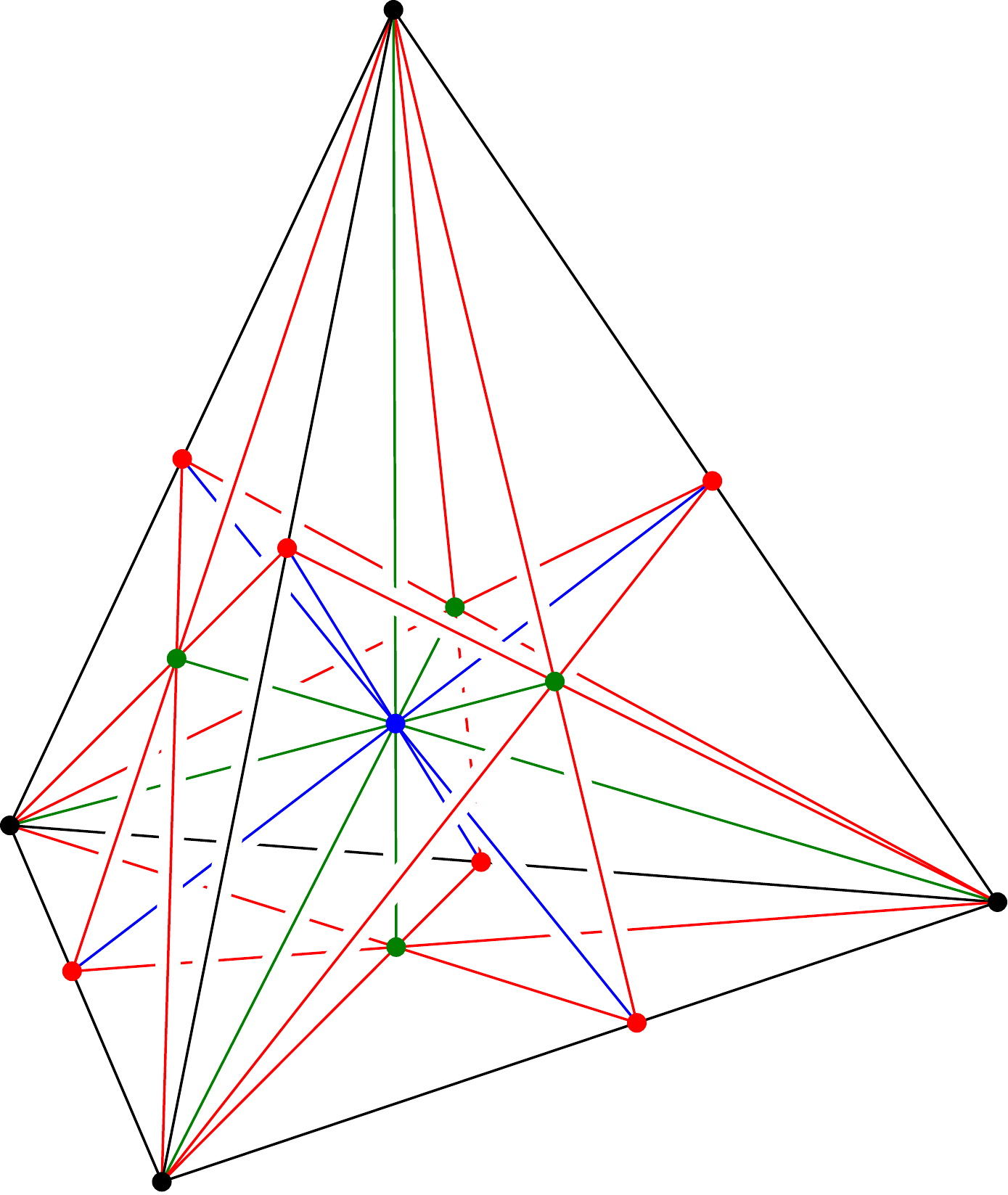}
}
\caption{Barycentric subdivision.} 
\label{fig_barycentric_subdivision}
\end{figure}
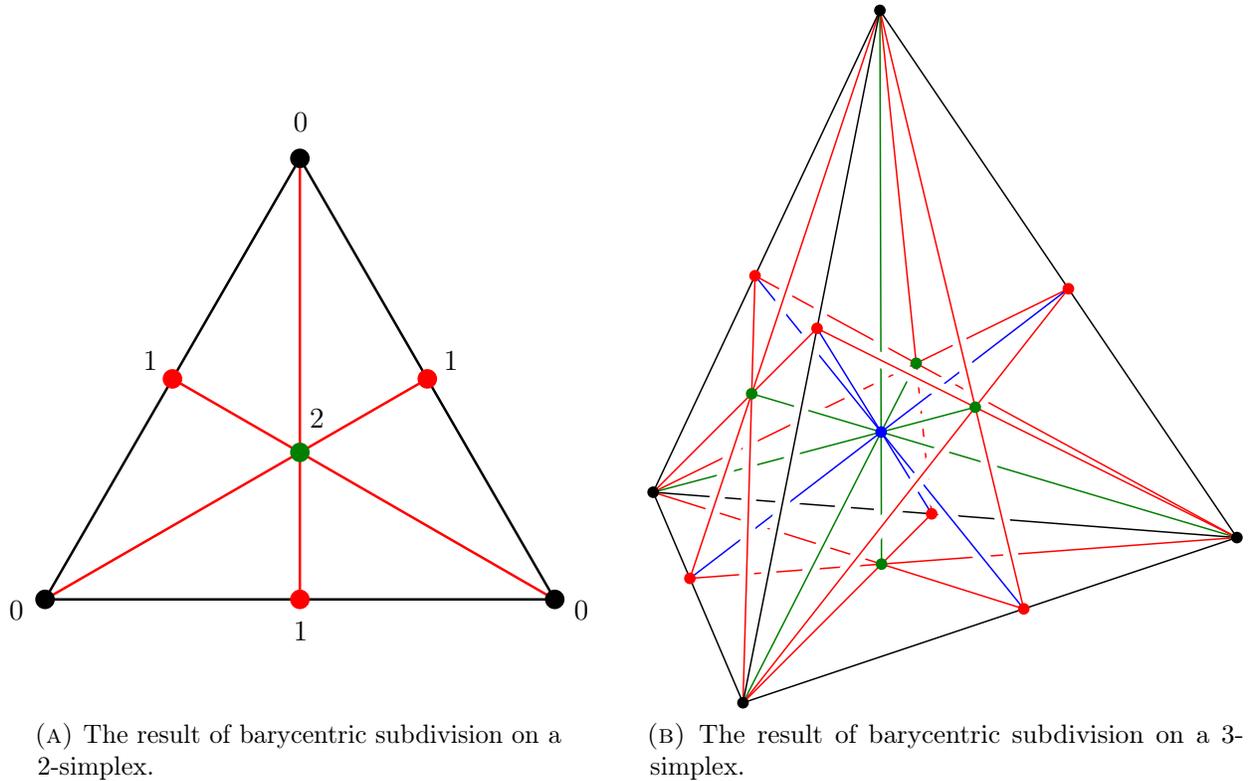

It follows from the definition of a triangulated pseudo-manifold that one can apply this procedure to each 3--simplex in the triangulation, and hence break the triangulation up into a finer triangulation having 24-times as many tetrahedra. It turns out that this can be achieved using stellar moves. Hence, by the previous section, barycentric subdivision can be realised using 1-4, 2-3, 3-2 moves on the original triangulation.

To achieve barycentric subdivision, first apply a 1-4 move to each 3--simplex in the original triangulation $\tri$ to obtain a new triangulation $\tri_0$ with four times as many tetrahedra. There are facets and edges that persist from $\tri$ to $\tri_0.$ On each of these facets, perform a stellar move, giving a new triangulation $\tri_1$ with three times as many tetrahedra as $\tri_0.$ The edges that persisted from $\tri$ to $\tri_0$ are still present in $\tri_1.$ Now perform stellar moves on all of these edges, giving a triangulation $\tri_2$ with twice as many tetrahedra as $\tri_1.$ So in total we have $4 \times 3 \times 2 = 24$ times as many tetrahedra as we started with, and we hope the reader is convinced that the overall effect is the same as applying a barycentric subdivision to each tetrahedron! The moves are illustrated in Figures~\ref{barycentric_subdiv_2d}--\ref{barycentric_subdiv_3d_edge_top_view}.

\begin{figure}[htbp]
\centering
\subfloat[]{
\centering
\labellist
\tiny\hair 2pt
\pinlabel {0} at -11 161
\pinlabel {0} at 395 161
\pinlabel {0} at 191 35
\pinlabel {0} at 191 285
\endlabellist
\includegraphics[width=0.28\textwidth]{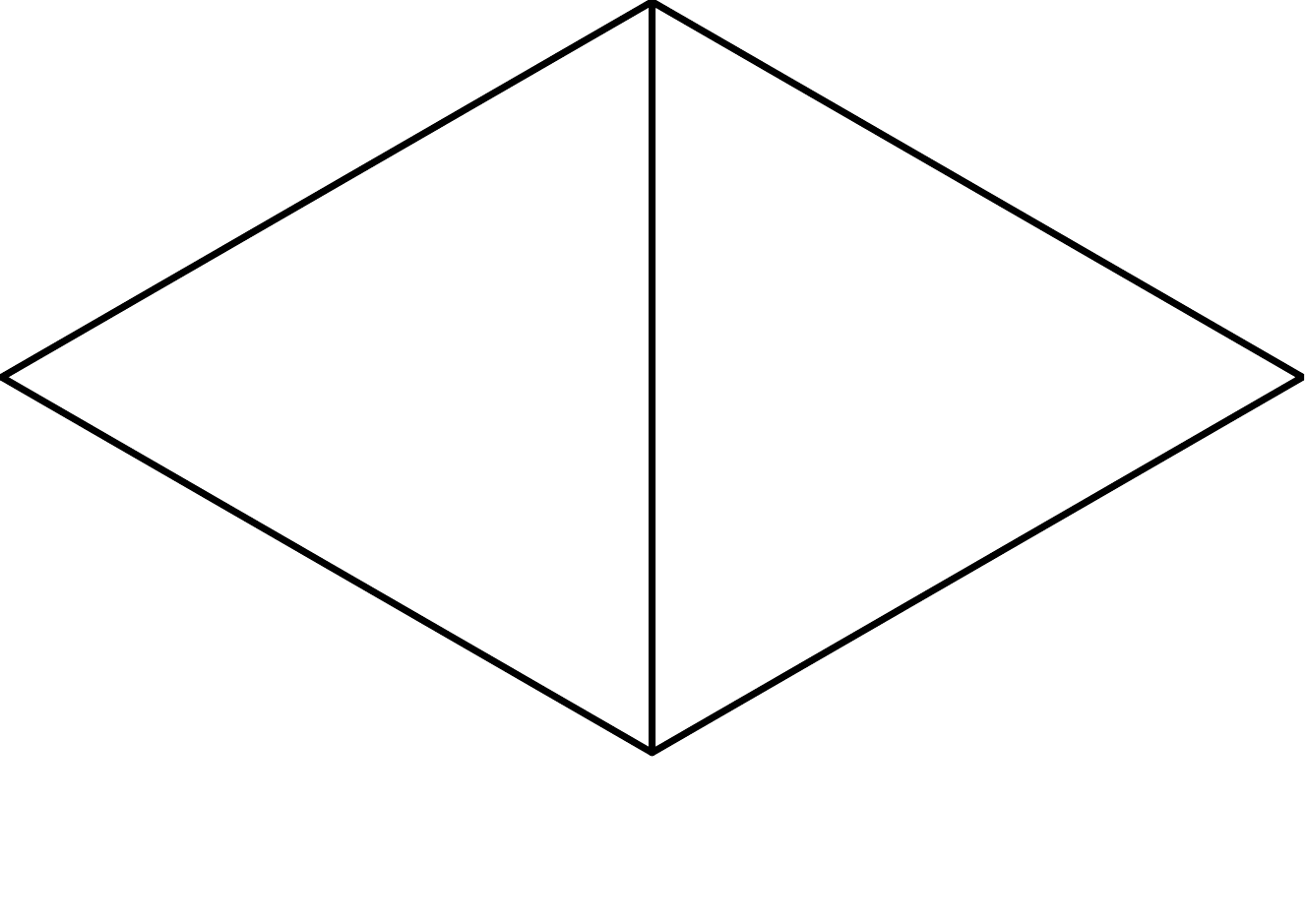}
}
\quad
\subfloat[]{
\centering
\labellist
\tiny\hair 2pt
\pinlabel {0} at -11 161
\pinlabel {0} at 395 161
\pinlabel {0} at 191 35
\pinlabel {0} at 191 285
\pinlabel {2} at 121 177
\pinlabel {2} at 261 177
\endlabellist
\includegraphics[width=0.28\textwidth]{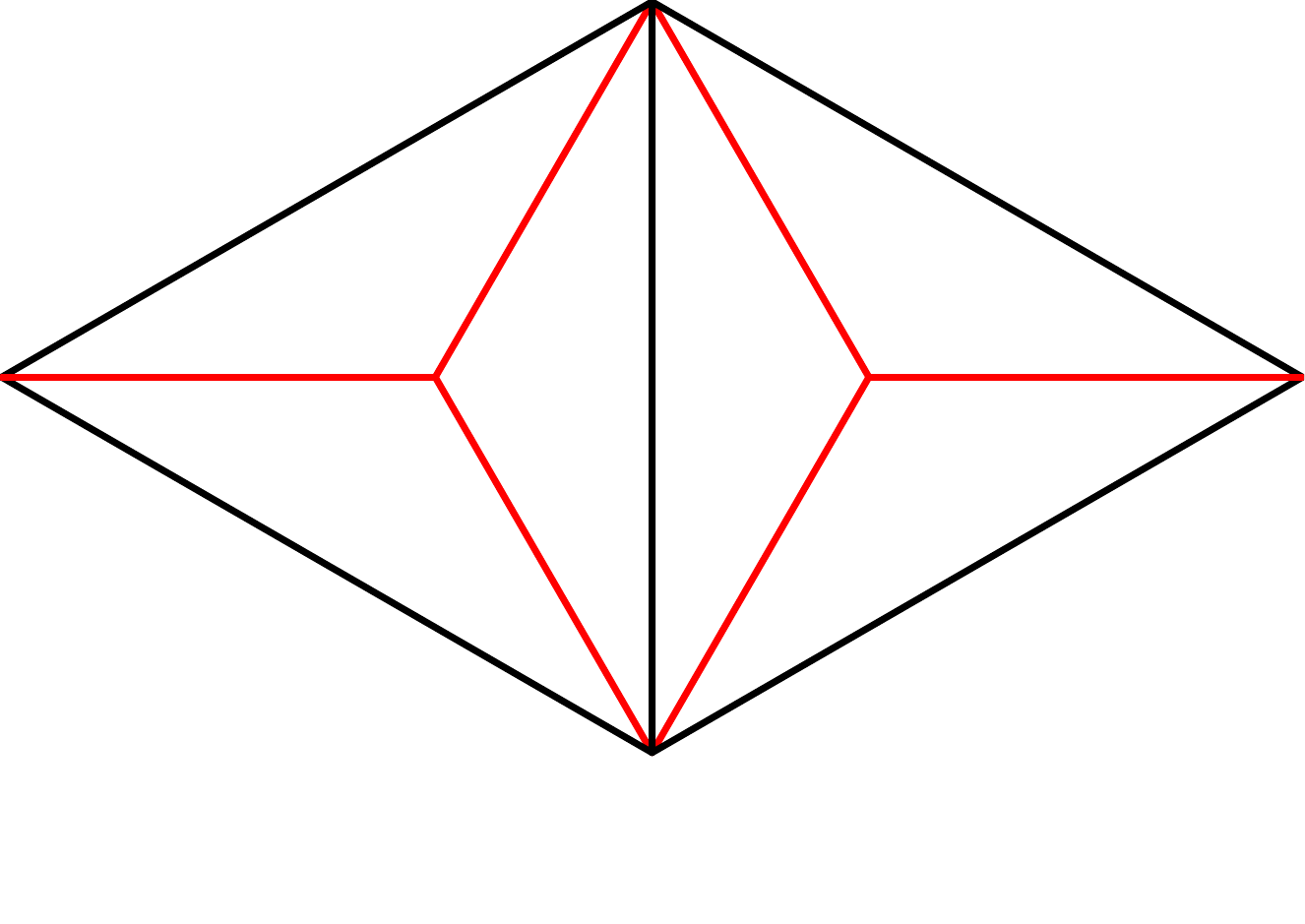}
}
\hspace{-1.2cm}
\subfloat[]{
\centering
\labellist
\tiny\hair 2pt
\pinlabel {0} at 191 35
\pinlabel {0} at 191 285
\pinlabel {2} at 114 161
\pinlabel {2} at 268 161
\pinlabel {1} at 224 177
\endlabellist
\includegraphics[width=0.28\textwidth]{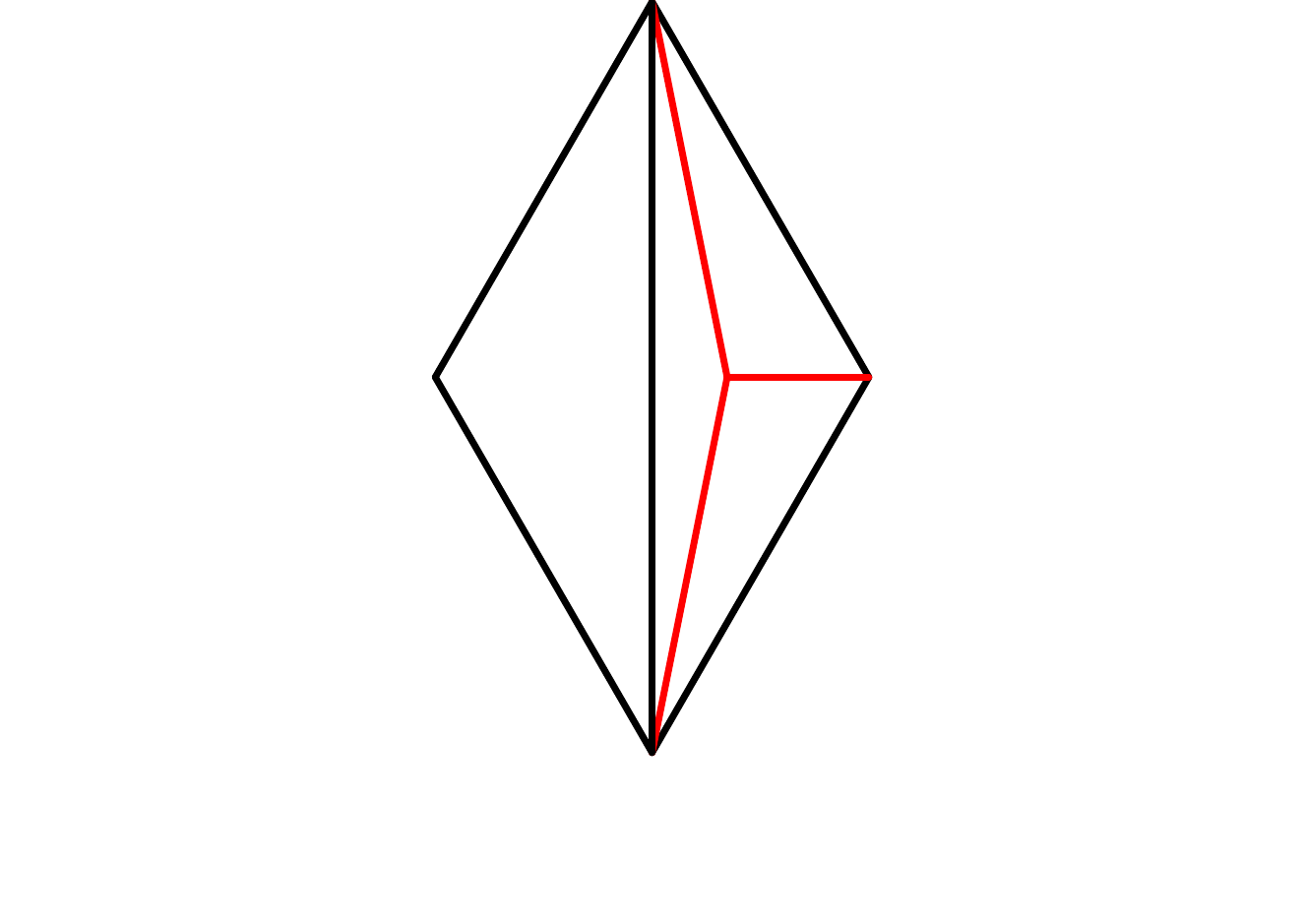}
}
\hspace{-2.8cm}
\subfloat[]{
\centering
\labellist
\tiny\hair 2pt
\pinlabel {0} at 191 35
\pinlabel {0} at 191 285
\pinlabel {2} at 114 161
\pinlabel {2} at 268 161
\pinlabel {1} at 224 177
\endlabellist
\includegraphics[width=0.28\textwidth]{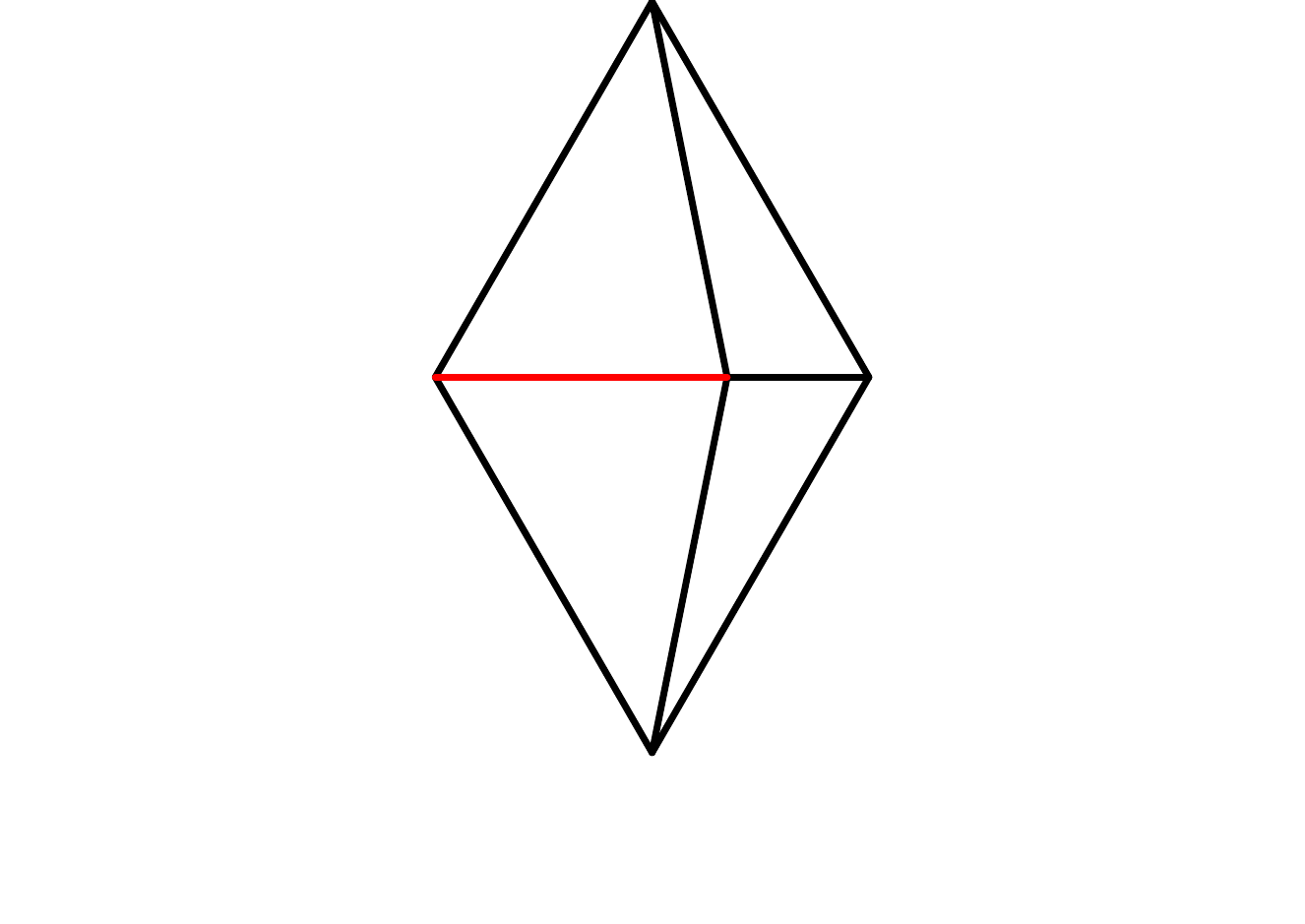}
}
\caption{Barycentric subdivision for two-dimensional triangulations via bistellar moves.} 
\label{barycentric_subdiv_2d}
\end{figure}
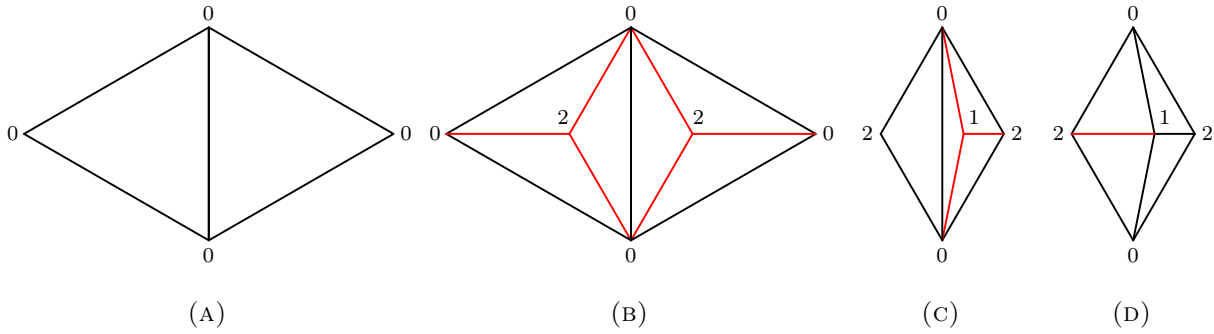

\begin{figure}[htbp]
\centering
\subfloat[$\cT$]{
\labellist
\tiny\hair 2pt
\pinlabel {0} at -14 142
\pinlabel {0} at 162 -19
\pinlabel {0} at 270 56
\pinlabel {0} at 225 336
\pinlabel {0} at 506 121
\endlabellist
\includegraphics[width=0.28\textwidth]{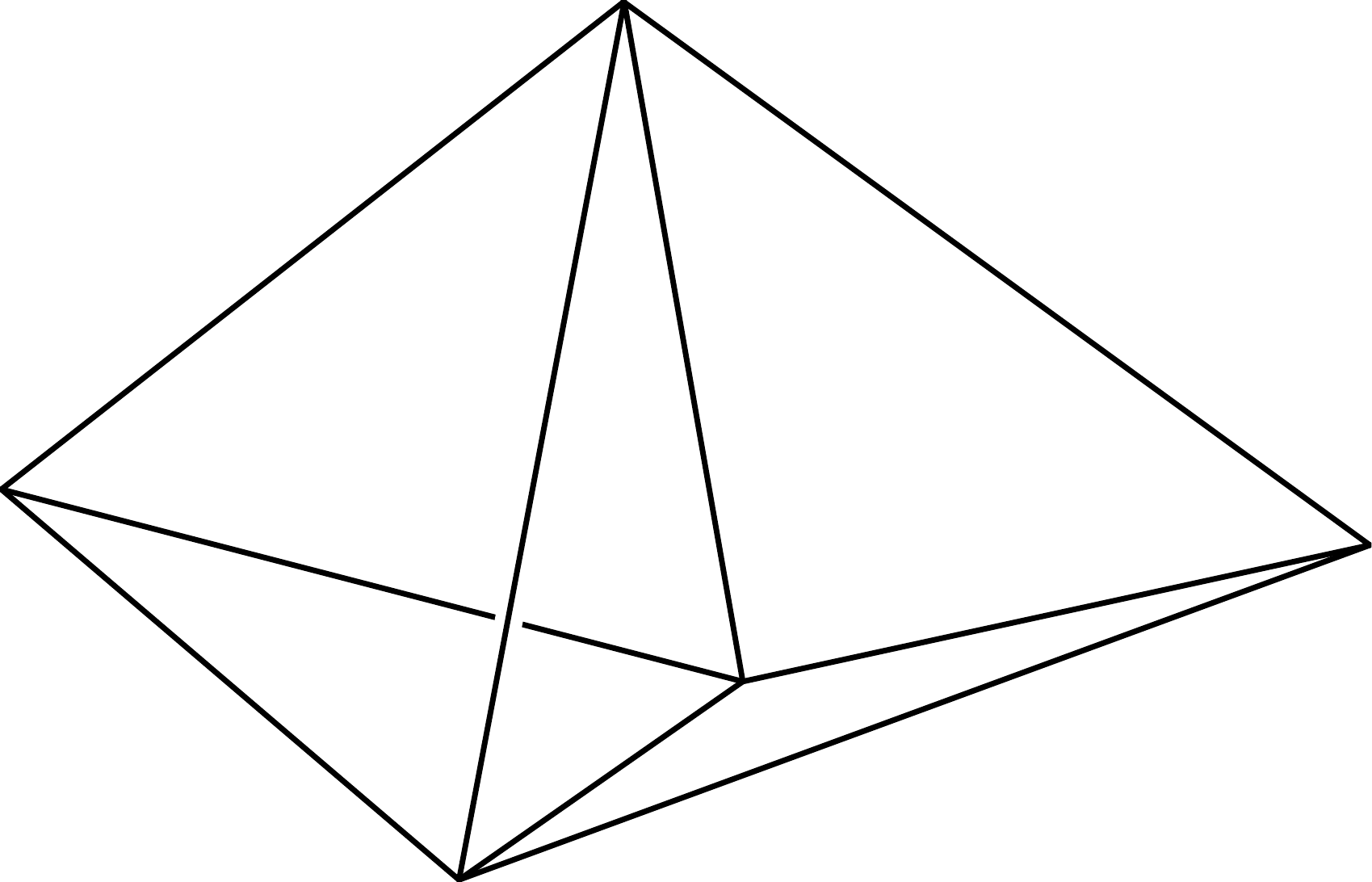}
}
\quad
\subfloat[$\cT'$]{
\labellist
\tiny\hair 2pt
\pinlabel {0} at -14 142
\pinlabel {0} at 162 -19
\pinlabel {0} at 270 56
\pinlabel {0} at 225 336
\pinlabel {0} at 506 121
\pinlabel {3} at 155 155
\pinlabel {3} at 298 150
\endlabellist
\includegraphics[width=0.28\textwidth]{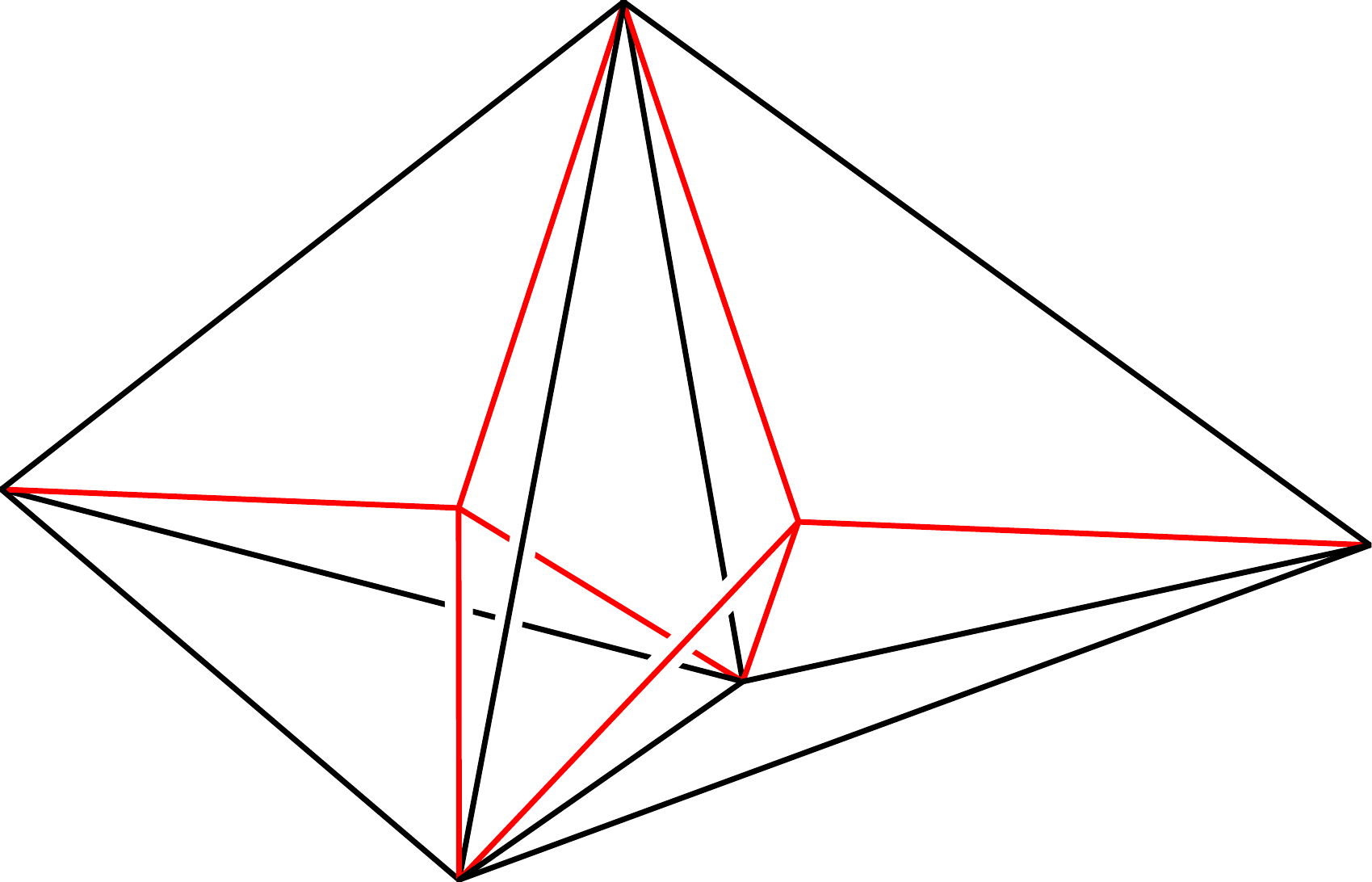}
}
\hspace{-1.2cm}
\subfloat[]{
\labellist
\tiny\hair 2pt
\pinlabel {0} at 162 -19
\pinlabel {0} at 272 53
\pinlabel {0} at 225 336
\pinlabel {3} at 147 139
\pinlabel {3} at 302 132
\pinlabel {2} at 227 140
\endlabellist
\includegraphics[width=0.28\textwidth]{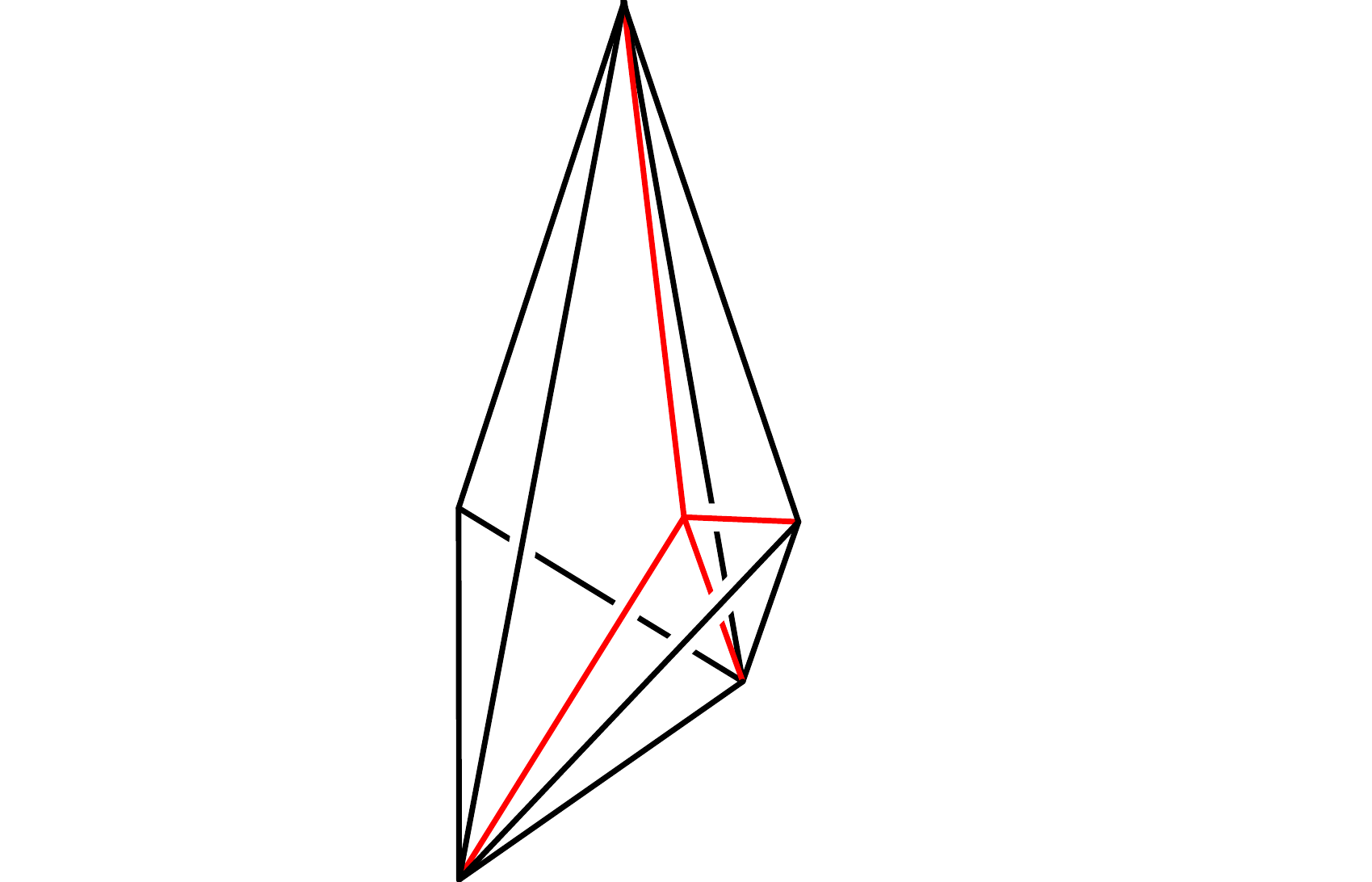}
}
\hspace{-2.8cm}
\subfloat[$\cT''$]{
\labellist
\tiny\hair 2pt
\pinlabel {0} at 162 -19
\pinlabel {0} at 272 53
\pinlabel {0} at 225 336
\pinlabel {3} at 147 139
\pinlabel {3} at 302 132
\pinlabel {2} at 227 151
\endlabellist
\includegraphics[width=0.28\textwidth]{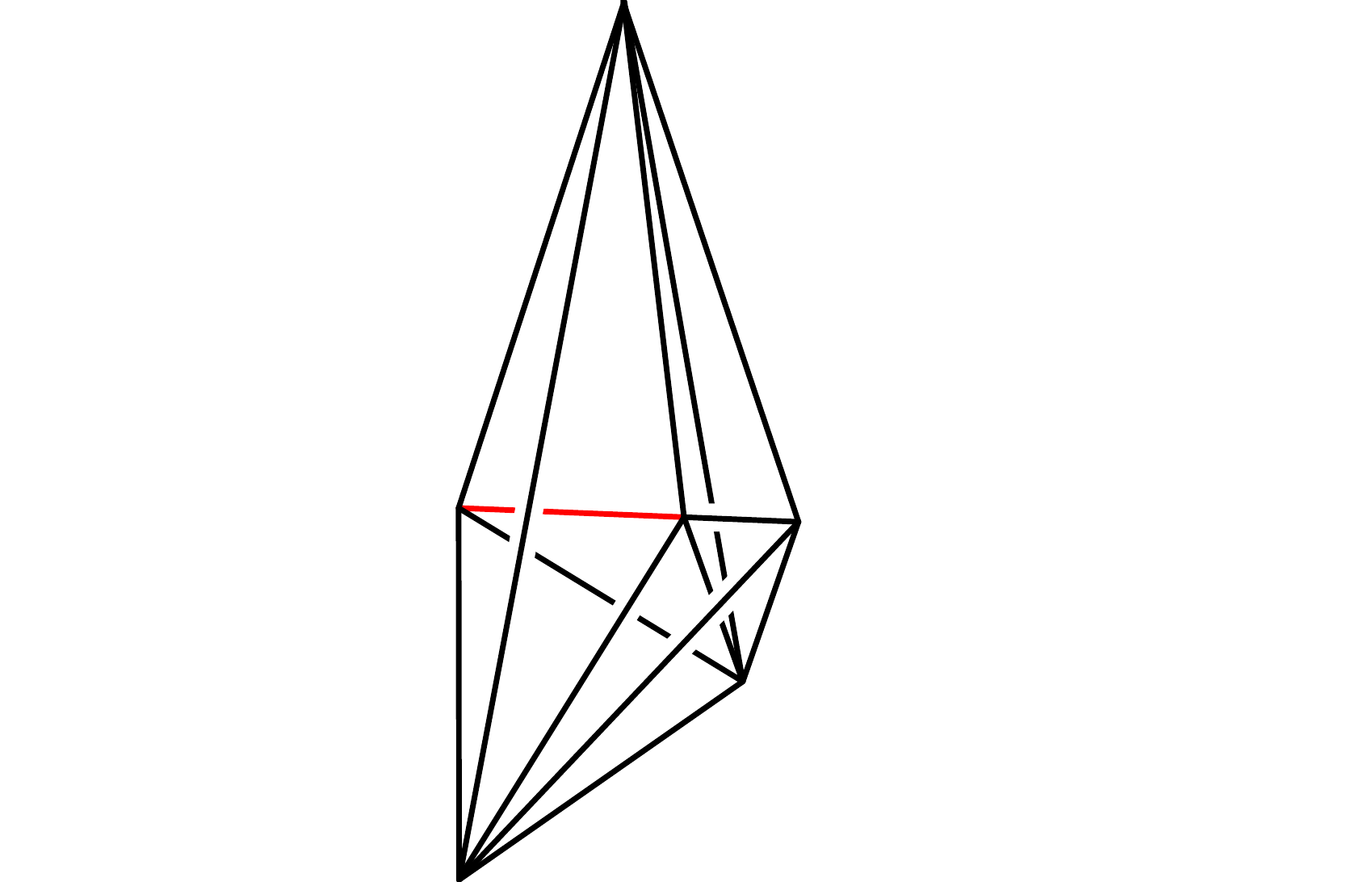}
}
\caption{Introducing the vertices corresponding to the three- and two-dimensional simplices for the barycentric subdivision of a three-dimensional triangulation. } 
\label{barycentric_subdiv_3d}
\end{figure}

\begin{figure}[htbp]
\centering
\subfloat[Before adding the vertex corresponding to an edge.]{
\labellist
\tiny\hair 2pt
\pinlabel {0} at 174 10
\pinlabel {0} at 174 385
\pinlabel {2} at -8 196
\pinlabel {3} at 27 120
\pinlabel {3} at 115 243
\pinlabel {2} at 253 103
\pinlabel {2} at 282 233
\pinlabel {3} at 365 173
\small
\pinlabel {$e$} at 186 176
\endlabellist
\includegraphics[width=0.31\textwidth]{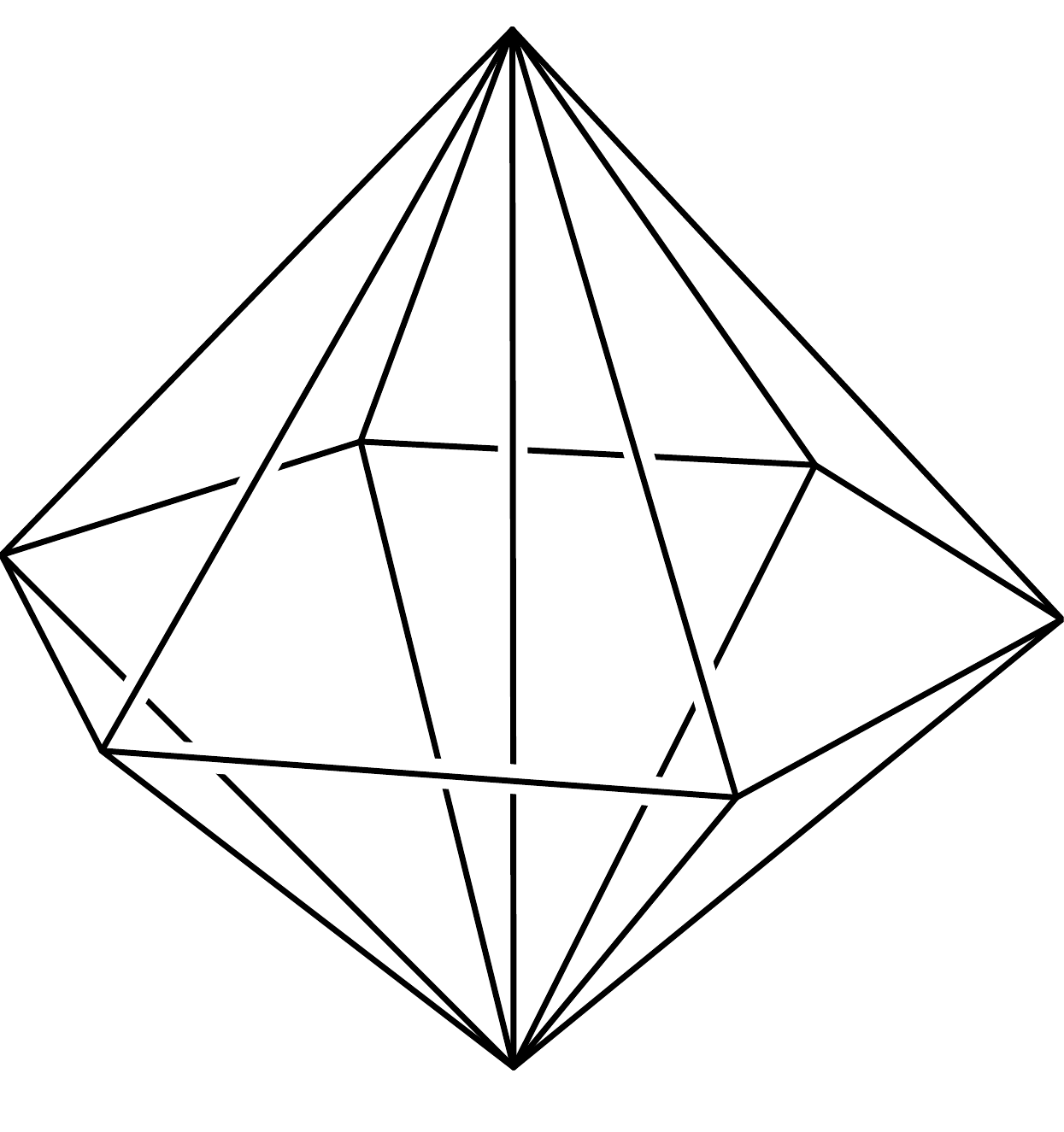}
}
\thinspace
\subfloat[The first step is to apply a 1-4 move.]{
\labellist
\tiny\hair 2pt
\pinlabel {0} at 174 10
\pinlabel {0} at 174 385
\pinlabel {2} at -8 196
\pinlabel {3} at 27 120
\pinlabel {3} at 115 243
\pinlabel {2} at 253 103
\pinlabel {2} at 282 233
\pinlabel {3} at 365 173
\pinlabel {1} at 150 165
\small
\pinlabel {$e$} at 186 176
\endlabellist
\includegraphics[width=0.31\textwidth]{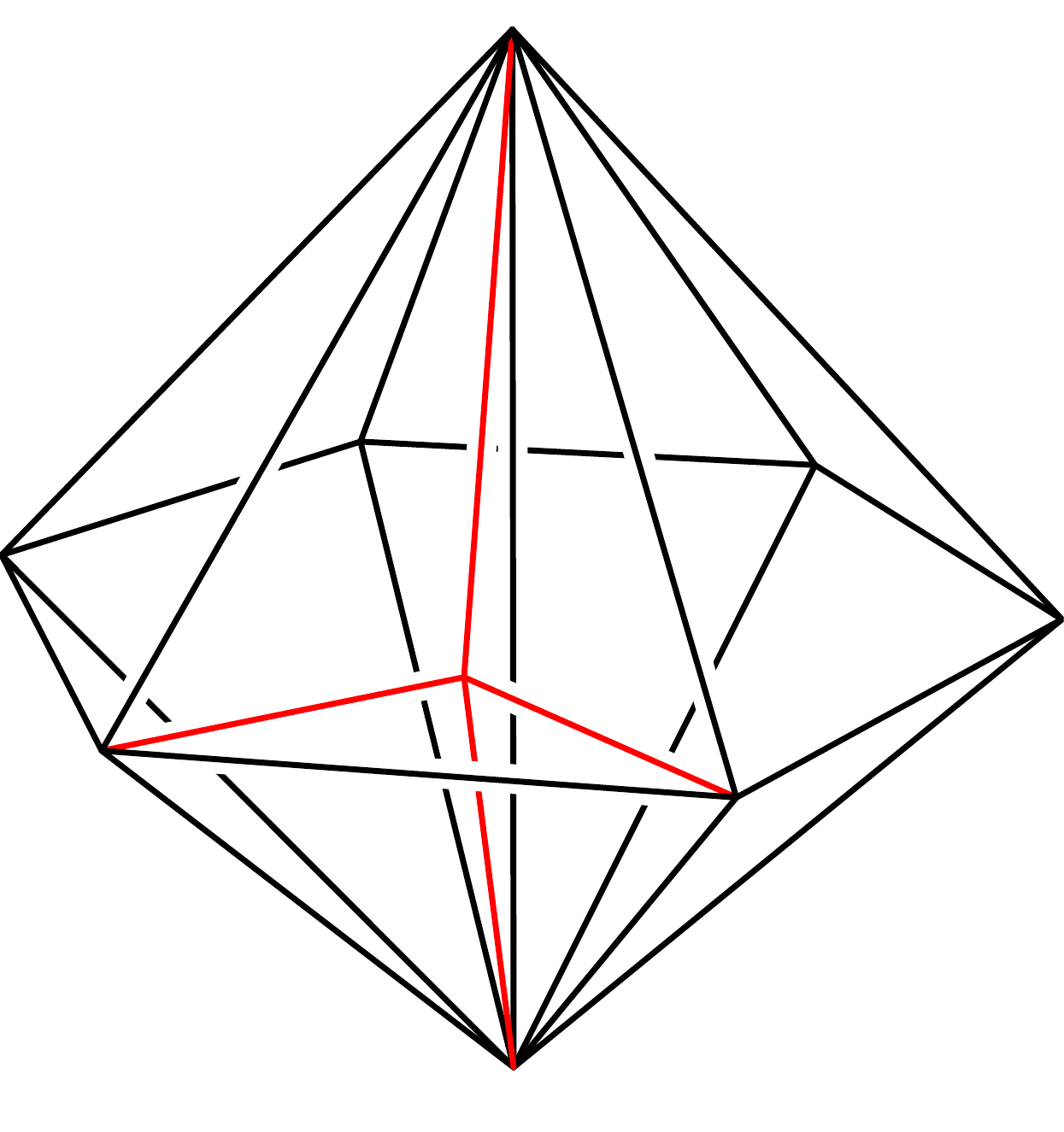}
}
\thinspace
\subfloat[The end result, after applying the subsequent steps shown in Figure~\ref{barycentric_subdiv_3d_edge_top_view}.]{
\labellist
\tiny\hair 2pt
\pinlabel {0} at 174 10
\pinlabel {0} at 174 385
\pinlabel {2} at -8 196
\pinlabel {3} at 27 120
\pinlabel {3} at 115 243
\pinlabel {2} at 253 103
\pinlabel {2} at 282 233
\pinlabel {3} at 365 173
\pinlabel {1} at 164 165
\endlabellist
\includegraphics[width=0.31\textwidth]{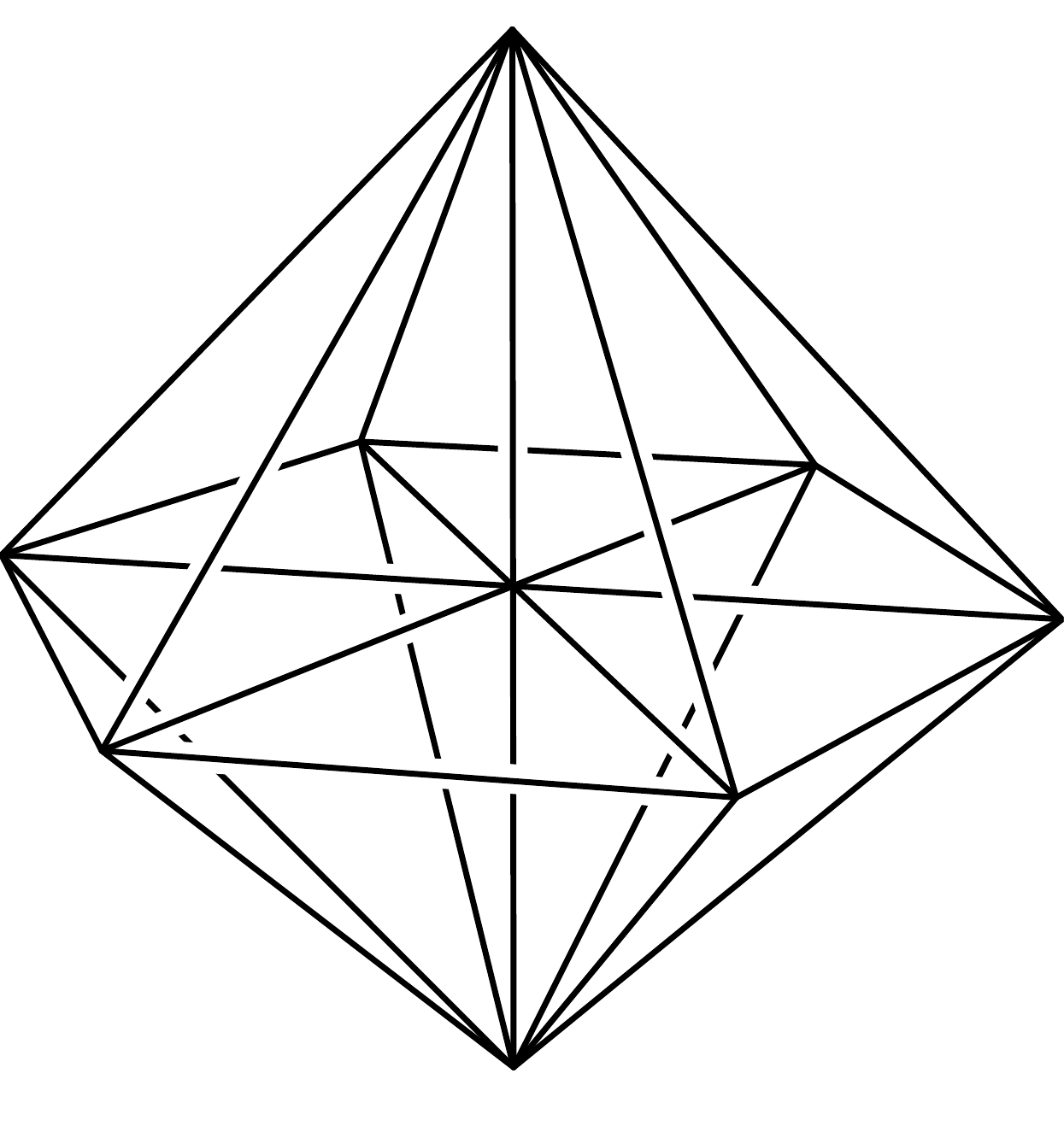}
}
\caption{Introducing the vertex corresponding to a one-dimensional simplex for the barycentric subdivision of a three-dimensional triangulation.}
\label{barycentric_subdiv_3d_edge}
\end{figure}

\begin{figure}[htbp]
\centering
\subfloat[]{
\includegraphics[width=0.15\textwidth]{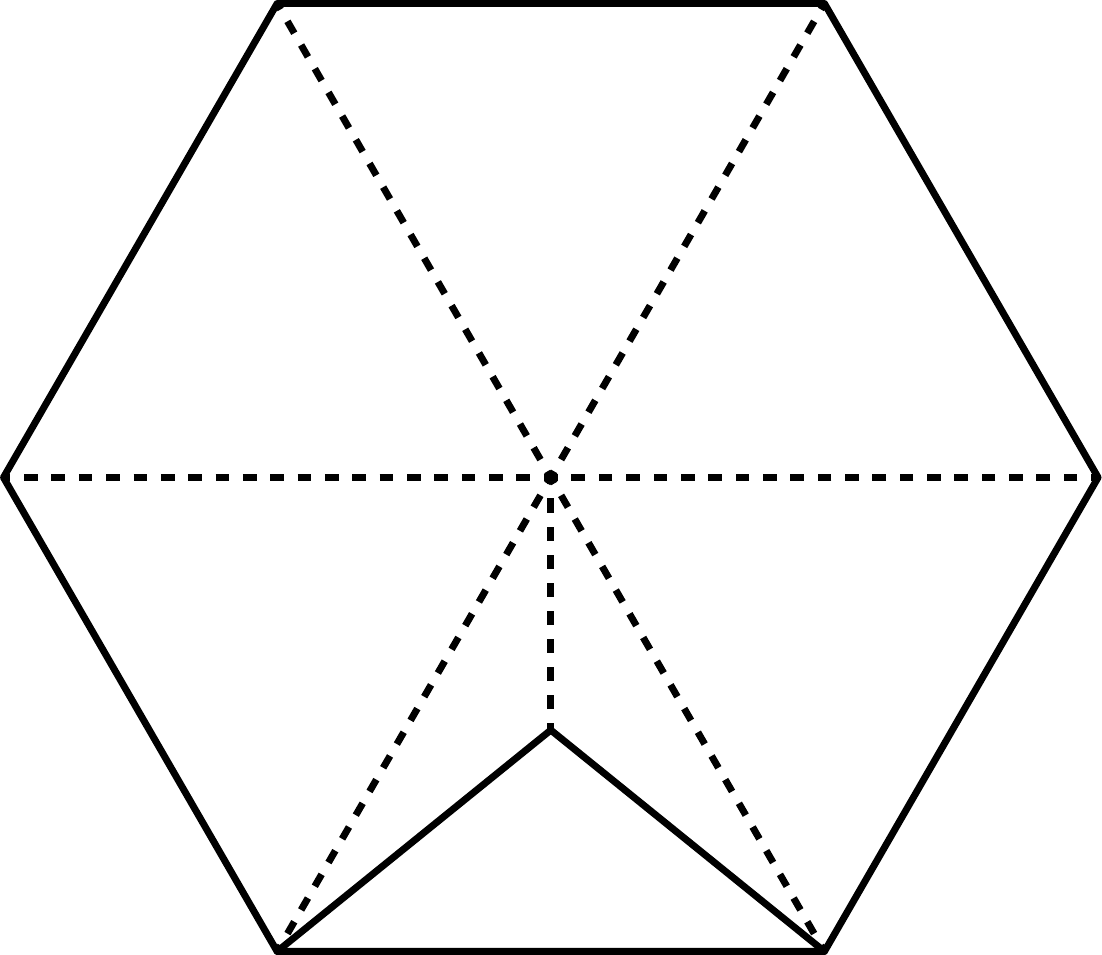}
}
\subfloat[]{
\includegraphics[width=0.15\textwidth]{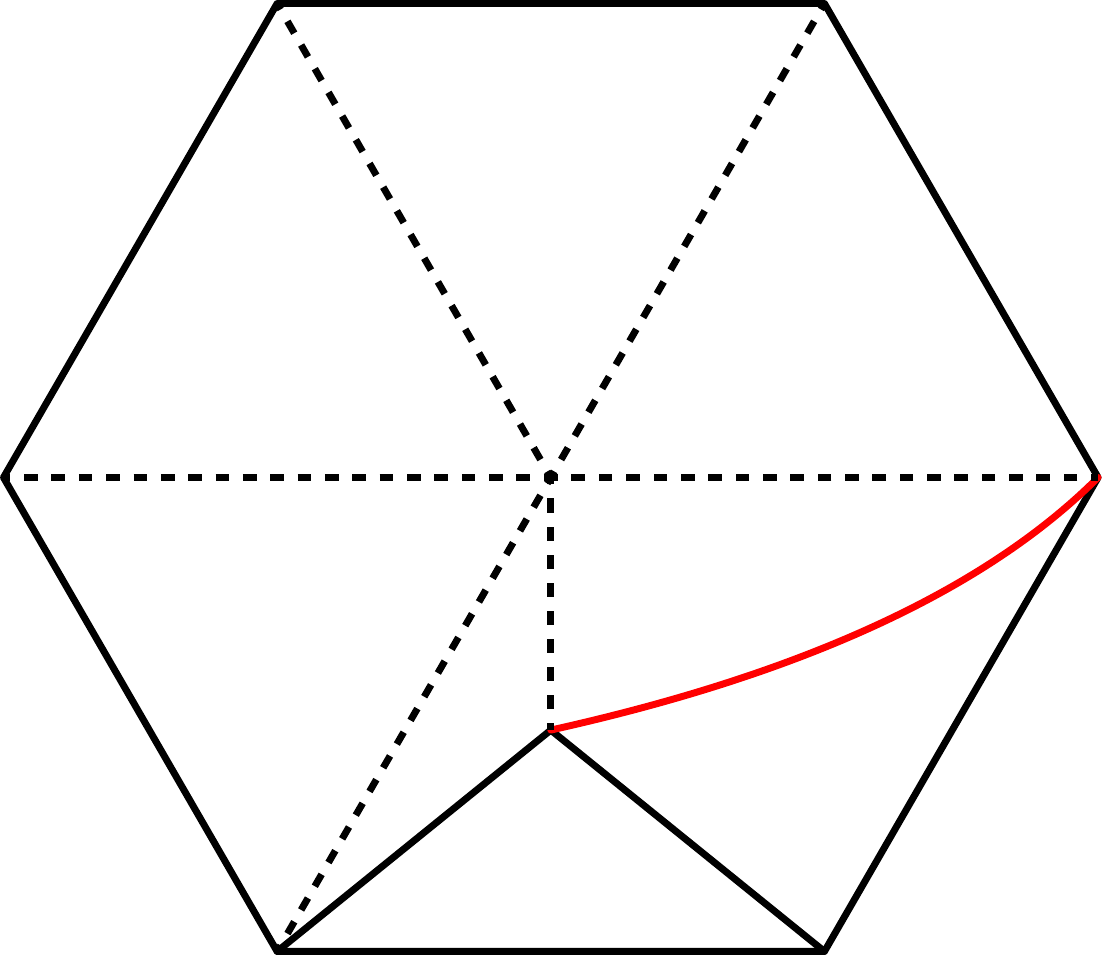}
}
\subfloat[]{
\includegraphics[width=0.15\textwidth]{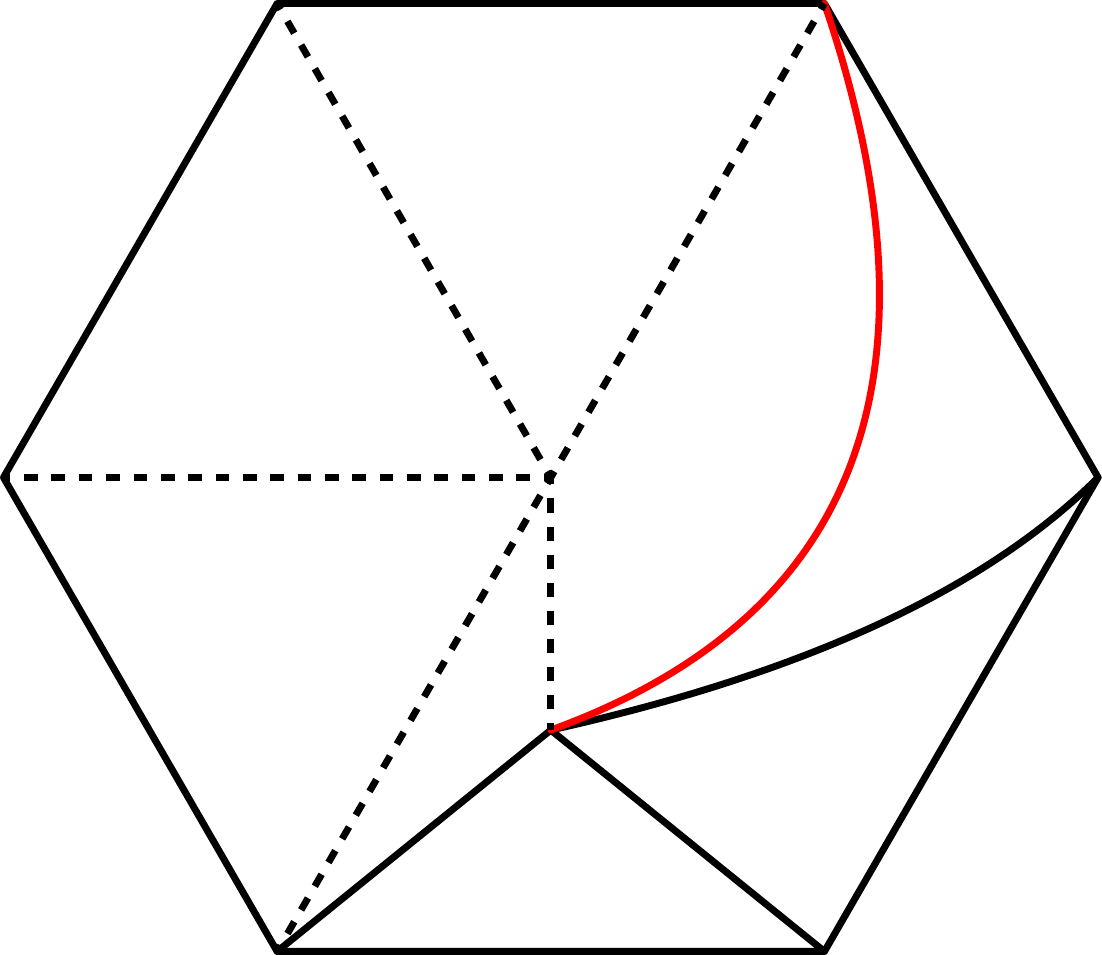}
}
\subfloat[]{
\includegraphics[width=0.15\textwidth]{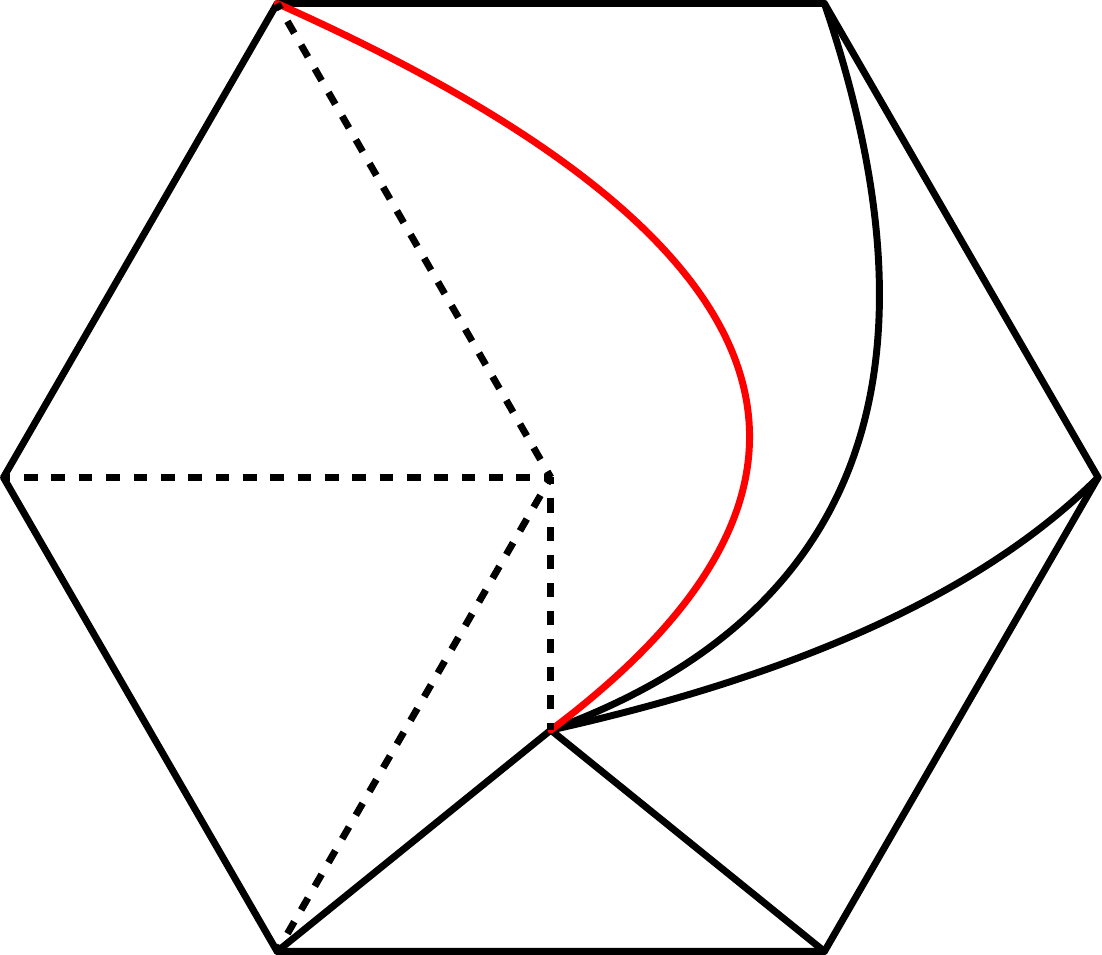}
}
\subfloat[]{
\includegraphics[width=0.15\textwidth]{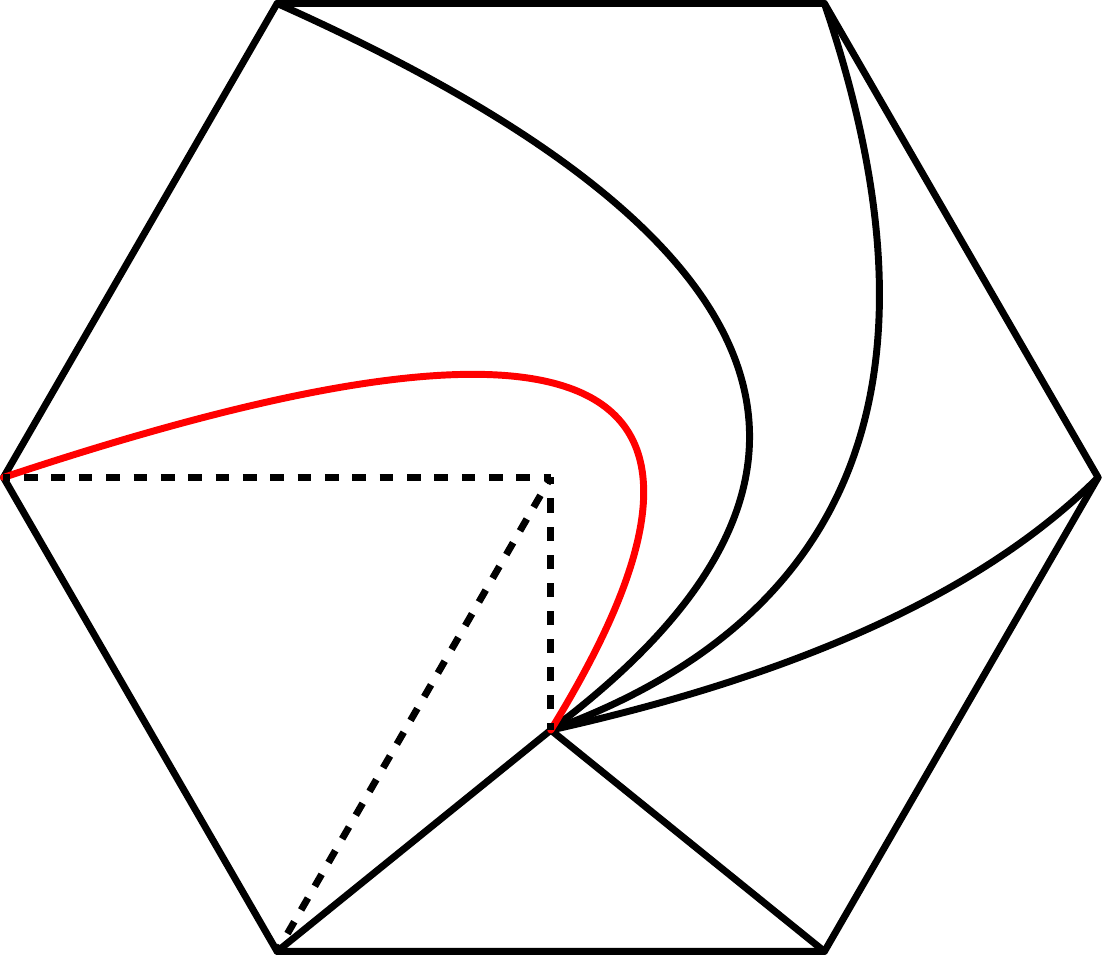}
}
\subfloat[]{
\includegraphics[width=0.15\textwidth]{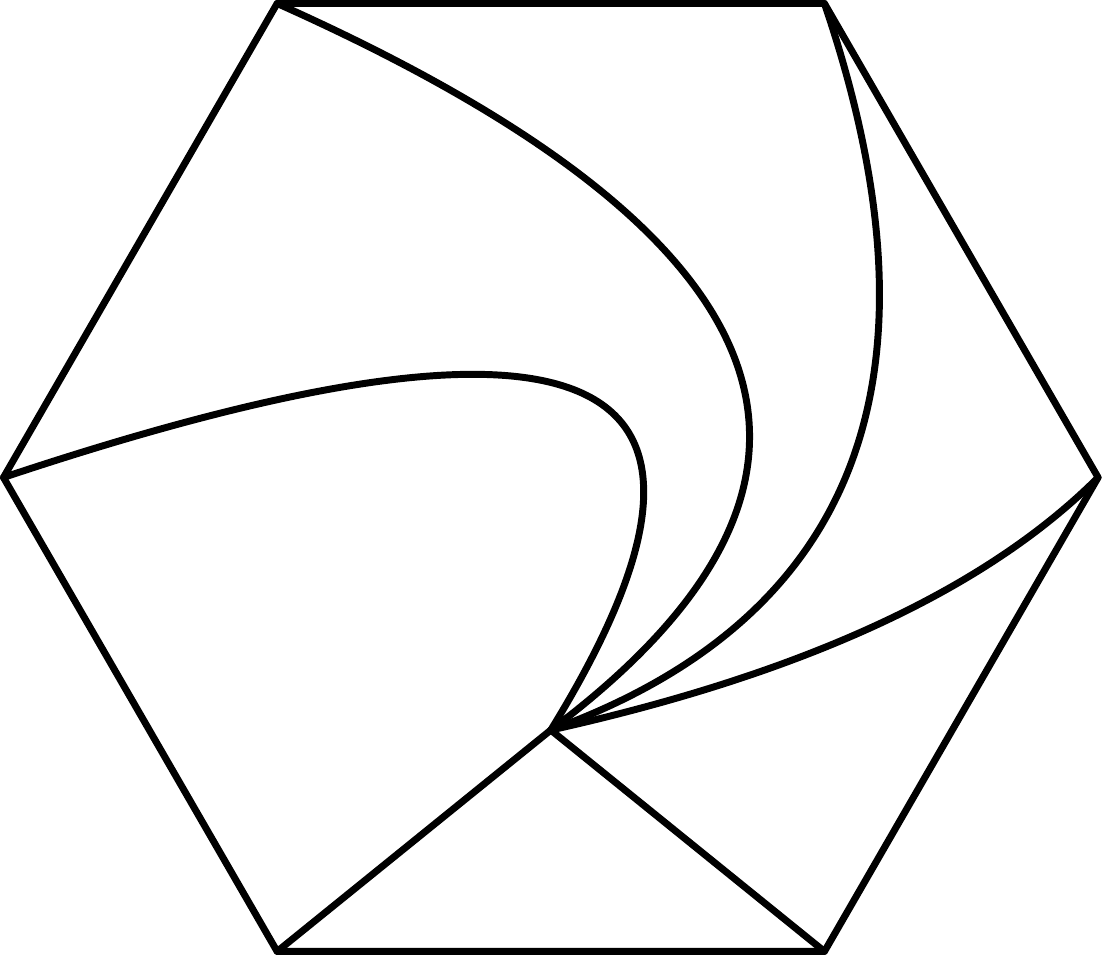}
}
\caption{The subsequent steps in adding the vertices corresponding to the one-dimensional simplices for the barycentric subdivision of a three-dimensional triangulation. The horizontal cross-section through the tetrahedra is shown. Edges are drawn with a solid line, while intersections of faces with the cross-section are drawn with a dotted line.} 
\label{barycentric_subdiv_3d_edge_top_view}
\end{figure}


\subsection{Common refinements}
\label{sec:refinements}

Even though we stated in the introduction that we will prove our main result assuming Theorems~\ref{pachner} and \ref{thm:Banagl-Friedman}, we cannot resist sketching proofs of these results. 

We are given triangulations $\cT_0$ and $\cT_1$ of a 3--manifold $M.$ To apply the classical results, we first apply two barycentric subdivisions to each in order to obtain the simplicial triangulations $\cT'_0 = ( \widetilde{\Delta}_0, \Phi_0, p_0)$ and $\cT'_1 = ( \widetilde{\Delta}_1, \Phi_1, p_1)$ respectively. We have quotient maps $p_0\co  \widetilde{\Delta}_0\to M$ and $p_1\co  \widetilde{\Delta}_1\to M.$ We view $M$ as the quotient space $\Delta_0/ \Phi_0,$ so that $p_0$ is a linear map on simplices. The work of Moise (see also Brown~\cite{Brown-hauptvermutung-1969} and Hamilton~\cite{Hamilton-triangulation-1976}) now allows us to compose $p_1$ with an isotopy $\iota\co M\to M$  so that the restriction of $\iota\circ p_1$ is also linear on each simplex. In particular, Alexander's theorem~\cite{Alexander-combinatorial-1930} (see also Lickorish~\cite{lickorish_simplicial_moves} and Turaev-Viro~\cite{Turaev-state-1992}) now allows us to transform $\cT'_0$ and $\cT'_1$ to a common refinement by applying stellar moves. 

We have shown above that barycentric subdivisions and stellar moves can be performed through sequences of bistellar moves. This completes our sketch of the proof of Theorem~\ref{pachner}.

The class of pseudo-manifolds considered in this paper is less general than that considered by Banagl and Friedman~\cite{banagl-friedman}. Our pseudo-manifolds are merely compact manifolds with boundary, where we have added a cone to each boundary component. It is therefore straightforward to sketch a proof of Theorem~\ref{thm:Banagl-Friedman} analogous to the above. We have the same set-up as before, and apply two barycentric subdivisions, giving quotient maps $p_0\co  \widetilde{\Delta}_0\to M$ and $p_1\co  \widetilde{\Delta}_1\to M.$ The work of Brown~\cite{Brown-hauptvermutung-1969} (see \cite[Theorem 2.13]{banagl-friedman}) now allows us to compose $p_1$ with an isotopy $\iota\co M\to M$ so that the restriction of $\iota\circ p_1$ is also linear on each simplex. Due to the topology of the pseudo-manifold, the isolated ideal vertices are necessarily fixed under the isotopy. Alexander's theorem again provides stellar moves resulting in a common refinement, and we conclude as above that there is the desired sequence of bistellar moves.


\subsection{Special spines}
\label{sec:spines}

We now describe a 2--dimensional object, a \emph{spine}, that encodes $M$, a 3--dimensional manifold or pseudo-manifold. To begin with, suppose you have a triangulation $\tri$ of $M,$ and pass to the second barycentric subdivision $\tri''$ of $\tri.$ This has the following nice property. Let $v$ be a vertex of $\tri''$ that corresponds to a vertex of $\tri.$ Then the collection $\Lk(v)$ of all tetrahedra in $\tri''$ containing $v$ is a cone on a triangulated surface with cone point $v,$ and for any two distinct vertices of $\tri$, the two corresponding surfaces are disjoint. Hence if we remove $\Lk(v)$ for each vertex $v$ of $\tri$ from $\tri'',$ we obtain the triangulation of a compact 3--dimensional manifold $M^c$ with boundary, and $M$ is obtained from $M^c$ by adding the cones on its boundary components. Hence we can reconstruct $M$ from $M^c.$ 

\begin{figure}[htbp]
\centering
\subfloat[A butterfly inside of a tetrahedron.]{
\includegraphics[width=0.2\textwidth]{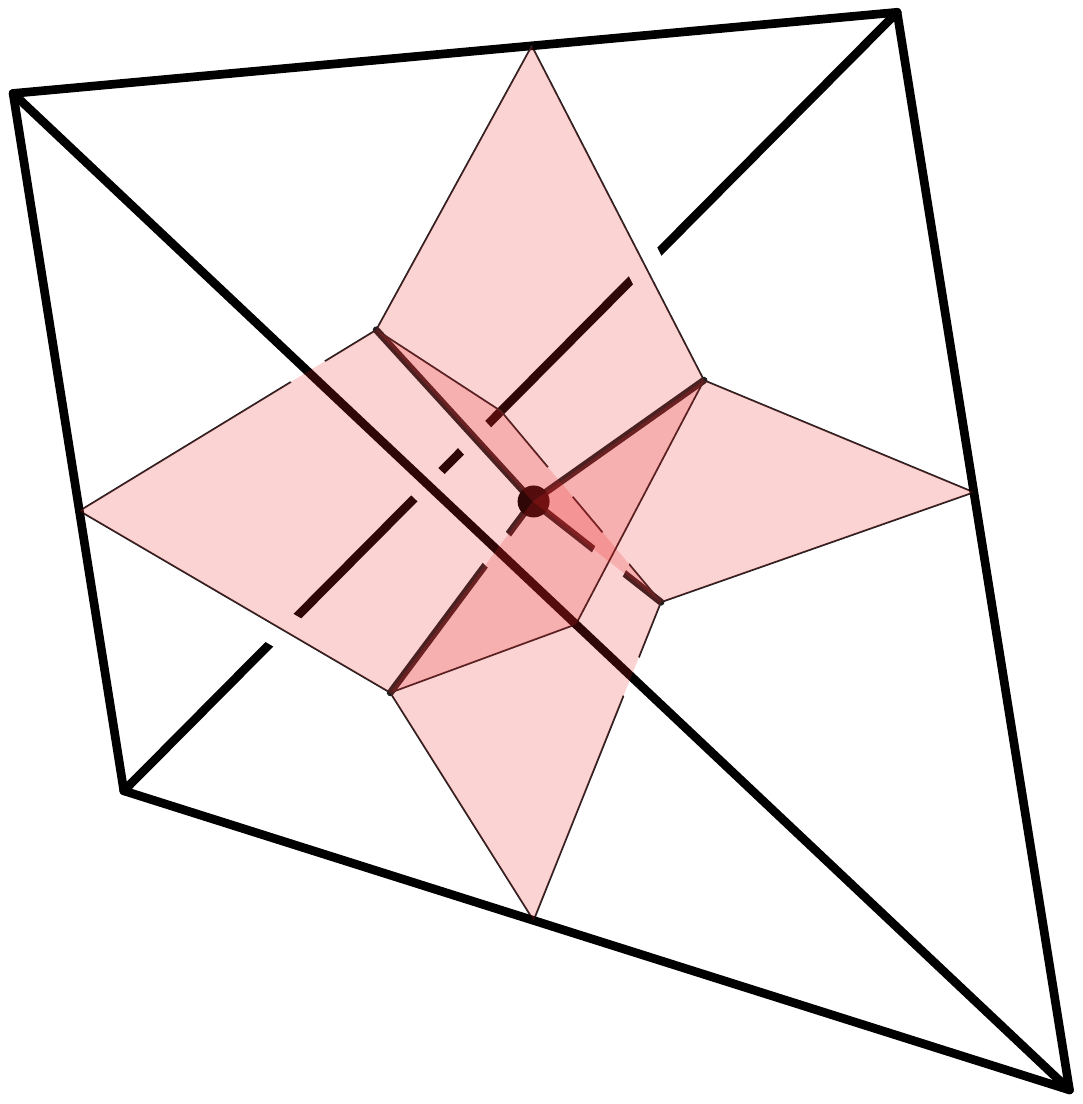}
\label{butterfly}
}
\quad
\subfloat[Three different types of point on a spine. From left to right: a non-singular point, a triple point, and a vertex.]{
\includegraphics[width=0.65\textwidth]{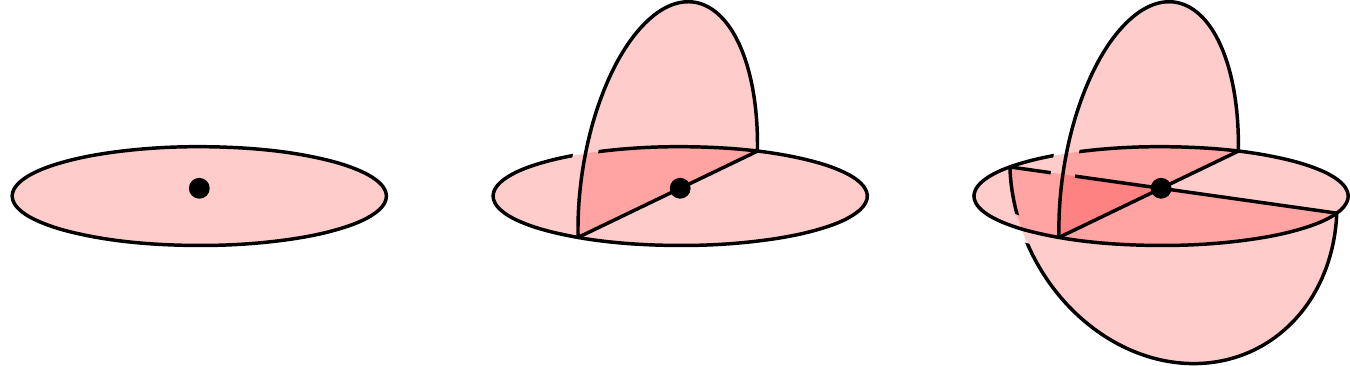}
\label{simple_spine_points}
}
\caption{Spines.} 
\label{spine_definitions}
\end{figure}

One can now collapse $M^c$ onto a 2--dimensional complex $S$ that is the union of all triangles in the first barycentric subdivision of $\tri$ with the property that they do not contain any vertex of $\tri.$ For each tetrahedron of $\tri$ this is the butterfly shown in Figure~\ref{butterfly}. Note that $M^c$ is homeomorphic to a regular neighbourhood of $S$ in $M^c.$ We claim that $S$ determines $M^c$ (and hence $M$) \emph{independent} of its embedding in $M^c.$ The important properties of $S$ that make this possible are the following. 

Each point in $S$ has a neighbourhood that is homeomorphic to one of the three models shown in Figure~\ref{simple_spine_points}. 
{
A point is \emph{singular} if it does not have a neighbourhood homeomorphic to a disc. The set of all singular points consists of the triple points and vertices, and is naturally a graph. The nodes of this graph are called vertices of the spine.
}
The complement of the set of singular points in $S$ is a surface, and we claim that it is a disjoint union of open discs. This comes from the local structure: each connected component of the complement is the union of all triangles in $\tri'$ having a vertex on the barycentre of an edge in $\tri.$ If one removes the vertices from the singular set, then each connected component is an interval---this is clear because we have a graph. The relationship with the triangulation is again evident: the interval connects the barycentres of adjacent tetrahedra through the barycentre of a common face. These properties are summarised by saying that $S$ is a \emph{special spine} of $M.$

Having observed this, one can show that if $S_0$ and $S_1$ are special spines of $M_0$ and $M_1,$ with the property that $S_0$ and $S_1$ are homeomorphic, then $M_0$ and $M_1$ are homeomorphic. The reader may find pleasure in proving this result, or can find it in Matveev's book~\cite[Theorem 1.1.17]{matveev_book}. 


Special spines and triangulations are dual objects: each vertex of a spine is dual to a tetrahedron; each edge of a spine is dual to a triangle; each face of a spine is dual to an edge of the triangulation. The complementary regions of a spine in the pseudo-manifold are dual to the vertices of the triangulation. Indeed, each such region is the cone over a surface, where the surface is homeomorphic to the vertex link. In particular, material vertices correspond to three-balls.

\subsection{Moves on triangulations and spines}

Using this duality, we should be able to view moves performed on triangulations in the dual spine picture. For example, the bistellar moves we introduced earlier can be seen as moves performed on spines, as shown in Figure~\ref{1-4 and 2-3 moves}. Pictures of spines are often more flexible, with many ways to draw the same spine. Figure~\ref{alt-spine_slide_2-3} shows an alternative picture for the 2-3 move.

\begin{figure}[htbp]
\subfloat[The 1-4 and 4-1 moves.]{
\centering
\labellist
\small\hair 2pt
\pinlabel 1-4 at 400 365
\pinlabel 4-1 at 400 278
\endlabellist
\includegraphics[width=0.5\textwidth]{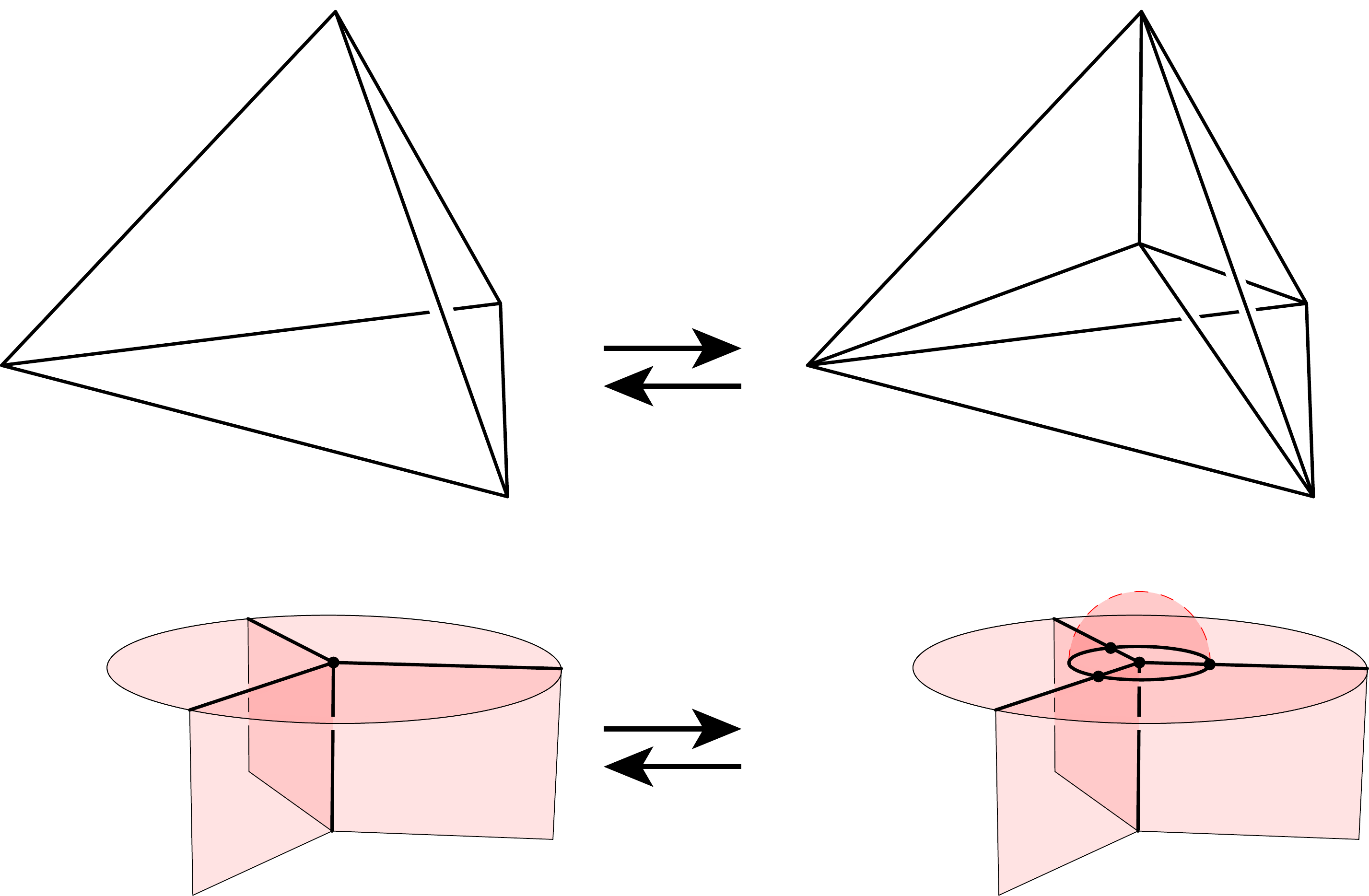}
}
\qquad
\subfloat[The 2-3 and 3-2 moves.]{
\centering
\labellist
\small\hair 2pt
\pinlabel 2-3 at 380 400
\pinlabel 3-2 at 380 314
\endlabellist
\includegraphics[width=0.4\textwidth]{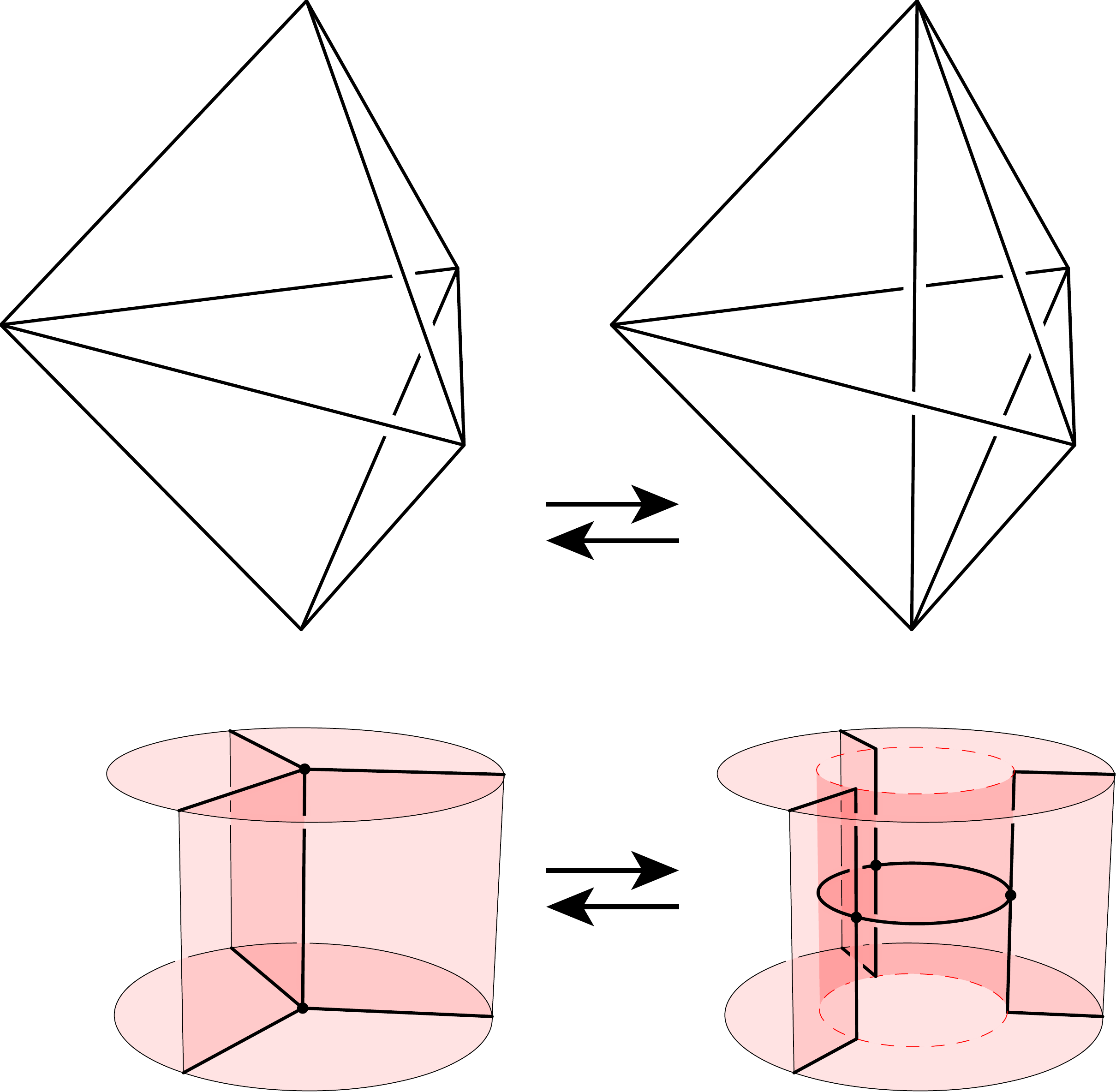}
}

\caption{Bistellar moves on triangulations and their dual special spines.}
\label{1-4 and 2-3 moves}
\end{figure}

\begin{figure}[htbp]

\centering
\labellist
\small\hair 2pt
\pinlabel 2-3 at 320 165
\endlabellist
\includegraphics[width=0.6\textwidth]{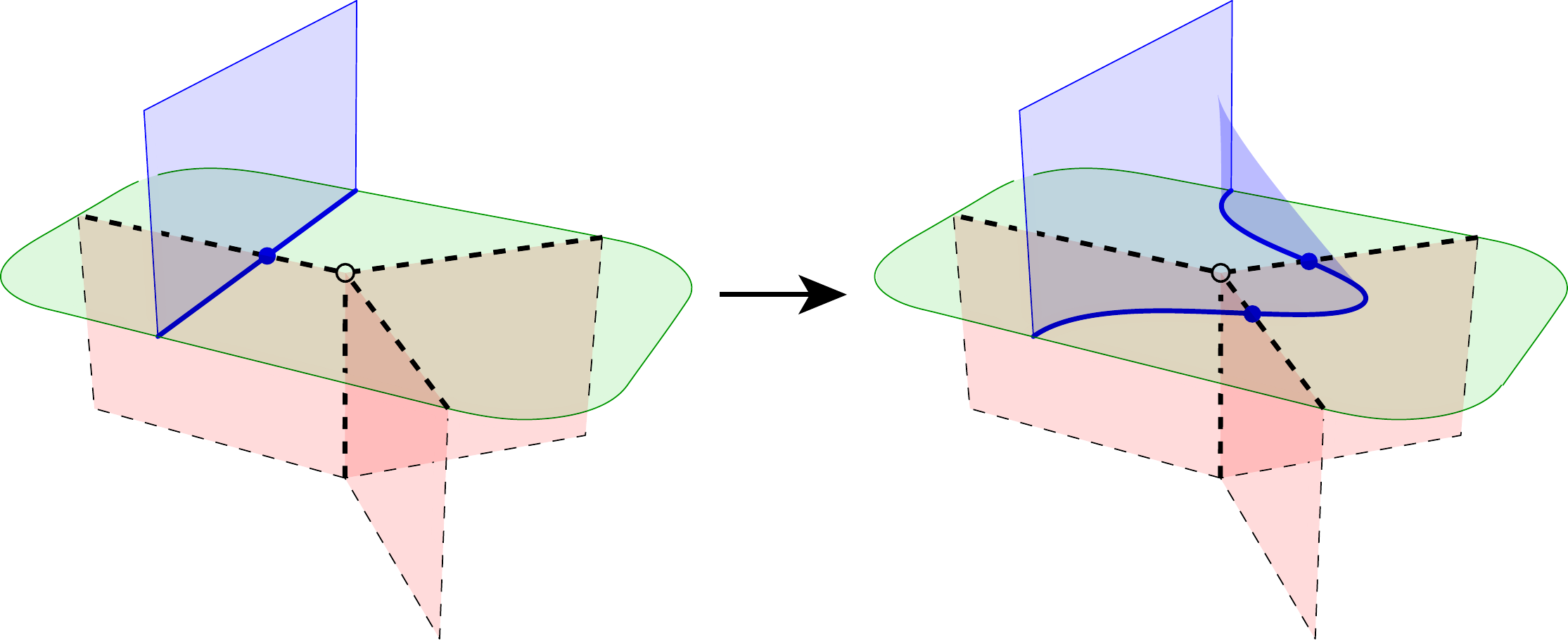}

\caption{An alternative picture of the 2-3 move on a spine.}
\label{alt-spine_slide_2-3}
\end{figure}


Similarly, we can see our barycentric subdivision moves (as implemented using bistellar moves) in the dual spine picture. These are shown in Figures~\ref{barycentric_subdiv_3d_spine} and \ref{barycentric_subdiv_3d_edge_spine}.

\begin{figure}[htbp]
\centering
\subfloat[$\cT$]{
\includegraphics[width=0.22\textwidth]{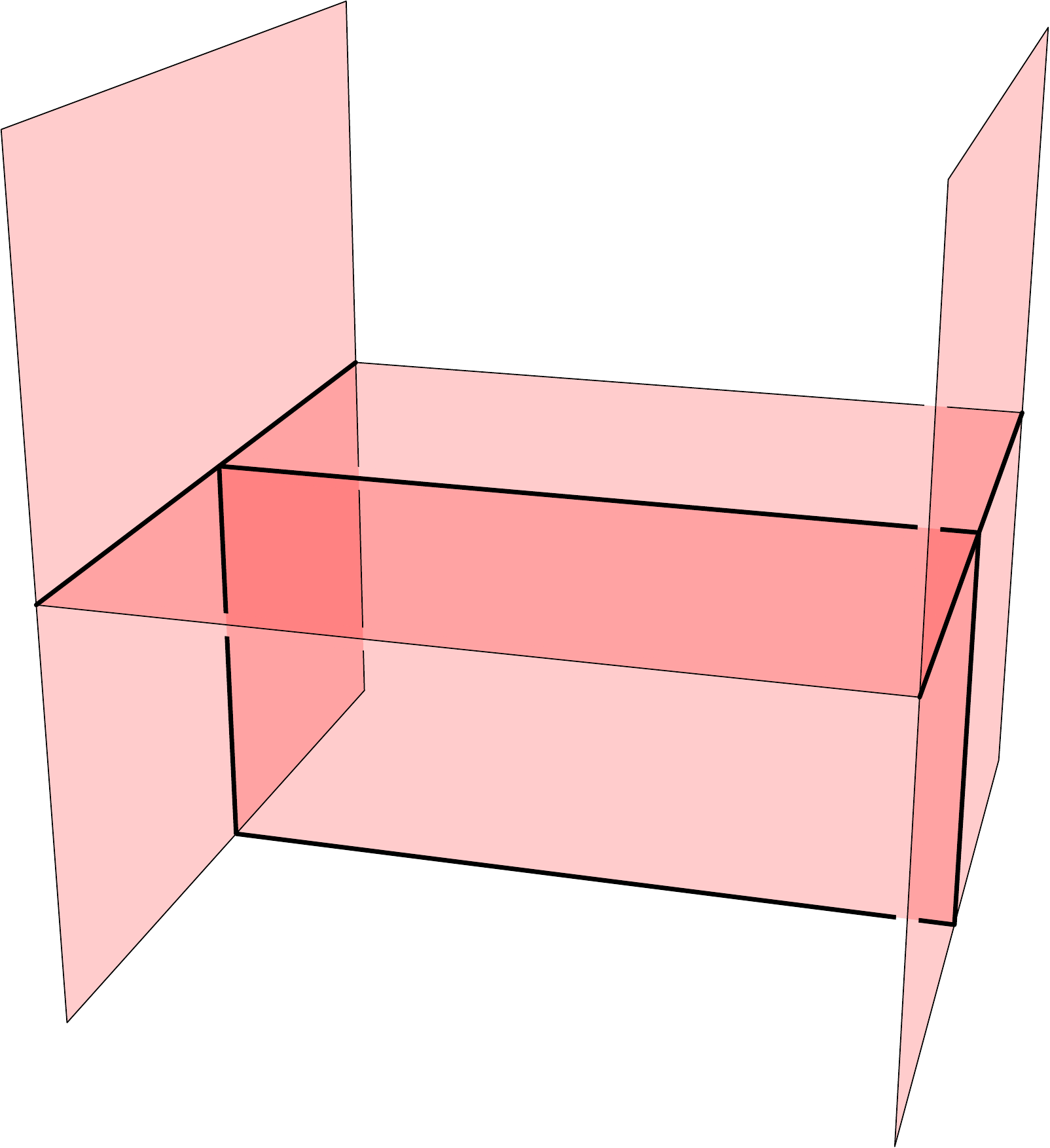}
}
\subfloat[$\cT'$]{
\includegraphics[width=0.22\textwidth]{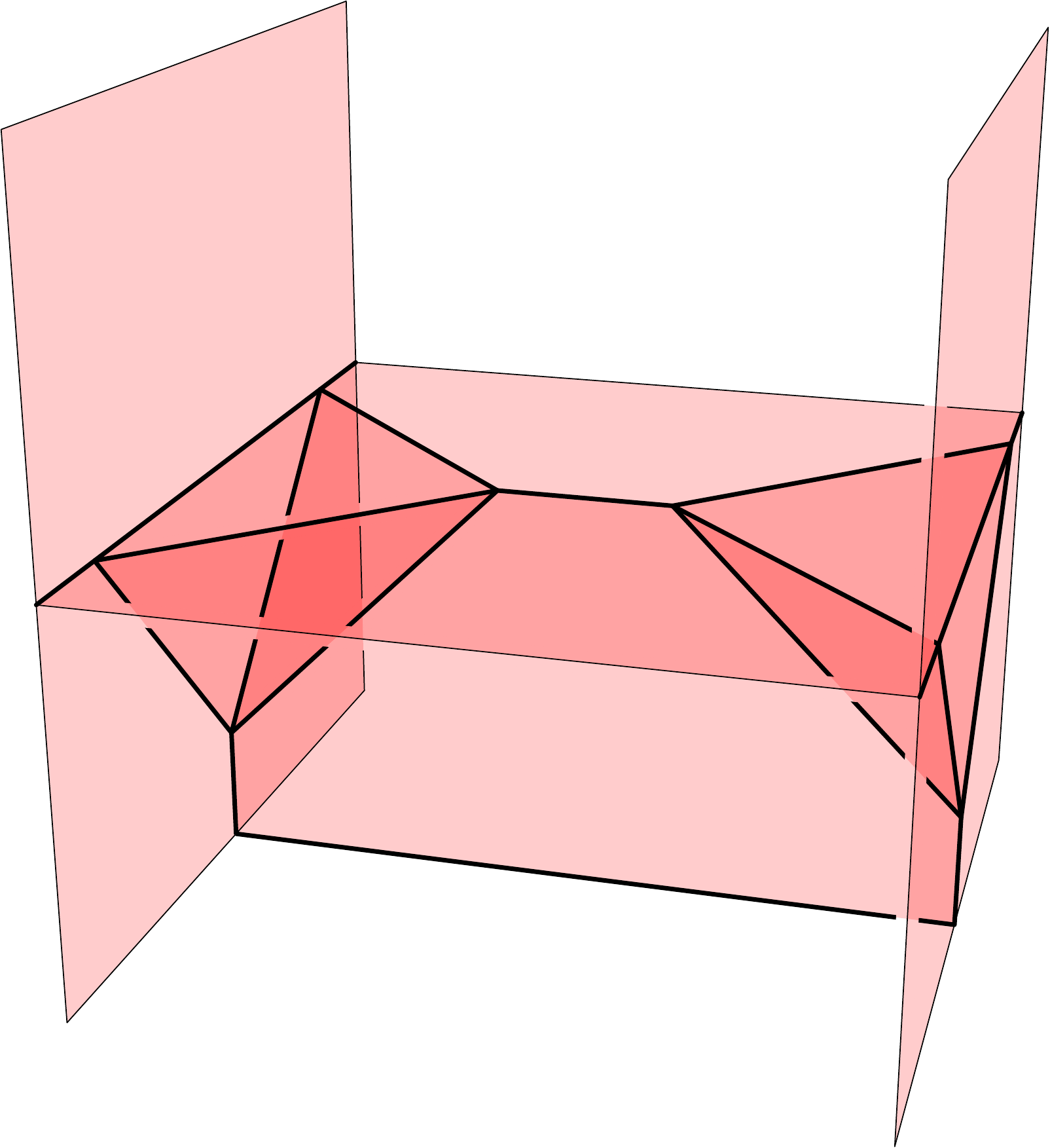}
}
\subfloat[]{
\includegraphics[width=0.22\textwidth]{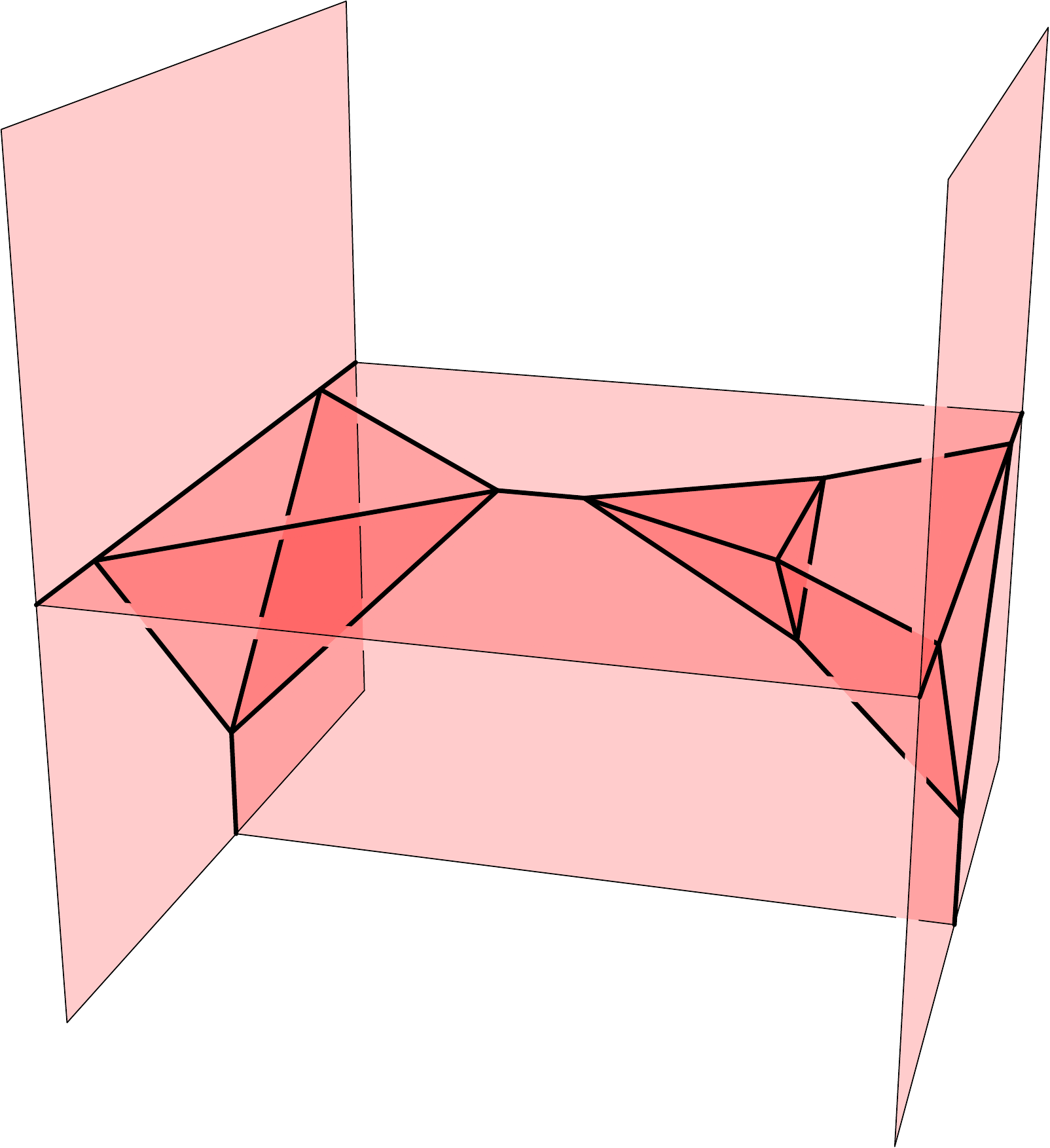}
}
\subfloat[$\cT''$]{
\includegraphics[width=0.22\textwidth]{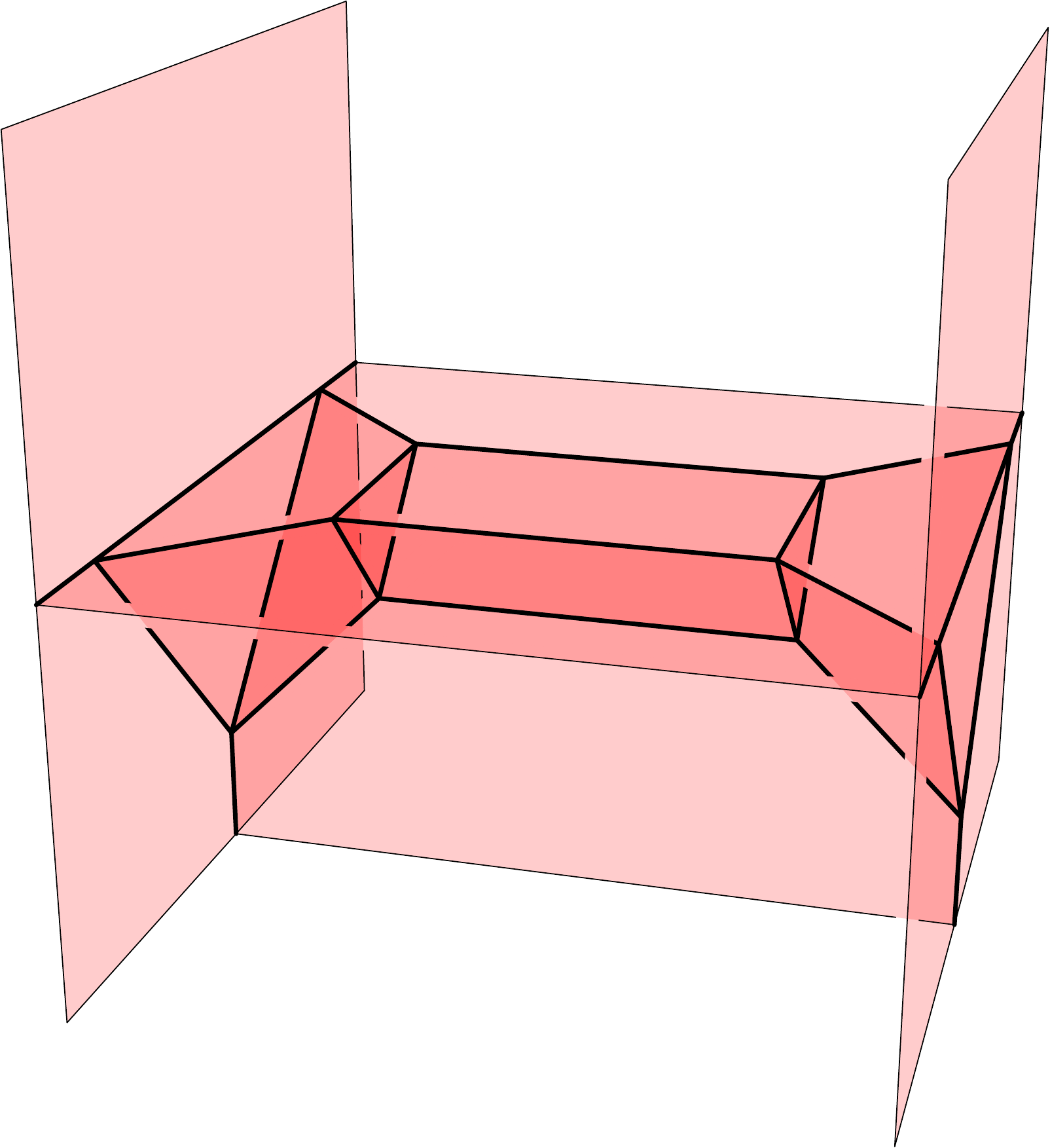}
}
\caption{Figure~\ref{barycentric_subdiv_3d}, drawn in the dual spine picture. } 
\label{barycentric_subdiv_3d_spine}
\end{figure}

\begin{figure}[htbp]
\centering
\subfloat[]{
\includegraphics[width=0.24\textwidth]{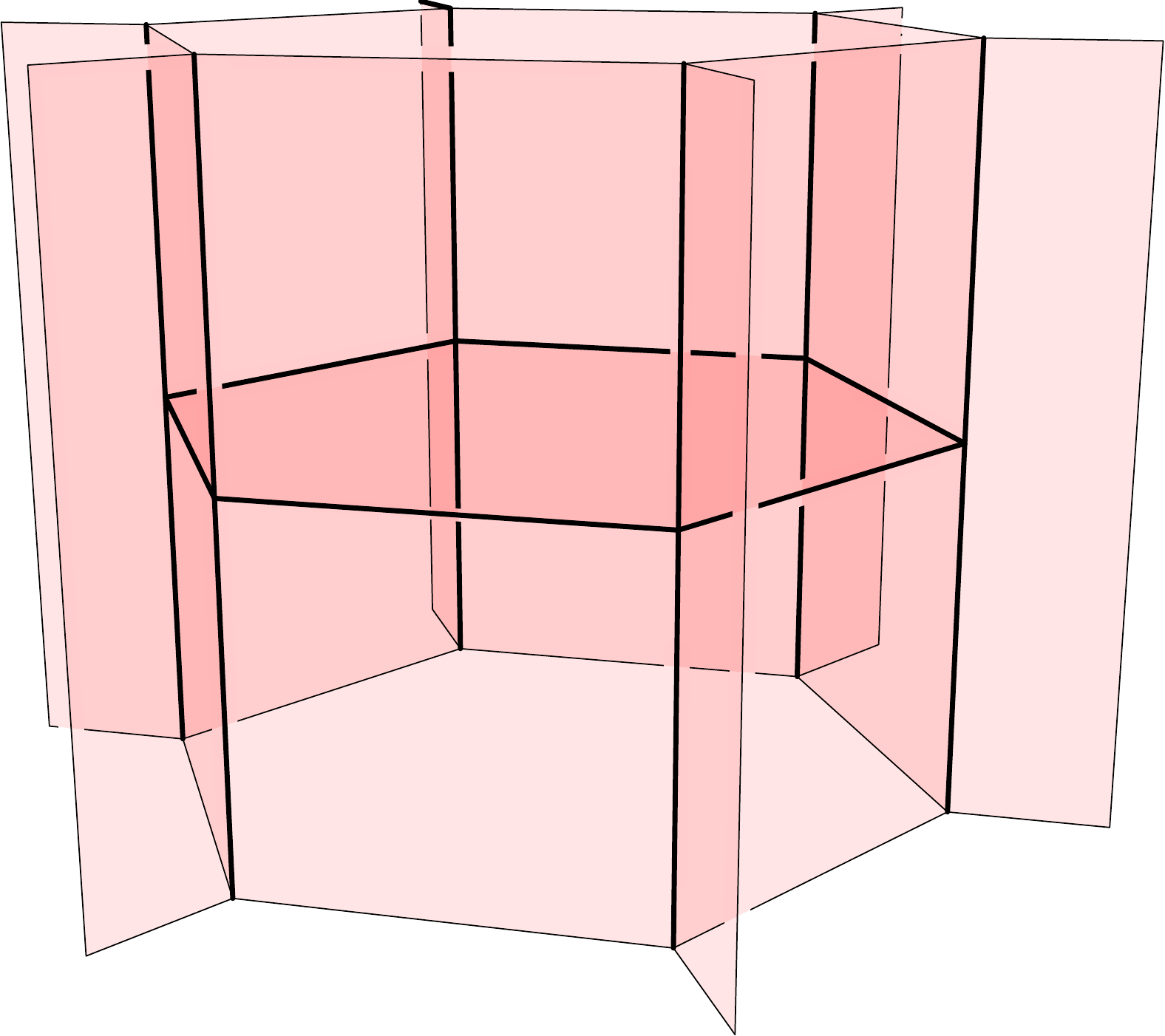}
\label{barycentric_subdiv_3d_spine_edge_step_0}
}
\subfloat[]{
\includegraphics[width=0.24\textwidth]{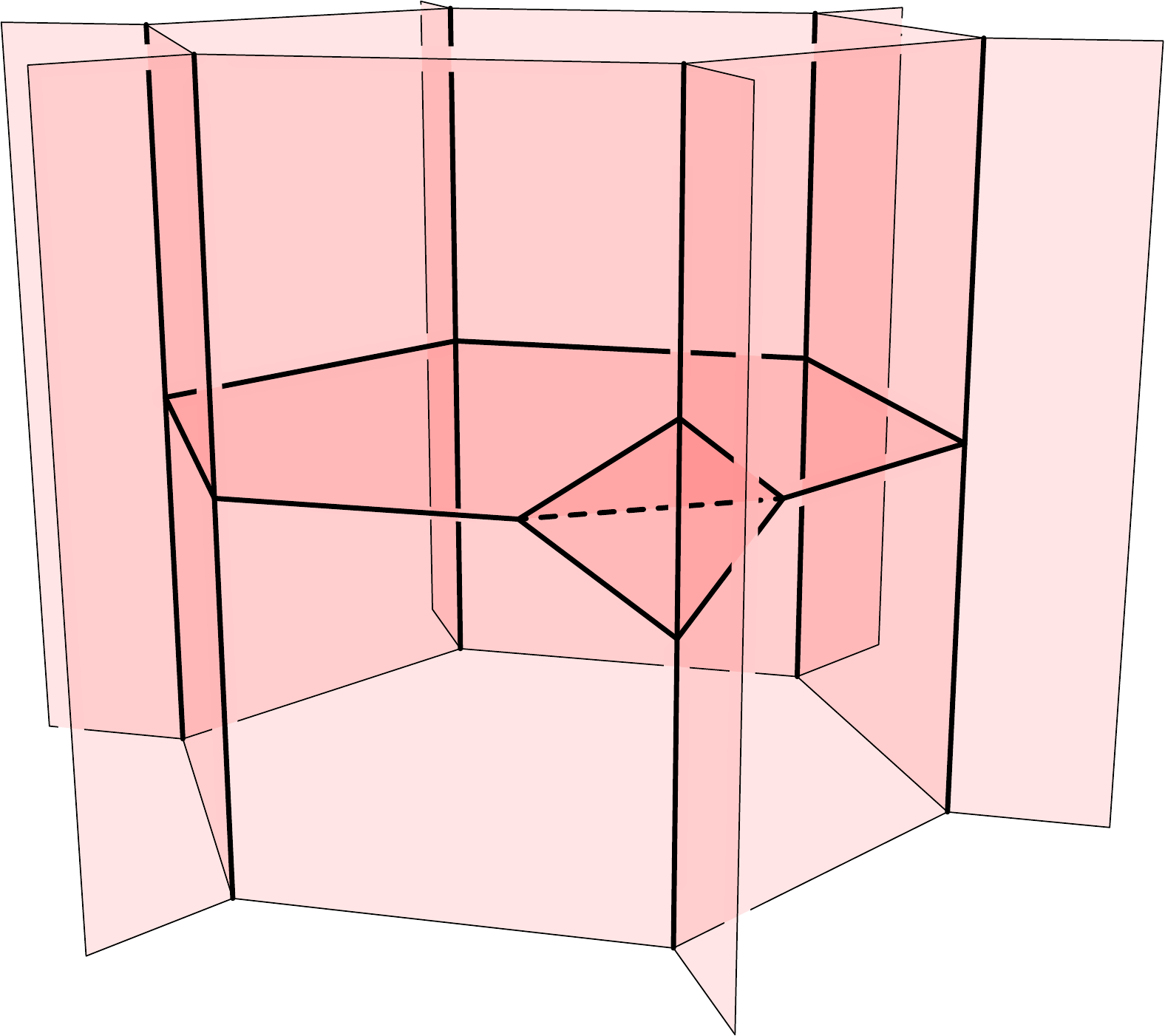}
\label{barycentric_subdiv_3d_spine_edge_step_1}
}
\subfloat[]{
\includegraphics[width=0.24\textwidth]{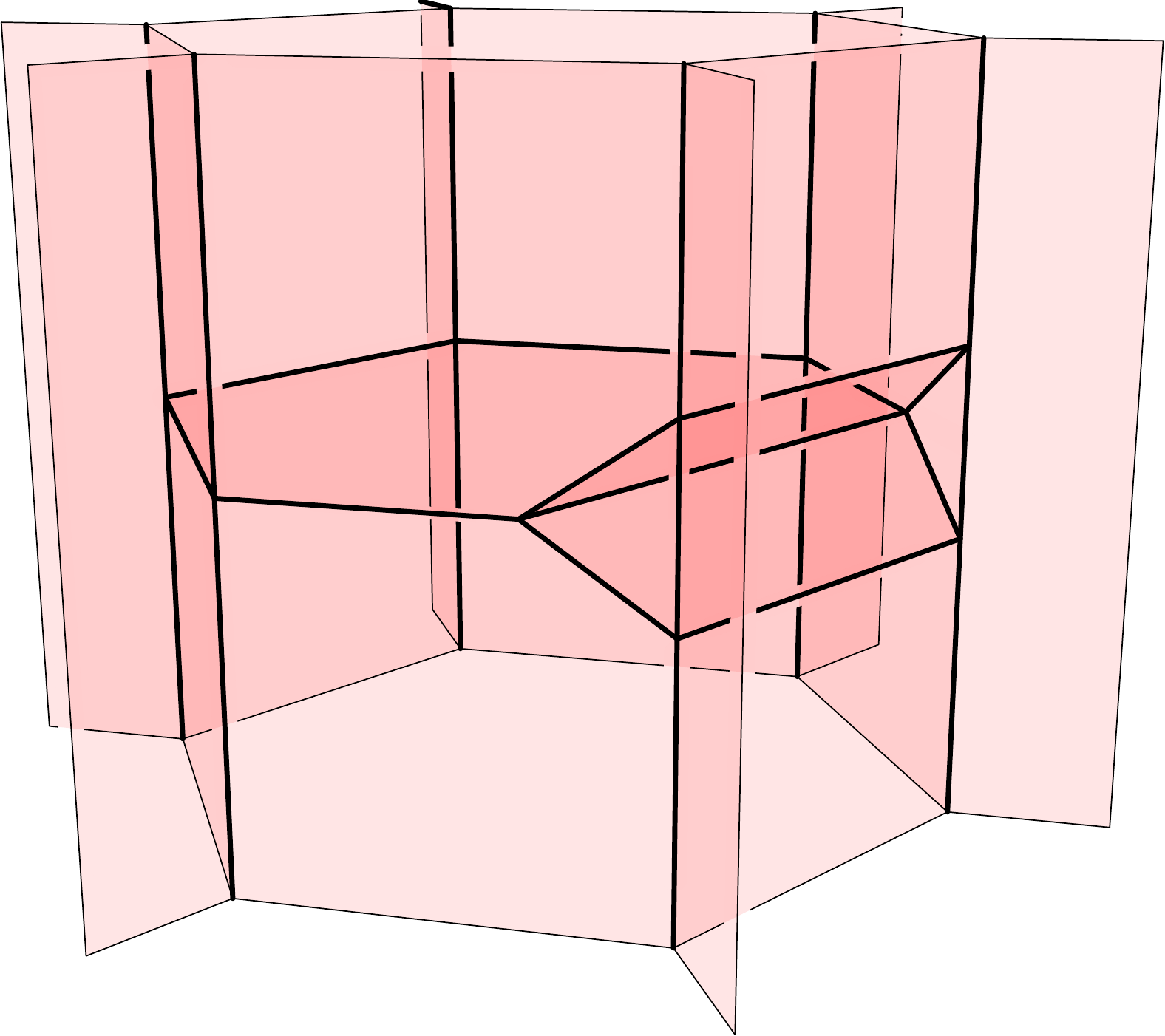}
\label{barycentric_subdiv_3d_spine_edge_step_2}
}
\subfloat[]{
\includegraphics[width=0.24\textwidth]{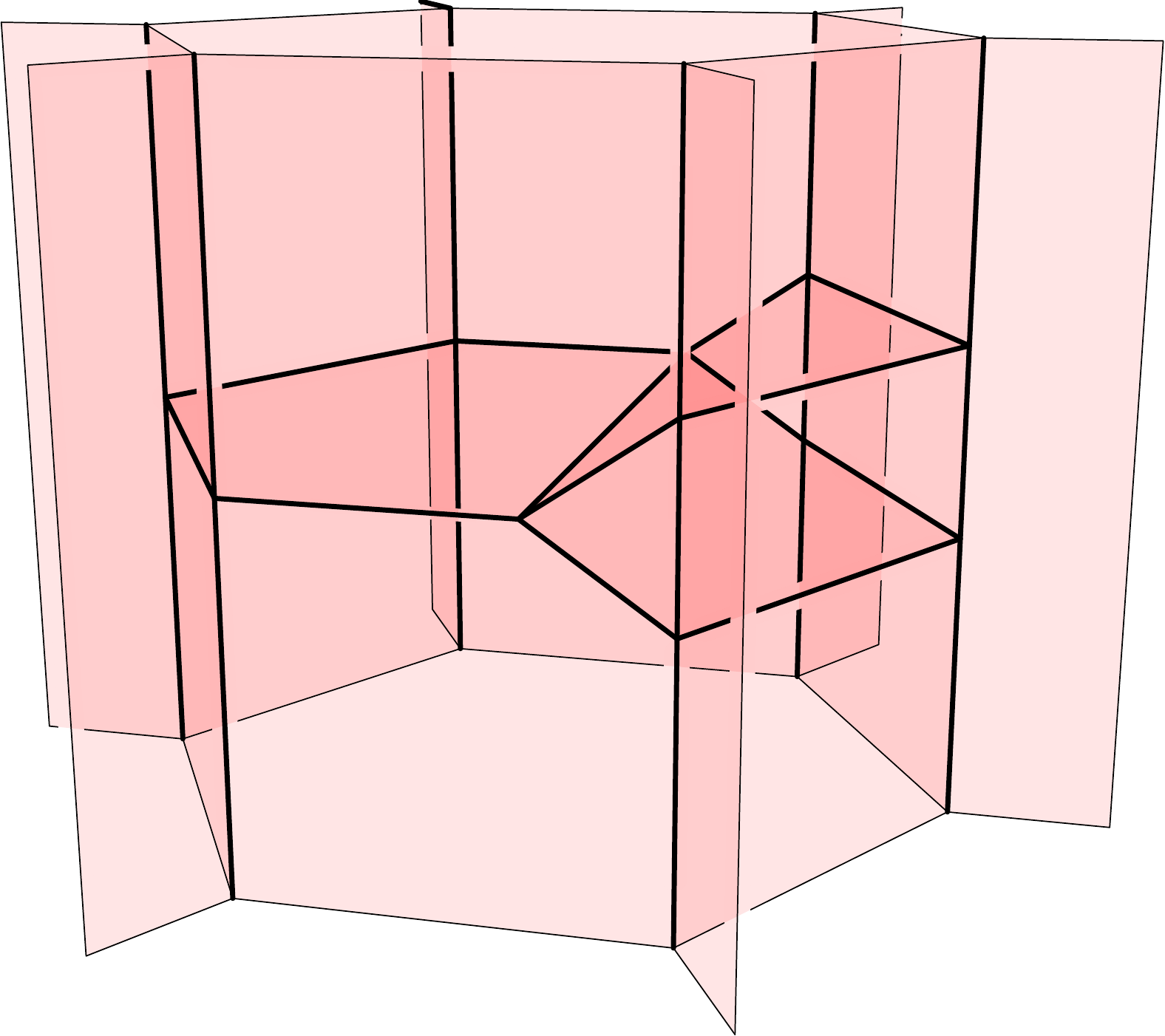}
\label{barycentric_subdiv_3d_spine_edge_step_3}
}

\subfloat[]{
\includegraphics[width=0.24\textwidth]{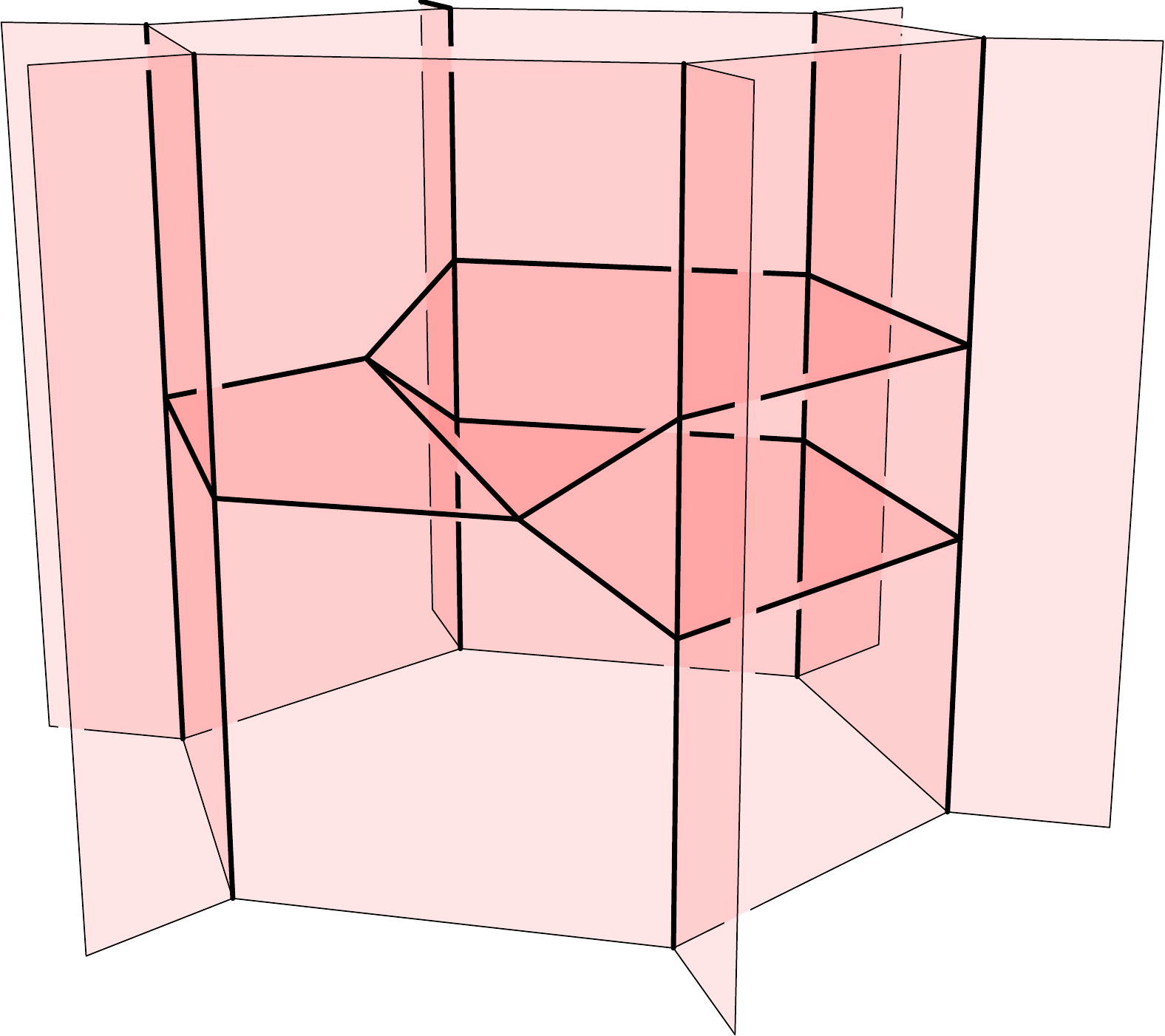}
\label{barycentric_subdiv_3d_spine_edge_step_4}
}
\subfloat[]{
\includegraphics[width=0.24\textwidth]{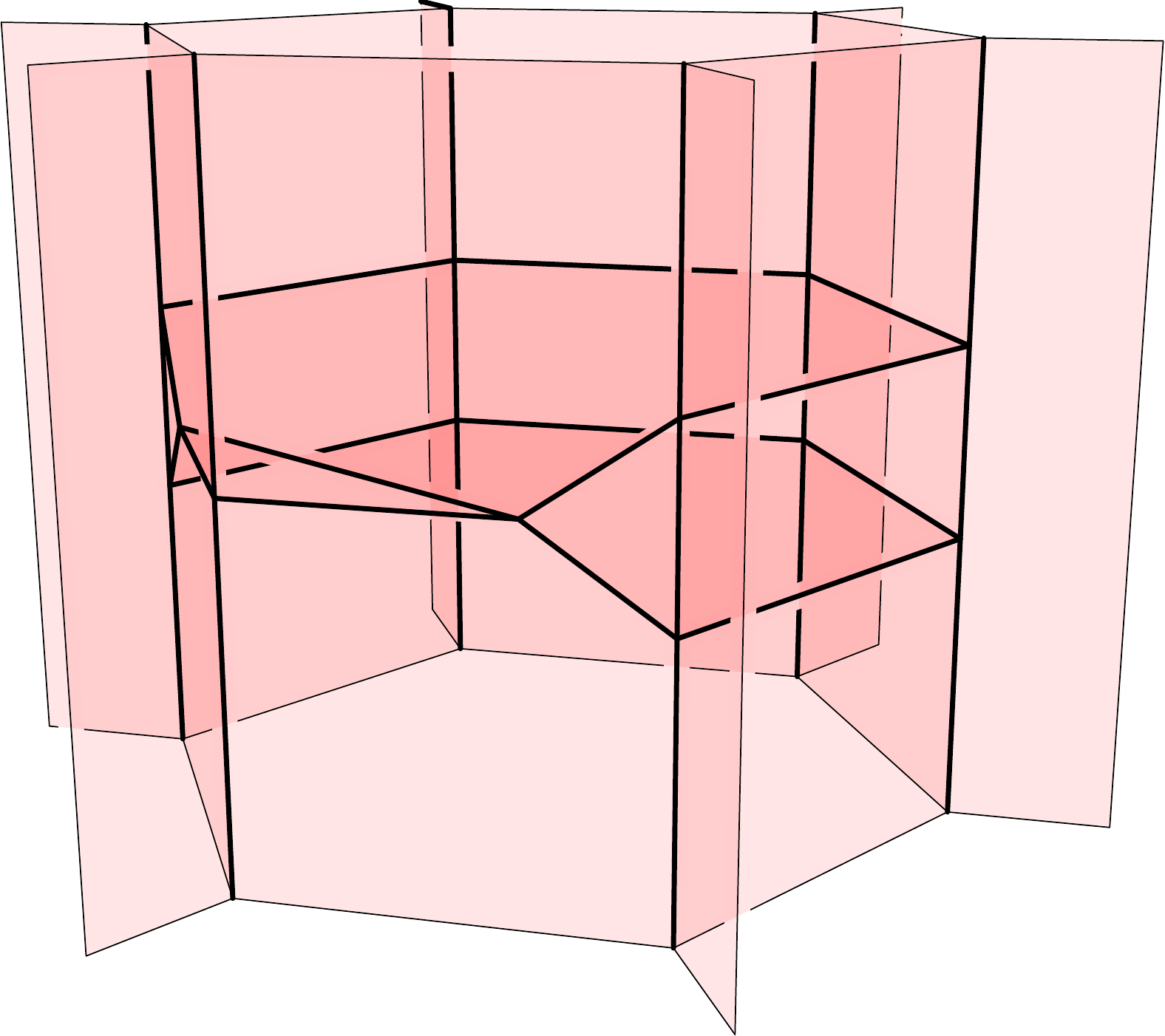}
\label{barycentric_subdiv_3d_spine_edge_step_5}
}
\subfloat[]{
\includegraphics[width=0.24\textwidth]{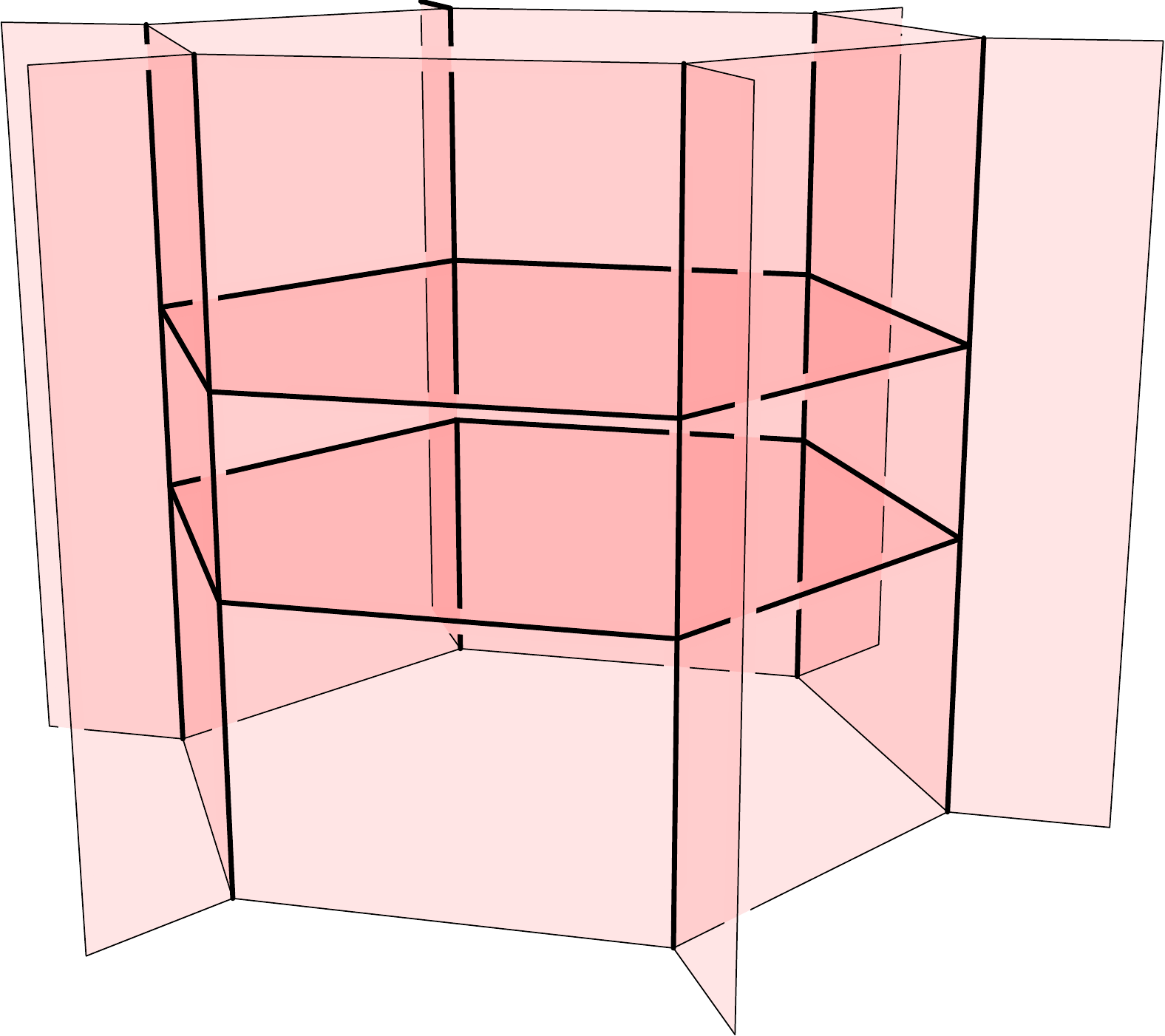}
\label{barycentric_subdiv_3d_spine_edge_step_6}
}
\caption{Figures \ref{barycentric_subdiv_3d_edge} and \ref{barycentric_subdiv_3d_edge_top_view}, drawn in the dual spine picture.} 
\label{barycentric_subdiv_3d_edge_spine}
\end{figure}

The spine picture is sometimes clearer than the triangulation picture when illustrating certain sequences of moves. For example, Matveev introduces the \emph{V-move}, which replaces one vertex of a spine with three vertices.
Matveev shows that the V-move can be implemented using 2-3 and 3-2 moves, assuming that a spine has more than one vertex (or dually that the triangulation has more than one tetrahedron), see Proposition 1.2.8 of \cite{matveev_book}. For the convenience of the reader, we reproduce the argument in Figure~\ref{fig_V-move}. This sequence of moves would be confusing and difficult to draw in the triangulation picture. (Try it!)

Note that by applying symmetries of the spine in the initial drawing, we can apply any of the three possible V-moves to any vertex of a spine that contains more than one vertex. The procedure requires an auxilliary neighbouring vertex -- any vertex in a spine has such a neighbour unless it is incident to itself along all four arcs of triple points, but then it would be the only vertex in the spine.

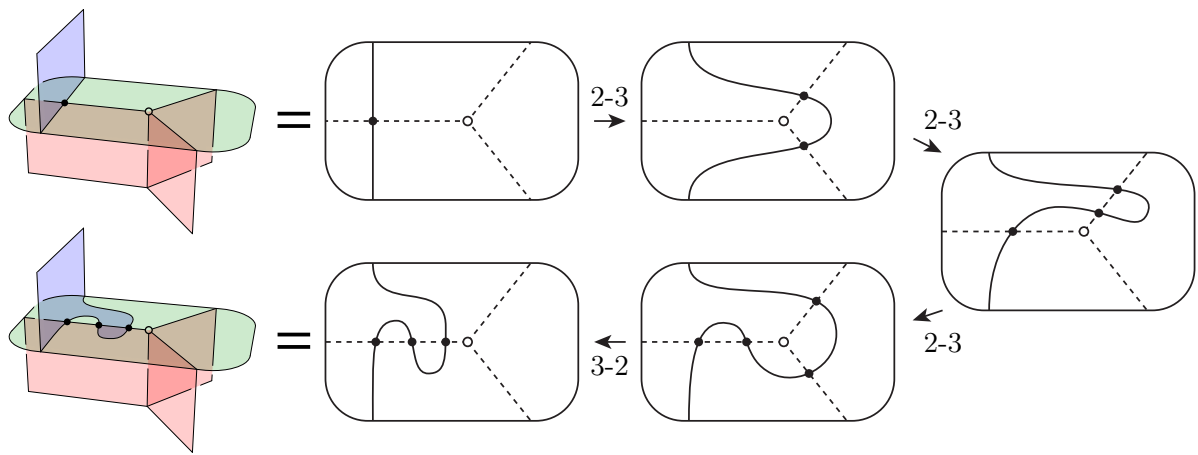
\begin{figure}[htbp]
\centering
\labellist
\small\hair 2pt
\pinlabel 2-3 at 345 203
\pinlabel 2-3 at 535 190
\pinlabel 2-3 at 535 63
\pinlabel 3-2 at 345 50
\endlabellist
\includegraphics[width=0.9\textwidth]{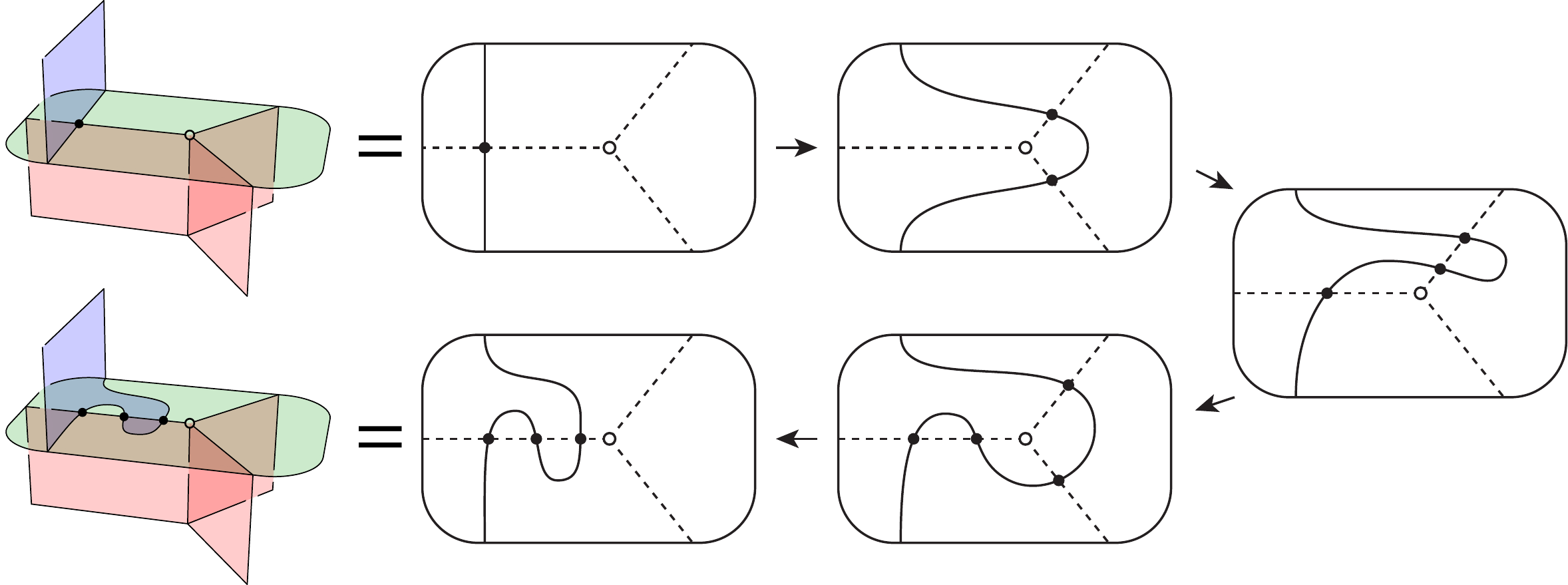}
\caption{The V-move and its construction from 2-3 and 3-2 moves. } 
\label{fig_V-move}
\end{figure}


\section{A warm-up exercise}
\label{sec:warm-up}

As a warm-up exercise for our proof of Theorem \ref{thm:main}, we prove the following lemma. This introduces some of the ideas of the main proof in an easier context, and will be useful later. 
In the proof, we introduce the \emph{triangular 0-2 move} and \emph{pillow marks}. They give a first taste for what is to come.

\begin{lemma} \label{material_vertices_floor}
Let $\cT_A$ and $\cT_B$ be two triangulations of a given 3--manifold $M$ with the same number, $k$, of material vertices. Then there is a sequence of triangulations $\cT_A=\cT_1, \cT_2, \ldots, \cT_n=\cT_B$, with adjacent triangulations related by 2-3, 3-2, 1-4 and 4-1 moves, such that each triangulation $\cT_i$ has at least $k$ material vertices.
\end{lemma}

\begin{proof}
The statement of this lemma is essentially the statement of Theorem \ref{thm:Banagl-Friedman}, except that we also require that the number of material vertices does not go below the number, $k$, of material vertices that $\cT_A$ and $\cT_B$ have. We will work from a sequence $\mathcal{S}$ of triangulations $\cT_A=\cT_1, \cT_2, \ldots, \cT_n=\cT_B$ given to us by Theorem \ref{thm:Banagl-Friedman}, and will modify it if and when the number of material vertices goes below $k$.

In fact, we may assume that $\cT_2, \ldots, \cT_{n-1}$ have precisely $k-1$ vertices, and thus $\cT_2$ is obtained from $\cT_1$ by a 1-4 move, $\cT_n$ is obtained from $\cT_{n-1}$ by a 4-1 move, and all other moves are 2-3 and 3-2 moves. If we are able to replace this by a sequence of triangulations for which the number of vertices is always at least $k$, then by induction we obtain the result.

Our tool to increase the number of vertices in a triangulation is the \emph{triangular 0-2 move}. Its reverse is the \emph{triangular 2-0 move}. See Figure~\ref{0-2_triangular}. The triangular 0-2 move is performed by ungluing one of the triangular gluings of the triangulation, and inserting into the resulting gap a \emph{triangular pillow} -- a three-ball formed from two tetrahedra, glued to each other along three faces, so that the three-ball has an internal vertex. The triangular 0-2 move can be performed on any triangular face of the triangulation, and implemented using a 1-4 move followed by a 3-2 move, as shown in Figure~\ref{0-2_triangular}. The inverse move is called the \emph{triangular 2-0 move}, and can be performed on any triangular pillow as long as the two outer faces of the pillow are not glued to each other. \footnote{If they were glued then the entire triangulation would consist of only the triangular pillow, so performing the 2-0 move would result in a triangulation with no tetrahedra!}

\begin{figure}[htbp]
\centering
\labellist
\small\hair 2pt
\pinlabel {0-2} at 167 429
\pinlabel {1-4} at 242 135
\pinlabel {3-2} at 484 135
\endlabellist
\includegraphics[width=0.8\textwidth]{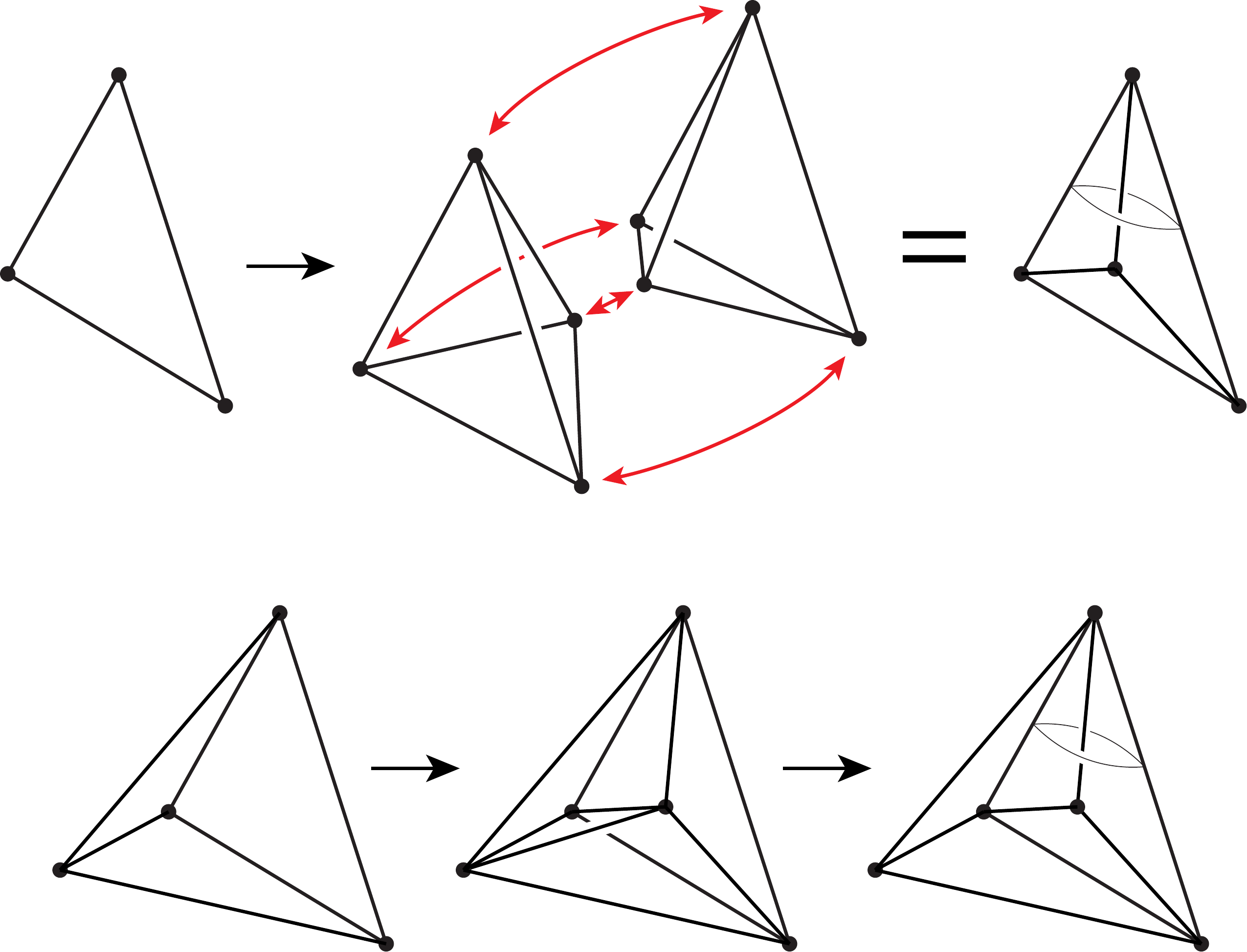}

\caption{Above: the triangular 0-2 move adds a vertex and two tetrahedra to a triangulation. Below: the triangular 0-2 move can be implemented using a 1-4 move applied to a tetrahedron adjacent to the triangle, followed by a 3-2 move. We draw a bigon to show the two external faces of the triangular pillow.} 
\label{0-2_triangular}
\end{figure}

We modify our sequence of triangulations $\cT_A=\cT_1, \cT_2, \ldots, \cT_n=\cT_B$ as follows. Immediately before the initial 4-1 move takes us to $k-1$ material vertices, we perform a triangular 0-2 move on some triangle $\tri$ of the triangulation, increasing the number of material vertices to $k+1$. If this triangle $\tri$ is not deleted by any of the subsequent moves taking us to $\cT_n$, then we perform those moves as before, with the resulting triangulations modified from the original sequence by our triangular 0-2 move. After performing the final 1-4 move (taking us to $k+1$ material vertices), we remove the triangular pillow with the triangular 2-0 move, taking us back to $k$ material vertices, and we have arrived at $\cT_B$.

If there is no triangle of $\cT_A$ that persists throughout the sequence of moves then the above process will not work. Whichever choice we make for $\tri$, at some point a 2-3 or 3-2 move in the sequence (or the final 4-1 move) is supposed to delete $\tri$. However, we will be unable to perform this move because the triangular pillow blocks it. In these cases, we must first move the triangular pillow elsewhere. 

In order to organise the argument, for each triangulation of $\mathcal{S}$, we will mark a triangle with a \emph{pillow mark}, showing where we want the triangular pillow to be. We denote by $\cT'_i$ the \emph{marked triangulation} resulting from adding a pillow mark to the triangulation $\cT_i$.
Denote the result of inserting a triangular pillow into a marked triangulation $\cT'_i$ according to the pillow mark by $\widehat{\cT}'_i$. We refer to these triangulations $\widehat{\cT}'_i$ as \emph{waypoint triangulations}. Our plan then is to build a new sequence of triangulations $\widehat{\mathcal{S}}$, which again connects $\cT_A$ to $\cT_B$, but for which the number of material vertices never goes below $k$. We do this by going from $\cT_A = \cT_1$ to $\widehat{\cT}'_1$ via a triangular 0-2 move (i.e. a 1-4 followed by a 3-2), then connecting each adjacent pair $\widehat{\cT}'_i$ and $\widehat{\cT}'_{i+1}$ of waypoint triangulations using 1-4, 2-3, 3-2 and 4-1 moves, then finally going from $\widehat{\cT}'_n$ to $\cT_n = \cT_B$ via a triangular 2-0 move (i.e. a 2-3 followed by a 4-1).

\begin{tikzpicture}
\end{tikzpicture}

\begin{figure}[htbp]
\[
\begin{tikzcd}
 \cT_A \arrow[r,equal] & \cT_1 \arrow[d] \arrow[r, "4-1"] & \cT_2 \arrow[d] \arrow[r] & \cT_3 \arrow[d] \arrow[r] & \cdots \arrow[r] & \cT_{n-1} \arrow[d] \arrow[r, "1-4"] & \cT_n \arrow[d] \arrow[r, equal] & \cT_B \\
& \widehat{\cT}'_1 \arrow[r, dashed]  & \widehat{\cT}'_2 \arrow[r, dashed] & \widehat{\cT}'_3 \arrow[r, dashed] & \cdots \arrow[r, dashed] & \widehat{\cT}'_{n-1} \arrow[r, dashed] &\widehat{\cT}'_n &
\end{tikzcd}
\]
\caption{The original sequence of triangulations $\mathcal{S}$ is on the upper row. The modified waypoint triangulations are below. Dashed arrows indicate sequences of 1-4, 2-3, 3-2 and 4-1 moves. All vertical arrows indicate triangular 0-2 moves.}
\label{triangular_pillow_sidestep}
\end{figure}
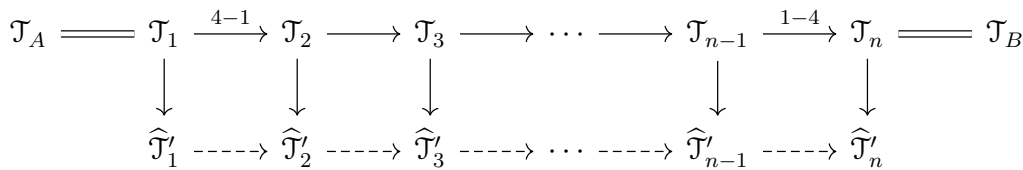

First, we describe where the pillow mark is on each $\cT'_i$. The idea is to have the pillow mark of $\cT'_i$ move when it is involved in the bistellar move relating $\cT'_i$ and $\cT'_{i+1}$.
Having done this, we will describe sequences of 1-4, 2-3, 3-2 and 4-1 moves to take each waypoint triangulation $\widehat{\cT}'_i$ to the subsequent $\widehat{\cT}'_{i+1}$. See Figure~\ref{triangular_pillow_sidestep}.

In $\cT'_1$, we choose the pillow mark arbitrarily. There are two cases under which we determine the new pillow mark on $\cT'_{i+1}$ based on the location of the pillow mark on $\cT_i$:

If a pillow mark is on a triangle of $\cT'_i$ that is not deleted by the bistellar move relating $\cT_i$ and $\cT_{i+1}$, then we leave it where it is, duplicating it on $\cT'_{i+1}$. If a pillow mark is on a triangle of $\cT'_i$ that is deleted by the bistellar move, then we must first move it to any triangle of $\cT_i$ that is not deleted by the move. We denote the resulting marked triangulation by $\cT''_i$. This done, we perform the bistellar move and get $\cT'_{i+1}$, where the pillow mark doesn't move between $\cT''_i$ and $\cT'_{i+1}$. Note that there is always a triangle of $\cT_i$ that is not deleted by the bistellar move, since every triangle on the boundary of the region altered by a bistellar move remains after the move.

Having constructed these markings on the triangulations of $\mathcal{S}$, the resulting waypoint triangulations $\widehat{\cT}'_i$ never have fewer than $k$ material vertices. In fact, $\widehat{\cT}'_1$ and $\widehat{\cT}'_n$ have $k+1$ vertices, while all others have $k$ vertices. All that remains is to show how to connect $\cT_1$ to $\widehat{\cT}'_1$, how to connect $\widehat{\cT}'_i$ to $\widehat{\cT}'_{i+1}$ for each subsequent pair of waypoint triangulations, and how to connect $\widehat{\cT}'_n$ to $\cT_n$ by sequences of 1-4, 2-3, 3-2 and 4-1 moves, without reducing the number of material vertices below $k$. 

To connect $\cT_1$ to $\widehat{\cT}'_1$, we apply a triangular 0-2 move. This is implemented by a 1-4 move followed by a 3-2 move, as shown in Figure~\ref{0-2_triangular}, and takes us from a triangulation with $k$ vertices to one with $k+1$ vertices. The connection between $\widehat{\cT}'_n$ to $\cT_n$ is the same but in reverse.

In the case that we did not move the pillow mark between $\cT'_i$ and $\cT'_{i+1}$, then the triangular pillow is not in the way of the corresponding 2-3 or 3-2 move on $\widehat{\cT}'_i$, and so we simply apply the move to obtain $\widehat{\cT}'_{i+1}$, which doesn't alter the number of vertices.
In the case that we did move the pillow mark, we get from $\widehat{\cT}'_i$ to $\widehat{\cT}''_i$ by first performing a triangular 0-2 move in the location of the new pillow mark, and then perform a triangular 2-0 move to remove the old triangular pillow and obtain $\widehat{\cT}'_{i+1}$. Thus the number of material vertices goes up by one to $k+1$ and then down by one to $k$ in this process, but never goes below $k$. Having moved the pillow mark out of the way of the 2-3 or 3-2 move that converts $\cT_i$ into $\cT_{i+1}$, we apply the move to $\widehat{\cT}''_i$ to obtain $\widehat{\cT}'_{i+1}$, which again doesn't alter the number of vertices.

Making these moves connects the waypoint triangulations together, and so connects $\cT_A$ to $\cT_B$ by a sequence of bistellar moves, in such a way that all triangulations in the sequence have at least $k$ material vertices.
\end{proof}

\begin{rmk}
One might ask why we use the triangular 0-2 move to increase the number of vertices instead of the simpler 1-4 move. The answer is that the triangular pillow is easier to ``hide'' from the other Pachner moves than the result of the 1-4 move.  For example, the sequence $\mathcal{S}$ might have a triangulation with only two tetrahedra on which we are supposed to perform a 2-3 move. Neither of the two tetrahedra would be valid locations to hide the extra vertex if we used the 1-4 move. 
\end{rmk}


\section{Proof strategy}
\label{proof_strategy}

In order to prove Theorem~\ref{thm:main}, suppose that
we are given two triangulations,  $\cT_A, \cT_B$ of a given manifold $M$ with the same number, $k$ say, of material vertices. By Theorems~\ref{pachner} and \ref{thm:Banagl-Friedman}, we have a sequence $\mathcal{S}$ of triangulations $\cT_A=\cT_1, \cT_2, \ldots, \cT_n=\cT_B$, with adjacent triangulations related by 2-3, 3-2, 1-4 and 4-1 moves. We convert this sequence $\mathcal{S}$ of triangulations into another sequence $\widehat{\mathcal{S}}$ that connects $\cT_A$ to $\cT_B$, but that does not use 1-4 or 4-1 moves.

Note that every triangulation of the desired sequence $\widehat{\mathcal{S}}$ must have $k$ material vertices, since the number of vertices only changes under 1-4 and 4-1 moves.
By Lemma \ref{material_vertices_floor}, we may assume that every triangulation in our sequence $\mathcal{S}$ has at least $k$ material vertices, so we must avoid intermediate triangulations with more than $k$ material vertices. Again, as in the proof of the lemma, it suffices to consider an innermost pair of triangulations with the property that the two triangulations have $k$ material vertices, and the triangulations between them $k+1$. In particular, we may assume that the first bistellar move is a 1-4 move applied to $\cT_A$, followed by a sequence of 2-3 and 3-2 moves, and the sequence is then completed with a 4-1 move resulting in $\cT_B$.

We need a number of tools to achieve this, primarily the \emph{arch}. The arch is a structure we can add in to a spine (or dually a triangulation). The move of \emph{inserting an arch} has the effect of connecting together two three-dimensional regions in the complement of the spine, while only altering the spine in a very localised fashion. Dually, this move unglues a triangular face $\tri$ at which two (or possibly one) tetrahedra are glued, and inserts the arch (made from a single tetrahedron). The effect is to identify two of the three vertices incident to $\tri$. This move cannot be achieved by 2-3 and 3-2 moves, because these moves do not change the number of vertices of the triangulation. We discuss the arch in detail in Section \ref{Sec:arch}.

Since the triangulations in our sequence $\mathcal{S}$ have at most $k+1$ material vertices, we will need to insert at most one arch in each triangulation. In order to organise the argument, we will mark the triangle of each triangulation where we wish to insert an arch. Each such \emph{arch mark} shows both the triangle for an arch to be introduced into, and the pair of incident vertices to be connected. See Figure~\ref{arch_mark}. We denote by $\cT'_i$ the \emph{marked triangulation} resulting from adding an arch mark to the triangulation $\cT_i$.
Dually, an arch mark is associated to an edge of the spine, and indicates which pair of incident three-dimensional regions are to be connected by inserting the arch.

\begin{figure}[h]
\centering
\includegraphics[width=0.2\textwidth]{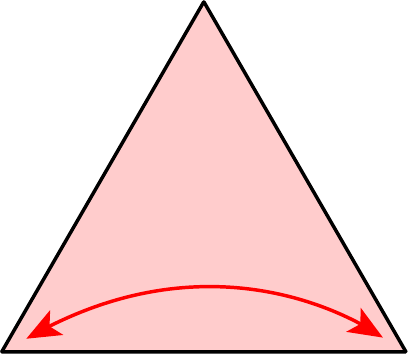}
\caption{An arch mark on a triangle shades the triangle and adds arrows showing the vertices to be identified by the arch.}
\label{arch_mark}
\end{figure}

We will choose where to put each arch mark in such a way that after performing the arch insertions, the resulting triangulations have the same number of material vertices as $\cT_A$ and $\cT_B$. Denote the result of inserting an arch into a marked triangulation $\cT'_i$ by $\widehat{\cT}'_i$. We refer to these triangulations $\widehat{\cT}'_i$ as \emph{waypoint triangulations}. Our plan then is to connect each adjacent pair $\widehat{\cT}'_i$ and $\widehat{\cT}'_{i+1}$ of waypoint triangulations using only 2-3 and 3-2 moves, producing a new sequence of triangulations $\widehat{\mathcal{S}}$, which again connects $\cT_A$ to $\cT_B$, but that only uses 2-3 and 3-2 moves. See Figure~\ref{arch_sidestep}.
We give the details of the proof in Section~\ref{sec:details}.

\begin{figure}[h]
\[
\begin{tikzcd}
  \cT_A \arrow[r,equal] &\cT_1 \arrow[r, "1-4"] \arrow[dr, dashed] & \cT_2 \arrow[d] \arrow[r] & \cT_3 \arrow[d] \arrow[r] & \cdots \arrow[r] & \cT_{n-1} \arrow[d] \arrow[r, "4-1"] & \cT_n \arrow[r,equal] & \cT_B\\
 & & \widehat{\cT}'_2 \arrow[r, dashed] & \widehat{\cT}'_3 \arrow[r, dashed] & \cdots \arrow[r, dashed] & \widehat{\cT}'_{n-1} \arrow[ur, dashed] & &
\end{tikzcd}
\]
\caption{The original sequence of triangulations $\mathcal{S}$ is on the upper row. The waypoint triangulations are below. Dashed arrows indicate sequences of 2-3 and 3-2 moves. All vertical arrows indicate arch insertion moves.}
\label{arch_sidestep}
\end{figure}
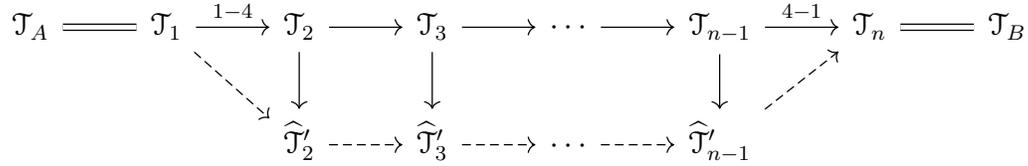


\section{The arch and associated constructions}
\label{sec:arch and friends}

Two of the tools that are key to the visual simplicity and conceptual elegance of working with spines are Matveev's \emph{arch} and \emph{arch-with-membrane}. This section describes these constructions and their equivalent moves on triangulations.

\subsection{The arch~\cite[page 7]{matveev_book}}

\label{Sec:arch}

Figure~\ref{fig_arches} shows the arch, and the related structure of an \emph{arch-with-membrane} in spine form. Figure~\ref{fig_tetrahedron_arch} shows the arch in the dual triangulation picture. Also see Fig. 1.29 of \cite{matveev_book}.

\begin{figure}[htbp]
\centering
\subfloat[An edge of the spine.]{
\centering
\labellist
\small\hair 2pt
\pinlabel $A$ at 360 230
\pinlabel $f$ at 250 230
\pinlabel $B$ at 100 230
\pinlabel $g$ at 120 100
\pinlabel $A'$ at 390 40
\endlabellist
\includegraphics[width=0.3\textwidth]{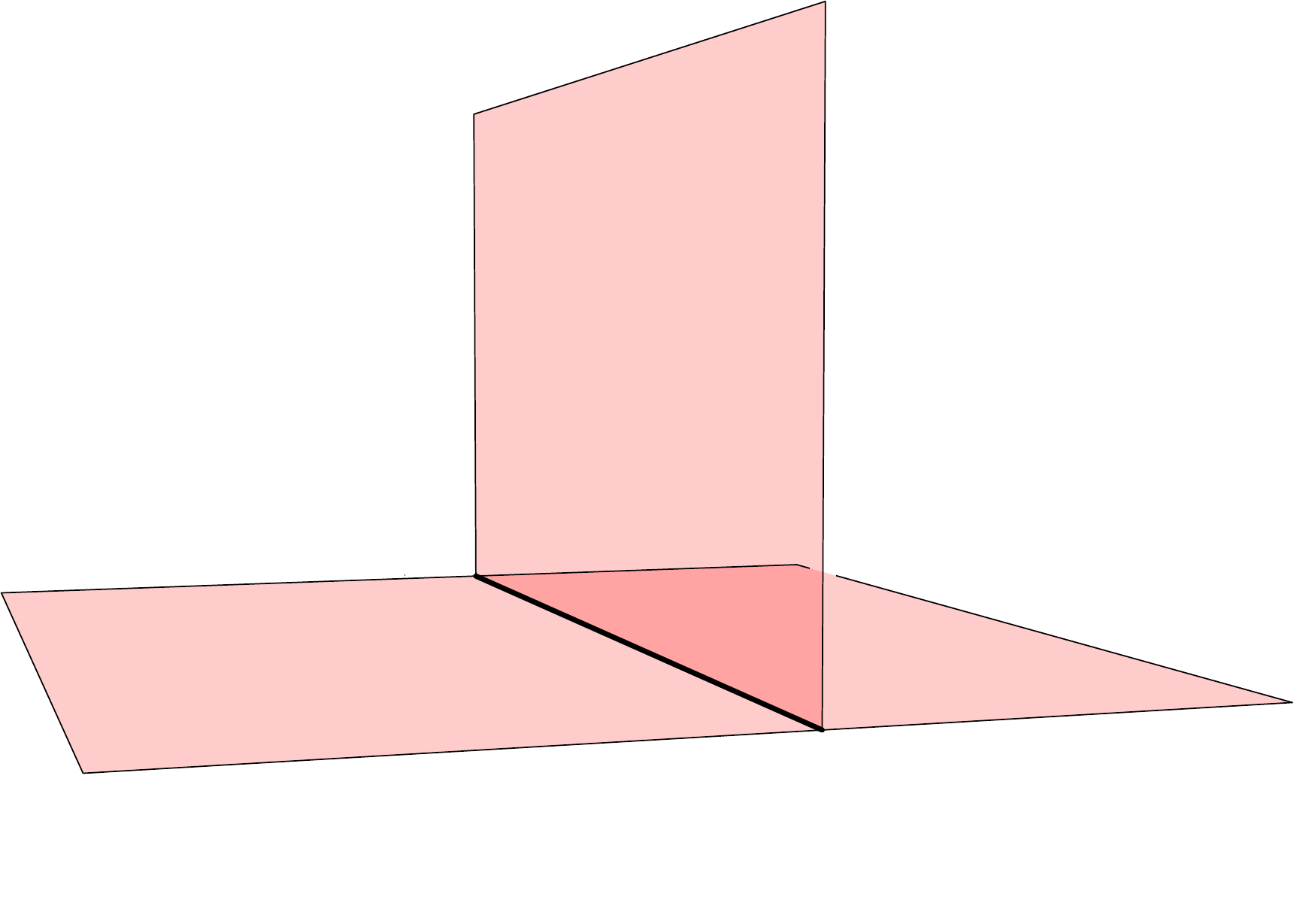}
\label{fig_arch_before}
}
\thinspace
\subfloat[An arch.]{
\centering
\includegraphics[width=0.3\textwidth]{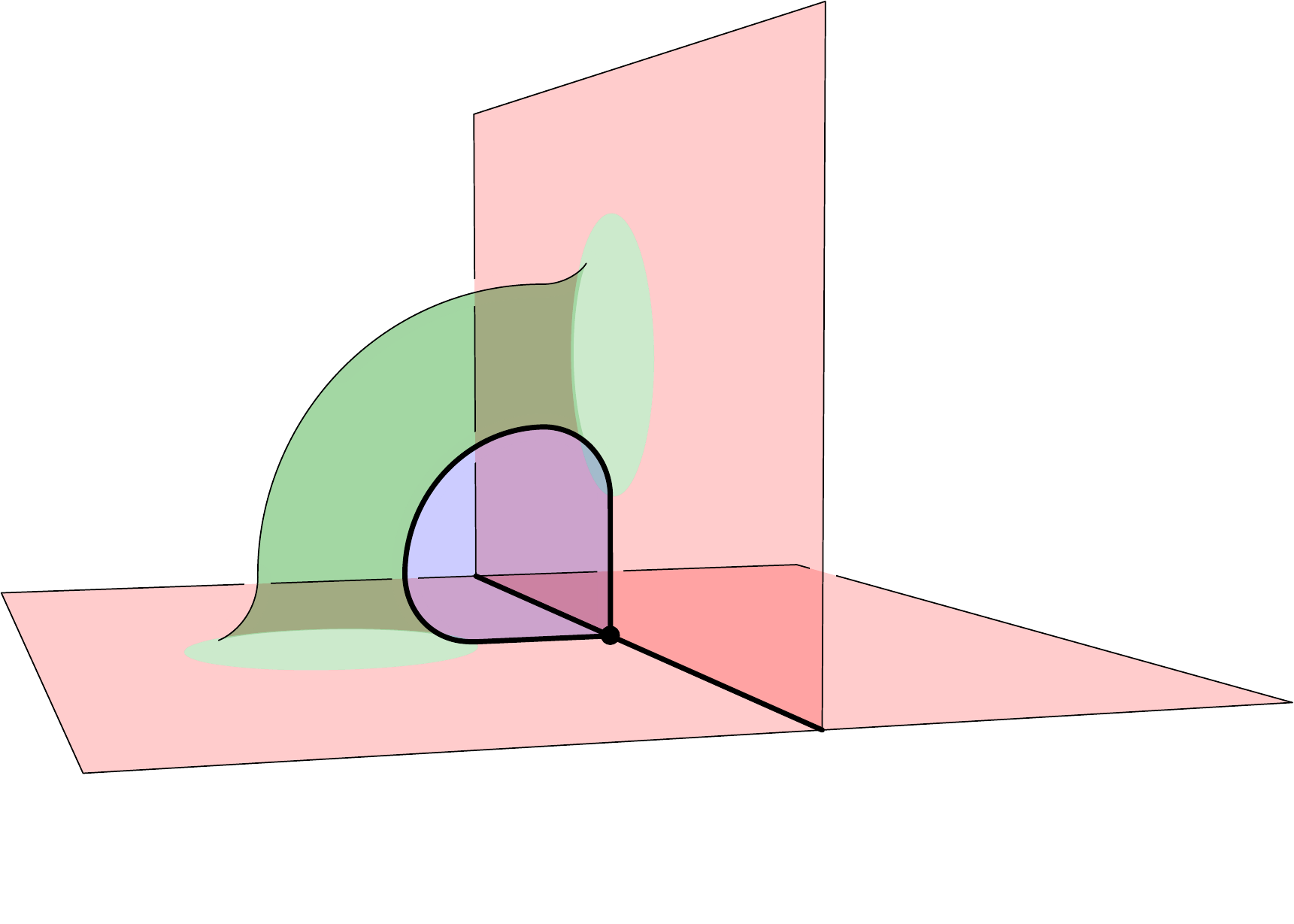}
\label{fig_arch}
}
\thinspace
\subfloat[An arch-with-membrane.]{
\includegraphics[width=0.3\textwidth]{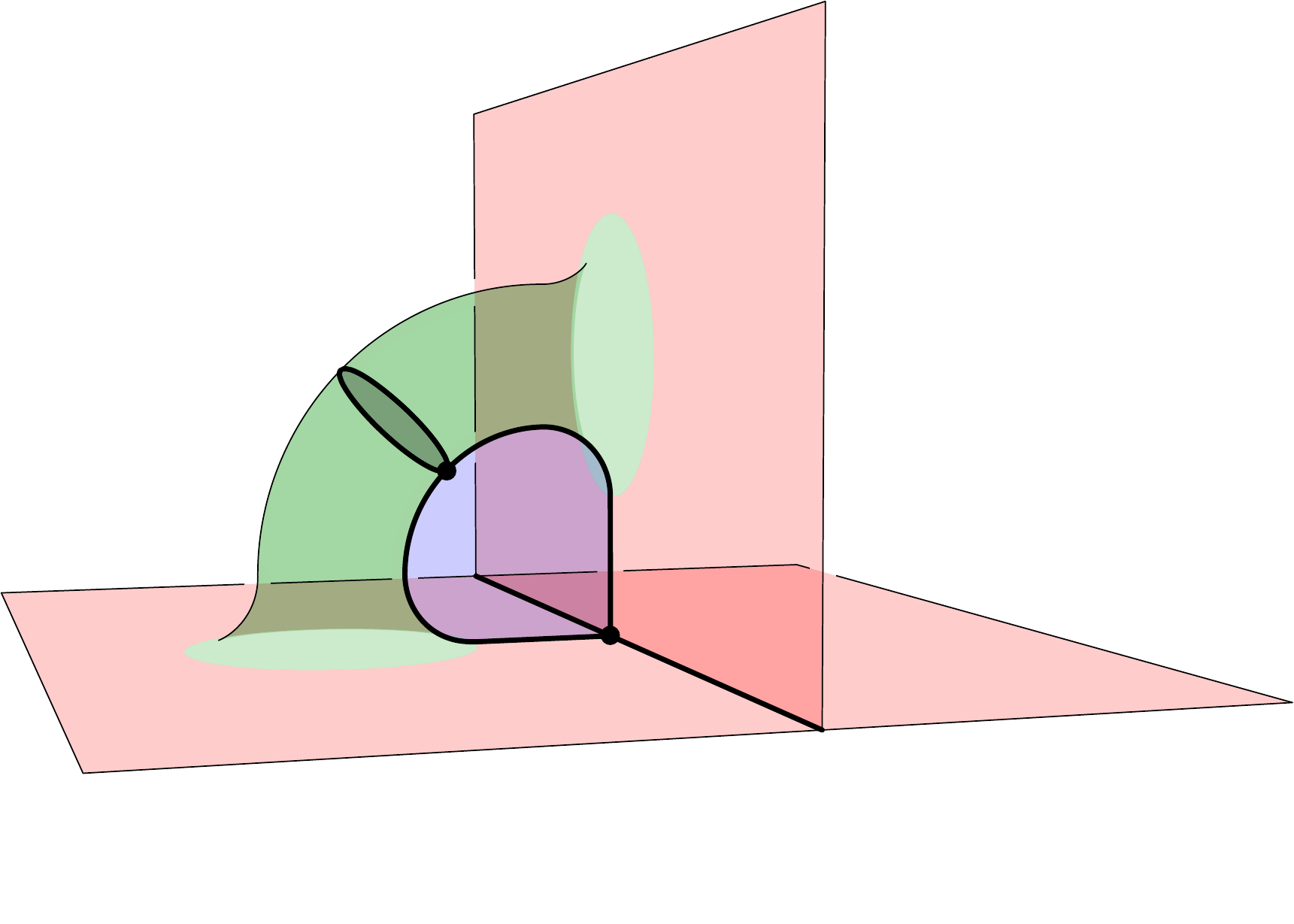}
\label{fig_arch_w_membrane}
}
\caption{Introducing an arch (i.e. going from Figure~\ref{fig_arch_before} to  \ref{fig_arch}) connects together the regions $A$ and $A'$. Figure~\ref{fig_arch_w_membrane} shows the related structure of an arch-with-membrane.} 
\label{fig_arches}
\end{figure}

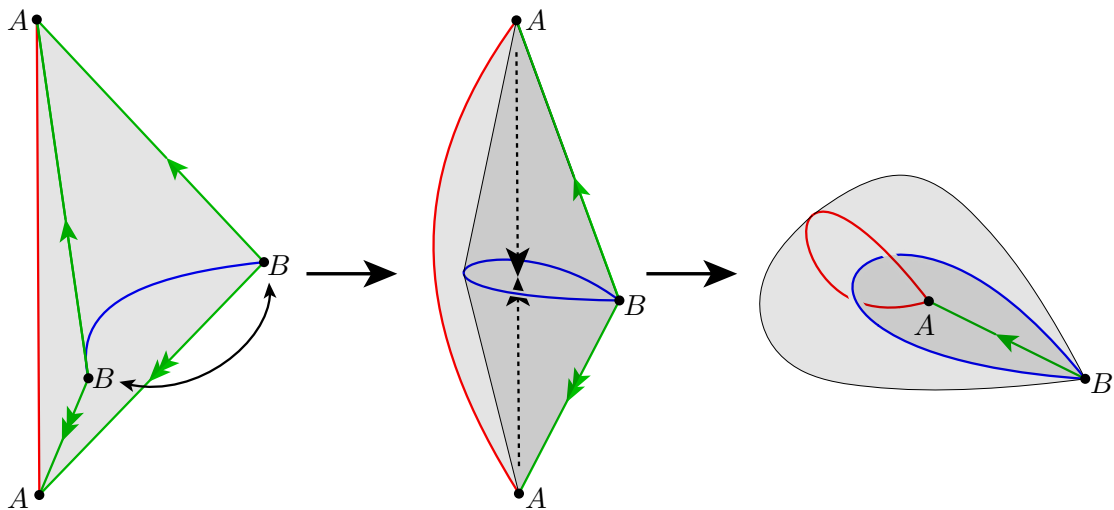
\begin{figure}[htbp]
\centering
\labellist
\small\hair 2pt
\pinlabel $A$ at -12 423
\pinlabel $A$ at -12 2
\pinlabel $A$ at 443 423
\pinlabel $A$ at 443 2
\pinlabel $A$ at 784 154
\pinlabel $B$ at 63 107
\pinlabel $B$ at 217 207
\pinlabel $B$ at 530 172
\pinlabel $B$ at 940 103
\endlabellist
\includegraphics[width=0.8\textwidth]{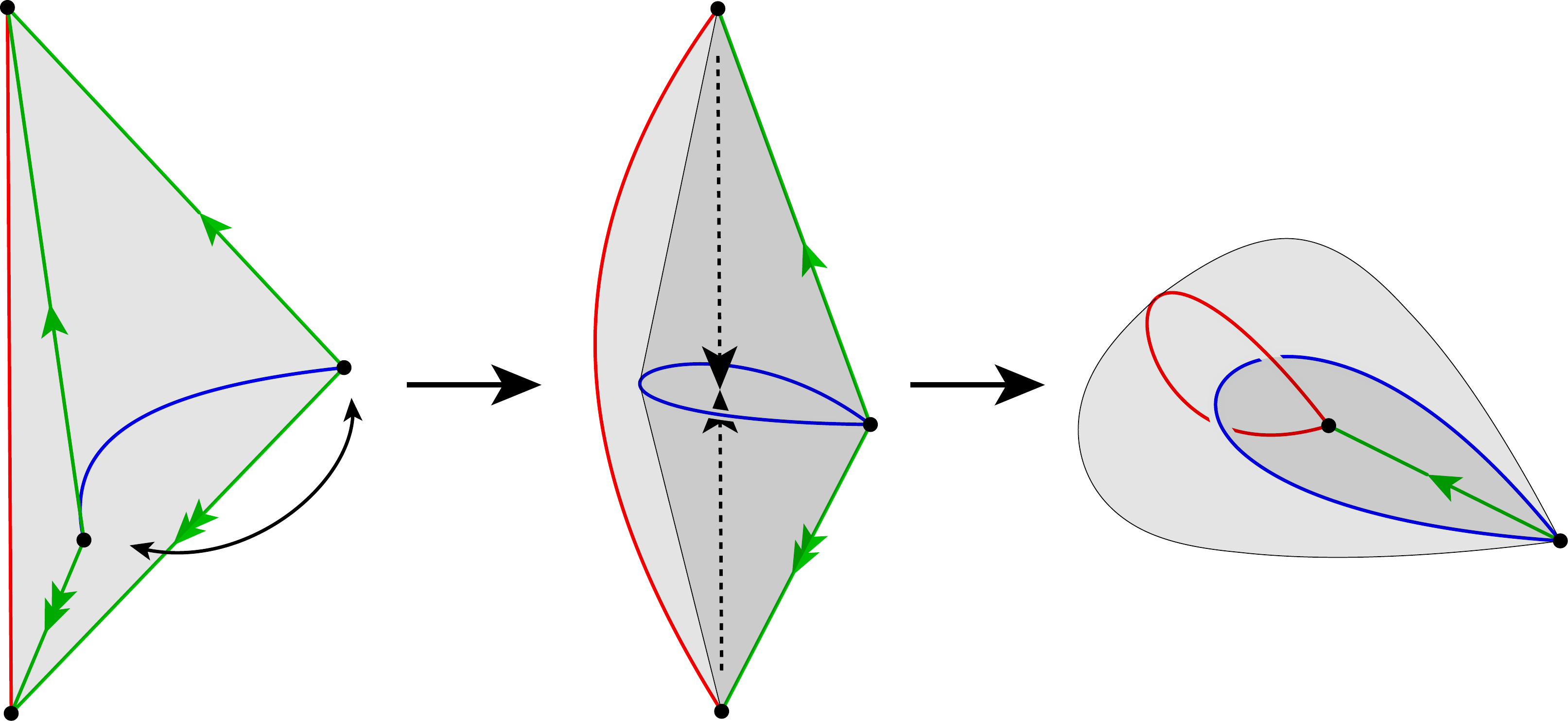}
\caption{The arch is harder to see in the dual picture of the triangulation. Here we show the shape of the single tetrahedron. First we identify the two edges marked with a single arrow, and separately the two edges marked with a double arrow. This produces the middle diagram: a shape with two triangular faces on the exterior, and two cone-shaped faces on the interior. Second, we glue the two interior faces together, which also identifies the two edges with arrows. The result has an exterior consisting of two triangular faces, with two of the vertices of these triangles identified with each other.} 
\label{fig_tetrahedron_arch}
\end{figure}

\begin{prop}
The move of inserting an arch into a spine, connecting two distinct three-dimensional regions, at least one of which is a three-ball, results in a spine of the same pseudo-manifold.
\end{prop}

\begin{proof}
We first consider the set of all triple points. It is easy to see that this is a union of open intervals: no loops have been introduced. It is slightly more subtle to see that all of the two-dimensional surfaces (the non-singular points) are again disks. The monogon region in Figure~\ref{fig_arch} is a disk. The surface containing the tube is formed by starting with the two disks $f$ and $g$ as shown in Figure~\ref{fig_arch_before}, taking their connect sum, and then cutting the result along the boundary of the monogon. This cuts open the surface to form a disk, as long as $f$ and $g$ were distinct. If $f=g$, then we have two cases. Consider the transverse coorientation on $f$ pointing from $A$ into $B$. Following this coorientation around, if we find that it also points from $A'$ into $B$, then $A=A'$, contradicting the fact that the two regions to be connected by the arch are distinct. If instead we find that the transverse coorientation also points from $B$ into $A'$, then we find that $A=B$ and $B=A'$, so we reach the same contradiction. 

Breaking symmetry, suppose that $A$ is homeomorphic to a ball. Since $A$ and $A'$ are distinct, the result of connecting together $A$ and $A'$ is homeomorphic to $A'$. The topology of the other complementary regions is unchanged.
\end{proof}

\subsection{Building an arch-with-membrane}
\label{sec_build_arch-with-membrane}

{Using 2-3 and 3-2 moves, we can make an arch-with-membrane on any edge of a spine which has more than one vertex. The moves are applied in a neighbourhood of a vertex of the spine incident to the edge. This is illustrated in Figure~\ref{fig_make_3d_arch_w_membrane}. }

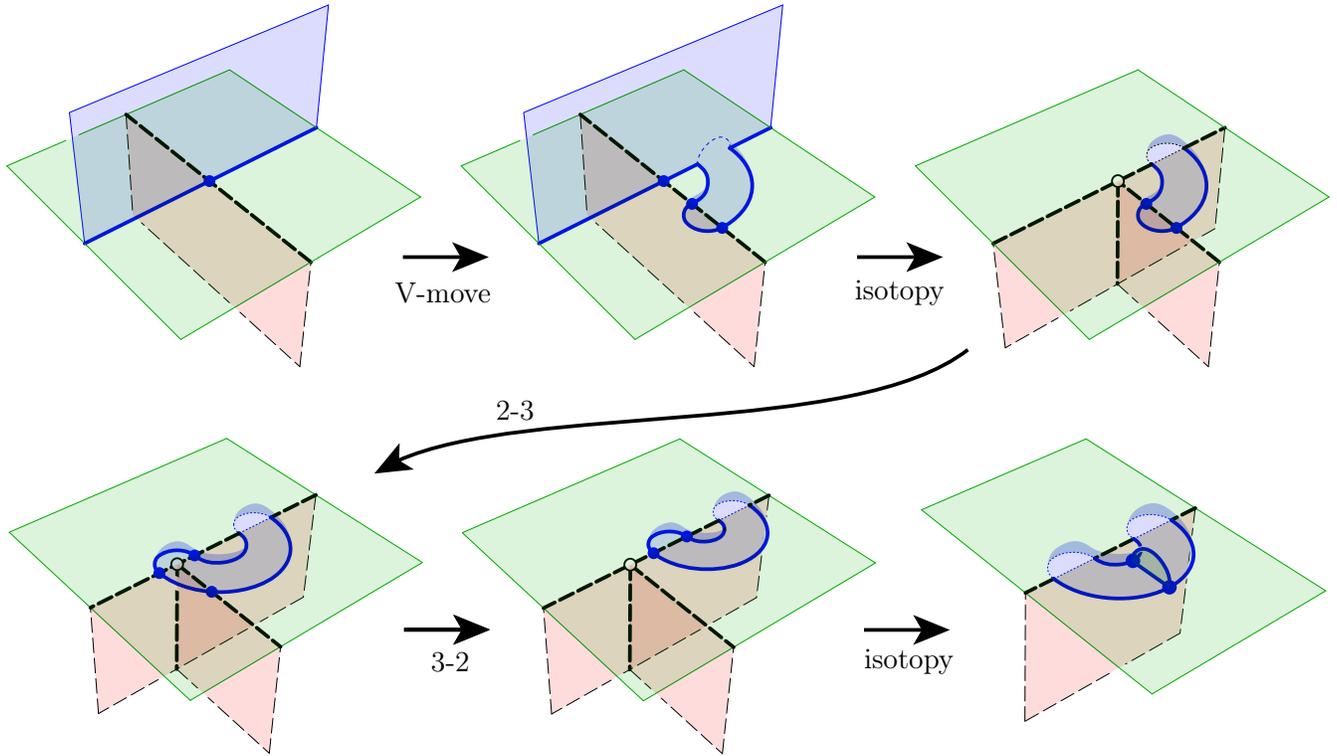
\begin{figure}[htbp]
\centering
\labellist
\small\hair 2pt
\pinlabel V-move at 236 250
\pinlabel isotopy at 483 250
\pinlabel 2-3 at 275 186
\pinlabel 3-2 at 240 50
\pinlabel isotopy at 488 50
\endlabellist
\includegraphics[width=\textwidth]{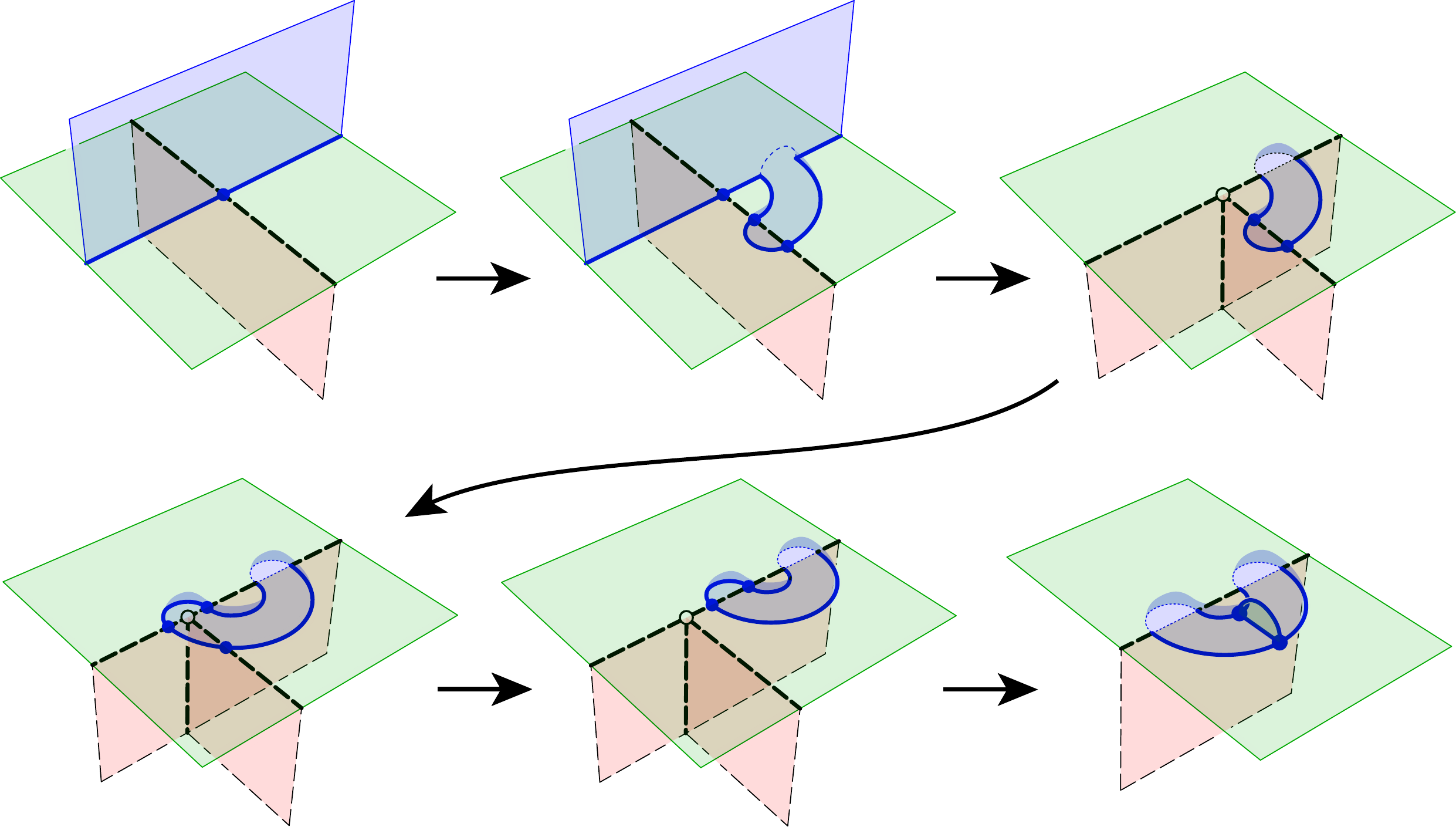}

\caption{Building an arch-with-membrane. We perform a V-move, a 2-3 move and then a 3-2 move. An isotopy moves the bigon along the tube. (In the last diagram here we have zoomed in to the arch-with-membrane and no longer see the original vertex.) A further isotopy (not shown), pushing the bigon in the final diagram forward gives the arch-with-membrane as shown in Figure~\ref{fig_arch_w_membrane}. } 
\label{fig_make_3d_arch_w_membrane}
\end{figure}

\subsection{Implementing a 1-4 move followed by introducing an arch using 2-3 and 3-2 moves}
\label{sec_1-4_arch_implementation}

{Similarly, we cannot perform a 1-4 move using only 2-3 and 3-2 moves. However, we can recreate the result of performing a 1-4 move followed by inserting an arch using 2-3 and 3-2 moves.}
This is achieved by constructing an arch-with-membrane, then applying two 2-3 moves. See Figure~\ref{fig_1-4_arch}.

\begin{figure}[htbp]
\centering
\labellist
\small\hair 2pt
\pinlabel 1-4 at 220 475
\pinlabel \text{insert arch} at 610 465
\pinlabel \text{build arch-with-membrane} at 90 210
\pinlabel 2-3 at 435 40
\pinlabel 2-3 at 640 225
\endlabellist
\includegraphics[width=\textwidth]{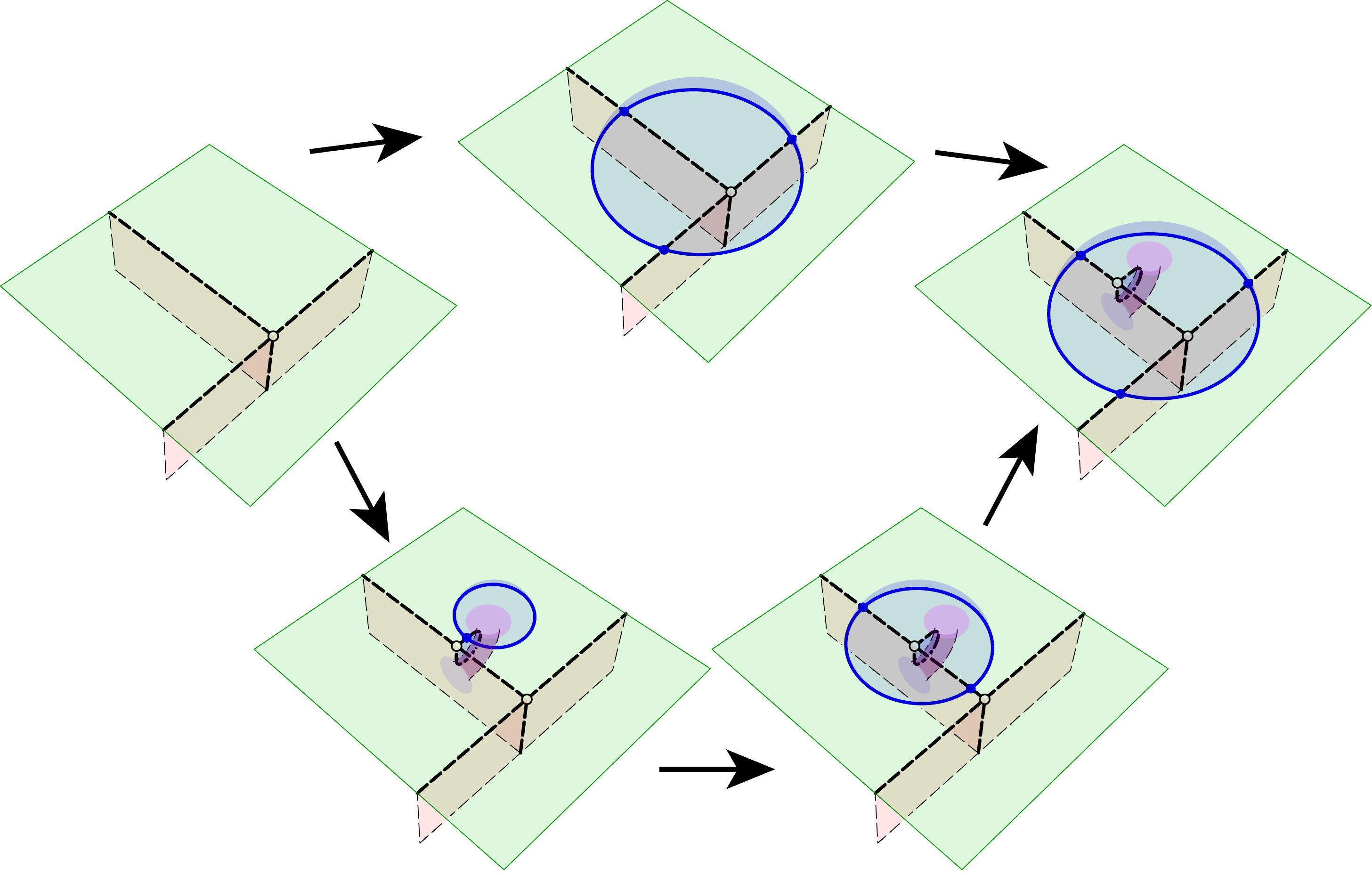}

\caption{By building an arch-with-membrane and performing two 2-3 moves, we implement the composition of a 1-4 move with inserting an arch while keeping the number of material vertices constant.} 
\label{fig_1-4_arch}
\end{figure}


\section{Details in the main proof}
\label{sec:details}

The point of time has come where we cannot further postpone diving into the intricacies of the proof of the main theorem. This section explains how we position and, if required, move arch marks,  and how we connect waypoint triangulations with 2-3 and 3-2 moves. The overall aim is to preprocess the triangulations in such a way, that all that remains is a step that is easy to describe but requires work to justify in detail: sweeping a membrane across a ball, from one arch to another. This very last step is given in \S\ref{move_membrane}.

\subsection{Arch mark positioning}

First, we describe where the arch mark is on each of $\cT'_2$ through $\cT'_{n-1}$. Then we describe sequences of 2-3 and 3-2 moves to take $\cT_1$ to $\widehat{\cT}'_2$, to take each waypoint triangulation $\widehat{\cT}'_i$ to the subsequent $\widehat{\cT}'_{i+1}$, and to take $\widehat{\cT}'_{n-1}$ to $\cT_n$.

In each triangulation $\cT'_i$, for $i\in\{2,\ldots,n-1\}$, the arch mark must connect a material vertex to some other vertex (either material or ideal).\footnote{Note that joining a material vertex to itself would produce an ideal vertex, and connecting two ideal vertices by an arch would act as a connect sum on their links.}

In $\cT'_2$, we choose the arch mark to connect the new vertex formed by the 1-4 move to one of the other vertices incident to it. See Figure~\ref{fig_1-4_arch_mark}.
There are a number of cases under which we determine the new arch mark on $\cT'_{i+1}$ based on the arch mark on $\cT'_i$.
If an arch mark is on a triangle that is not deleted by the bistellar move relating $\cT_i$ and $\cT_{i+1}$, then we leave it where it is, duplicating it on $\cT'_{i+1}$. If an arch mark is on a triangle that is deleted by a 2-3 or 3-2 move, then we must first move it to another triangle of $\cT_i$ that is not deleted by the move. We denote the resulting marked triangulation by $\cT''_i$. This done, we perform the bistellar move and get $\cT'_{i+1}$. We will require that the arch marks on $\cT'_i$ and $\cT''_i$ share a material vertex. This will allow us to move the arch itself from $\widehat{\cT}'_i$ to $\widehat{\cT}''_i$ in Section~\ref{move_arch}.
We deal with the following cases:

\begin{figure}
\centering
\labellist
\small\hair 2pt
\pinlabel $\cT_1$ at 23 90
\pinlabel $\cT'_2$ at 183 90
\endlabellist
\includegraphics[width=0.4\textwidth]{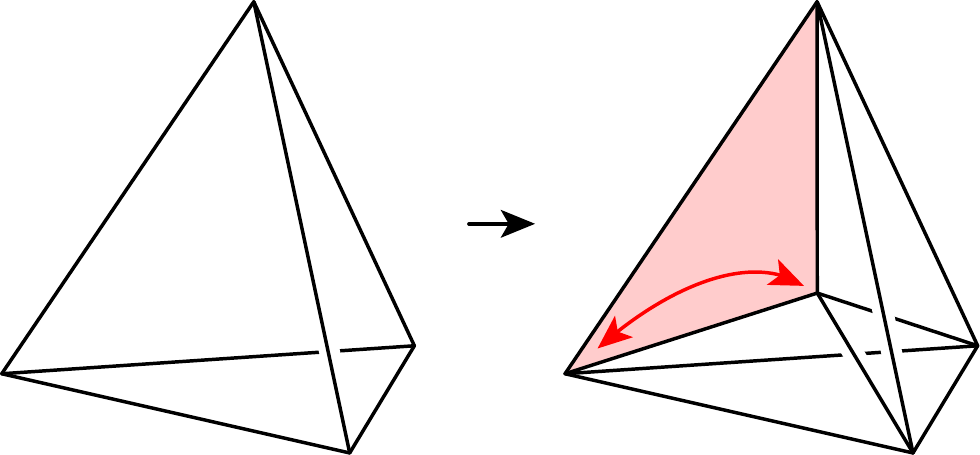}
\caption{We draw an arch mark to connect the vertex created by the 1-4 move to one of the existing vertices.}
\label{fig_1-4_arch_mark}
\end{figure}

\begin{enumerate}
\item In the case of a 2-3 move, one triangle is removed from the interior of the six-sided polyhedron involved in the move. An arch mark on this triangle connects together two of the three vertices on the ``equator'' of the polyhedron. We move such a mark to either of the two triangles on the boundary of the polyhedron that share the same two vertices. This done, the arch mark connects the same two vertices together as before, so the arch marks on $\cT'_i$ and $\cT''_i$ share a material vertex. See Figure~\ref{2-3_arch_mark}.
\item In the case of a 3-2 move, three triangles are removed from the interior of the six-sided polyhedron involved in the move. An arch mark on one of these triangles is in one of two cases: either it connects a vertex on the equator of the polyhedron to one of the two polar vertices of the polyhedron, or it connects the two polar vertices together. In the first case, we again move such a mark to one of the two triangles of the polyhedron that share the same two vertices. Again, the new arch mark connects the same two vertices so the arch marks on $\cT'_i$ and $\cT''_i$ share a material vertex. See Figure~\ref{3-2_arch_mark_case_1}. In the second case, label the two polar vertices $N$ and $S$, and let $A$ be one of the equatorial vertices. Without loss of generality, assume that $N$ is a material vertex. If $A$ is distinct from $N$ then we connect $A$ to $N$, again by adding an arch mark on a boundary triangle of the polyhedron. Otherwise we connect $A$ to $S$,  again by adding an arch mark on a boundary triangle of the polyhedron. By hypothesis, $A = N$ is distinct from $S$. Here, the arch marks on $\cT'_i$ and $\cT''_i$ share the material vertex $N$. See Figure~\ref{3-2_arch_mark_case_2}.
\end{enumerate}

\begin{figure}
\centering
\labellist
\small\hair 2pt
\pinlabel $\cT'_i$ at 20 20
\pinlabel $\cT''_i$ at 185 20
\pinlabel $\cT'_{i+1}$ at 350 20
\endlabellist
\includegraphics[width=0.6\textwidth]{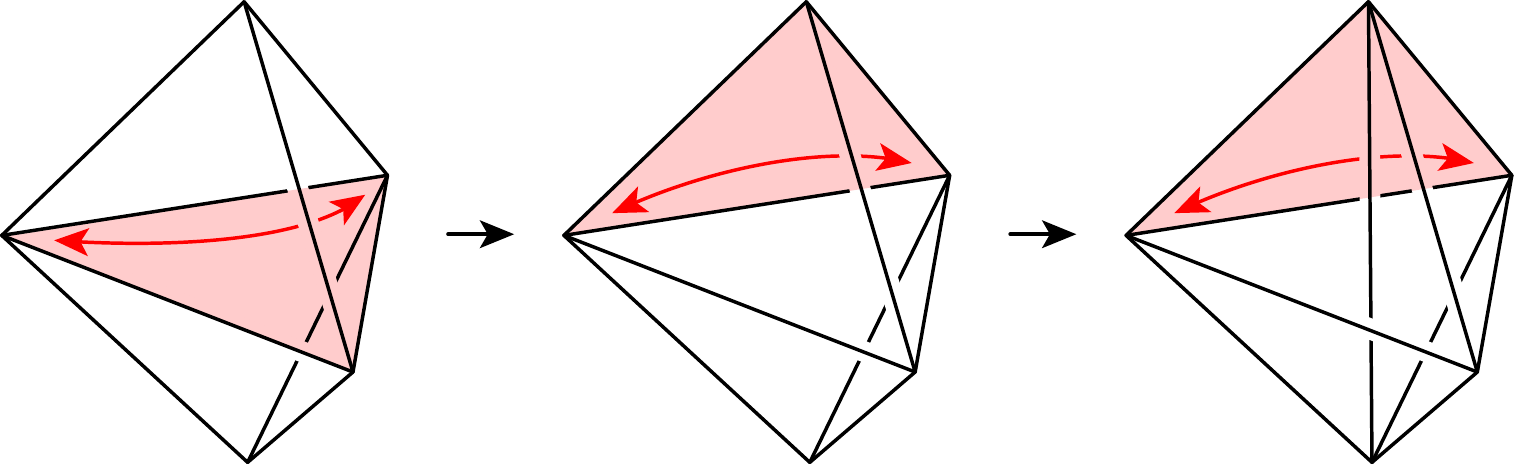}
\caption{If a 2-3 move deletes the face that the arch mark is on, we move the arch mark so as to connect the same pair of vertices before performing the 2-3 move.}
\label{2-3_arch_mark}
\end{figure}

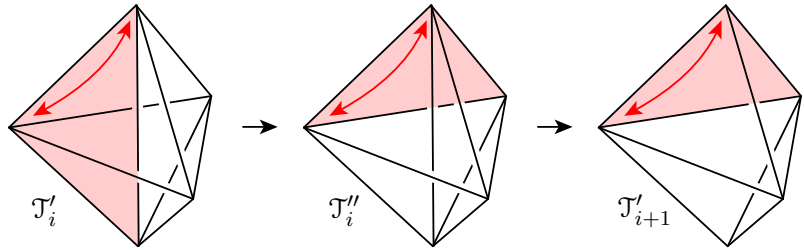
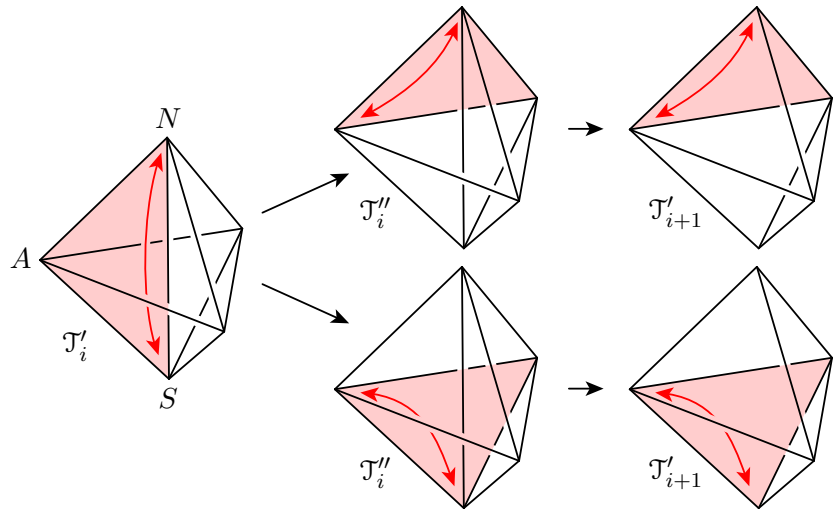
\begin{figure}

\subfloat[Case 1. Here, we move the arch mark so as to connect the same pair of vertices before performing the 3-2 move.]{
\centering
\labellist
\small\hair 2pt
\pinlabel $\cT'_i$ at 20 20
\pinlabel $\cT''_i$ at 185 20
\pinlabel $\cT'_{i+1}$ at 350 20
\endlabellist
\includegraphics[width=0.6\textwidth]{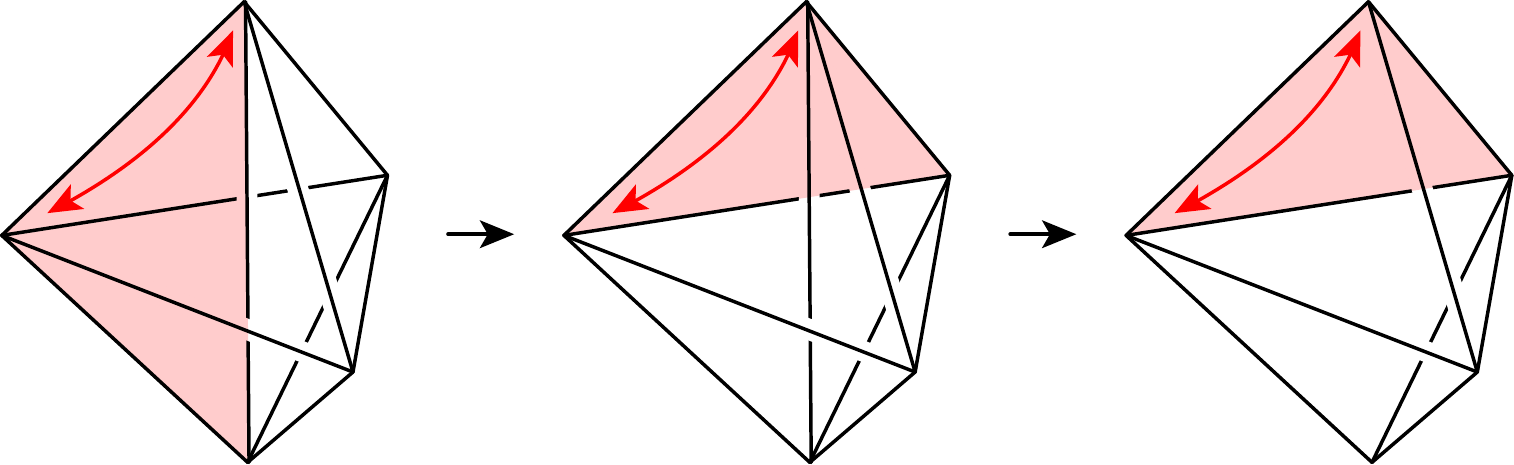}
\label{3-2_arch_mark_case_1}
}

\subfloat[Case 2. We take the upper path if $A$ is distinct from $N$, and the lower path if not. Here, the new arch mark may or may not connect the same pair of vertices together as the old arch mark.]{
\centering
\labellist
\small\hair 2pt
\pinlabel $N$ at 71 215
\pinlabel $S$ at 71 61
\pinlabel $A$ at -10 138
\pinlabel $\cT'_i$ at 20 90
\pinlabel $\cT''_i$ at 185 20
\pinlabel $\cT'_{i+1}$ at 350 20
\pinlabel $\cT''_i$ at 185 163
\pinlabel $\cT'_{i+1}$ at 350 163
\endlabellist
\includegraphics[width=0.6\textwidth]{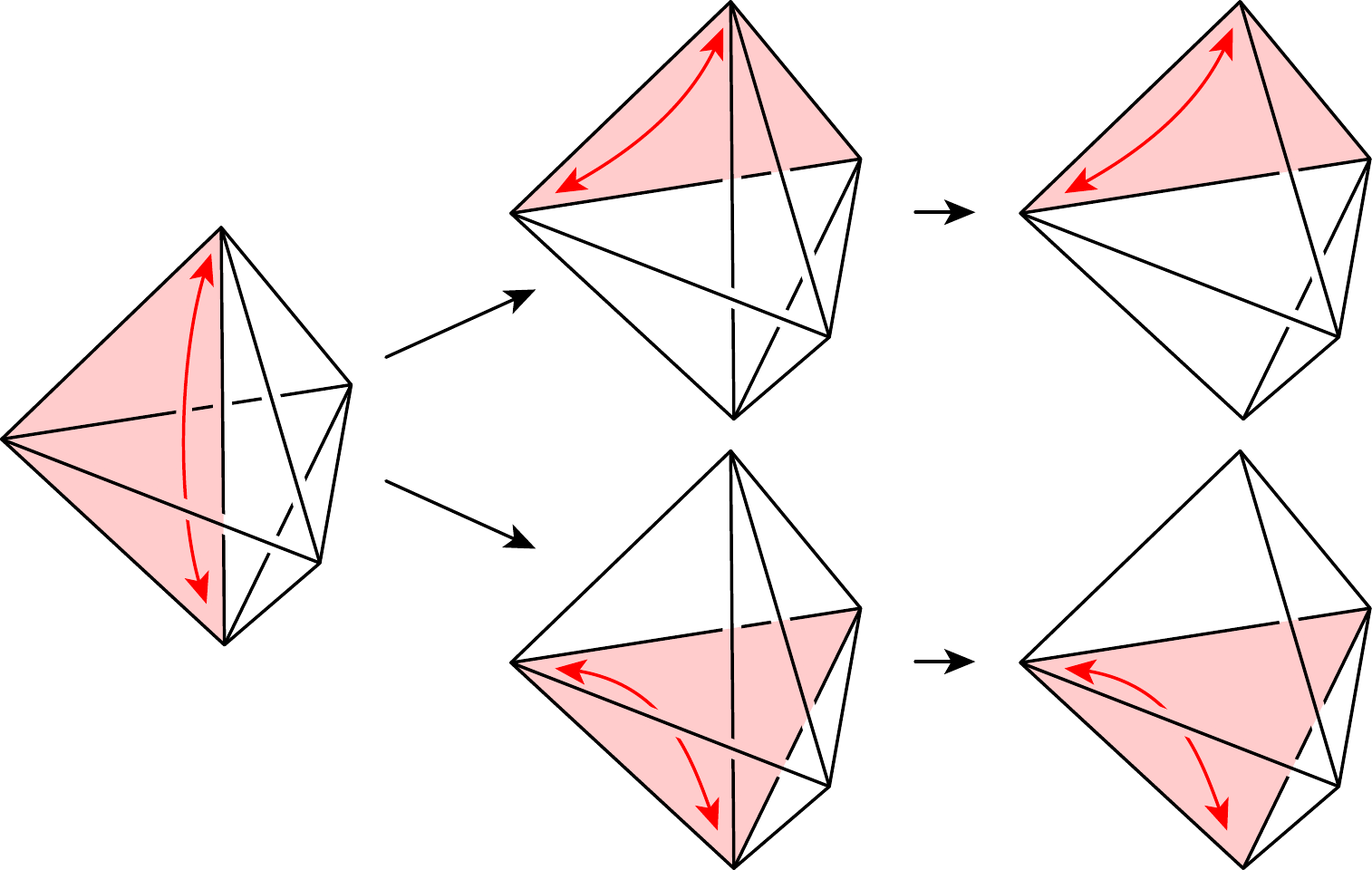}
\label{3-2_arch_mark_case_2}
}
\caption{If a 3-2 move deletes the face that the arch mark is on, we move the arch mark before performing the 3-2 move.}
\label{3-2_arch_mark}
\end{figure}

With the above moves, we can add arch marks to all triangulations $\cT_2, \ldots \cT_{n-1}$ in the sequence, and produce all of the waypoint triangulations $\widehat{\cT}'_i$, all of which have the same number of material vertices as $\cT_A$ and $\cT_B$ by construction. 

\subsection{Moving from $\cT_1$ to $\widehat{\cT}'_2$}
\label{T1_to_T2}

The triangulation $\cT_1$ is related to $\cT_2$ by a 1-4 move. we need a sequence of 2-3 and 3-2 moves that produces the result of a 1-4 move followed by adding an arch, converting $\cT_1$ to $\widehat{\cT}_2$. We gave details of this in Section \ref{sec_1-4_arch_implementation}.

\subsection{Moving between waypoint triangulations using 2-3 and 3-2 moves}
\label{move_arch}

If the arch mark does not move between $\cT'_i$ and $\cT'_{i+1}$ (since the 2-3 or 3-2 move does not destroy its triangle), then $\widehat{\cT}'_i$ and $\widehat{\cT}'_{i+1}$ are related by the same 2-3 or 3-2 move, so we connect them using that move.

If however we had to move the arch mark first, via a marked triangulation $\cT''_i$, then we need a way to move the arch via 2-3 and 3-2 moves, connecting the triangulations $\widehat{\cT}'_i$ and $\widehat{\cT}''_i$. Having moved the arch, we then connect $\widehat{\cT}''_i$ to $\widehat{\cT}'_{i+1}$ by the same 2-3 or 3-2 move that connects $\cT_i$ to $\cT_{i+1}$.

We have arranged matters so that the arch marks on $\cT'_i$ and $\cT''_i$ share a material vertex, $v$ say.
In the dual spine picture for $\cT_i$, there is a three-ball $B$ dual to $v$ since $v$ is a material vertex. In $\widehat{\cT}'_i$, $B$ is connected via a cylinder in the neck of the arch to the dual region $R$ to the vertex on the other end of the arch mark on $\cT'_i$. See Figure~\ref{move_arch_across_ball_step0}. Here the arch is to the right of the figure, and in the middle we have drawn part of the boundary of $B$, as seen from within $B$. 

First we build a \emph{arch-with-membrane} (see Figure~\ref{fig_arch_w_membrane}) in the location of the new arch mark. This can be made using 2-3 and 3-2 moves -- see Section \ref{sec_build_arch-with-membrane}. The result is shown in Figure~\ref{move_arch_across_ball_step1}. This done, we see $B$ with an arch glued to one end and an arch-with-membrane glued to the other. 

Next, we sweep the membrane from the arch-with-membrane, through $B$, and into the other arch. This turns the arch-with-membrane into an arch, and the arch into an arch-with-membrane. See Figure~\ref{move_arch_across_ball_step2}. Then, we deconstruct the new arch-with-membrane, following the reverse of the sequence of moves used to build such an arch-with-membrane. Having performed these moves, we have moved the arch. See Figure~\ref{move_arch_across_ball_step3}. We give details of this sweep move in Section~\ref{move_membrane}.

\begin{figure}[htbp]
\centering
\labellist
\small\hair 2pt
\pinlabel $B$ at 435 75
\pinlabel $R$ at 750 125
\endlabellist
\subfloat[We start with a three-ball $B$ connected via a cylinder in the neck of an arch to some other three-dimensional region $R$.]{
\includegraphics[width=0.45\textwidth]{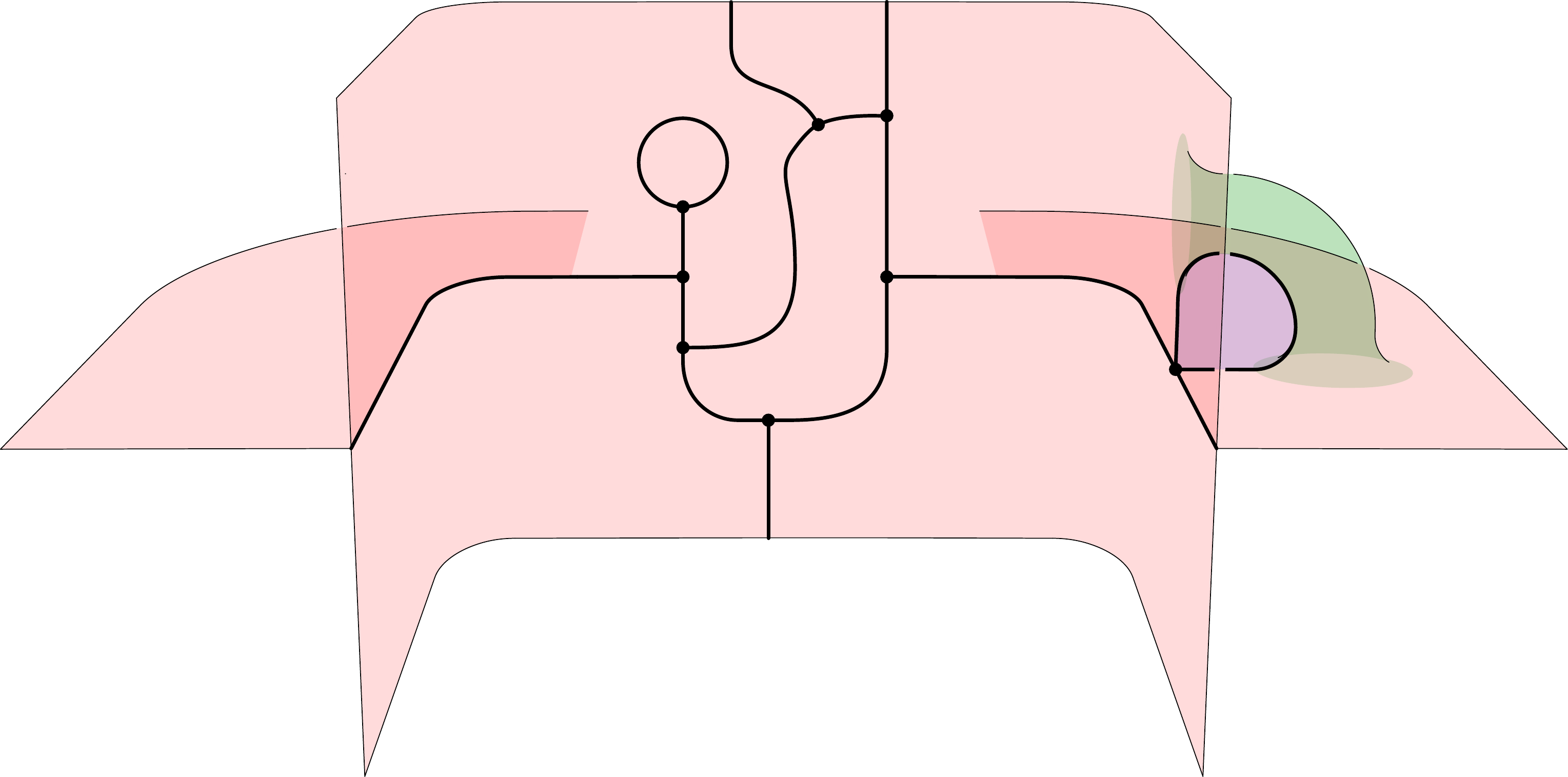}
\label{move_arch_across_ball_step0}
}\quad
\subfloat[First we build an arch-with-membrane in the desired new location for the arch.]{
\includegraphics[width=0.45\textwidth]{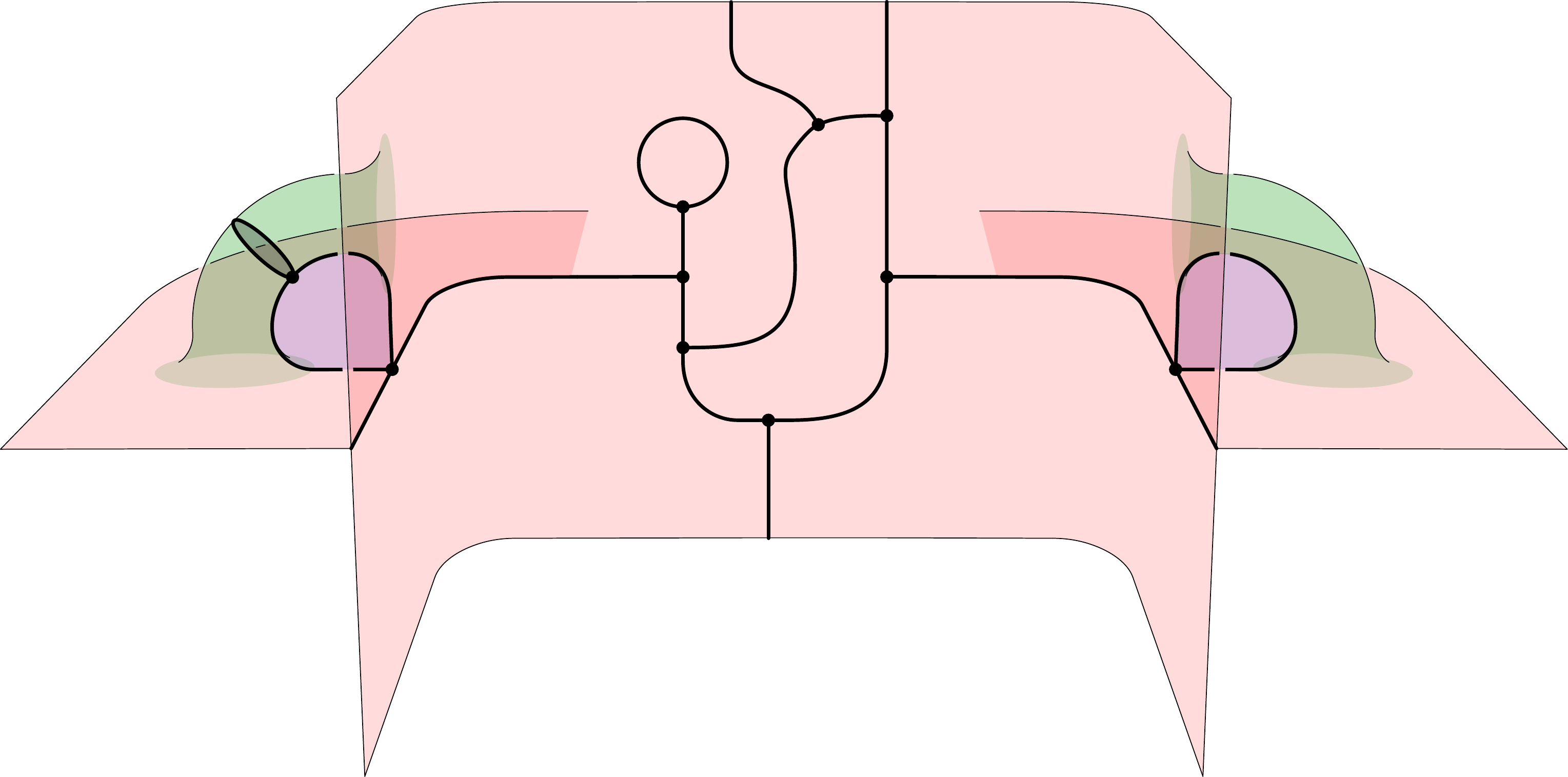}
\label{move_arch_across_ball_step1}
}

\subfloat[Next we sweep the membrane across $B$, from the arch-with-membrane into the arch.]{
\includegraphics[width=0.45\textwidth]{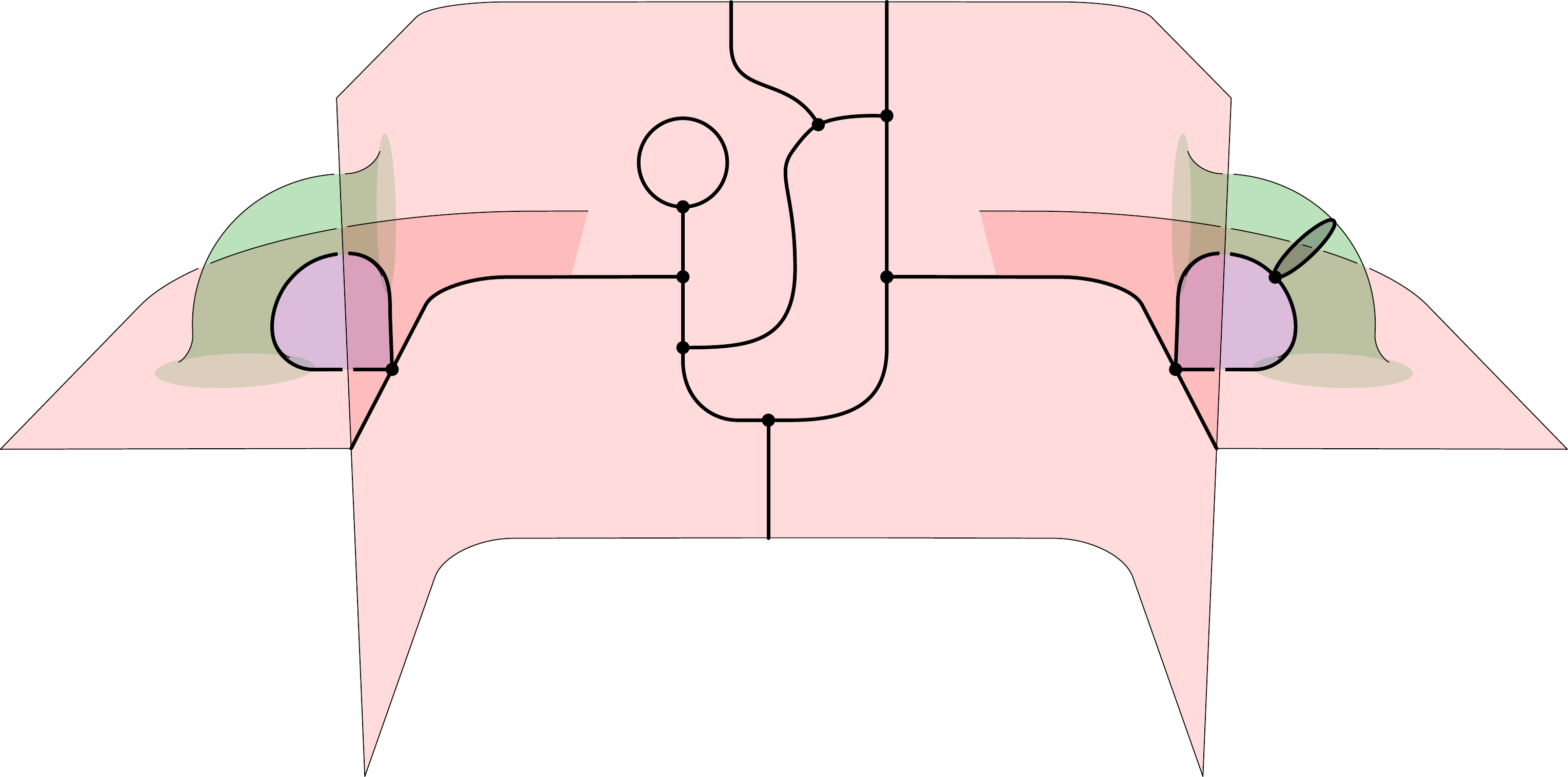}
\label{move_arch_across_ball_step2}
}\quad
\subfloat[Finally we deconstruct the resulting arch-with-membrane.]{
\includegraphics[width=0.45\textwidth]{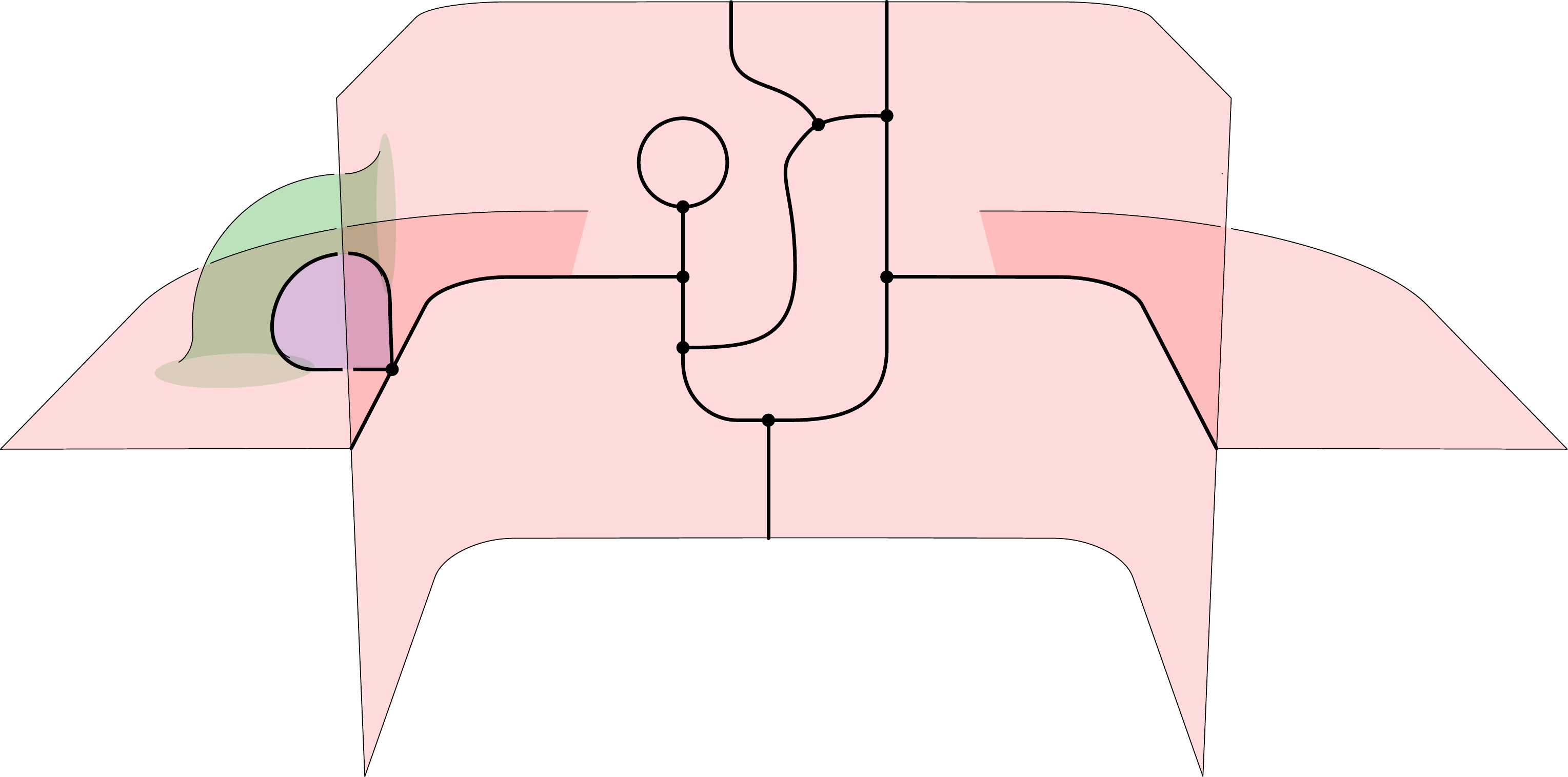}
\label{move_arch_across_ball_step3}
}

\caption{Moving an arch across a material vertex, drawn in the dual spine picture.}
\label{fig_move_arch}
\end{figure}

\subsection{Moving from $\widehat{\cT}'_{n-1}$ to $\cT_n$}

If our arch mark is connected to the material vertex $w$ we are about to remove in the 4-1 move then we perform the reverse procedure to Section~\ref{T1_to_T2} and we are done. If however the arch mark is connected to some other material vertex $v$, then we must do some further work first. If $v$ and $w$ share an edge $e$, then we move the arch mark to a triangle containing $e$ so that it connects $v$ to $w$. As in the previous section, moving the corresponding arch can be performed by sweeping across a three-ball, here it is the three-ball dual to $v$. If $v$ and $w$ do not share an edge, then we move $w$ until they do. After the 4-1 move our triangulation is $\cT_n$, and we may think of $\cT_{n-1}$ as $\cT_n$ with one of its tetrahedra, $\sigma$ say, subdivided into four. See Figure~\ref{move_4-1_vertex_from_tet_to_tet_step0}. By performing a 3-2 move on one of the four edges incident to $w$, we convert $\cT_{n-1}$ into the result of inserting a triangular pillow into a face of $\cT_n$ incident to $\sigma$. See Figure~\ref{move_4-1_vertex_from_tet_to_tet_step1}. Note that we can perform this 3-2 move because the three tetrahedra incident to this edge are distinct, since we were able to perform a 4-1 move on them and one other tetrahedron. The 3-2 move creates a new face that forms one of the two outer faces of the triangular pillow. By performing a 2-3 move on the other outer face, we then get the result of subdividing a tetrahedron $\sigma'$ of $\cT_n$ incident to $\sigma$. See Figure~\ref{move_4-1_vertex_from_tet_to_tet_step2}. Note that we can perform this 2-3 move because the two tetrahedra incident to the face are distinct: one is part of the triangular pillow, while the other is not.

These two bistellar moves transport the vertex $w$ from $\sigma$ to $\sigma'$, and once again we are a 4-1 move away from $\cT_n$. By repeatedly moving $w$ in this way, we can arrange for $v$ and $w$ to share an edge. Also note that in moving $w$ in this way, we do not destroy the triangle containing the arch mark, because in order to do so, $w$ and $v$ would already share an edge.

\begin{figure}[htbp]
\centering

\subfloat[]{
\includegraphics[width=0.3\textwidth]{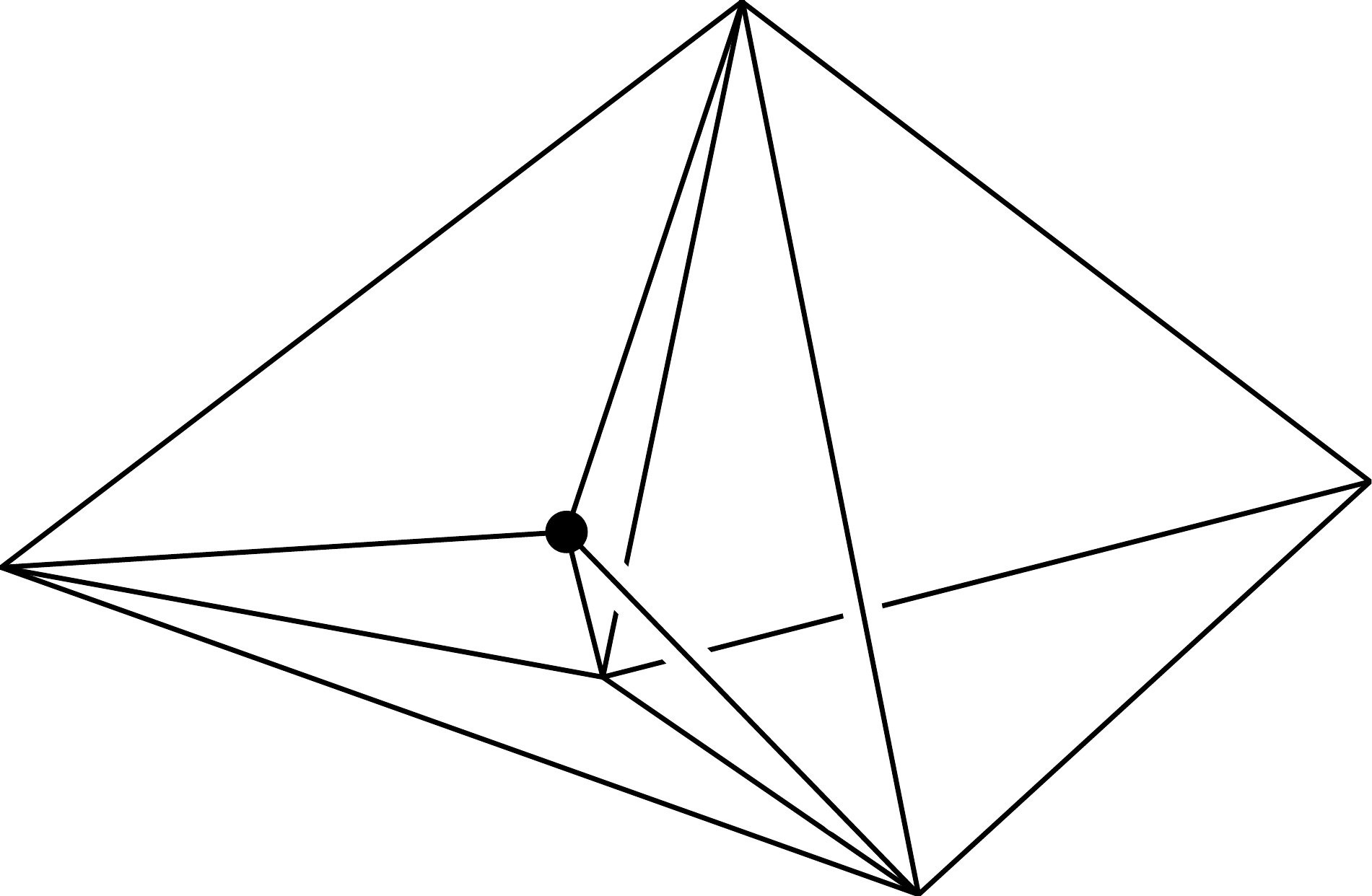}
\label{move_4-1_vertex_from_tet_to_tet_step0}
}\thinspace
\subfloat[]{
\includegraphics[width=0.3\textwidth]{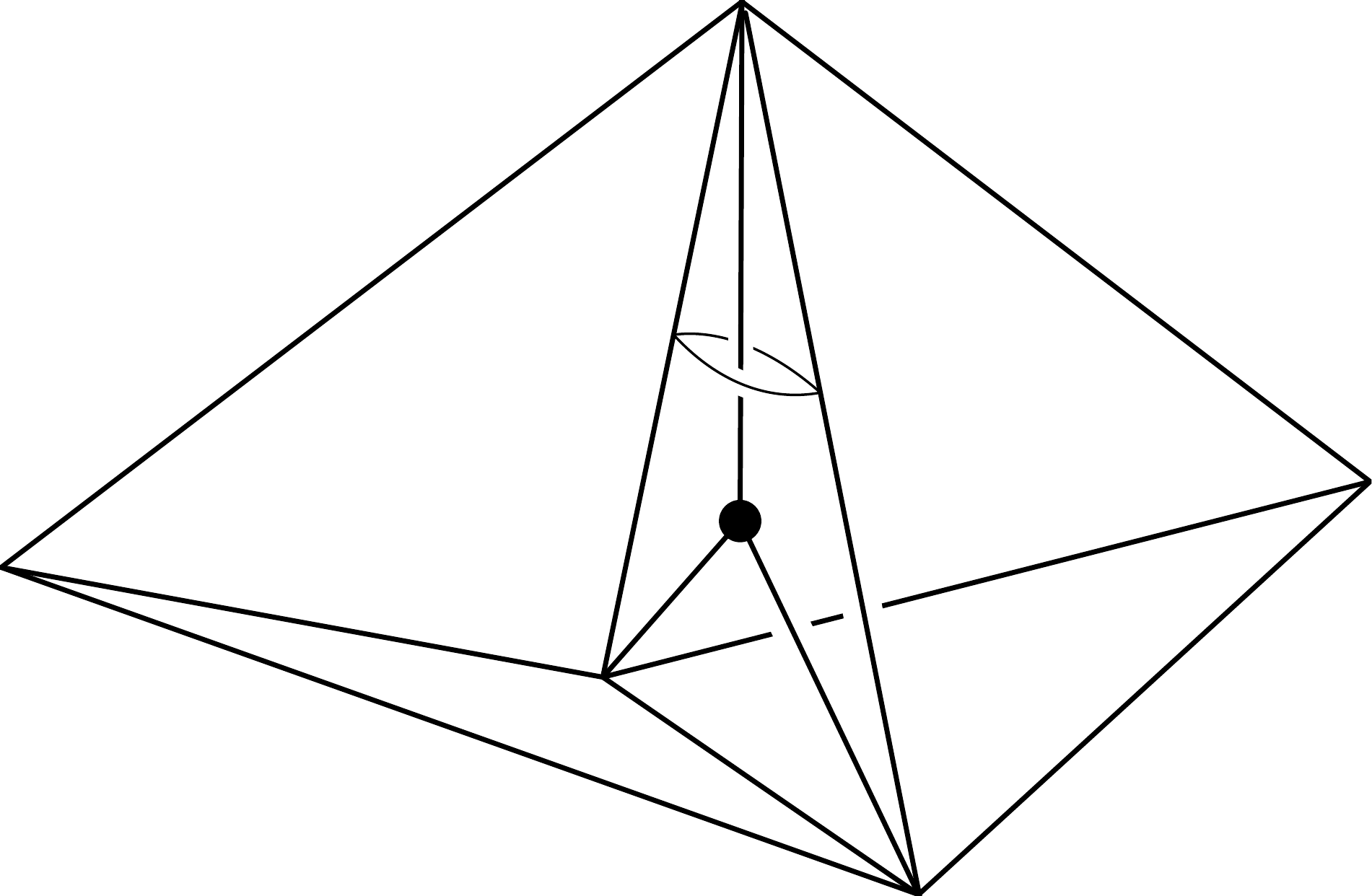}
\label{move_4-1_vertex_from_tet_to_tet_step1}
}\thinspace
\subfloat[]{
\includegraphics[width=0.3\textwidth]{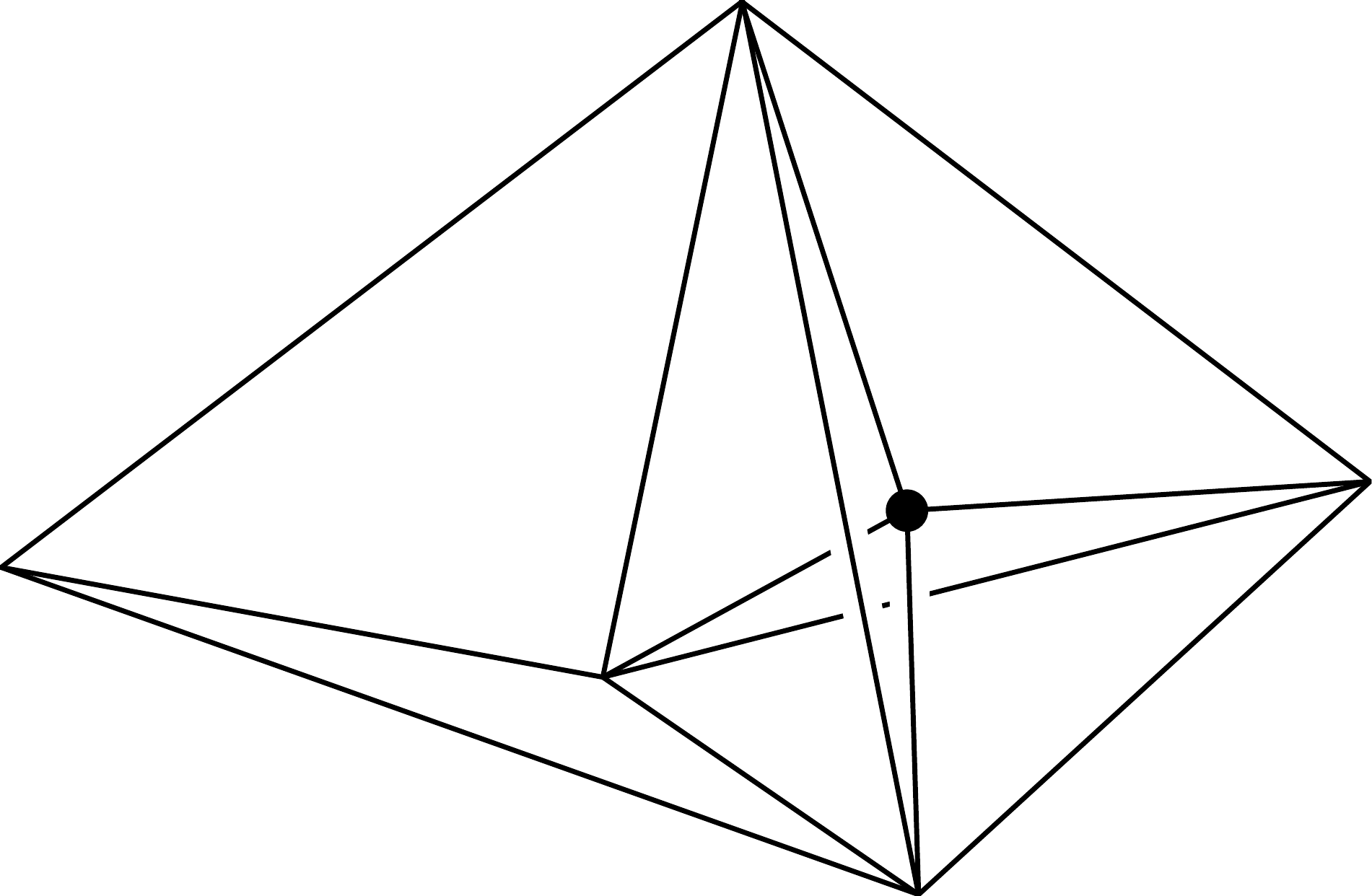}
\label{move_4-1_vertex_from_tet_to_tet_step2}
}

\caption{The vertex $w$ can be moved from one tetrahedron of $\cT_n$ to a neighbouring tetrahedron by a 3-2 move followed by a 2-3 move.}
\label{fig_move_4-1_vertex}
\end{figure}


\section{Moving a membrane across a ball, from one arch to another}
\label{move_membrane}

There are two key steps at which we depart from the proofs of Matveev and Amendola in how we sweep a membrane across a ball. We point these out in the discussion below.

We must connect together triangulations $\cT_\alpha$ and $\cT_\beta$, say, shown in dual spine form in Figures~\ref{move_arch_across_ball_step1} and \ref{move_arch_across_ball_step2} respectively. We note that this is the setting of Lemma 1.2.16 (2-cell replacement) of \cite{matveev_book}. Matveev's method is to slide the membrane across the ball from the arch-with-membrane into the arch. The combinatorial changes as one slides the membrane are 2-3 and 3-2 moves, and what we call \emph{quadrilateral 0-2 moves}. In the literature this move is often called simply a ``0-2 move'', but we add the adjective ``quadrilateral'' to distinguish them from our triangular 0-2 move.\footnote{The name comes from the dual triangulation picture: the quadrilateral 0-2 move inserts a quadrilateral pillow consisting of two tetrahedra, instead of the triangular pillow inserted in the triangular 0-2 move.} Matveev refers to a quadrilateral 0-2 move as a \emph{lune move}. The implementation of the lune move in terms of 2-3 and 3-2 moves is quite subtle in full generality, due to self-gluings in the vicinity of the move. In our setting, we will be able to avoid these subtleties (although at the cost of adding other subtleties!)
The use of 2-3, 3-2 and lune moves to sweep the membrane across the ball is also somewhat tricky. The situation is (again) complicated by the fact that the ball may have self gluings, so when we isotope the membrane across a self-glued face of the ball, it potentially interacts with parts of itself meeting that face from the other side. Thus, Matveev's proof that the moves can be done to achieve the sweep involves a PL topology argument. 

In contrast, our method is entirely combinatorial. Our strategy is to first apply moves to each of the triangulations $\cT_\alpha$ and $\cT_\beta$ to convert them to triangulations $\cT'_\alpha$ and $\cT'_\beta$, which again only differ by the position of the membrane, and for which the ball has no self-gluings. This means that we can then organise the isotopy of the membrane from one arch to the other in a combinatorial way, and convert $\cT'_\alpha$ into $\cT'_\beta$.

First of all, we see what to do in the case that the ball has no self-gluings.

\subsection{The no self-gluings case}
\label{no_self-gluings}

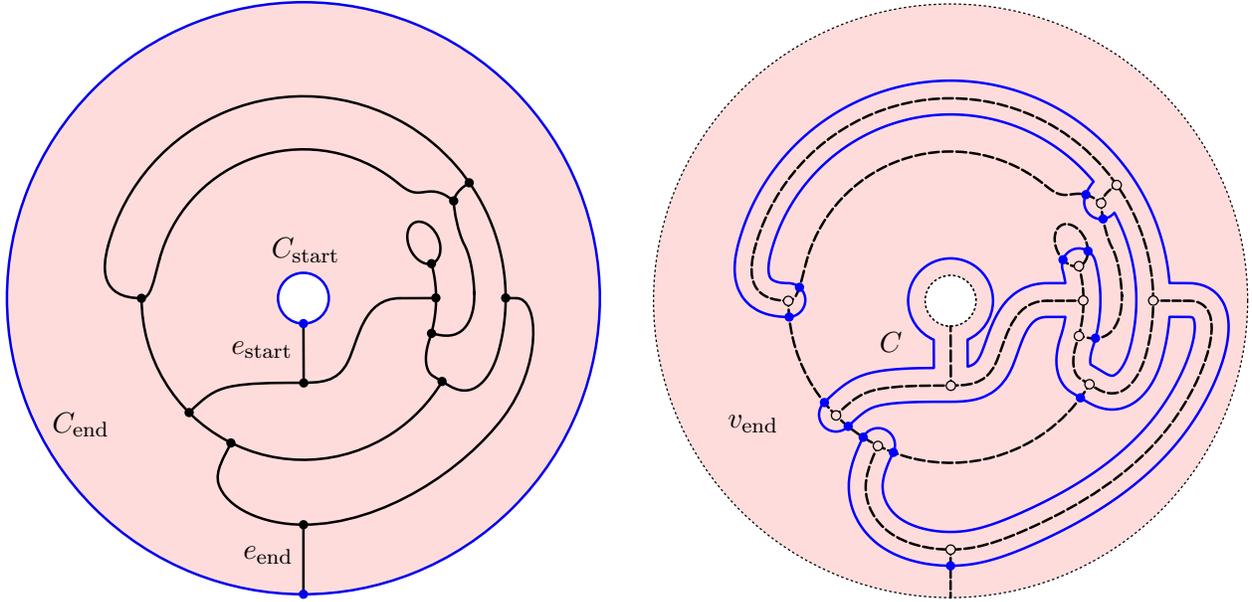
\begin{figure}[htbp]
\centering
\subfloat[The annulus $A$. The boundary of the membrane is shown at its initial position $C_\text{start}$ and final position $C_\text{end}$, which are also the boundary curves of $A$.]{
\labellist
\small\hair 2pt
\pinlabel $C_\text{start}$ at 240 280
\pinlabel $e_\text{start}$ at 205 200
\pinlabel $e_\text{end}$ at 210 35
\pinlabel $C_\text{end}$ at 60 140
\endlabellist
\includegraphics[width=0.45\textwidth]{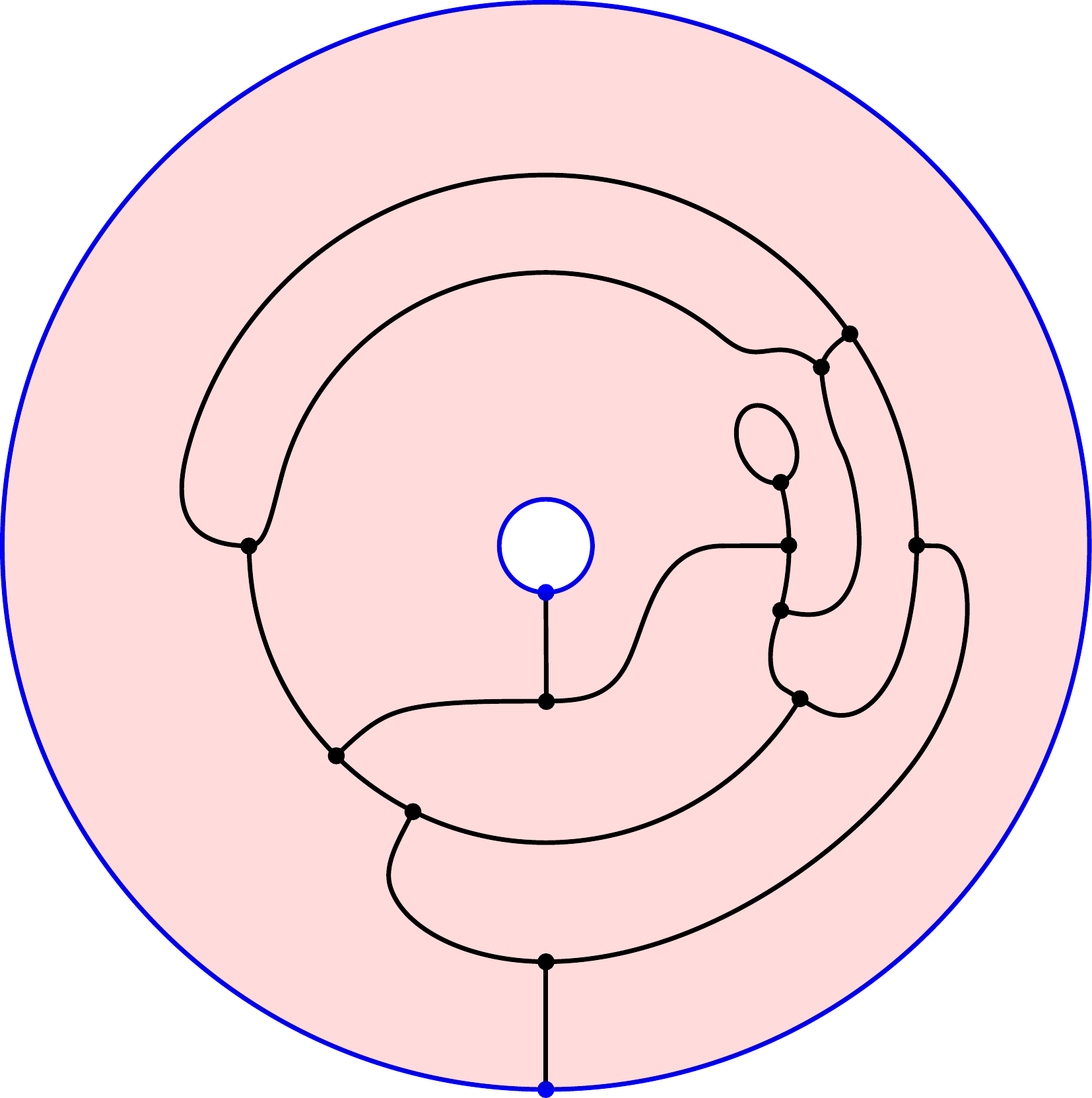}
\label{fig_annulusA_2}
}
\quad
\subfloat[After pushing $C$ over the spanning tree $T$.]{
\labellist
\small\hair 2pt
\pinlabel $C$ at 190 205
\pinlabel $v_\text{end}$ at 80 140
\endlabellist
\includegraphics[width=0.45\textwidth]{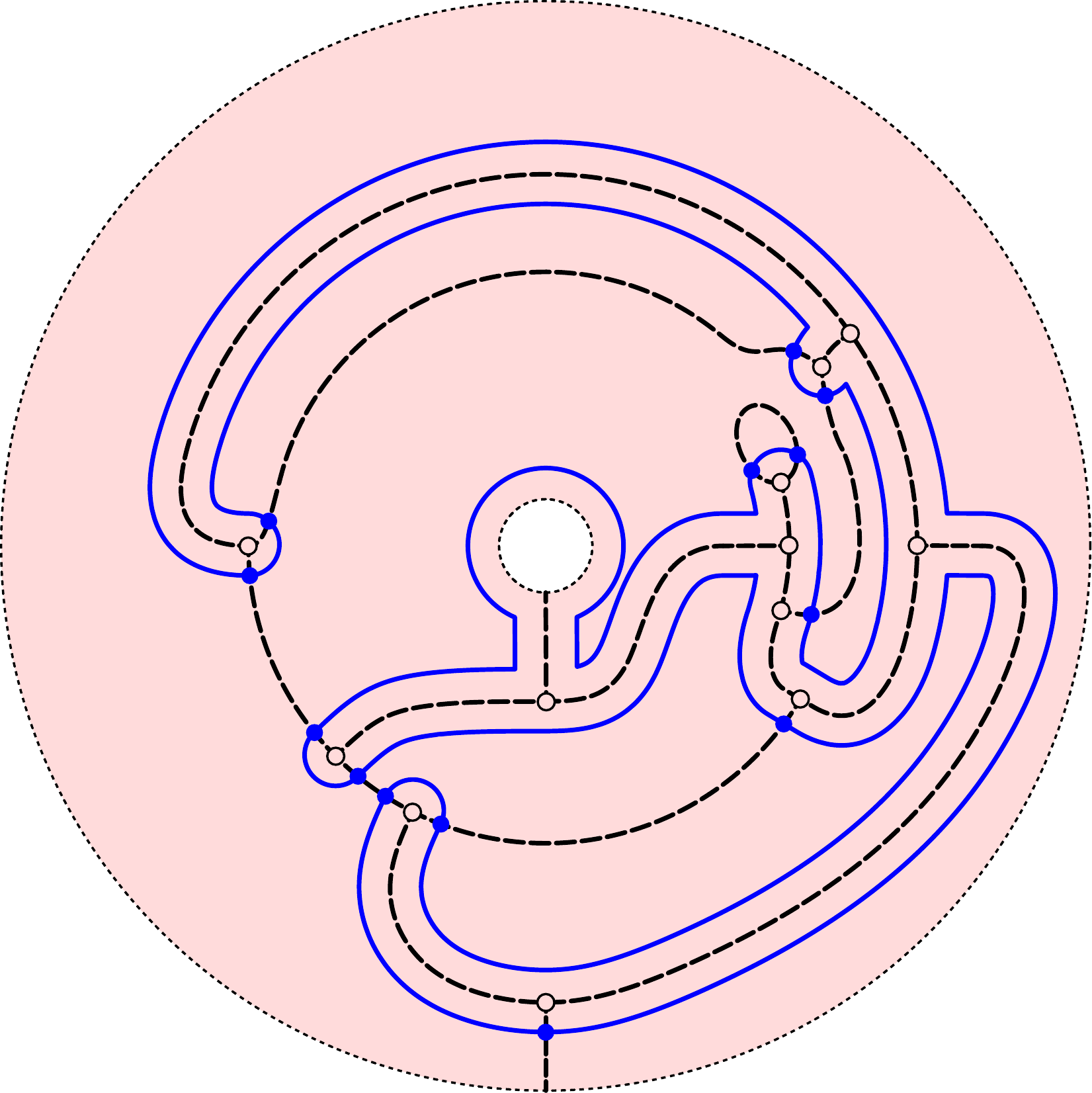}
\label{sweep_membrane_annulus_step_1}
}
\caption{The setting in which we sweep a membrane across a three-ball, from an arch-with-membrane to another arch, assuming that the ball has no self-gluings. } 
\label{fig_sweep_membrane_across_3_ball_both}
\end{figure}

Again, we must get from Figure~\ref{move_arch_across_ball_step1} to Figure~\ref{move_arch_across_ball_step2} by sweeping the membrane disk across the three-ball $B$, which we assume for the moment has no self-gluings.
Equivalently, we have to slide the boundary of the disk, a circle $C$, across the annulus $A$ formed from the boundary of the $B$ minus the initial and final positions of the membrane. Let $C_\text{start}$ and $C_\text{end}$ be the start and end positions of $C$, which are therefore also the two boundary components of $A$. 
The key point is that the pattern on the annulus $A$ formed from the intersection of the boundary of $B$ with parts of the spine outside of $B$ is a trivalent graph $\Gamma$. Let $e_\text{start}$ and $e_\text{end}$ be the edges of $\Gamma$ that intersect $C_\text{start}$ and $C_\text{end}$ respectively. See Figure~\ref{fig_annulusA_2}.

\begin{figure}[htbp]
\centering
\subfloat[The 2-3 move.]{
\labellist
\small\hair 2pt
\pinlabel 2-3 at 320 168
\endlabellist
\includegraphics[width=0.45\textwidth]{spine_slide_2-3}
\label{spine_slide_2-3}
}
\quad
\subfloat[The quadrilateral 0-2 move.]{
\labellist
\small\hair 2pt
\pinlabel 0-2 at 320 148
\endlabellist
\includegraphics[width=0.45\textwidth]{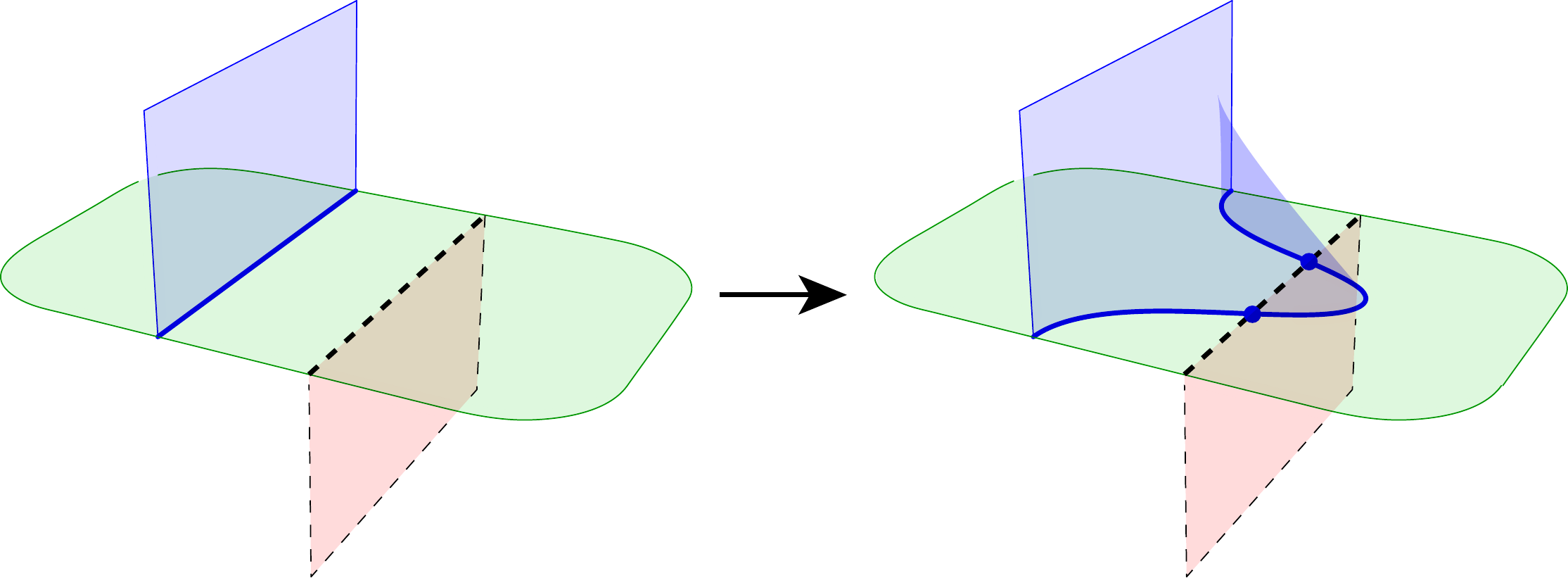}
\label{spine_slide_0-2}
}
\caption{Sliding the upper disk past parts of the spine below the plane is achieved by the 2-3 move and (the inverse of the) quadrilateral 0-2 move.} 
\label{fig_spine_slide_2-3_and_0-2}
\end{figure}

We slide $C$ across $A$ in two stages. First, we push $C$ off $C_\text{start}$ slightly into $A$, then choose a spanning tree $T$ for $\Gamma$ (not including  $e_\text{end}$), and push $C$ over $T$. This is achieved by a sequence of 2-3 moves, as shown in Figure~\ref{spine_slide_2-3}. Figure~\ref{sweep_membrane_annulus_step_1} shows the result in our example. In the second stage, we slide $C$ over the 2-cells of $A$, as follows. These 2-cells themselves form a tree $T'$, dual to $T$. There is one special vertex $v_\text{end}$ of $T'$, corresponding to the 2-cell incident to $C_\text{end}$. We may collapse $T'$ onto $v_\text{end}$ by performing a sequence of quadrilateral 2-0 moves (Figure~\ref{spine_slide_0-2} shows the reverse move). We perform a quadrilateral 2-0 move on a leaf of $T'$ to remove it from the tree as we push $C$ over the corresponding 2-cell. Since $T'$ is a tree, we can always find a leaf vertex other than $v_\text{end}$ to perform a quadrilateral 2-0 move on.

\subsection{Performing the quadrilateral 2-0 moves}

In Lemma 1.2.11 of \cite{matveev_book}, Matveev shows that the quadrilateral 0-2 move can be performed using 2-3 and 3-2 moves in a very general setting. In this paper however, 
we only prove that it can be performed in our more restrictive setting. 

Here we state the moves purely in terms of the circle intersecting the planar graph.
The meaning of ``0-2'', ``2-3'' and ``3-2'' moves is as for the boundary of the upper disk moving against the boundary of the spine below the plane, as shown in Figure~\ref{fig_spine_slide_2-3_and_0-2}.

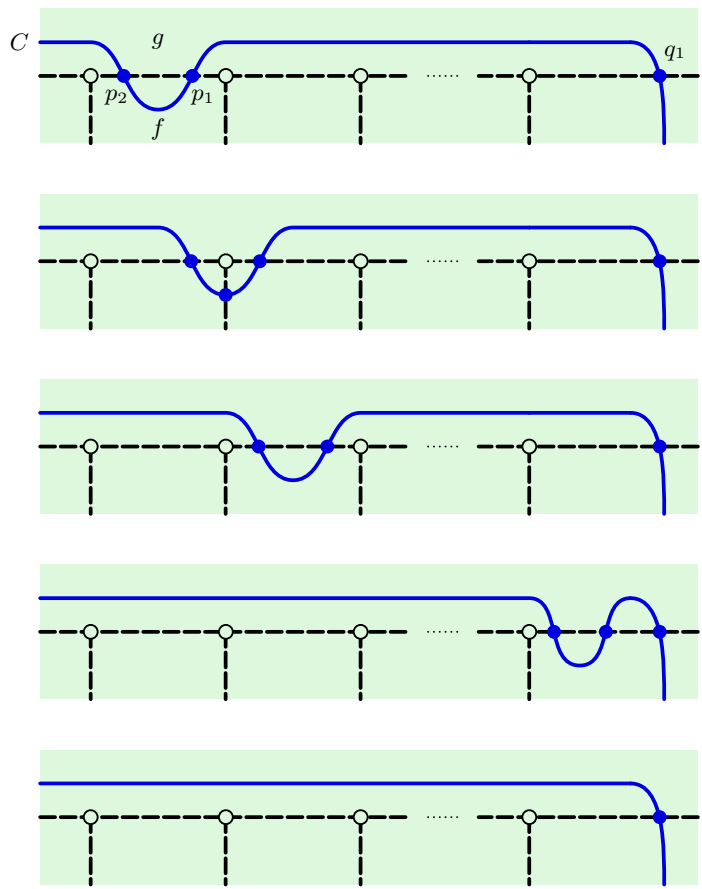
\begin{figure}[htbp]
\centering
\labellist
\scriptsize\hair 2pt
\pinlabel $q_1$ at 340 446
\pinlabel $C$ at -10 452
\pinlabel $g$ at 64 452
\pinlabel $p_2$ at 42 422
\pinlabel $p_1$ at 88 422
\pinlabel $f$ at 64 405
\endlabellist
\includegraphics[width=0.5\textwidth]{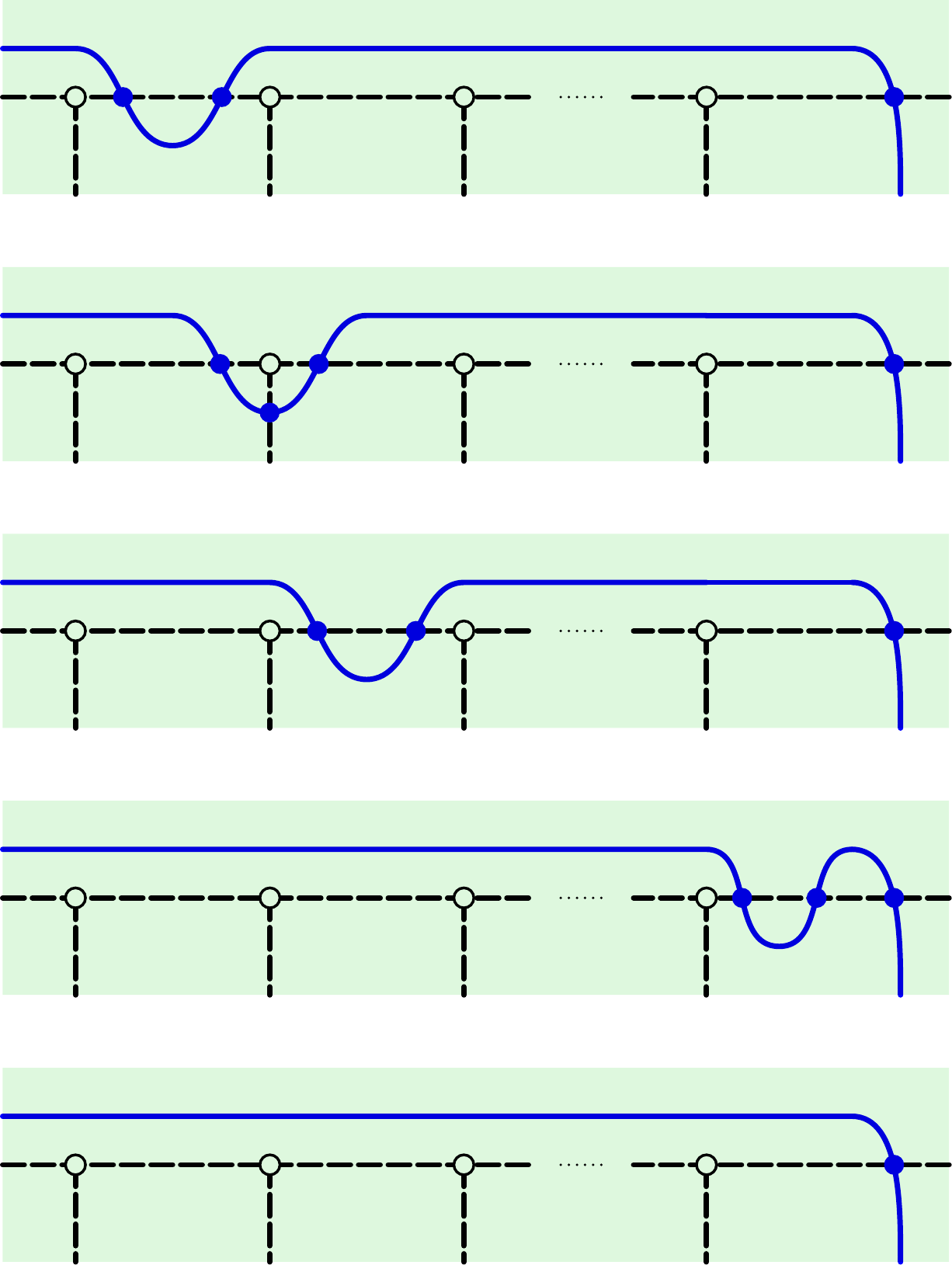}
\caption{Moves to perform a quadrilateral 2-0 move.} 
\label{fig_2-0_quad_move}
\end{figure}

See Figure~\ref{fig_2-0_quad_move}. 
We have a ``bump'' in the curve $C$, which protrudes into the face $f$. We wish to push the bump into $g$. There are always two components $A_\text{start}$ and $A_\text{end}$ of $A-C$, which contain $C_\text{start}$ and $C_\text{end}$ respectively, and we want to push the bump from the $A_\text{start}$ side towards the $A_\text{end}$ side.

The bump intersects $\bdry g$ in two points, say $p_1$ and $p_2$. The circle $C$ also intersects $\bdry g$ at some other point on $\bdry g$. This follows since if not, after the 2-0 move $C$ would lie in the interior of the face $g$ of the spine. But $C$ cannot contract to either side since it is essential in $A$. Therefore $g$ cannot be a disk, which is a contradiction.  Let $q_1$ and $q_2$ be the first such intersection points moving away from the bump in the directions of $p_1$ and $p_2$ respectively. Let $a_1$ and $a_2$ be the subarcs of $\bdry g$ between $p_1$ and $q_1$, and between $p_2$ and $q_2$ respectively. At least one of $a_1$ and $a_2$ does not contain $e_\text{start}$. Relabelling if necessary, we assume that $a_1$ does not contain $e_\text{start}$.

In our setting, $a_1$ is in $A_\text{start}$. Therefore there are no intersections between $C$ and $a_1$, and by our previous choice $a_1$ does not contain $e_\text{start}$.
So, we can move the bump along $a_1$ via a sequence of 2-3 and 3-2 moves, staying inside of $A$, as shown in Figure~\ref{fig_2-0_quad_move}, until it is next to $q_1$. Finally, the inverse of a V-move removes the bump. 

After performing all of these moves, the position of $C$ is identical to its original position, except that the original bump has been removed. This then achieves the quadrilateral 2-0 move.

In this process, we should be concerned that some moves we make might be blocked by self-identifications in our graph $\Gamma$.  It is possible that some of the edges of $a_1$ are identified with each other, corresponding to identifications among the sides of the face $g$. Note that any identification between these edges must be orientation reversing since the graph is planar. Moreover, there can be no identification between adjacent edges since the graph is trivalent. 
Each such step of the bump from sitting in one edge to sitting in the next involves only a small neighbourhood of the vertex between the two edges. By the above comments, there can be no identification between parts of $\Gamma$ involved in the step. 

\subsection{Removing self-gluings}

{Now that we know how to sweep a membrane across a ball when the ball has no self-gluings, we now have to deal with balls with self-gluings.}

Fixing some notation, we have membranes $m_\alpha$ and $m_\beta$ in $\cT_\alpha$ and $\cT_\beta$ respectively, and the difference in the position of the two membranes is the only difference between the two spines. Let $\cT_{\alpha\beta} = (\cT_\alpha \cup m_\beta) = (\cT_\beta \cup m_\alpha)$. See Figure~\ref{T_alphabeta}.

The ball $B$ is one of the connected regions of $M$ minus the spine $\cT_{\alpha\beta}$. Let $R_\alpha$ be the other incident region to $m_\alpha$, and similarly for $R_\beta$ and $m_\beta$. We have $R_\alpha \neq B \neq R_\beta$, although $R_\alpha$ may or may not be the same region as $R_\beta$. The 1-cell $\bdry m_\alpha$ is incident to two 2-cells other than $m_\alpha$, namely $f_\alpha$ and $g_\alpha$. The 2-cell $f_\alpha$ separates $B$ from another region $C_\alpha$, while $g_\alpha$ separates $R_\alpha$ from $C_\alpha$. The region $C_\alpha$ meets itself along a 2-cell $d_\alpha$ associated to the arch-with-membrane. We have similar notation for the $\beta$ side of the picture. 
The regions $C_\alpha$ and $C_\beta$ may each be the same as any of the other regions listed, including  $B$ if  $f_\alpha$ or $f_\beta$ is a self-gluing of $B$. Moreover, we may have $f_\alpha = f_\beta$ if the intersection of the 1-skeleton of the spine with $\bdry B$ does not separate the two arches, and we may have $g_\alpha = g_\beta$, connecting around outside of $B$.

\begin{figure}[htbp]
\centering
\labellist
\small\hair 2pt
\pinlabel $g_\alpha$ at 53 207
\pinlabel $C_\alpha$ at 73 377
\pinlabel $m_\alpha$ at 113 322
\pinlabel $R_\alpha$ at 113 122
\pinlabel $d_\alpha$ at 178 255
\pinlabel $f_\alpha$ at 208 377
\pinlabel $B$ at 435 75
\pinlabel $g_\beta$ at 820 207
\pinlabel $C_\beta$ at 800 377
\pinlabel $m_\beta$ at 766 322
\pinlabel $R_\beta$ at 766 122
\pinlabel $d_\beta$ at 703 255
\pinlabel $f_\beta$ at 673 377
\endlabellist
\includegraphics[width=\textwidth]{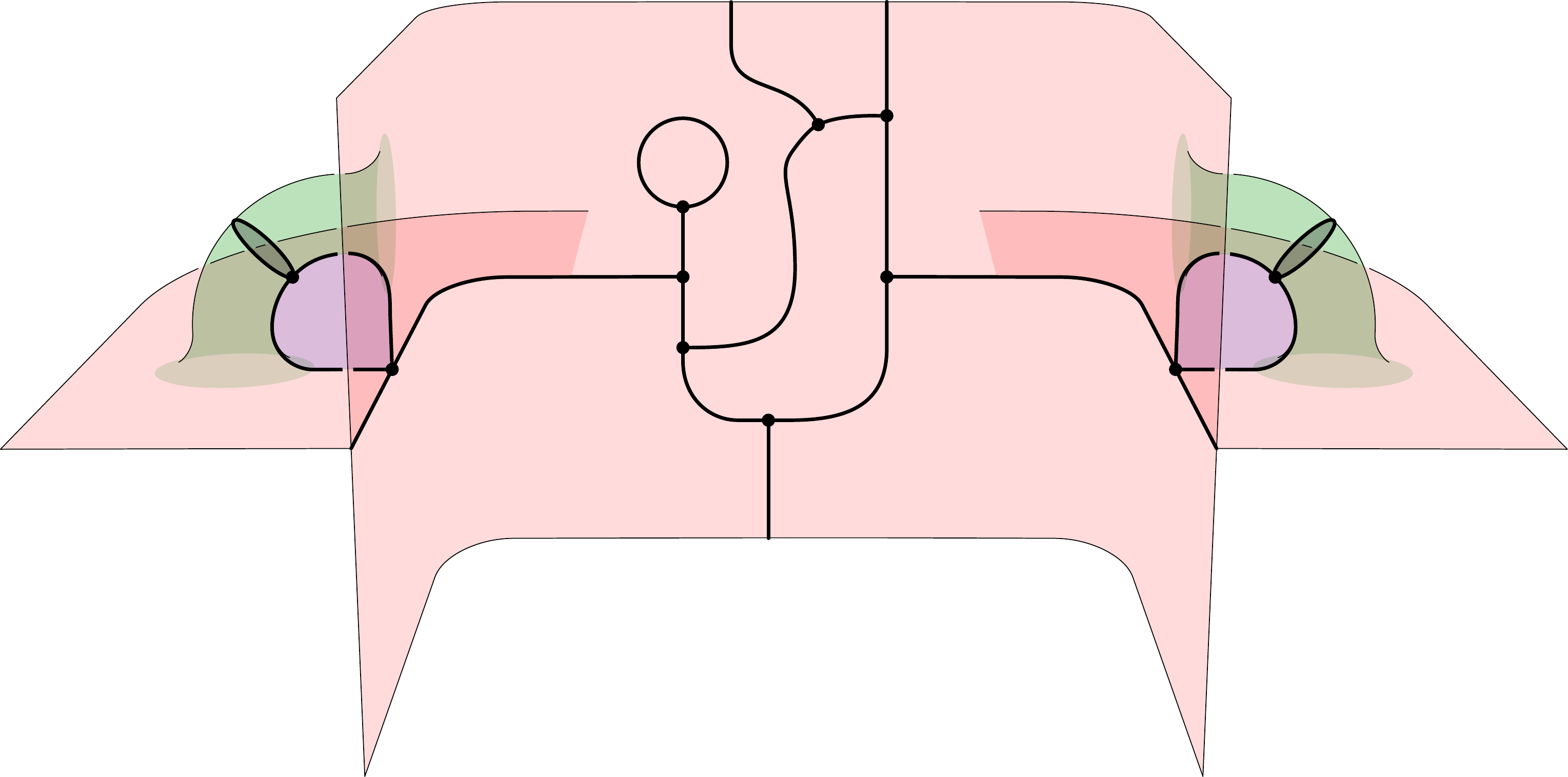}
\caption{The spine $\cT_{\alpha\beta}$. } 
\label{T_alphabeta}
\end{figure}

Note that $\bdry B$ (before we glue up the faces of $B$) consists of $(m_\alpha\cup m_\beta)$, together with an annulus $A$. After gluing, the ball $B$ may be glued to itself along some faces, which means that the annulus has some faces  identified.   

The pattern on $\bdry B \cup g_\alpha \cup g_\beta$ formed from the intersection of the boundary of $B$ with parts of the spine outside of $B$ (and parts of $\bdry d_\alpha$ and $\bdry d_\beta$ intersecting $g_\alpha$ and $g_\beta$) is a trivalent graph $\Gamma$. 

We will remove self-gluings of $B$, using moves analogous to the barycentric subdivision moves from Section~\ref{sec:Barycentric subdivision}. As we modify our triangulations to remove self-gluings, we abuse notation and continue to refer to the results of our various moves as $\cT_\alpha$,  $\cT_\beta$ and $\cT_{\alpha\beta}$. Similarly, we will maintain the same names for the various cells of our spines, even as they are modified or moved.
Having removed all self-gluings, we move the membrane from its position in $\cT_\alpha$ to its position in $\cT_\beta$, similarly to as in Section~\ref{no_self-gluings} (with a few modifications we will give after this section). This connects the two triangulations.

 In this process, we may do moves to $\cT_\alpha$ in the vicinity of $m_\beta$. We will alter the corresponding moves applied to $\cT_\beta$ (and $\cT_{\alpha\beta}$), so as to give the same result. To be more precise, we will apply various \emph{composite moves} to $\cT_\alpha$, $\cT_\beta$ and $\cT_{\alpha\beta}$. The same composite moves on these triangulations may consist of a different sequence of bistellar moves, to deal with the fact that membranes may or may not be present in different positions for the three spines. However, after each such composite move, we require that the identity $\cT_{\alpha\beta} = (\cT_\alpha \cup m_\beta) = (\cT_\beta \cup m_\alpha)$ holds. Our composite moves are given in the following sections. 

\subsubsection{Ensure that there are no $BBB$ spine edges.}
\label{remove_BBB_edges}

First, we will require that there are no edges of $\cT_{\alpha\beta}$ (thought of as a spine) whose only incident region is the ball $B$ (``$BBB$-edges'').
We can remove 
$BBB$-edges as follows. First, if there is a $BBB$-edge then there must be a vertex $v$ of the spine $\cT_{\alpha\beta}$ incident to $B$ three times and some other region, $R$ say, distinct from $B$. If no such vertex can be found, then the $BBB$-edges and their incident regions form a connected component of the manifold, which therefore contains only $B$. But we know that $B$ is incident to a vertex of the spine inside the arch-with-membrane, which is incident to $R_\alpha \neq B$, a contradiction. Having found such a vertex $v$, we perform a 2-3 move along its incident $BBB$-edge $e$, noting that the 2-3 move is possible: $e$ cannot connect $v$ to itself since no other edge incident to $v$ is a $BBB$-edge. This 2-3 move reduces the number of $BBB$-edges by one. By induction we can remove all $BBB$-edges. 

If $\bdry d_\beta$ is a $BBB$-edge in $\cT_\alpha$, then this procedure will perform a 2-3 move along it. We must apply appropriate bistellar moves to $\cT_\beta$ so that $\cT_\beta$ is identical to $\cT_\alpha \cup m_\beta$ in the vicinity of $m_\beta$. This is done by using two 2-3 moves along $\bdry d_\beta$ instead of just one. See Figure~\ref{do_two_2-3_moves}. We apply a similar technique if $\bdry d_\alpha$ is a $BBB$-edge in $\cT_\beta$.

\begin{figure}[htbp]
\centering
\subfloat[]{
\includegraphics[width=0.28\textwidth]{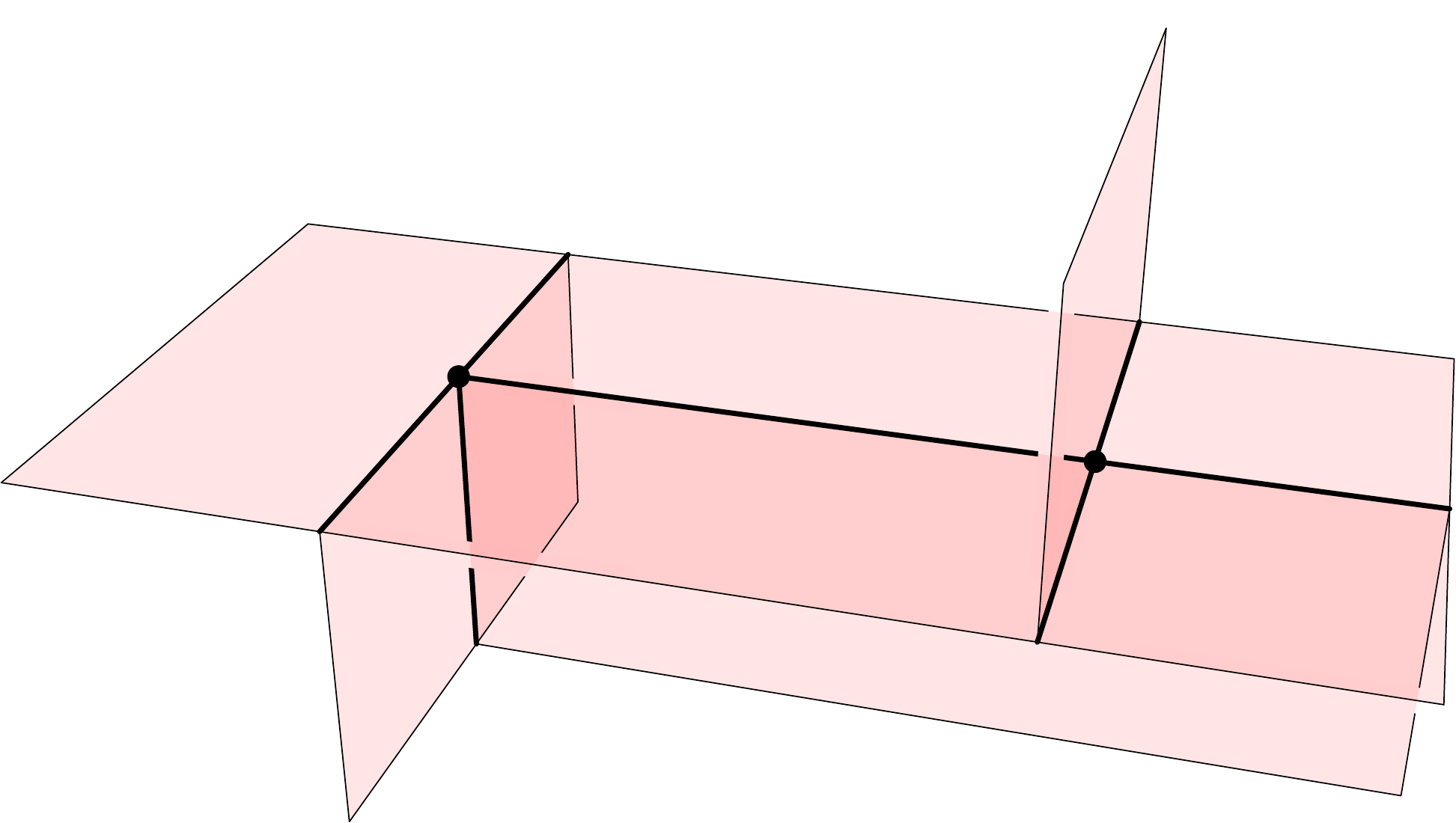}
}
\quad\qquad\qquad\qquad\qquad\qquad\qquad
\subfloat[]{
\includegraphics[width=0.28\textwidth]{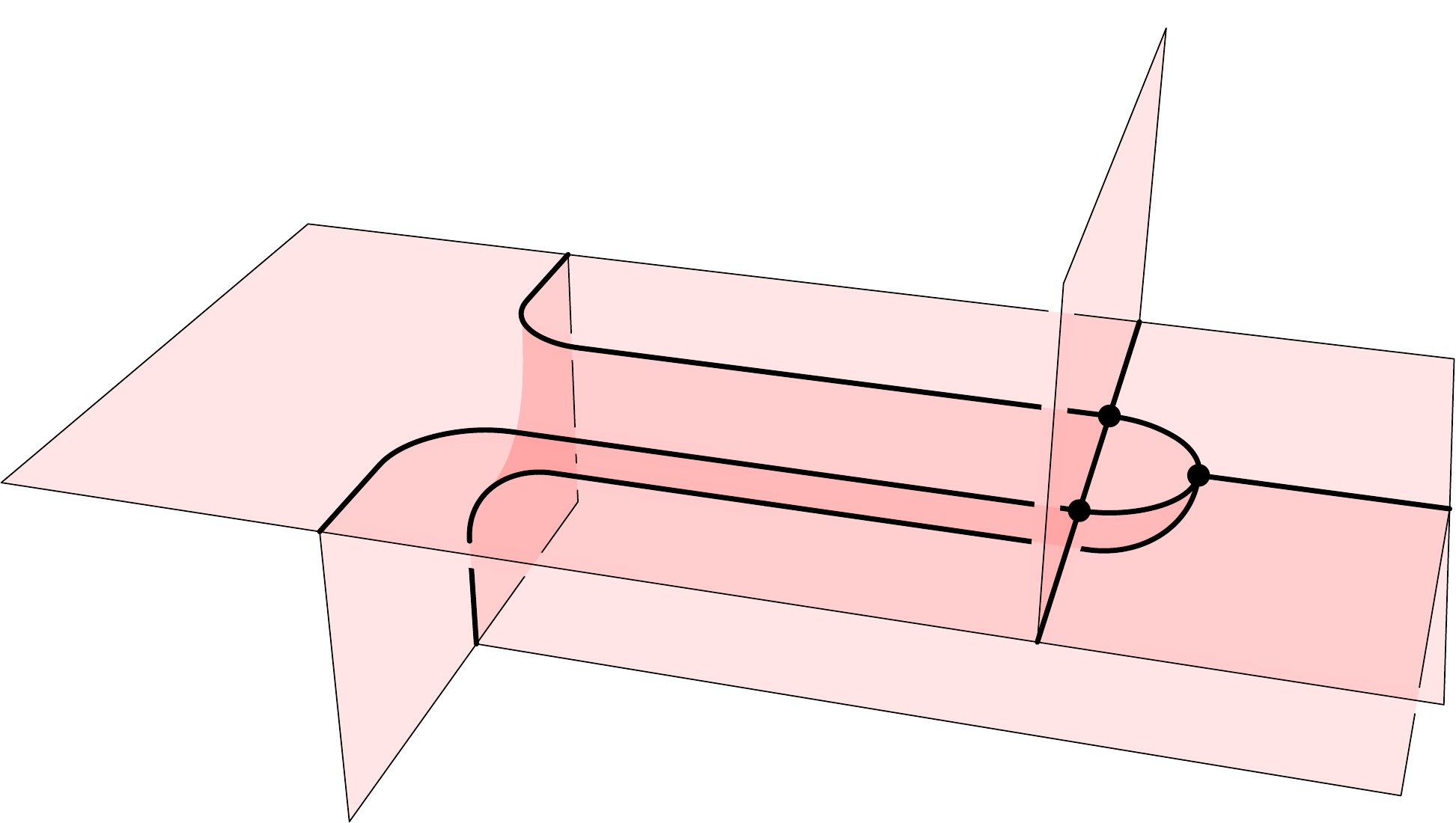}
}

\subfloat[]{
\includegraphics[width=0.28\textwidth]{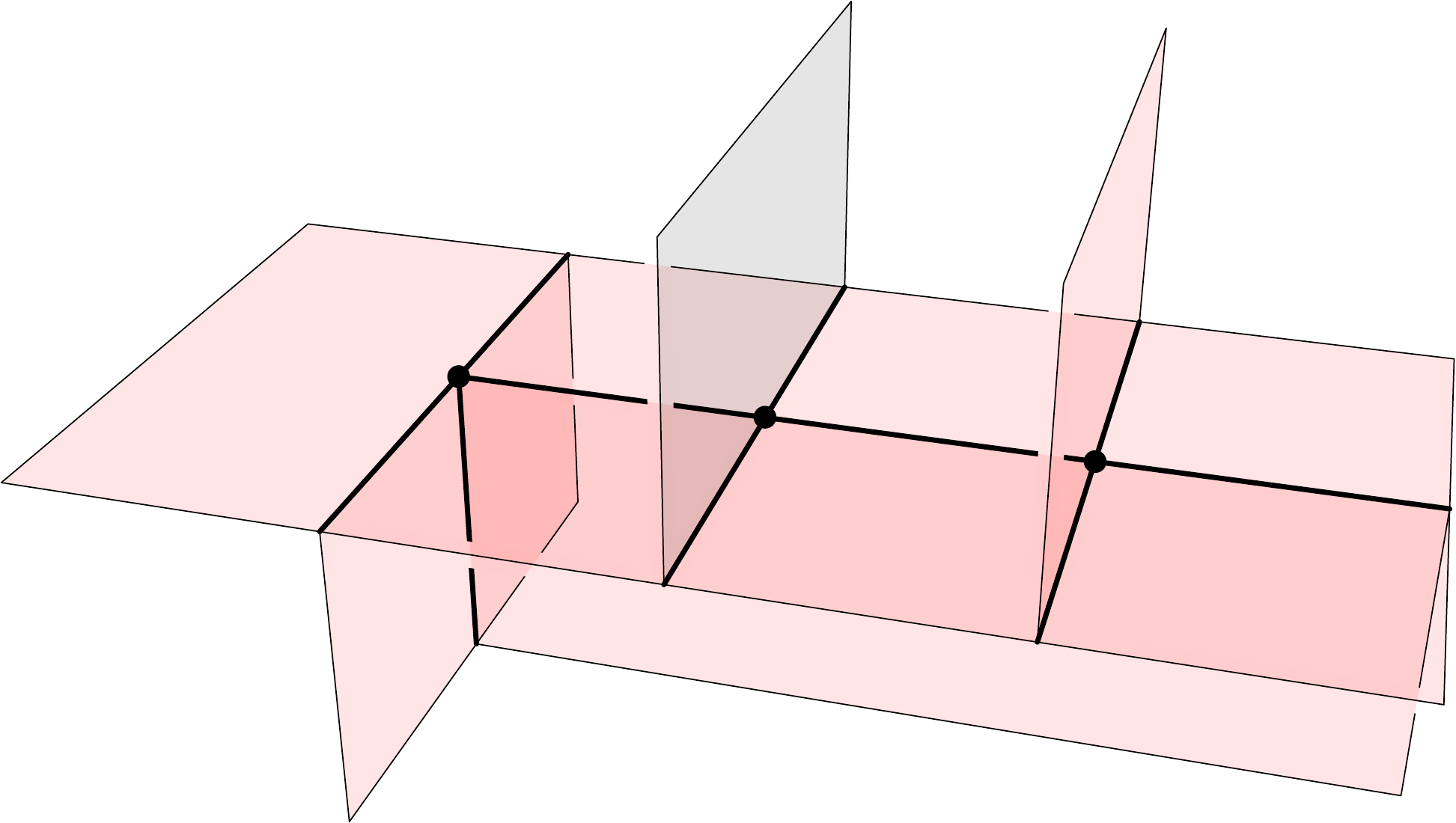}
}
\quad
\subfloat[]{
\includegraphics[width=0.28\textwidth]{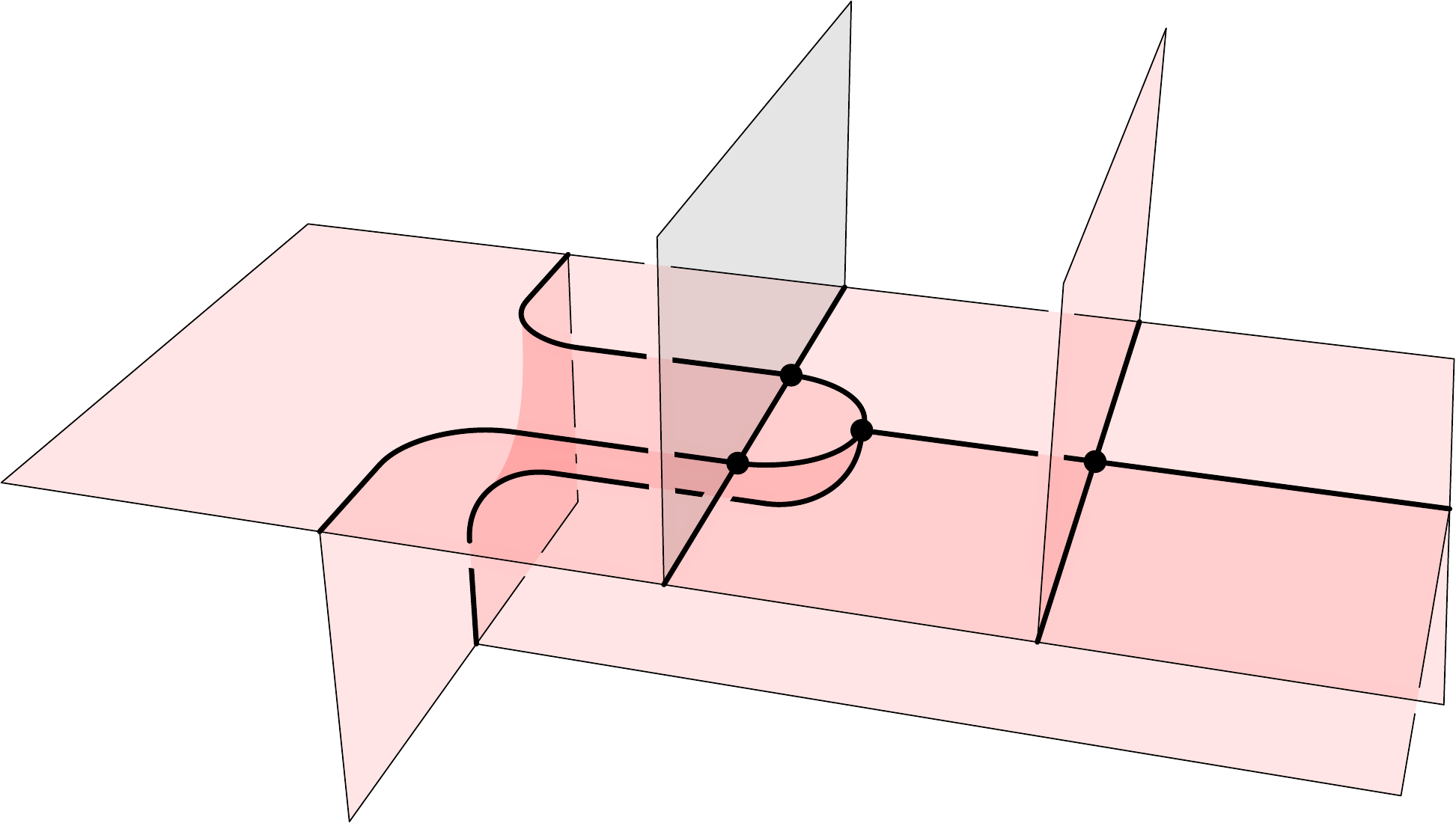}
}
\quad
\subfloat[]{
\includegraphics[width=0.28\textwidth]{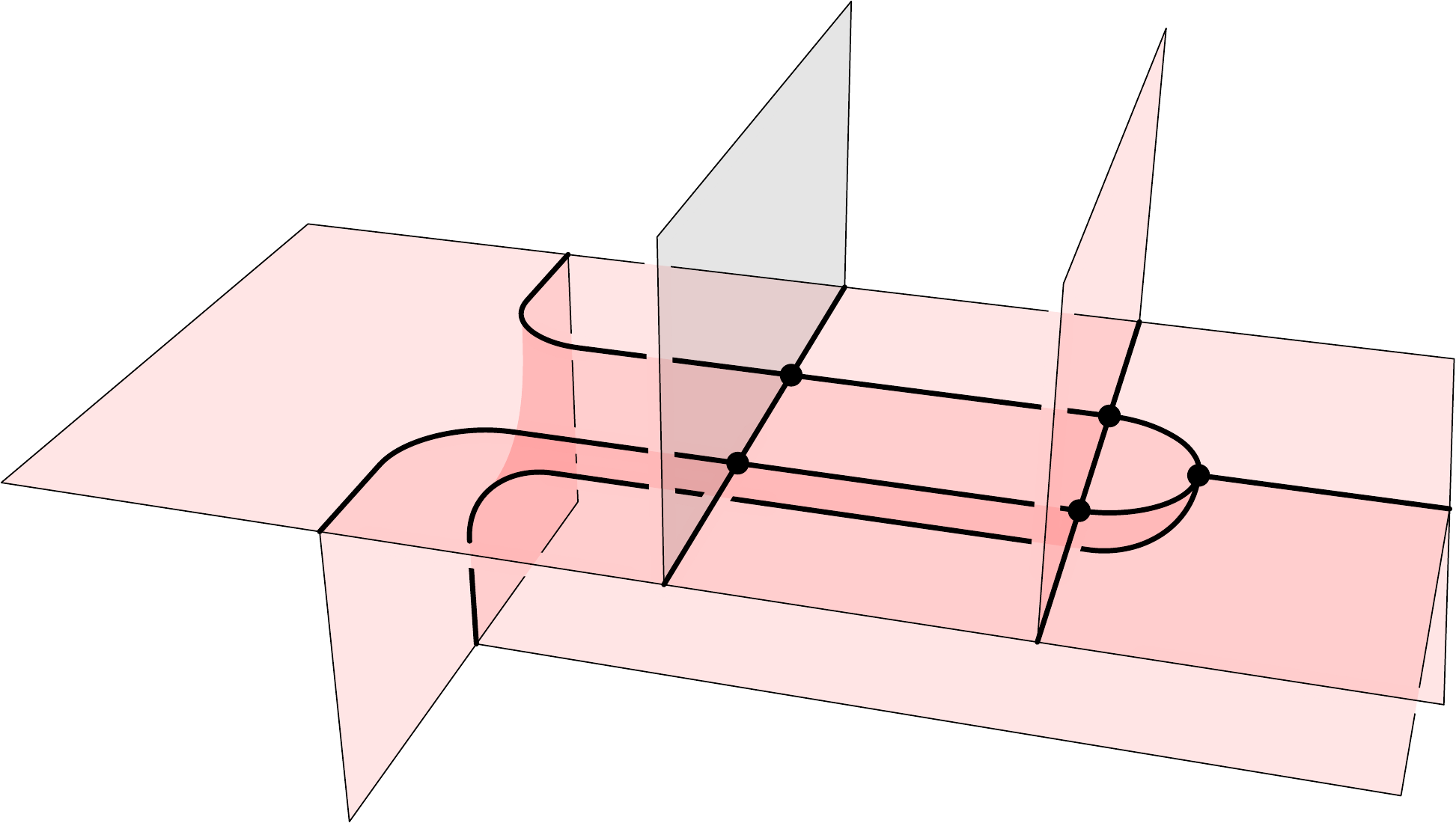}
}
\caption{Top row: a 2-3 move applied to an edge. Bottom row: the same result is achieved despite a membrane disk (shaded gray) being in the way by applying two 2-3 moves.} 
\label{do_two_2-3_moves}
\end{figure}

Having completed this step, no triangle of $\cT_{\alpha\beta}$ (thought of as a triangulation) has three $B$ vertices.

\subsubsection{Thickening the 1-skeleton of $B$}
\label{thicken_K}

Next, let $K$ be the 1-skeleton of $B$ (in $\cT_{\alpha\beta}$, thought of as a spine), union the edges in $\bdry g_\alpha,$ and $\bdry g_\beta$, and those edges that intersect $\bdry m_\alpha$ and $\bdry m_\beta$, but minus $\bdry m_\alpha$ and $\bdry m_\beta$ themselves. We wish to ``thicken up'' $K$, inspired by the implementation of barycentric subdivision given in Section~\ref{sec:Barycentric subdivision}, which uses 1-4 moves. Of course we cannot perform any 1-4 moves here, but we can perform the composition of a 1-4 move with adding an arch, as shown in Figure~\ref{fig_1-4_arch}. Moreover, since every vertex of $K$ is incident to the region $B$ at most twice, we can arrange matters so that the arch is on an edge of the resulting spine which is not incident to $B$. In the dual triangulation picture, we apply a 1-4 move and can indicate the arch with an arch mark. The picture is the same as shown in Figure~\ref{fig_1-4_arch_mark}. Here, we put the arch mark on the internal triangle incident to two non-$B$ vertices. Following this convention, after these moves our triangulation $\cT_\alpha$ is the result of adding arches at arch marks drawn on a triangulation $\cT'_\alpha$ say, where $\cT'_\alpha$ has an extra material vertex for every 1-4 move. On $\cT'_\alpha$, we label the new vertices of the triangulation ``3'', as in the first stage of barycentric subdivision, given in Section~\ref{sec:Barycentric subdivision}. This thickens up the vertices of $K$. To thicken up the edges of $K$, we continue with the barycentric subdivision procedure, adding vertices labelled ``2'' to $\cT'_\alpha$, again with arches connecting the new vertices back to non-$B$ vertices, with arch marks placed on triangles with vertices labelled 2, 3 and a non-$B$ 0 vertex. Again this is always possible to do, since no face of the triangulation has three $B$ vertices. In the second part of this stage of barycentric subdivision, we perform a 2-3 move on a face of the triangulation with three 0 vertices. We are able to do this without destroying one of our arch marks because all of our arch marks are on triangles with at least one non-0 vertex. This done, we have thickened $K$.

When we thicken an edge that intersects $\bdry m_\alpha$, on $\cT_\alpha$ we apply two 2-3 moves instead of one, just as in Section~\ref{remove_BBB_edges}, in order to keep up with $\cT_\beta$. The same applies with $\alpha$ and $\beta$ swapped.

Having completed this step, each face of the spine $\cT_{\alpha\beta}$ (and also $\cT'_{\alpha\beta}$) incident to $B$ has all of its incident vertices and edges distinct. Moreover, in $\cT'_{\alpha\beta}$, no arch mark is on a triangle with a $B$-vertex.

\subsubsection{Thickening self-gluings of $B$}

We now wish to continue the barycentric subdivision theme, thickening up the faces of $B$ that are self-gluings. For such a face, we are precisely in the setting of adding a vertex labelled 1, as in Figure~\ref{barycentric_subdiv_3d_edge}. 
Again here we implement the 1-4 and arch move on $\cT_\alpha$ with 2-3 and 3-2 moves. The corresponding arch mark is placed on a triangle of $\cT'_\alpha$ not incident to a $B$ vertex. In fact we may place it on a triangle with vertices labelled 1, 2 and 3. Next we follow the subsequent steps, as shown in Figure~\ref{barycentric_subdiv_3d_edge_top_view}. These are all 2-3 moves, deleting triangles with two $B$ vertices, which therefore do not have arch marks on them, and a final 3-2 move. Each of the three triangles deleted in the final 3-2 move has a $B$ vertex, and so again there are no arches in the way of this move.

In this step, special handling for $m_\alpha$ (and $m_\beta$ respectively) is needed in the case that $B = C_\alpha$, so $f_\alpha$ is a self-gluing of $B$. In this case, for $\cT_\beta$, the membrane $m_\alpha$ is not present, so $f_\alpha = g_\alpha$, and the entire face is thickened. So, we must do the same for $\cT_\alpha$, despite $m_\alpha$ being in the way. To achieve this, first of all let's break the operation of thickening up a face, as shown in Figure~\ref{barycentric_subdiv_3d_edge_spine}, into two stages.

\begin{enumerate} 
\item First, from Figure~\ref{barycentric_subdiv_3d_spine_edge_step_0} to \ref{barycentric_subdiv_3d_spine_edge_step_1} we create a new three-cell region (here, connected via an arch to one of the neighbouring regions). 
\item Second, from Figure~\ref{barycentric_subdiv_3d_spine_edge_step_1} onwards, we expand the new region outwards, \emph{collapsing} the face away. 
\end{enumerate}

The data determining a face collapse move consists of 
\begin{enumerate}
\item the face $f$ to be collapsed, together with 
\item the \emph{collapsing edge} of $f$ which expands into $f$ as $f$ collapses.
\end{enumerate}
In Figure~\ref{barycentric_subdiv_3d_spine_edge_step_1}  the collapsing edge is drawn dashed.
In order to perform a face collapse, we require that all vertices of the face be distinct, and that there are at least three vertices.

\begin{figure}[htbp]
\centering
\subfloat[]{
\labellist
\scriptsize\hair 2pt
\pinlabel $f_\alpha$ at 282 150
\pinlabel $m_\alpha$ at 300 256
\pinlabel $g_\alpha$ at 355 200
\endlabellist
\includegraphics[width=0.3\textwidth]{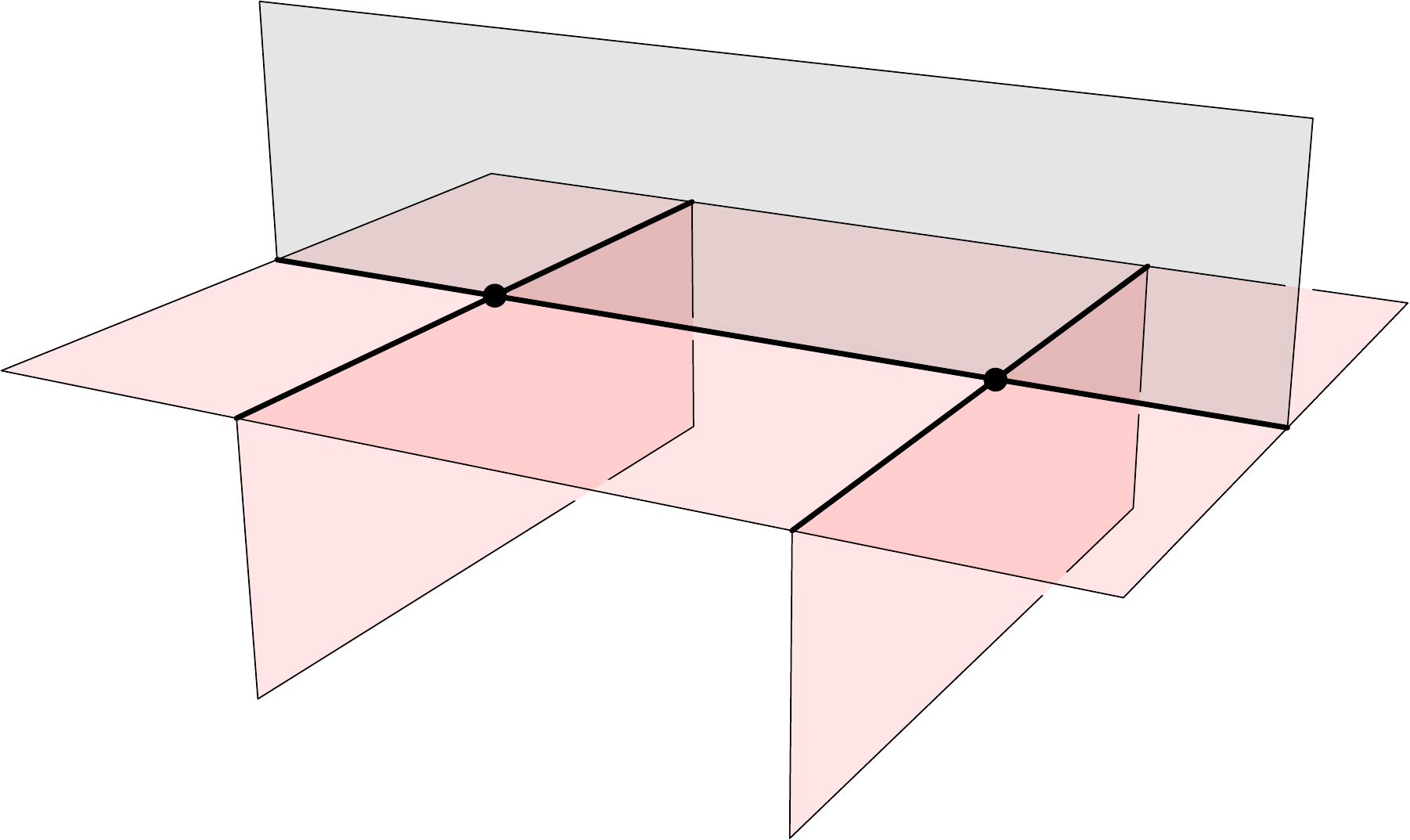}
}
\subfloat[]{
\includegraphics[width=0.3\textwidth]{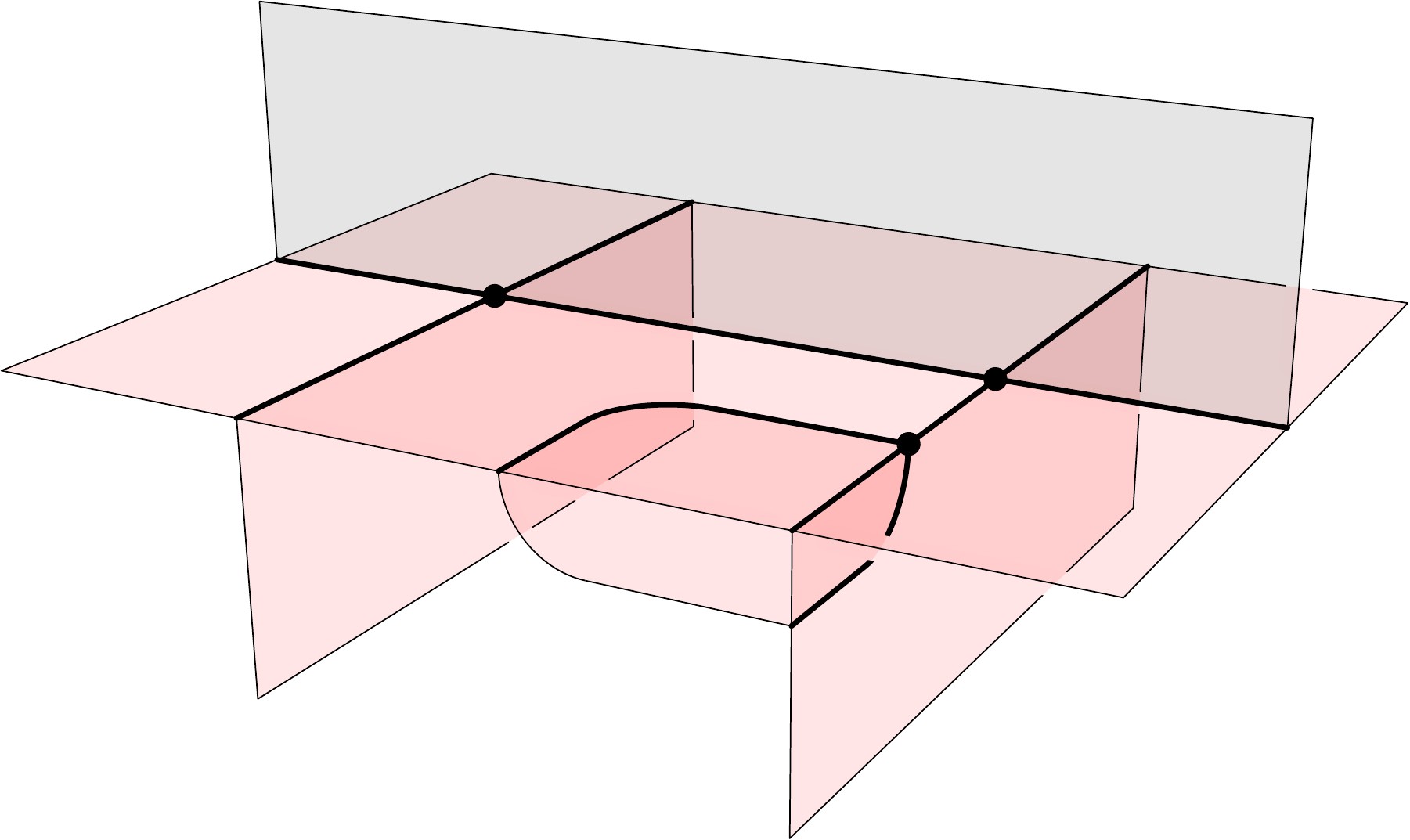}
}
\subfloat[]{
\includegraphics[width=0.3\textwidth]{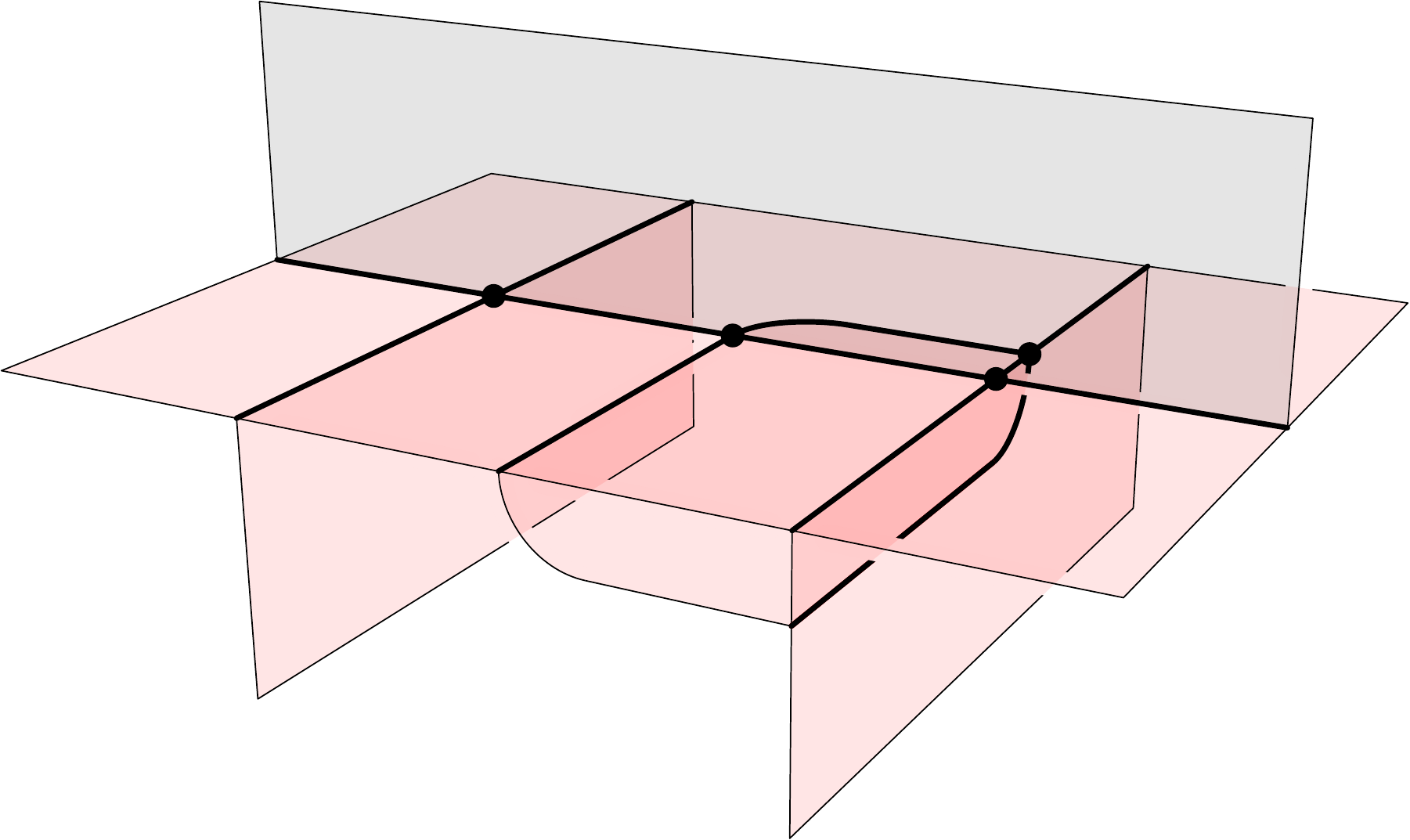}
}

\subfloat[]{
\includegraphics[width=0.3\textwidth]{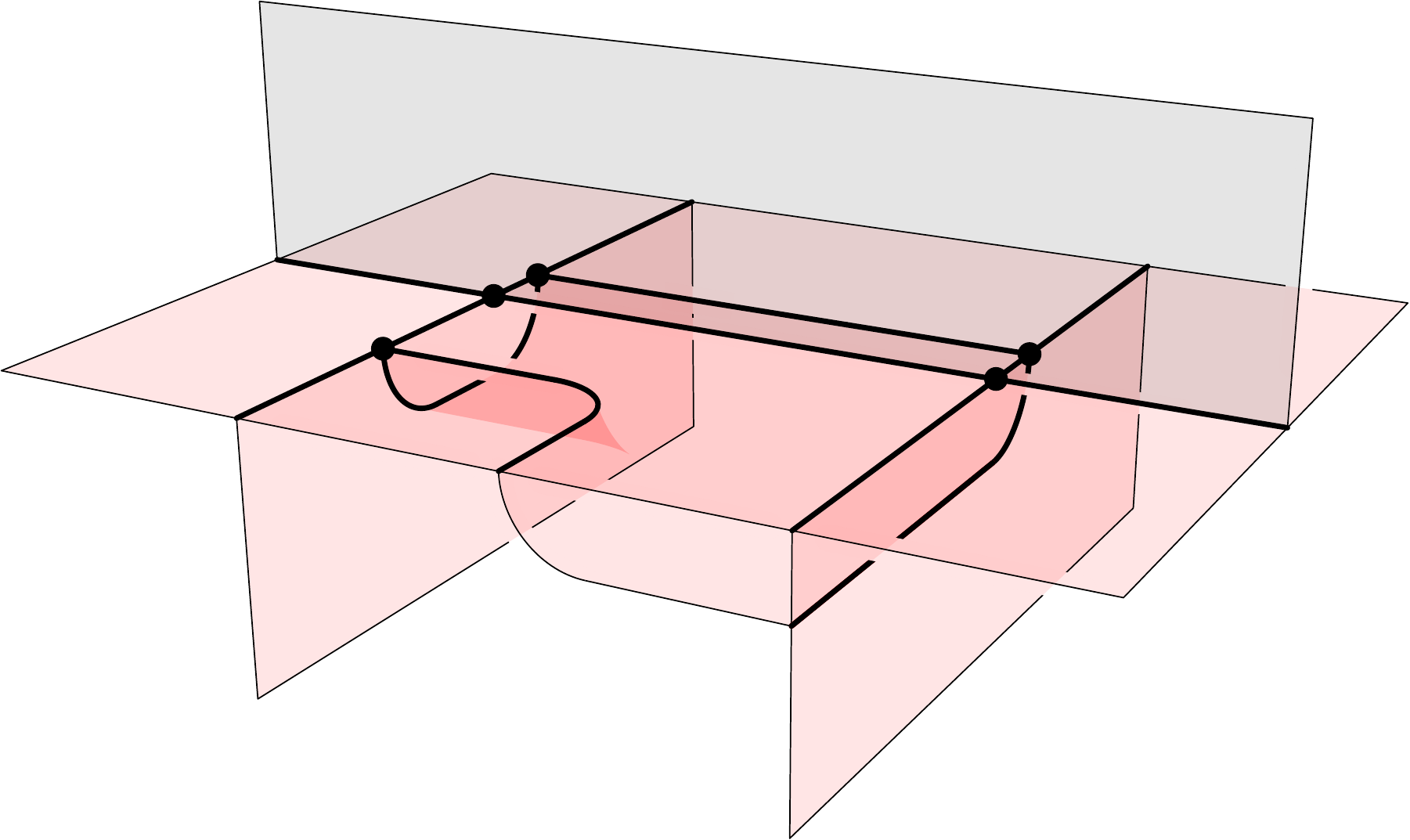}
}
\subfloat[]{
\includegraphics[width=0.3\textwidth]{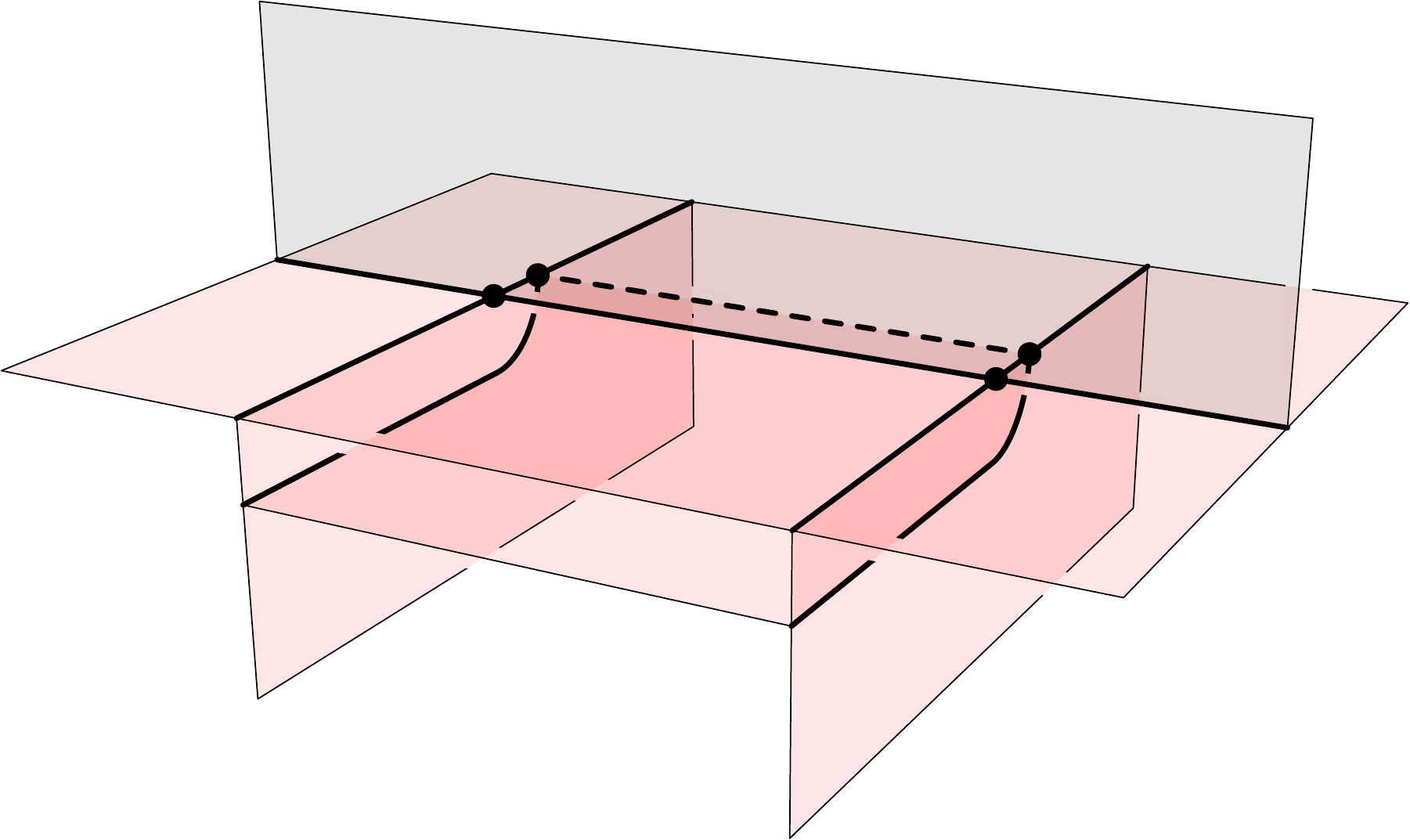}
\label{split_past_membrane_step_4}
}
\caption{Collapsing a face past a membrane (shaded gray).} 
\label{split_past_membrane}
\end{figure}

Figure~\ref{split_past_membrane} shows the effect of thickening a face $f_\alpha$ incident to a membrane $m_\alpha$. (Note that this is not a completely general picture: the degree of the remains of $f_\alpha$ could reach three earlier than shown, in which case we perform a 3-2 move rather than further 2-3 moves.) Now, to continue the thickening of $f_\alpha$ into $g_\alpha$, we perform a second collapsing step, this time collapsing $g_\alpha$, without first creating a new region. Instead we use the region already created to thicken $f_\alpha$, collapsing $g_\alpha$ with the collapsing edge shown in Figure~\ref{split_past_membrane_step_4}. 
Note that by construction, at the start of the second collapsing step the degree of $g_\alpha$ is at least three, so we are able to perform the collapse: If the degree of $g_\alpha$ were two, then before the thickening move we would have had an arc of $\Gamma$ cross $\bdry m_\alpha$ and return, forming a bigon. But our previous moves cannot produce such a configuration.

Also note that if $g_\alpha = g_\beta$, then it is possible that we collapse $g_\alpha$ first, which then requires us to collapse into $f_\alpha$. Finally, it is also possible that we have that $f_\alpha = f_\beta$, since the graph $\Gamma$ does not separate them, and moreover that $g_\alpha = g_\beta$, connecting around outside of $B$. In this case, the membrane (in either position) cannot be ignored - without it there is one face $f_\alpha = f_\beta = g_\beta = g_\alpha$, which is therefore not a disk. When we collapse this face, we continue on from Figure~\ref{split_past_membrane_step_4} to Figure~\ref{split_past_membrane_2_step_1}. Completing the collapse in the rest of Figure~\ref{split_past_membrane_2}, we gain an extra rectangular face, parallel to the membrane, which separates the region that expanded into $f_\alpha = f_\beta = g_\beta = g_\alpha$ from itself. Here, we need to do a little more work to maintain $(\cT_\alpha \cup m_\beta) = (\cT_\beta \cup m_\alpha)$. In particular, after thickening $f_\alpha = f_\beta = g_\beta = g_\alpha$, we have a rectangle near $m_\alpha$ in $\cT_\alpha$, but one near $m_\beta$ in $\cT_\beta$.

\begin{figure}[htbp]
\centering
\subfloat[]{
\includegraphics[width=0.3\textwidth]{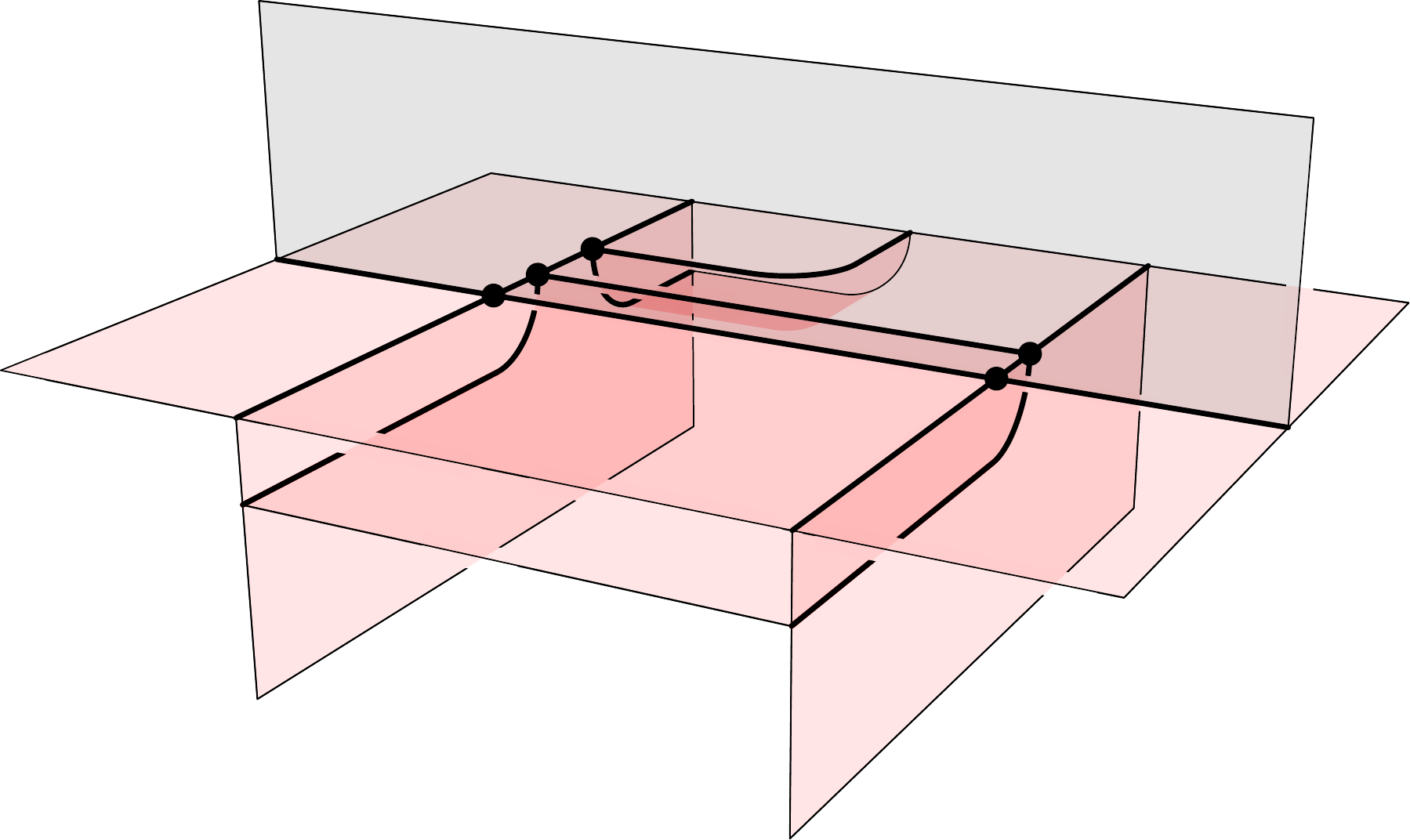}
\label{split_past_membrane_2_step_1}
}
\subfloat[]{
\includegraphics[width=0.3\textwidth]{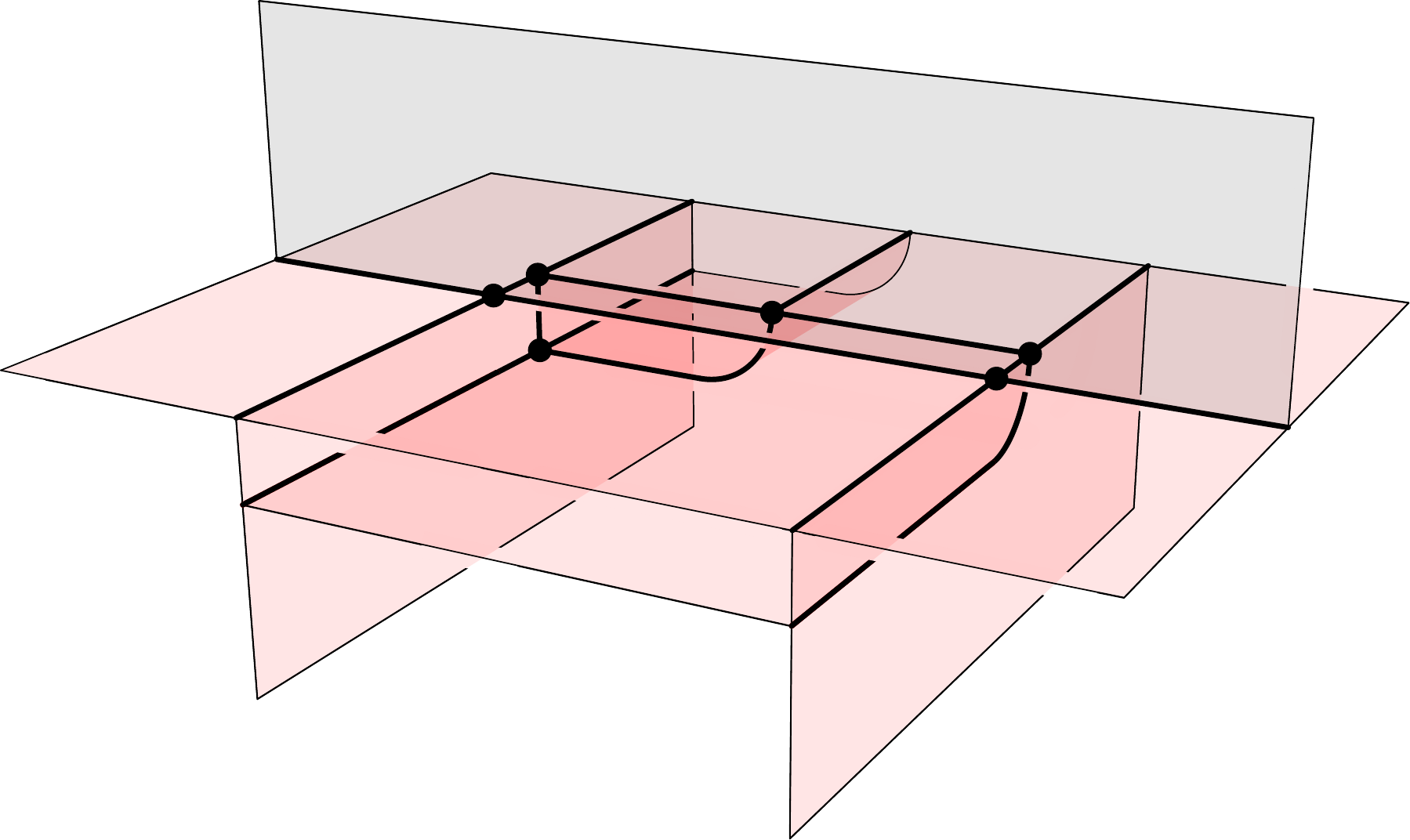}
}

\subfloat[]{
\includegraphics[width=0.3\textwidth]{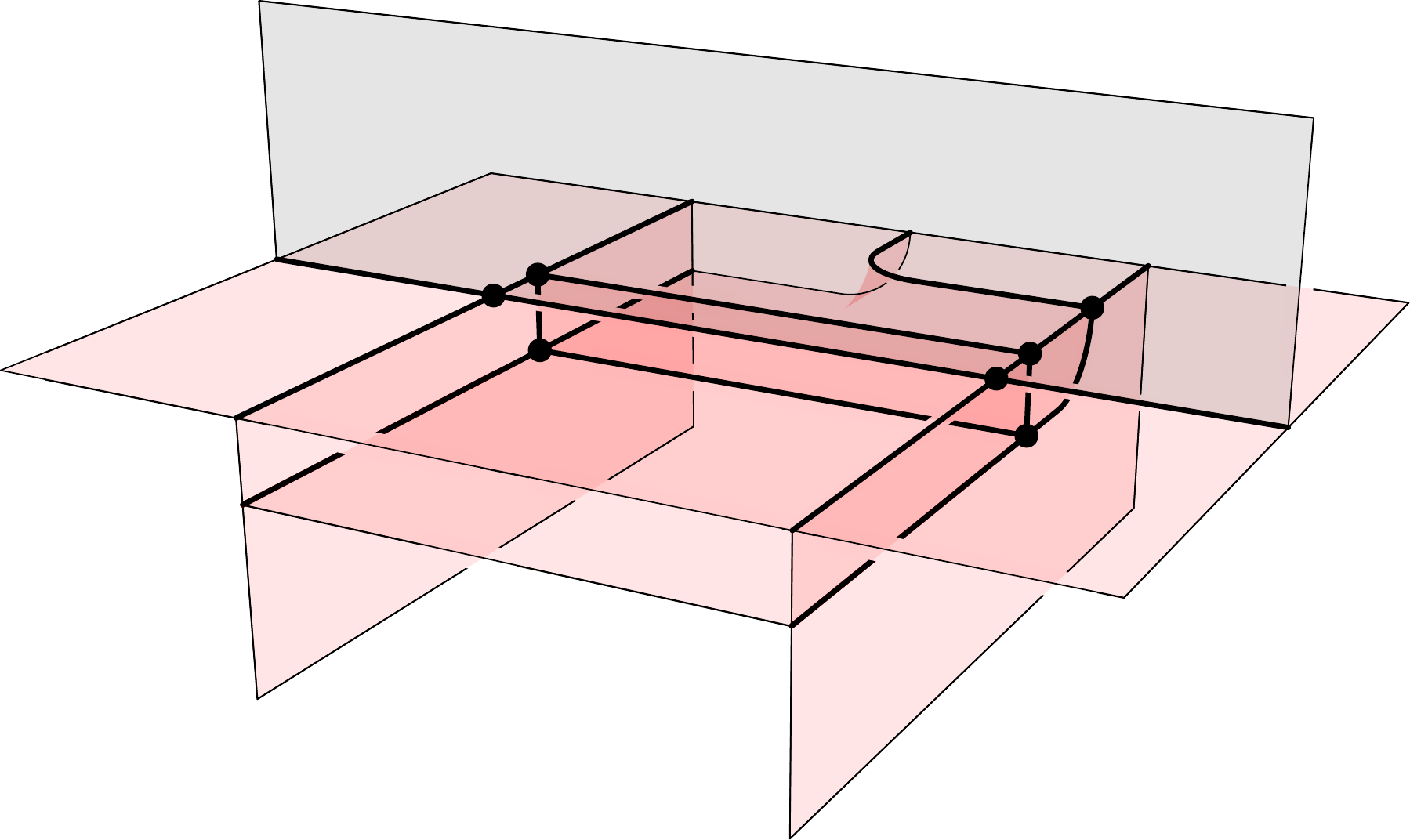}
}
\subfloat[]{
\includegraphics[width=0.3\textwidth]{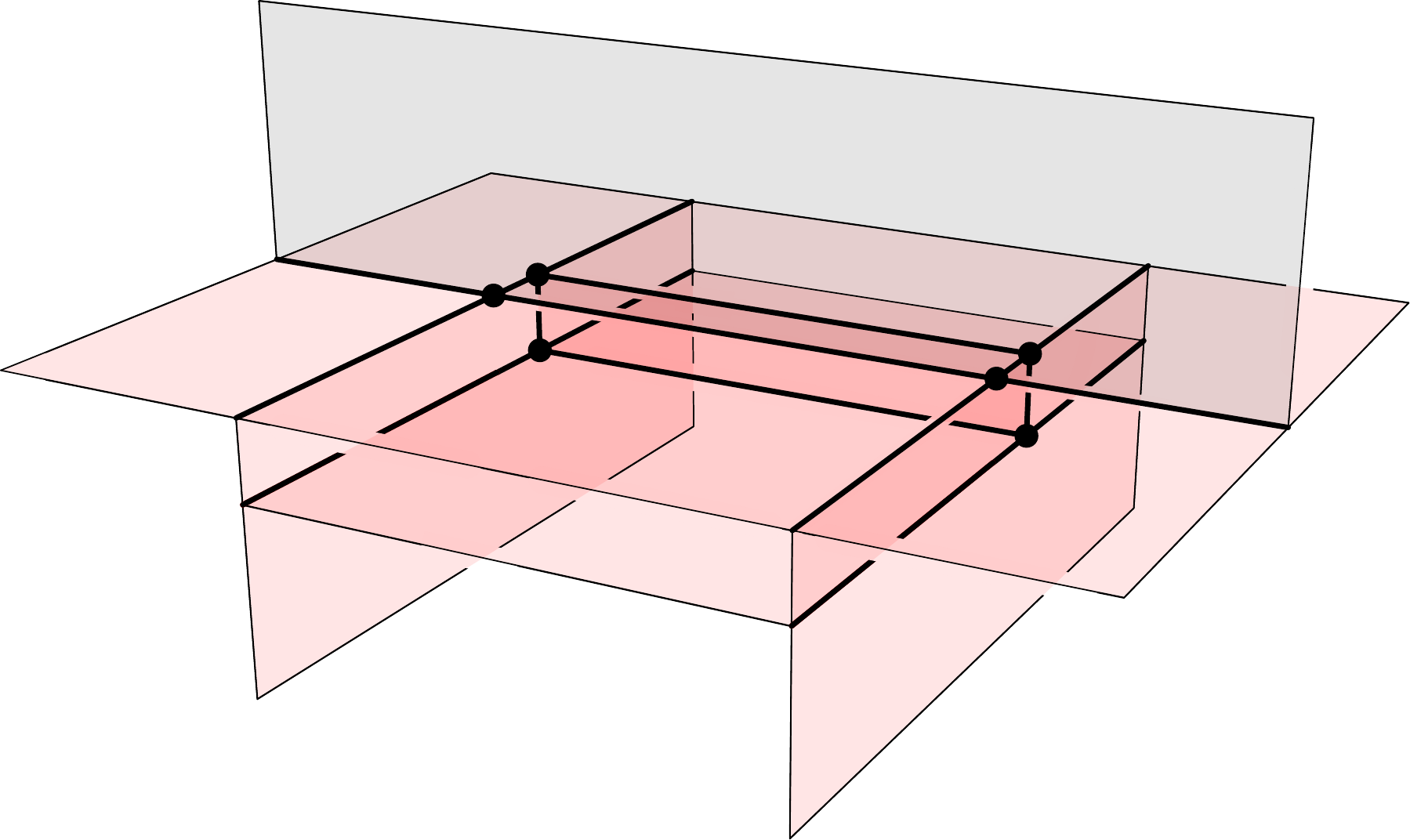}
}
\caption{Collapsing a face past a membrane (shaded gray) when the membrane is a self-gluing.} 
\label{split_past_membrane_2}
\end{figure}

We can move the rectangle past a membrane using a 2-3 followed by a 3-2 move, and we can move it around the boundary of the thickened face using similar moves. See Figure~\ref{slide_rectangle}. Note that although we are sliding a disk (the rectangle) through a ball (the thickened face), this is a much simpler setting than our original problem: there are no self-identifications and the thickened face has a much simpler graph on its boundary. These two rectangle moves suffice to transport it from its location near $m_\alpha$ to its location near $m_\beta$, as long as we don't attempt to move it past the arch connecting the thickened face to the outside. Thankfully, we can either move the rectangle through $B$ or around the outside of $B$, and one of these two paths does pass by the arch.

\begin{figure}[htbp]
\centering
\subfloat[]{
\includegraphics[width=0.3\textwidth]{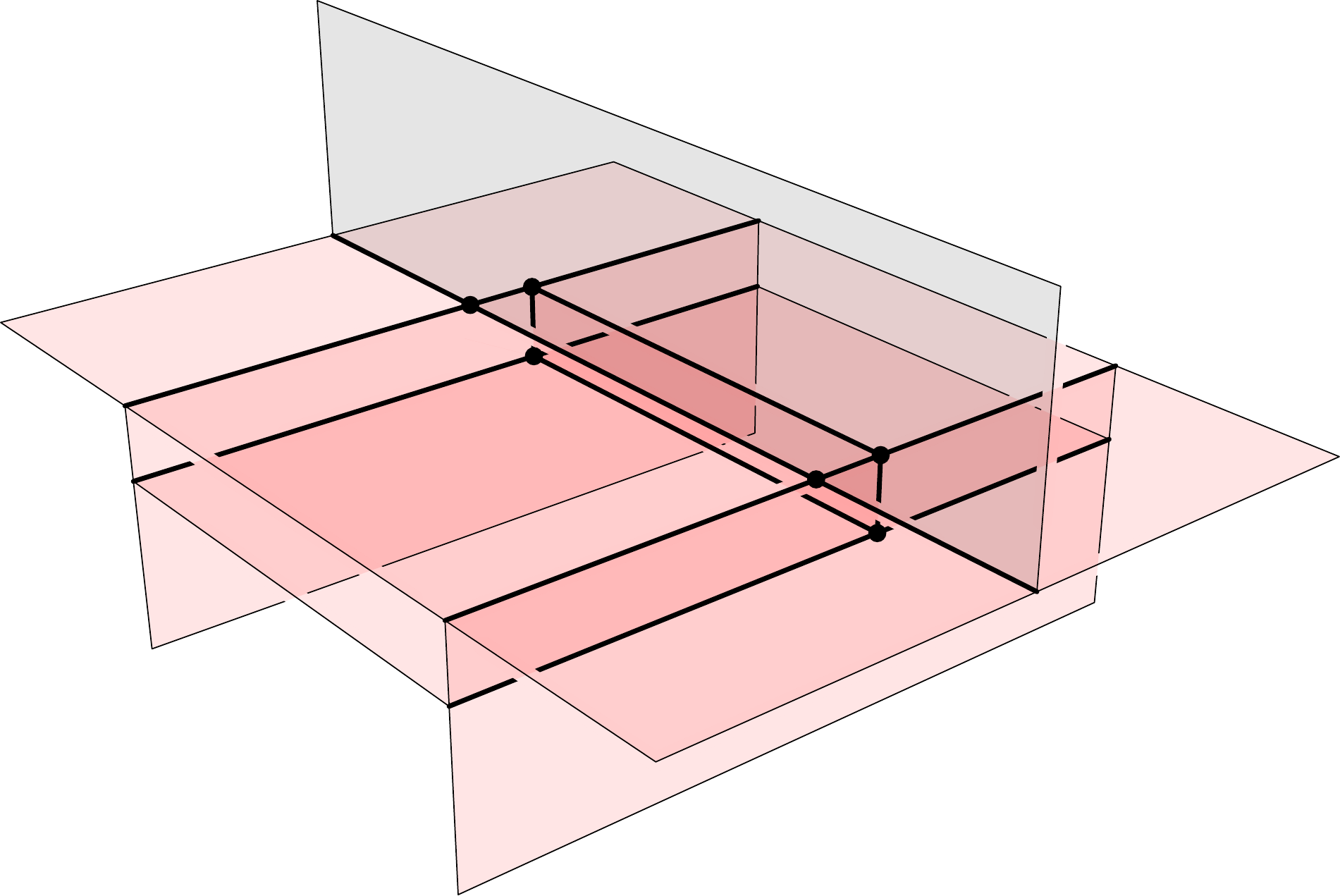}
}
\subfloat[]{
\includegraphics[width=0.3\textwidth]{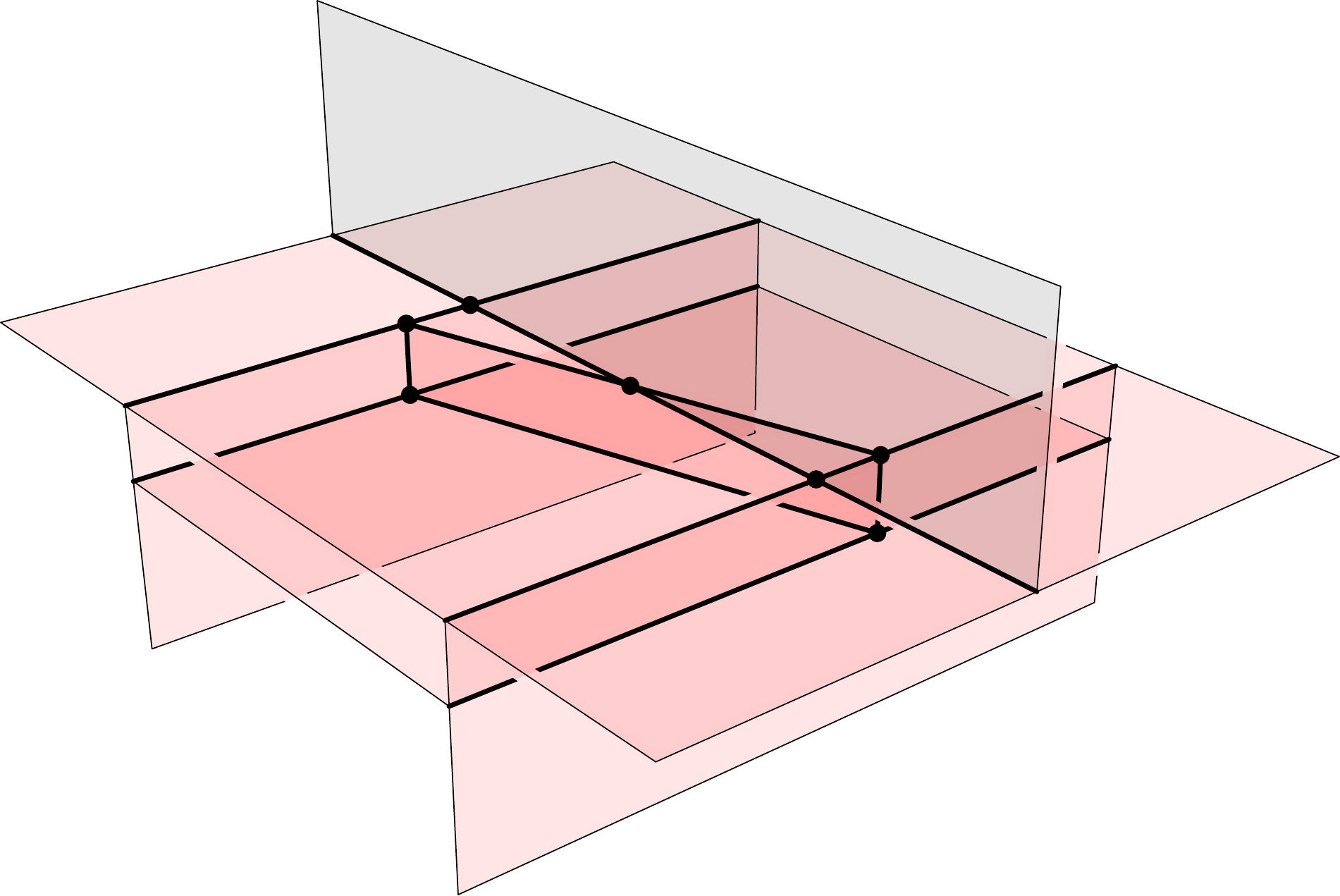}
}
\subfloat[]{
\includegraphics[width=0.3\textwidth]{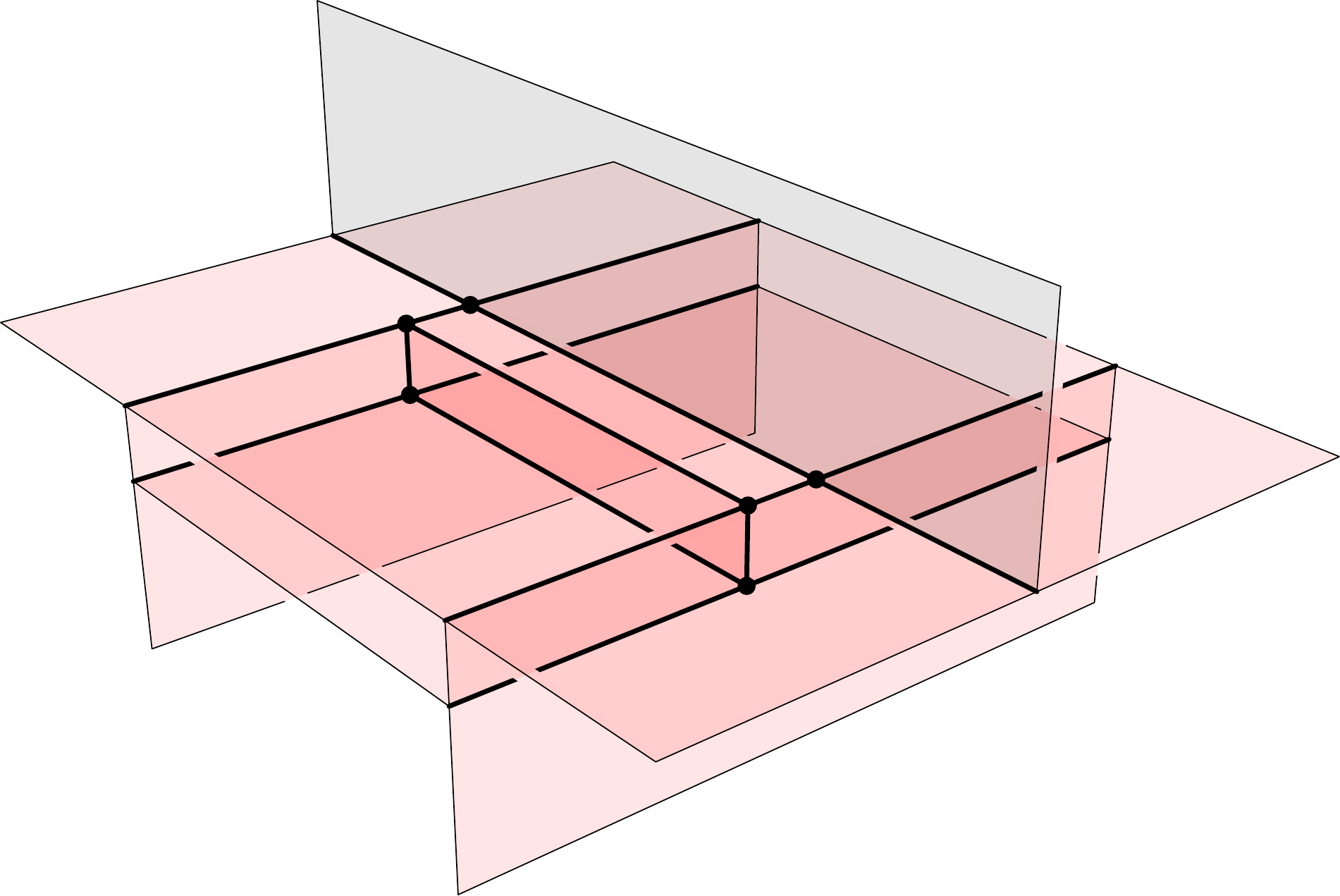}
}

\subfloat[]{
\includegraphics[width=0.3\textwidth]{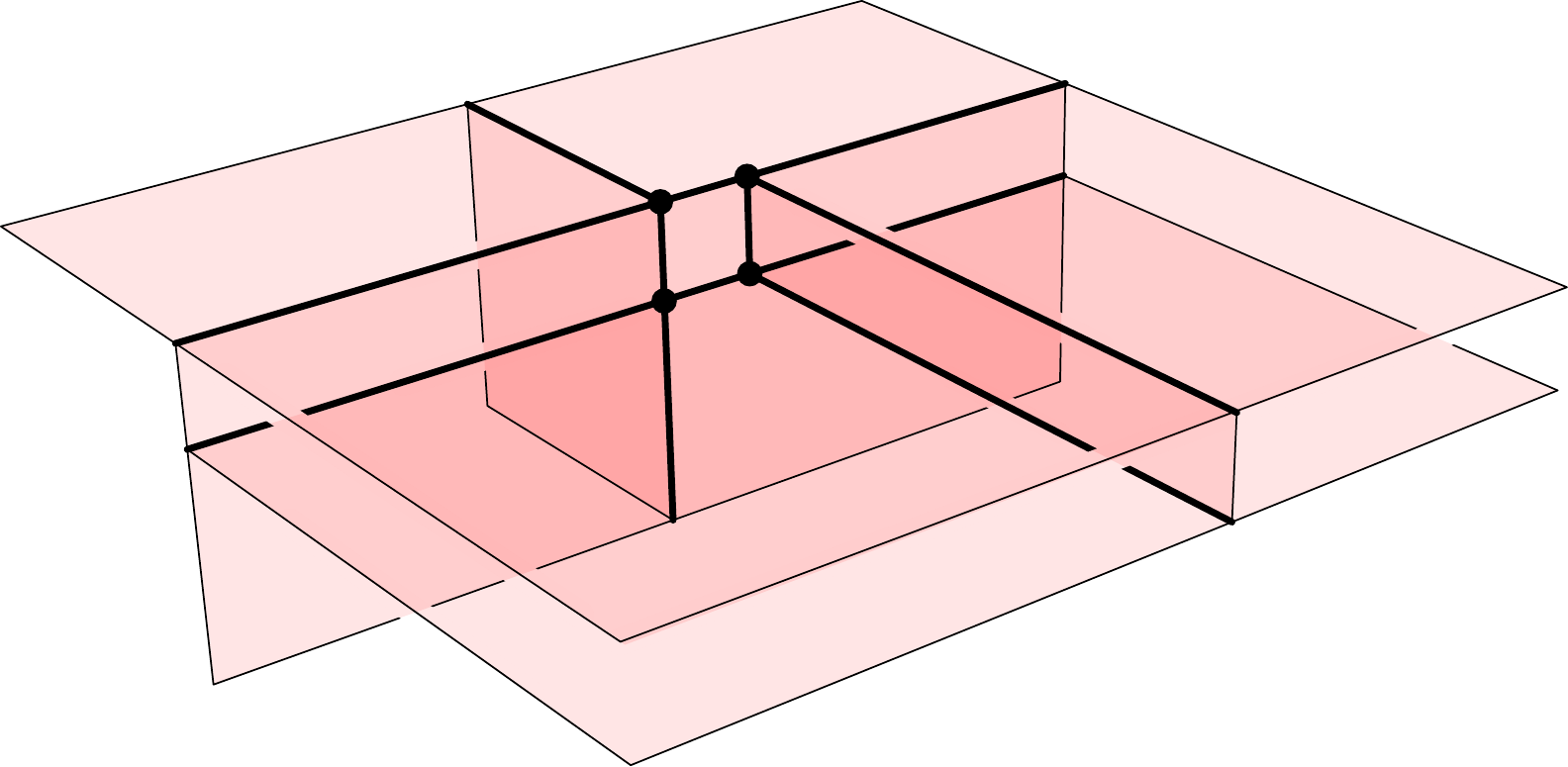}
}
\subfloat[]{
\includegraphics[width=0.3\textwidth]{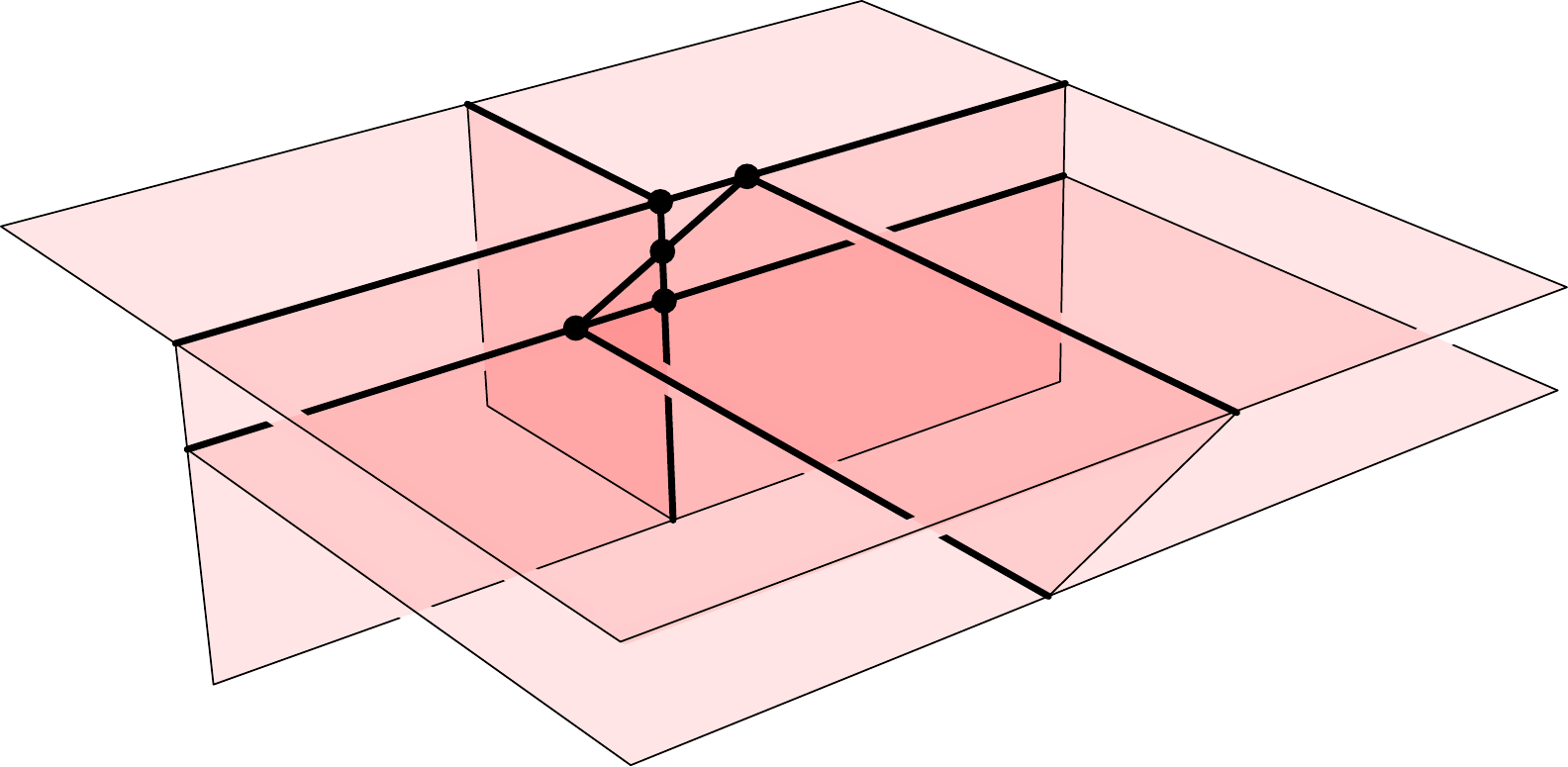}
}
\subfloat[]{
\includegraphics[width=0.3\textwidth]{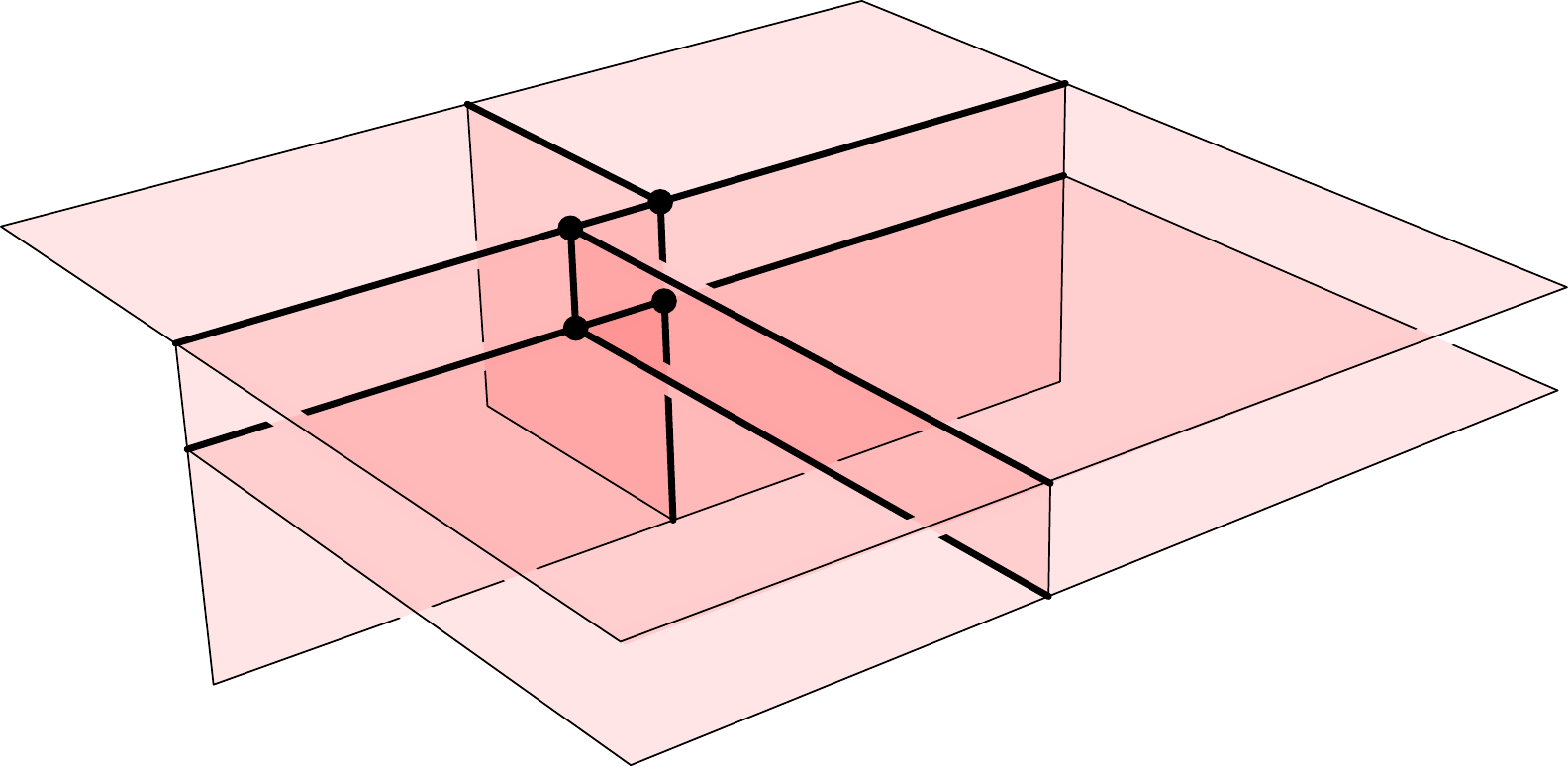}
}

\caption{Sliding the rectangle past a membrane (first row) or a corner of the thickened face (second row). In both cases, this is achieved by applying a 2-3 move followed by a 3-2 move.} 
\label{slide_rectangle}
\end{figure}

\subsection{Modifications to moving a membrane across a ball}

The only difference in the setup here in comparison to that in Section~\ref{no_self-gluings} is that there may be more than one edge of $\Gamma$ exiting the annulus $A$ through each of its boundary curves $C_\text{start}$ and $C_\text{end}$: If $\bdry d_\alpha$ is a $BBB$-edge then in Section~\ref{remove_BBB_edges} we remove it. The resulting spine has two edges intersecting $\bdry m_\alpha$. When we thicken up edges in Section~\ref{thicken_K}, this again doubles the number of edges intersecting $\bdry m_\alpha$. The same is true for edges intersecting $\bdry m_\beta$. Our algorithm is however almost identical to the one given in Section~\ref{no_self-gluings}. All we need to do is choose one of the edges intersecting $C_\text{start}$ to be $e_\text{start}$, and have the spanning tree include that $e_\text{start}$ but no other edges intersecting $C_\text{start}$ or $C_\text{end}$. The rest of the algorithm then goes through identically.


\subsection*{Acknowledgements}

{Research of the first and third authors is supported in part by the Australian Research Council under the Discovery Projects funding scheme (DP160104502). The second author was supported in part by National Science Foundation grant DMS-1708239.} The authors thank Robert L\"owe for helpful comments on a previous draft. We are most grateful to the referee for their careful reading and helpful remarks.

\bibliographystyle{hamsplain}
\bibliography{connectivity} 

\providecommand{\bysame}{\leavevmode\hbox to3em{\hrulefill}\thinspace}
\providecommand{\href}[2]{#2}
\begin{thebibliography}{10}

\bibitem{Adams-knot-2004}
Colin~C. Adams, \emph{The knot book}, American Mathematical Society,
  Providence, RI, 2004, An elementary introduction to the mathematical theory
  of knots, Revised reprint of the 1994 original.

\bibitem{Agol-computational-2006}
Ian Agol, Joel Hass, and William Thurston, \emph{The computational complexity
  of knot genus and spanning area}, Trans. Amer. Math. Soc. \textbf{358}
  (2006), no.~9, 3821--3850.

\bibitem{Akbulut-Casson-1990}
Selman Akbulut and John~D. McCarthy, \emph{Casson's invariant for oriented
  homology {$3$}-spheres}, Mathematical Notes, vol.~36, Princeton University
  Press, Princeton, NJ, 1990, An exposition.

\bibitem{Alexander-combinatorial-1930}
James~W. Alexander, \emph{The combinatorial theory of complexes}, Ann. of Math.
  (2) \textbf{31} (1930), no.~2, 292--320.

\bibitem{Amendola-calculus-2005}
Gennaro Amendola, \emph{A calculus for ideal triangulations of three-manifolds
  with embedded arcs}, Math. Nachr. \textbf{278} (2005), no.~9, 975--994.

\bibitem{banagl-friedman}
Markus Banagl and Greg Friedman, \emph{Triangulations of 3-dimensional
  pseudomanifolds with an application to state-sum integrals}, Algebraic \&
  Geometric Topology \textbf{4} (2004), 521--542.

\bibitem{Baseilhac-analytic-2015}
St\'{e}phane Baseilhac and Riccardo Benedetti, \emph{Analytic families of
  quantum hyperbolic invariants}, Algebr. Geom. Topol. \textbf{15} (2015),
  no.~4, 1983--2063.

\bibitem{Benedetti-finite-1995}
Riccardo Benedetti and Carlo Petronio, \emph{A finite graphic calculus for
  {$3$}-manifolds}, Manuscripta Math. \textbf{88} (1995), no.~3, 291--310.

\bibitem{Benedetti-branched-1997}
\bysame, \emph{Branched standard spines of {$3$}-manifolds}, Lecture Notes in
  Mathematics, vol. 1653, Springer-Verlag, Berlin, 1997.

\bibitem{Brown-hauptvermutung-1969}
Edward~M. Brown, \emph{The {H}auptvermutung for {$3$}-complexes}, Trans. Amer.
  Math. Soc. \textbf{144} (1969), 173--196.

\bibitem{regina}
Benjamin~A. Burton, Ryan Budney, William Pettersson, et~al., \emph{Regina:
  Software for low-dimensional topology}, {\tt http://\allowbreak
  regina-normal.\allowbreak github.\allowbreak io/}, 1999--2017.

\bibitem{Burton-courcelle-2017}
Benjamin~A. Burton and Rodney~G. Downey, \emph{Courcelle's theorem for
  triangulations}, J. Combin. Theory Ser. A \textbf{146} (2017), 264--294.

\bibitem{Carbone-wigner-2000}
Gaspare Carbone, Mauro Carfora, and Annalisa Marzuoli, \emph{Wigner symbols and
  combinatorial invariants of three-manifolds with boundary}, Comm. Math. Phys.
  \textbf{212} (2000), no.~3, 571--590.

\bibitem{Coulson-computing-2000}
David Coulson, Oliver~A. Goodman, Craig~D. Hodgson, and Walter~D. Neumann,
  \emph{Computing arithmetic invariants of 3-manifolds}, Experiment. Math.
  \textbf{9} (2000), no.~1, 127--152.

\bibitem{Coward-reidemeister-2014}
Alexander Coward and Marc Lackenby, \emph{An upper bound on {R}eidemeister
  moves}, Amer. J. Math. \textbf{136} (2014), no.~4, 1023--1066.

\bibitem{snappy}
Marc Culler, Nathan Dunfield, and Jeffrey~R. Weeks, \emph{{SnapPy}}, a computer
  program for studying the geometry and topology of 3-manifolds,
  http://snappy.computop.org.

\bibitem{Dimofte-quantum-2011}
Tudor Dimofte and Sergei Gukov, \emph{Quantum field theory and the volume
  conjecture}, Chern-{S}imons gauge theory: 20 years after, AMS/IP Stud. Adv.
  Math., vol.~50, Amer. Math. Soc., Providence, RI, 2011, pp.~19--42.

\bibitem{Dunfield-spanning-2011}
Nathan~M. Dunfield and Anil~N. Hirani, \emph{The least spanning area of a knot
  and the optimal bounding chain problem}, Computational geometry ({SCG}'11),
  ACM, New York, 2011, pp.~135--144.

\bibitem{Epstein-curves-1966}
D.~B.~A. Epstein, \emph{Curves on {$2$}-manifolds and isotopies}, Acta Math.
  \textbf{115} (1966), 83--107.

\bibitem{Epstein-euclidean-1988}
D.~B.~A. Epstein and R.~C. Penner, \emph{Euclidean decompositions of noncompact
  hyperbolic manifolds}, J. Differential Geom. \textbf{27} (1988), no.~1,
  67--80.

\bibitem{Francis-conway-1999}
George~K. Francis and Jeffrey~R. Weeks, \emph{Conway's {ZIP} proof}, Amer.
  Math. Monthly \textbf{106} (1999), no.~5, 393--399.

\bibitem{Frigerio-constructing-2004}
Roberto Frigerio and Carlo Petronio, \emph{Construction and recognition of
  hyperbolic 3-manifolds with geodesic boundary}, Trans. Amer. Math. Soc.
  \textbf{356} (2004), no.~8, 3243--3282.

\bibitem{Haken-verfahren-1961}
Wolfgang Haken, \emph{Ein {V}erfahren zur {A}ufspaltung einer
  {$3$}-{M}annigfaltigkeit in irreduzible {$3$}-{M}annigfaltigkeiten}, Math. Z.
  \textbf{76} (1961), 427--467.

\bibitem{Haken-theorie-1961}
\bysame, \emph{Theorie der {N}ormalfl\"achen}, Acta Math. \textbf{105} (1961),
  245--375.

\bibitem{Haken-uber-1962}
\bysame, \emph{\"uber das {H}om\"oomorphieproblem der 3-{M}annigfaltigkeiten.
  {I}}, Math. Z. \textbf{80} (1962), 89--120.

\bibitem{Haken-some-1968}
\bysame, \emph{Some results on surfaces in {$3$}-manifolds}, Studies in
  {M}odern {T}opology, Math. Assoc. Amer. (distributed by Prentice-Hall,
  Englewood Cliffs, N.J.), 1968, pp.~39--98.

\bibitem{Hamilton-triangulation-1976}
A.~J.~S. Hamilton, \emph{The triangulation of {$3$}-manifolds}, Quart. J. Math.
  Oxford Ser. (2) \textbf{27} (1976), no.~105, 63--70.

\bibitem{Hass-computational-1999}
Joel Hass, Jeffrey~C. Lagarias, and Nicholas Pippenger, \emph{The computational
  complexity of knot and link problems}, J. ACM \textbf{46} (1999), no.~2,
  185--211.

\bibitem{Hatcher-kirby-2013}
A.~{Hatcher}, \emph{{The Kirby torus trick for surfaces}}, ArXiv e-prints
  (2013), \mbox{1312.3518}.

\bibitem{Hatcher-triangulations-1991}
Allen Hatcher, \emph{On triangulations of surfaces}, Topology Appl. \textbf{40}
  (1991), no.~2, 189--194.

\bibitem{Hoffman-verified-2016}
Neil Hoffman, Kazuhiro Ichihara, Masahide Kashiwagi, Hidetoshi Masai, Shin'ichi
  Oishi, and Akitoshi Takayasu, \emph{Verified computations for hyperbolic
  3-manifolds}, Exp. Math. \textbf{25} (2016), no.~1, 66--78.

\bibitem{Jaco-algorithm-1984}
William Jaco and Ulrich Oertel, \emph{An algorithm to decide if a
  {$3$}-manifold is a {H}aken manifold}, Topology \textbf{23} (1984), no.~2,
  195--209.

\bibitem{Jaco-minimal-2009}
William Jaco, Hyam Rubinstein, and Stephan Tillmann, \emph{Minimal
  triangulations for an infinite family of lens spaces}, J. Topol. \textbf{2}
  (2009), no.~1, 157--180.

\bibitem{Jaco-minimal-1988}
William Jaco and J.~Hyam Rubinstein, \emph{P{L} minimal surfaces in
  {$3$}-manifolds}, J. Differential Geom. \textbf{27} (1988), no.~3, 493--524.

\bibitem{Jaco-equivariant-1989}
\bysame, \emph{P{L} equivariant surgery and invariant decompositions of
  {$3$}-manifolds}, Adv. in Math. \textbf{73} (1989), no.~2, 149--191.

\bibitem{Jaco-0-efficient-2003}
\bysame, \emph{{$0$}-efficient triangulations of 3-manifolds}, J. Differential
  Geom. \textbf{65} (2003), no.~1, 61--168.

\bibitem{Jaco-layered-2006}
William {Jaco} and J.~Hyam {Rubinstein}, \emph{{Layered-triangulations of
  3-manifolds}}, ArXiv Mathematics e-prints (2006), \mbox{math/0603601}.

\bibitem{Kabaya-pre-2007}
Yuichi Kabaya, \emph{Pre-{B}loch invariants of 3-manifolds with boundary},
  Topology Appl. \textbf{154} (2007), no.~14, 2656--2671.

\bibitem{Kashaev-matrix-2011}
R.~M. Kashaev, \emph{{$R$}-matrix knot invariants and triangulations},
  Interactions between hyperbolic geometry, quantum topology and number theory,
  Contemp. Math., vol. 541, Amer. Math. Soc., Providence, RI, 2011, pp.~69--81.

\bibitem{Kneser-geschlossene-1929}
Hellmuth Kneser, \emph{Geschlossene fl\"achen in dreidimensionalen
  mannigfaltigkeiten}, Jahresber. Dtsch. Math.-Ver. \textbf{38} (1929),
  248--260.

\bibitem{Li-Heegaard-2011}
Tao Li, \emph{An algorithm to determine the {H}eegaard genus of a 3-manifold},
  Geom. Topol. \textbf{15} (2011), no.~2, 1029--1106.

\bibitem{lickorish_simplicial_moves}
W~B~Raymond Lickorish, \emph{Simplicial moves on complexes and manifolds},
  Proceedings of the Kirbyfest, Geometry \& Topology Monographs, vol.~2, 1999,
  pp.~299--320.

\bibitem{Ludwig-elementary-2006}
Monika Ludwig and Matthias Reitzner, \emph{Elementary moves on triangulations},
  Discrete Comput. Geom. \textbf{35} (2006), no.~4, 527--536.

\bibitem{Manolescu-triangulations-2014}
Ciprian Manolescu, \emph{Triangulations of manifolds}, ICCM Not. \textbf{2}
  (2014), no.~2, 21--23.

\bibitem{Manolescu-lectures-2016}
\bysame, \emph{Lectures on the triangulation conjecture}, Proceedings of the
  {G}\"okova {G}eometry-{T}opology {C}onference 2015, G\"okova
  Geometry/Topology Conference (GGT), G\"okova, 2016, pp.~1--38.

\bibitem{Manolescu-Pin-2016}
\bysame, \emph{Pin(2)-equivariant {S}eiberg-{W}itten {F}loer homology and the
  triangulation conjecture}, J. Amer. Math. Soc. \textbf{29} (2016), no.~1,
  147--176.

\bibitem{Matveev-transformations-1987}
Sergei~V. Matveev, \emph{Transformations of special spines, and the {Z}eeman
  conjecture}, Izv. Akad. Nauk SSSR Ser. Mat. \textbf{51} (1987), no.~5,
  1104--1116, 1119.

\bibitem{matveev_book}
\bysame, \emph{Algorithmic topology and classification of 3-manifolds}, second
  ed., Springer, 2007.

\bibitem{Moise-affine-1952}
Edwin~E. Moise, \emph{Affine structures in {$3$}-manifolds. {V}. {T}he
  triangulation theorem and {H}auptvermutung}, Ann. of Math. (2) \textbf{56}
  (1952), 96--114.

\bibitem{Mosher-tiling-1988}
Lee Mosher, \emph{Tiling the projective foliation space of a punctured
  surface}, Trans. Amer. Math. Soc. \textbf{306} (1988), no.~1, 1--70.

\bibitem{Murakami-current-2013}
Hitoshi Murakami, \emph{Current status of the volume conjecture}, Sugaku
  Expositions \textbf{26} (2013), no.~2, 181--203.

\bibitem{Neumann-realizing-2011}
Walter~D. Neumann, \emph{Realizing arithmetic invariants of hyperbolic
  3-manifolds}, Interactions between hyperbolic geometry, quantum topology and
  number theory, Contemp. Math., vol. 541, Amer. Math. Soc., Providence, RI,
  2011, pp.~233--246.

\bibitem{Neumann-arithmetic-1992}
Walter~D. Neumann and Alan~W. Reid, \emph{Arithmetic of hyperbolic manifolds},
  Topology '90 ({C}olumbus, {OH}, 1990), Ohio State Univ. Math. Res. Inst.
  Publ., vol.~1, de Gruyter, Berlin, 1992, pp.~273--310.

\bibitem{Neumann-bloch-1999}
Walter~D. Neumann and Jun Yang, \emph{Bloch invariants of hyperbolic
  {$3$}-manifolds}, Duke Math. J. \textbf{96} (1999), no.~1, 29--59.

\bibitem{Neumann-volumes-1985}
Walter~D. Neumann and Don Zagier, \emph{Volumes of hyperbolic three-manifolds},
  Topology \textbf{24} (1985), no.~3, 307--332.

\bibitem{Newman-foundation-1926}
M.~H.~A. Newman, \emph{On the foundations of combinatorial analysis situs},
  Proc. Royal Acad. Amsterdam \textbf{29} (1926), 610--641.

\bibitem{Pachner-bistellare-1978}
Udo Pachner, \emph{Bistellare \"aquivalenz kombinatorischer
  {M}annigfaltigkeiten}, Arch. Math. (Basel) \textbf{30} (1978), no.~1, 89--98.

\bibitem{Penner-decorated-1987}
R.~C. Penner, \emph{The decorated {T}eichm\"{u}ller space of punctured
  surfaces}, Comm. Math. Phys. \textbf{113} (1987), no.~2, 299--339.

\bibitem{Penner-decorated-2012}
Robert~C. Penner, \emph{Decorated {T}eichm\"{u}ller theory}, QGM Master Class
  Series, European Mathematical Society (EMS), Z\"{u}rich, 2012, With a
  foreword by Yuri I. Manin.

\bibitem{Piergallini-standard-1988}
Riccardo Piergallini, \emph{Standard moves for standard polyhedra and spines},
  Rend. Circ. Mat. Palermo (2) Suppl. (1988), no.~18, 391--414, Third National
  Conference on Topology (Italian) (Trieste, 1986).

\bibitem{Rubinstein-algorithm-1995}
Joachim~H. Rubinstein, \emph{An algorithm to recognize the {$3$}-sphere},
  Proceedings of the {I}nternational {C}ongress of {M}athematicians, {V}ol.\ 1,
  2 ({Z}\"urich, 1994), Birkh\"auser, Basel, 1995, pp.~601--611.

\bibitem{Schleimer-NP-2011}
Saul Schleimer, \emph{Sphere recognition lies in {NP}}, Low-dimensional and
  symplectic topology, Proc. Sympos. Pure Math., vol.~82, Amer. Math. Soc.,
  Providence, RI, 2011, pp.~183--213.

\bibitem{Schubert-bestimmung-1961}
Horst Schubert, \emph{Bestimmung der {P}rimfaktorzerlegung von {V}erkettungen},
  Math. Z. \textbf{76} (1961), 116--148.

\bibitem{Springborn-hyperbolic-2017}
Boris Springborn, \emph{The hyperbolic geometry of {M}arkov's theorem on
  {D}iophantine approximation and quadratic forms}, Enseign. Math. \textbf{63}
  (2017), no.~3-4, 333--373.

\bibitem{Thompson-thin-1994}
Abigail Thompson, \emph{Thin position and the recognition problem for {$S^3$}},
  Math. Res. Lett. \textbf{1} (1994), no.~5, 613--630.

\bibitem{Thurston-three-1997}
William~P. Thurston, \emph{Three-dimensional geometry and topology. {V}ol. 1},
  Princeton Mathematical Series, vol.~35, Princeton University Press,
  Princeton, NJ, 1997, Edited by Silvio Levy.

\bibitem{Tillmann-algorithms-2016}
Stephan Tillmann and Sampson Wong, \emph{An algorithm for the {E}uclidean cell
  decomposition of a cusped strictly convex projective surface}, J. Comput.
  Geom. \textbf{7} (2016), no.~1, 237--255.

\bibitem{Turaev-quantum-1994}
V.~G. Turaev, \emph{Quantum invariants of knots and 3-manifolds}, De Gruyter
  Studies in Mathematics, vol.~18, Walter de Gruyter \& Co., Berlin, 1994.

\bibitem{Turaev-state-1992}
V.~G. Turaev and O.~Ya. Viro, \emph{State sum invariants of {$3$}-manifolds and
  quantum {$6j$}-symbols}, Topology \textbf{31} (1992), no.~4, 865--902.

\end{thebibliography}


\end{document}